\documentclass[12pt,a4paper,twoside]{amsart}
\usepackage{amssymb}
\usepackage{mathtools,xspace} 

\usepackage{titletoc}
\usepackage[shortlabels]{enumitem}

\usepackage{mathrsfs,autobreak}

\usepackage{nicefrac}

\usepackage[usenames,dvipsnames]{xcolor}
\definecolor{darkblue}{rgb}{0.0, 0.0, 0.55}
\definecolor{bordeaux}{rgb}{0.34, 0.01, 0.1}

\usepackage[pagebackref,colorlinks,linkcolor=bordeaux,citecolor=darkblue,urlcolor=black,hypertexnames=true]{hyperref}
\usepackage[baseline]{euflag}
\usepackage{tikz}
\usetikzlibrary{arrows.meta}
\usepackage{wrapfig}

\usepackage{bm}

\usepackage{ulem}

\allowdisplaybreaks

\usepackage[
top=3cm,
bottom=3cm,
inner=2.5cm,     
outer=2.5cm,
footskip=40pt    
]{geometry}

\renewcommand{\subset}{\subseteq}

\DeclareFontFamily{U}{mathx}{}
\DeclareFontShape{U}{mathx}{m}{n}{<-> mathx10}{}
\DeclareSymbolFont{mathx}{U}{mathx}{m}{n}
\DeclareMathAccent{\widehat}{0}{mathx}{"70}
\DeclareMathAccent{\widecheck}{0}{mathx}{"71}

\makeindex
\makeatletter
\newcommand{\mycontentsbox}{%
	{
		\parskip=-1.0pt
		\newpage\printindex\tableofcontents\addcontentsline{toc}{section}{Table of Contents}}}
\def\enddoc@text{\ifx\@empty\@translators \else\@settranslators\fi
	\ifx\@empty\addresses \else\@setaddresses\fi
	\newpage\mycontentsbox
}
\makeatother

\contentsmargin{2.55em} 

\numberwithin{equation}{section}

\makeatletter
\newsavebox\myboxA
\newsavebox\myboxB
\newlength\mylenA

\newcommand*\xoverline[2][0.75]{%
	\sbox{\myboxA}{$\m@th#2$}%
	\setbox\myboxB\null
	\ht\myboxB=\ht\myboxA%
	\dp\myboxB=\dp\myboxA%
	\wd\myboxB=#1\wd\myboxA
	\sbox\myboxB{$\m@th\overline{\copy\myboxB}$}
	\setlength\mylenA{\the\wd\myboxA}
	\addtolength\mylenA{-\the\wd\myboxB}%
	\ifdim\wd\myboxB<\wd\myboxA%
	\rlap{\hskip 0.5\mylenA\usebox\myboxB}{\usebox\myboxA}%
	\else
	\hskip -0.5\mylenA\rlap{\usebox\myboxA}{\hskip 0.5\mylenA\usebox\myboxB}%
	\fi}
\makeatother

\makeatletter
\theoremstyle{definition}
\newtheorem*{rep@def}{\rep@title}
\newcommand{\newrepdefinition}[2]{%
	\newenvironment{rep#1}[1]{%
		\def\rep@title{#2 \ref{##1}}%
		\begin{rep@def}}%
		{\end{rep@def}}}
\makeatother
\newrepdefinition{def}{Definition}
\makeatletter
\theoremstyle{plain}
\newtheorem*{rep@thm}{\rep@title}
\newcommand{\newreptheorem}[2]{%
	\newenvironment{rep#1}[1]{%
		\def\rep@title{#2 \ref{##1}}%
		\begin{rep@thm}}%
		{\end{rep@thm}}}
\makeatother

\newreptheorem{thm}{Theorem}
\newreptheorem{prop}{Proposition}
\newreptheorem{cor}{Corollary}

\theoremstyle{plain}

\newtheorem{theorem}            {Theorem}[section]
\newtheorem{corollary}          [theorem]{Corollary}
\newtheorem{proposition}        [theorem]{Proposition}
\newtheorem{lemma}              [theorem]{Lemma}

\theoremstyle{definition}

\newtheorem{definition}         [theorem]{Definition}
\newtheorem{remark}         [theorem]{Remark}

\newtheorem{example}            [theorem]{Example}

\newcommand{\Gat}{\Gamma}
\newcommand{\gat}{\gamma}
\newcommand{\PhiG}{\Gat} 

\newcommand{\gv}{{\tt{g}}}
\newcommand{\hv}{{\tt{h}}}
\newcommand{\rv}{{\tt{r}}}

\newcommand{\uapprox}{\underset{\raisebox{0.5ex}{\scriptsize $u$}}{\approx}}

\def\cA{ {\mathcal A} }
\def\cB{ {\mathcal B} }
\def\cC{ {\mathcal C} }

\def\cD{ {\mathcal D} }
\def\cE{ {\mathcal E} }

\def\cH{ {\mathcal H} }

\def\cK{ {\mathcal K} }

\def\cN{ {\mathcal N} }

\def\cR{{ \mathcal R }}
\def\cS{{\mathcal S} }

\def\cV{{\mathcal V}}

\def\cY{\mathcal Y}
\def\cZ{ {\mathcal Z} }

\def\gtupn{\mathbb{S}_n^\gv}

\def\bbS{{\mathbb S}}
\def\gtup{\mathbb{S}^\gv}

\def\C{\mathbb C}

\def\R{\mathbb R}

\newcommand{\opco}{\operatorname{opco}}

\newcommand{\ovopco}{\overline{\opco}}
\newcommand{\matco}{\operatorname{matco}}

\def\range{\operatorname{range}}

\mathchardef\mhyphen="2D

\newcommand{\CB}{\textcolor{black}}

\newcommand{\CCBB}{\color{black}}

\def\ax{\langle x \rangle}

\newcommand{\py}{\tilde{z}}
\newcommand{\PGat}{\Phi_\Gat}

\newcommand{\rank}{\operatorname{rank}}

\newcommand{\MM}{\mathbb{M}}
\newcommand{\Mm}{M}
\newcommand{\Ii}{I}

\newcommand{\Coup}{\mathcal{C}_\Gamma}
\newcommand{\N}{\mathbb{N}}

\newcommand{\spann}{\operatorname{span}}

\newcommand{\df}[1]{{\bf{#1}}{\index{#1}}}

\newcommand{\Hilby}{Hilbertian\xspace}
\newcommand{\agler}{Agler\xspace}
\newcommand{\bdcpt}{closed and bounded\xspace}
\newcommand{\rA}{{\operatorname{A}}}

\newcommand{\vs}{\mathscr{V}}

\newcommand{\UCP}{\operatorname{UCP}}  
\newcommand{\siz}{\operatorname{size}}
\newcommand{\mconv}{\operatorname{matco}}
\newcommand{\conv}{\operatorname{conv}}
\newcommand{\LG}{L^\Gat}
\newcommand{\cHi}{\cH} 
\newcommand{\sa}{\operatorname{sa}}
\newcommand{\SOT}{\operatorname{sot}}
\DeclareMathSymbol{\mhyphen}{\mathord}{AMSa}{"39}
\newcommand{\ext}{\operatorname{ext}}
\newcommand{\wstar}{weak$^*$\xspace}

\newcommand{\sR}{\mathscr{R}}
\newcommand{\sC}{\mathscr{C}}
\newcommand{\Cb}{\operatorname{CB}}
\newcommand{\fin}{\operatorname{mat}}
\newcommand{\mat}{\operatorname{mat}}
\newcommand{\BG}{\mathscr{B}_\Gat}
\newcommand{\proj}{\operatorname{proj}}
\newcommand{\Gatconv}{\Gat\mhyphen\conv}

\newcommand{\wvphi}[1]{\varphi(#1)}

\DeclareMathOperator{\QM}{QM}
\newcommand{\nuv}{\nu}

\DeclareMathOperator{\col}{col}

\DeclareMathOperator{\rmSOT}{SOT}
\DeclareMathOperator{\rmWOT}{WOT}

\title[$\Gat$-convex sets]{Duality, extreme points and hulls for \\[1mm] noncommutative partial convexity}

\author[I.\ Klep]{Igor Klep${}^{1,Q}$}
\address{Igor Klep, Department of Mathematics, University of Ljubljana 
\& Famnit, University of Primorska, Koper 
\& Institute of Mathematics, Physics and Mechanics,
Ljubljana, Slovenia}
\email{igor.klep@fmf.uni-lj.si}
\thanks{${}^1$Supported by the Slovenian Research Agency 
	program P1-0222 and grants J1-50002,
	J1-2453, N1-0217,
J1-60011, J1-50001, and J1-3004.}

\author[S. McCullough]{Scott McCullough}
\address{Scott McCullough, Department of Mathematics\\
	University of Florida\\ Gainesville 
}
\email{sam@math.ufl.edu}
\thanks{}

\author[T.\ \v Strekelj]{Tea \v Strekelj${}^{2}$}
\address{Tea \v Strekelj, Famnit, University of Primorska, Koper, Slovenia}
\email{tea.strekelj@famnit.upr.si}
\thanks{${}^2$Supported by the Slovenian Research Agency 
	grant J1-60011.}

\thanks{${}^Q$This work was performed within the project COMPUTE, funded within the QuantERA II Programme that has received funding from the EU's H2020 research and innovation programme under the GA No 101017733 {\normalsize\euflag}}



\numberwithin{equation}{section}

\makeindex

\begin{document}
	
	\setcounter{tocdepth}{3}
	\contentsmargin{2.55em} 
	\dottedcontents{section}[3.8em]{}{2.3em}{.4pc} 
	\dottedcontents{subsection}[6.1em]{}{3.2em}{.4pc}
	\dottedcontents{subsubsection}[8.4em]{}{4.1em}{.4pc}
	
	\begin{abstract}
This article studies  generalizations of (matrix) convexity,  including partial convexity and biconvexity, 
under the umbrella of $\Gat$-convexity.
Here $\Gat$ is a tuple of free symmetric polynomials determining the geometry of a $\Gat$-convex set. 
The paper introduces the notions of $\Gat$-operator systems and $\Gat$-ucp maps and establishes	a Webster-Winkler type categorical duality between $\Gat$-operator systems and $\Gat$-convex sets. 
Next, a notion of an extreme point for $\Gat$-convex sets is defined, paralleling the concept of a free extreme point for a matrix convex set. To ensure the existence of such points,  the matricial sets considered are extended to include an operator level. It is shown that the  $\Gat$-extreme points of an operator $\Gat$-convex set $\bm{K}$ are in  correspondence
with the free extreme points of the operator convex hull of $\Gat(\bm{K}).$ From this result, a Krein-Milman  theorem for $\Gat$-convex sets follows.
Finally, relying on the results of Helton and the first two authors, 
a  construction of an approximation scheme for the $\Gat$-convex hull of the matricial positivity domain, 
{also known as a free semialgebraic set,}
$\cD_p$ of a free symmetric polynomial $p$ is given. The approximation consists of a decreasing family of $\Gat$-analogs of free spectrahedra, whose projections, under mild assumptions,  in the limit yield the $\Gat$-convex hull of $\cD_p.$
	\end{abstract}

	\maketitle

\newpage

	\section{Introduction} 
Convexity is a fundamental concept  that plays a pivotal role across various branches of mathematics and its applications
 including  optimization, economics, and geometric analysis.  A set is  \df{convex} if for any two points in the set the line segment joining them lies entirely within the set. This property not only simplifies mathematical models but also ensures the tractability of optimization problems by guaranteeing that local minima are global. Hence the elegant structure of convex sets and functions  allows for the development of robust analytical tools. 
{Understanding convexity is thus essential for advancing theory and practice across diverse scientific disciplines.}
Expanding upon the classical notion of convexity, noncommutative convexity \cite{Arv08,DK15,DK,DHM,DM05,HM12}, also known as matrix convexity \cite{EW,EH,EHKM,HKM16,WW,zalar}, concerns convexity in the setting of matrix spaces. 
Noncommutative convexity 
is intimately connected with the study of operator systems and spaces \cite{Pa,HKM13,HKM17}, and
has applications in  quantum physics \cite{Effros, AC}, and  control theory and systems engineering \cite{SIG}.

	In this article we study generalizations of (noncommutative) convexity. An example is given by partially convex sets, i.e., the sets that are convex in some of the coordinates with the others held fixed \cite{HHLM}. The generalized notions of convexity we consider are summarized under the term $\Gat$-convexity \cite{JKMMP, JKMMP22}, where $\Gat$ is a tuple of free noncommutative polynomials. 
 
	The choice of the tuple $\Gat$ determines the geometry of a $\Gat$-convex set. 
	For example, by choosing $\Gat= (x,y,xy+yx, i(xy-yx))$ one obtains the class of $xy$-convex sets.
	Such sets arise in the study of
bilinear matrix inequalities (BMIs), {which appear in control theory, optimization, and various applied mathematics fields \cite{KSVS04}.} These are 
finite sums  of the form
\begin{equation}
	\label{e:BMI}
	A_0 + \sum_j A_j x_j + \sum_k B_k y_k + \sum_{p,q} C_{pq} x_p y_q \succeq0
\end{equation}
for self-adjoint matrices $A_j,B_k,C_{pq}$ \cite{VAB}.

Another commonly studied type of $\Gat$-convex sets is defined by the tuple $\Gat=(x,y,y^2).$ In analogy with matrix convex sets and BMIs one considers sets of 
$(x,y)=(x_1,\dots,x_\gv,y)$ describable in the form
\begin{align}\label{eq:pconv}
A_0+\sum A_jx_j + By + C y^2\succeq0.
\end{align}
Such sets are convex in the variables $x$ whenever $y$ is held fixed.

An important area where matrix inequalities such as \eqref{e:BMI} and \eqref{eq:pconv} arise is engineering systems problems governed by a signal flow diagram. 
Such problems naturally give rise to two classes of variables. The system variables depend on the choice of system parameters and produce polynomial inequalities in the state variables. The algebraic form of these inequalities involves noncommutative polynomials and depends only upon the flow diagram, and not the particular choice of system variables. See  \cite{dOH, dOHMP} and the citations therein.  Convexity (or linearity, as in Linear Matrix Inequalities) in the state variables, for a given choice of system variables, is then an important optimization consideration.

{\CCBB

The map $\Gat$ may also be viewed as a noncommutative analog of a
feature map: one adjoins nonlinear polynomial functions of the variables and
then applies linear convex-analytic tools in the lifted coordinates. This
perspective also suggests possible connections with quantum information theory
and quantum machine learning, though we do not pursue such applications here.
}

Broadly speaking, our results are of three types. We establish a $\Gat$-analog of the Webster-Winkler \cite{WW} duality between matrix
convex sets and operator systems; introduce extreme points and
establish a Krein-Milman theorem for operator $\Gat$-convex sets; and construct a version of the
Lasserre-Parrilo spectrahedral lift \cite{Las09, Par06} of the $\Gat$-convex hull of a semialgebraic set.  The remainder of this introduction is organized as follows:
Subsections~\ref{subsec polys} and \ref{ssec:prelimGat} contain basic notation and terminology surrounding matrix and 
$\Gat$-convex sets;
the Hahn-Banach
separation theorems in each case are reviewed in  Subsection~\ref{sec:HB}. With this background, we then preview the main results of
the paper in Subsection~\ref{sec:main}.

\subsection{Free polynomials and their evaluations}
  \label{subsec polys}

For $n,m \in \N$ and a vector space $\vs,$ let  \df{$\Mm_{n,m}(\vs)$}  denote the space of $n \times m$ matrices  with entries from $\vs.$ In the case $\vs=\mathbb{C},$ we use the abbreviation $\Mm_n = \Mm_{n,n},$ \index{$\Mm_n$} and denote by $\Ii_n \in \Mm_n$  \index{$\Ii_n$} the identity matrix. When $\vs = \mathbb{C}^\gv$ for some $\gv \in \N,$ the space $\Mm_n(\C^\gv)$ is canonically identified with $\Mm_n^\gv,$ the set of $\gv$-tuples of $n \times n$ complex matrices. Denote by $\mathbb{S}_n \subseteq \Mm_n$ the set of self-adjoint $n \times n$ complex matrices, and let  $\mathbb{S}^\gv$ denote \index{$\mathbb{S}_n$} \index{$\mathbb{S}^\gv$}
 the graded set  $(\mathbb{S}_n^\gv)_n.$ Likewise, let  $\MM(\vs)=( \Mm_n(\vs))_n$ and $\MM^\gv=\MM(\C^\gv).$  \index{$\MM(\vs)$} 
\index{$\MM$}

Let $x=(x_1,\dots,x_\gv)$ denote a tuple of $\gv$ (freely) noncommuting variables and
let $\ax$ denote the semigroup of  words in the variables $x_1,\dots,x_\gv.$ We use \index{word}
$1$ to denote the empty word $\emptyset.$ Now denote by \df{$\C\ax$} the free algebra consisting of finite $\C$-linear
combinations of the words in the variables $x_1,\dots,x_\gv.$ \index{$\ax$} \index{$\C\ax$} 
An element $p\in \C\ax$ is called a \df{noncommutative (nc) polynomial},  or synonymously a \df{free polynomial},
and it is of the form
\begin{equation}
	\label{e:defp}
	p =\sum_{w\in\ax}  p_w w,
\end{equation}
where the sum is finite and $p_w\in \C$. \index{involution ${}^*$}
A natural involution ${}^*$ on the semigroup $\ax$ is defined by $x_j^*=x_j$ for $1\le j\le \gv$
and $(wu)^* = u^* w^*$ for $u,w\in\ax$. 
The involution ${}^*$ naturally extends to $\C\ax$ by
\[
p^* = \sum_{w\in \ax}^{{\text{finite}}} \overline{p_w} w^*,
\]
where $p$ is as in \eqref{e:defp}.

Free polynomials are   \df{evaluated} at an $X\in\gtup$. For a word 
\[
w = x_{j_1}\, x_{j_2} \, \cdots \, x_{j_N} \in \ax,
\]
and $X\in \gtupn$, 
\[
w(X) = X_{j_1}\, X_{j_2}\, \cdots\, X_{j_N} \in \Mm_n;
\]
and  for $p$ as in  \eqref{e:defp}, \index{$p(X)$}
\[
p(X) =\sum_{w\in\ax}^{{\text{finite}}} p_w w(X) \in \Mm_n.  
\]
Thus $p$ determines a (graded)  function $p:\gtup \to \mathbb{M}$,
where $\mathbb{M}=(\Mm_n)_n,$ and similarly,
an $\rv$-tuple $p=(p_1,\dots,p_\rv)\in \C\ax^{1\times \rv} = M_{1,\rv}(\C\ax)$ determines
a mapping $p:\gtup\to \mathbb{M}^\rv$.  \index{$M(\C)$}  \index{mapping}

A polynomial that is invariant under the involution ${}^*$ is called \df{symmetric}. It is well-known and easy to see that $p\in \C\ax$
is symmetric if and only if  it satisfies  $p(X)^* =   p(X)$
 for all $X\in\gtup.$
 In this case $p$ determines a mapping $p:\gtup \to \mathbb S = \mathbb{S}^1$.
 Since $x_j^*=x_j,$ the variables $x$ are referred to as \df{symmetric variables}.

More generally, a polynomial of the form \eqref{e:defp} with the coefficients $p_w$ lying in  $\cB(\cH),$
 the bounded operators on a Hilbert space $\cH,$ 
  is a $\cB(\cH)$-valued noncommutative polynomial and \df{$\cB(\cH)\otimes \C\ax$} is the space of $\cB(\cH)$-valued
  noncommutative polynomials. In the case   $\cH=\C^\mu$ for some $\mu \in \N,$ we obtain the space of all $\mu\times\mu$ 
   matrix-valued noncommutative polynomials, denoted by $ M_\mu(\C\ax).$  Any polynomial
$p \in \cB(\cH)\otimes\C\ax$ is \df{evaluated} at a tuple $X\in \gtupn$ using the (Kronecker) tensor  product,
\[
p(X) = \sum p_w\otimes w(X) \in \cB(\cH) \otimes M_n,
\]
and $p$ is \df{symmetric} if  $p(X)^*=p(X)$
  for all $X\in \gtup$. Equivalently, a polynomial
$p \in \cB(\cH)\otimes \C\ax,$ or $p\in M_\mu(\C\ax)$ in the case $\cH=\C^\mu,$  is \df{symmetric} if $p_{w^*}=p_w^*$ for all words $w$.\looseness=-1

\subsection{Preliminaries on \texorpdfstring{$\Gat$}{Gamma}-convexity}\label{ssec:prelimGat}
Fix $\gv\in\N$ and 
let $\Gat=(\gat_1,\dots,\gat_\rv)$ denote a tuple  of symmetric noncommutative polynomials
with $\gat_j=x_j$ for $1\le j\le \gv\le \rv$. As in Subsection \ref{subsec polys} we also use  $\PhiG:\gtup \to \mathbb{S}^\rv$
to denote the mapping \index{$\PGat$} \index{$\Gat$}
\begin{equation*}
	\PhiG(X) = (\gat_1(X),\dots,\gat_\rv(X)).
\end{equation*}
 For instance, in the case $\gv=2$ and $\rv=4$ for $\Gat(x,y)=(x,y,xy+yx, i(xy-yx))$  and $(X,Y)\in M_n^2,$ 
\[
 \PhiG(X,Y) = (X,Y,XY+YX, i(XY-YX))\in M_n^4.
\]

\begin{definition}\label{def gamma} Let $\bm{K}= (K_n)_n \subseteq \gtup,$ where
	$K_n \subseteq \gtupn$ for each positive integer $n,$ be a given graded set.
\begin{enumerate}
\item  The graded set $\bm K$ is \df{(uniformly) bounded} if there is an $R\in\R_{\geq0}$ such that 
\[
   \sum_{j=1}^\gv X_j^2 \preceq R^2
\]
 for all $n$ and $X\in K_n.$  Equivalently, with the \df{norm} \df{$\|X\|$} of $X\in  K_n$ defined as the norm of
 the $n\times n\, \gv$ matrix 
 \[
  \begin{pmatrix} X_1 & X_2 & \dots & X_{\gv}\end{pmatrix},
 \]
  the set $\bm K$ is bounded if there is an $R\ge 0$ such that $\|X\|\le R$ for all $X\in \bm K.$
 
 \item The set  $\bm K$ is a 
	\df{free set} if it is closed under direct sums and
	simultaneous  unitary conjugations.\footnote{
		This means that if  $X\in K_n$ and $Y\in K_m$, then $X\oplus Y\in K_{n+m};$
		if $U \in \Mm_n$ is unitary, then $U^*XU=(U^*X_1 U,\dots,U^*X_\gv U)\in K_n.$}
		
 \item  A pair $(X,V)$, where $X\in \gtupn$
	and $V:\C^m\to \C^n$ is an isometry, is a \df{$\Gat$-pair} if it satisfies
	\[
	V^* \PhiG(X)V= \PhiG(V^*XV).
	\]
	Let \df{$\Coup$} denote the collection of all $\PhiG$-pairs, \index{$\Coup$} 
	 parameterized over all choices of positive integers $n,m.$

 \item  The set  $\bm{K} \subseteq \gtup$ is a \df{$\Gat$-convex set} if it is free and if 
	\[
	X\in \bm{K} \, \text{ and } \, (X,V)\in \Coup \, \implies V^*XV\in \bm{K}.
	\]
	
 \item The \df{$\Gat$-convex hull} of a free set $\bm{K}$ is the smallest $\Gat$-convex set that contains $\bm{K}.$ 
 It is obtained as the intersection\footnote{\CCBB The universe $\MM^\gv$ is $\Gamma$-convex.} of all $\Gat$-convex sets containing $\bm{K}$ and is
	denoted  by $\Gat\mhyphen\conv (\bm{K}).$ \qed
\end{enumerate}
\end{definition}

\begin{remark} 
\mbox{}\par
	(a) Note that for every $n\times n$ unitary matrix $U$
	and  $X\in \gtupn$, the tuple $(X,U)$ is a $\Gat$-pair.
	
\smallskip

	(b) If an isometry $V$ reduces a tuple $X,$ then $(X,V)$ is a $\Gat$-pair. Indeed, writing the decomposition of $X$ with respect to the range of $V$ as $X=Y\oplus Z,$ 
	$$V^*\Gat(X)V = V^* (\Gat(Y)\oplus\Gat(Z))V = \Gat(Y) = \Gat(V^*(Y\oplus Z)V) = \Gat(V^*XV).$$

\smallskip

	(c) In the case when $\rv=\gv,$ and thus  $\Gat(x)=x,$
	 the notion of  $\Gat$-convexity reduces to  \df{matrix convexity}. (See Subsection~\ref{s:finite-dim-case}.)
	 \qed
\end{remark}
	 
\begin{definition}
Given a positive integer $k,$ tuples  $X^{(i)} \in K_{n_i}$ and $V_i \in \Mm_{n_i,n}$ for $1\le i\le k$
  such that $\sum_{i=1}^k V_i^* V_i=\Ii_n,$ the tuple 
$$
\sum_{i=1}^k V_i^\ast X^{(i)} V_i \in M_n^\gv
$$
is a \df{matrix convex combination} often abbreviated by setting 
\begin{equation}\label{eq:preGatpair}
X = \oplus_{i=1}^k X^{(i)} \quad
\text{ and }
\quad
V = \text{col}(V_1,\ldots,V_k) =
\begin{pmatrix}
	V_1 \\
	\vdots \\
	V_k	
\end{pmatrix},
\end{equation}
so that 
$$
V^\ast X V = \sum_{i=1}^k V_i^\ast X^{(i)} V_i.
$$
In particular, the condition  $\sum_{i=1}^k V_i^*V_i=\Ii_n$ is equivalent to $V$ being an isometry.

    A matrix convex combination of the form
	\begin{equation*}
		\sum_{i=1}^k V_i^\ast X^{(i)} V_i,
	\end{equation*}
	is a \df{$\Gamma$-convex combination} if  $(X, V)$ 
of \eqref{eq:preGatpair}
	is a $\Gamma$-pair. \qed
\end{definition}

The next two propositions \cite[Proposition 2.1, Proposition 2.2]{JKMMP} give alternative descriptions of the $\Gat$-convex hull of a free set. We include the easy proofs for self-containment.

\begin{proposition} \label{prop:gammahull}
	A free set $\bm{K} \subseteq \mathbb{S}^{\texttt{g}}$  is a $\Gamma$-convex set if and only if it is closed under  $\Gamma$-convex combinations. Moreover, 
	\[
	\Gat\mhyphen\conv(\bm{K}) =\{V^*XV: X\in \bm{K}, \, (X,V)\in\Coup \}.
	\] 
\end{proposition}

\begin{proof}   The first statement is clear by the definition of a $\Gat$-convex set.
	For the second claim   it is easy to verify that
	$\mathscr{K} =\{V^*XV: X\in \bm{K}, \, (X,V)\in\Coup\}$ is a $\Gat$-convex set that contains $\bm{K}$. Since by definition, 
	$\Gat\mhyphen\conv (\bm{K})$ must contain $\mathscr{K},$ we have that $\mathscr{K} = \Gat\mhyphen\conv (\bm{K}).$ 
\end{proof}

In the case $r=\gv$ so that  $\Gat=x;$ that is, we are dealing with matrix convex sets, the $\Gat$-convex hull is the 
\df{matrix convex hull} of $\bm K$ and we use the standard notation \df{$\mconv(\bm K)$}
instead of $\Gat\mhyphen\conv (\bm{K})$.

\begin{proposition}
	\label{prop:jp}
	Suppose $\bm{K} \subset \gtup$ is a free set and $X\in \gtup$.  The point 
	$X$ is in $\Gat\mhyphen\conv(\bm{K})$ if and only if 
	$\PhiG(X)$ is in $\mconv(\PhiG(\bm{K})).$ Equivalently,
	\[
	\PhiG^{-1}(\mconv(\PhiG(\bm{K})))
	= \Gat\mhyphen\conv(\bm{K}).
	\]
\end{proposition}

\begin{proof}
	If $X\in \Gat\mhyphen\conv (\bm{K}),$ then by Proposition \ref{prop:gammahull},  there 
	is a $Y$ in $\bm{K}$ and an isometry $V$ such that 
	$V^*\Gat(Y)V= \Gat(V^*YV)$ and $X=V^*YV$. Thus,  $\PhiG(X)=V^*\PhiG(Y)V,$ which implies
	$\PhiG(X) \in  \mconv(\PhiG(\bm{K})).$
	
	For the converse assume $\PhiG(X)\in \mconv(\PhiG(\bm{K})).$ Thus there is a $Y\in \bm{K}$  and 
	an isometry $V$ such that $\PhiG(X)  = V^*\PhiG(Y)V.$  
	A comparison of the first $\gv$  coordinates gives $X=V^*YV,$ from which we deduce that $(Y,V)$ is a $\Gat$-pair.
	Since $Y\in \bm{K}$ and $(Y,V)$ is a $\Gat$-pair,  
	Proposition \ref{prop:gammahull} implies  $X\in \Gat\mhyphen\conv (\bm{K})$.
\end{proof}

\begin{example}
	\label{eg:TV}  
	In the case of two variables $(x,y)=(x_1,x_2)$
	 {\CCBB by \df{$y^2$-convex} we mean $\Gat$-convex for $\Gat=(x,y,y^2)$.} It was shown in \cite[$\S$4]{JKMMP}, in this case, 
 $((X,Y),V)$ is a $\Gat$-pair if and only if
	the isometry $V$ reduces $Y$ and that a free set $\bm{K}$ is $y^2$-convex
	if and only if it is convex in $x$ for any fixed $y$. That is,
	 $(X_1,Y),(X_2,Y)\in K_n$ implies $(\frac{X_1+X_2}{2},Y)\in K_n$.
	 An example of a $y^2$-convex set is 
	 $$
	  \bm K = \{ (X,Y) \ | \ -(Y^2+I) \preceq X \preceq Y^2 + I, \ Y^2 \preceq I\},
	 $$
	 which we consider later in Example \ref{eg:parabola}. \CB{(See also Example \ref{e:not:vac} for a different example of a $y^2$-convex set)}.\qed
\end{example}

\begin{example}
	  In the case $\Gat(x,y)=(x,y,xy+yx,i(xy-yx)),$
	  a free set that is  $\Gat$-convex is known as an \df{$xy$-convex set}. An $xy$-convex set is convex in both $x$ and $y$ separately.
	 See \cite{HHLM,BM,DHM,JKMMP,BHM+} for more results on this topic.\qed
\end{example}

{\CCBB
\begin{example}\label{ex:x^2y^2}
Let $\Gamma(x,y)=(x,y,x^2,y^2)$.  We claim that a free set
$\bm K\subseteq \mathbb{S}^2$ is $\Gamma$-convex if and only if it is closed under
restrictions to common reducing subspaces.

Indeed, let $(X,Y)\in \mathbb S^2_n$ and let $V:\C^m\to\C^n$
be an isometry.  Then $((X,Y),V)$ is a $\Gamma$-pair if and only if
\[
        V^*X^2V=(V^*XV)^2
        \qquad\text{and}\qquad
        V^*Y^2V=(V^*YV)^2 .
\]
As in Example \ref{eg:TV}, this is equivalent to $\range(V)$ reducing both $X$ and $Y$.

Consequently, a free set $\bm K\subseteq \mathbb S^2$ is
$(x,y,x^2,y^2)$-convex precisely when
\[
        (X,Y)\in \bm K
        \quad\iff\quad
        (X|_{\mathcal M},Y|_{\mathcal M})\in \bm K
\]
for every common reducing subspace $\mathcal M$ for $X$ and $Y$.
\end{example}

\begin{example}\rm
\label{ex:matco-all}
  At the opposite extreme of Example~\ref{ex:x^2y^2}, if $\bm K$ is matrix convex, then it is $\Gat$-convex for every $\Gamma.$
\end{example}
}

\subsection{A Hahn-Banach separation theorem for 
\texorpdfstring{$\Gat$}{Gamma}-convex sets}
\label{sec:HB}
This section presents  a Hahn-Banach separation theorem of an outlier from a closed $\Gat$-convex set, where the separation is given by a $\Gat$-analog of a linear pencil. We continue to let 
 $\Gat(x) = (\gamma_1, \ldots, \gamma_\rv),$ where the $\gamma_i$ are symmetric and $\gamma_i=x_i$  for $i=1, \ldots, \gv.$

For a (complex) Hilbert space $\cH$, let \df{$\mathbb{S}_\cH$} denote the self-adjoint (hermitian) operators on $\cH$ and
 let \df{$\mathbb{S}_\cH^\gv$} denote the set of $\gv$-tuples of elements of $\mathbb{S}_\cH.$
 
\begin{definition}
	A \df{$\Gamma$-pencil} with coefficients $A=(A_0, A_1,\ldots, A_\rv) \in \mathbb{S}^{\rv+1}_\cH$ is the $\cB(\cH)$-valued affine linear  polynomial
	\begin{equation*} 
	\LG(x)=\LG_A(x)  = A_0 + \sum_{i=1}^\gv A_i\,x_i + \sum_{i=\gv+1}^\rv A_i \gamma_i(x).  \index{$\LG$}
	\end{equation*}
	 For a tuple $X\in \Mm_n^\gv,$
	\[
	 \LG(X) = A_0\otimes \Ii_n + \sum_{i=1}^\gv A_i\otimes X_i + \sum_{i=\gv+1}^\rv A_i\otimes \gamma_i(X) \in \cB(\cH)\otimes M_n 
	   \cong \cB(H\otimes  \C^n)\cong \cB(H^n).
	\]
	The  matricial \df{positivity domain} of $\LG,$ denoted by $\mathcal{D}^\Gat_A,$  is the graded set
	\[
	\mathcal{D}^\Gat_A = 
	   \left (\mathcal{D}^\Gat_A(n) \right )_n= ( \{X \in \mathbb{S}_n^\texttt{g} \ | \ \LG(X) \succeq 0\} )_n
	\]
	and is known  as a \df{$\Gat$-spectrahedron} when $\cH$ 
	 is finite-dimensional and a  \df{\Hilby $\Gat$-spectrahedron}  in case $\cH$ is potentially infinite-dimensional.
		A $\Gat$-pencil is called \df{monic} if $A_0 = I.$
	\qed
\end{definition}

In the case that $\gv=\rv$ (and $\Gat(x)=x$) the expression,
\[
 L(x) =L_A(x) = A_0 + \sum_{j=1}^ \gv A_j \, x_j
\]
 is a \df{linear pencil} and the positivity domain  $\cD^\Gat_A= \cD_A$ \index{$\cD_A$}  is a \df{\Hilby free spectrahedron}. In particular, a $\Gat$-pencil $\LG$ has the form $\LG=L_A\circ \Gat,$ for a tuple $A\in \mathbb{S}_\cH^{\rv+1}.$

As examples,  a  monic $\Gat$-pencil with $\Gat= (x,y,xy+yx, i(xy-yx))$ is of the form
\[
 \LG(x,y) = I + A_x x + A_y y + Bxy + B^*yx,
\]
where $A_x,A_y$ are self-adjoint; and a  monic $\Gat$-pencil with $\Gat=(x,y,y^2)$ is of the form
\[
\LG(x,y) = I +A_x x + A_y y + By^2,
\] where $A_x, A_y$, and $B$ are self-adjoint.

We now recall a weaker form of the finite-dimensional version of the Effros-Winkler \cite{EW} Hahn-Banach separation theorem for matrix convex sets given in \cite[Proposition 6.4]{HM12} (or see, e.g., \cite[Lemma 2.1(a)]{Kr}).
Typically,  an adjective describing a free set is a \df{level-wise description}. For instance, a matrix convex set $K\subseteq \mathbb{S}^\gv$
 is \df{closed} if each $K_n$ is closed (in any locally convex topology, equivalently norm topology)  on $M_n^\gv.$ 
 An exception to this convention is boundedness, since, as it turns out, a matrix convex set $\bm K \subseteq\mathbb{S}^\gv$  is level-wise bounded if
 and only if it is bounded. See Lemma~\ref{l:level:bounded}.    If the origin $0=(0,\dots,0)\in \C^\gv$
  is contained in $K_1,$ then $0\in M_n^\gv,$  the tuple of  zero matrices, is in $K_n$ for each $n.$  A tuple $Y\in \mathbb{S}^g$ has \df{size} $\ell$ 
  means $Y\in \mathbb{S}^\gv_\ell.$

\begin{theorem}
	\label{t:EW}
	Let $\bm{K} \subset \mathbb{S}^\gv$ be a closed matrix convex set containing the origin. If $\ell\in \N$ and
	 $Y\in \mathbb{S}^\gv_\ell \setminus K_\ell,$ then there is a monic linear pencil $L$ of size $\ell$  such that $L(X)\succeq 0$ for all $X$ in $\bm{K},$
	but $L(Y)\not\succeq 0$. 
\end{theorem}

 Theorem~\ref{t:EW}  begets the following $\Gat$-analog of the Effros-Winkler matricial Hahn-Banach separation theorem.

\begin{theorem}[\protect{\cite[Theorem 1.4]{JKMMP}}]
	\label{t:thm:jp-intro} 
	Suppose $\Gamma(0)=0$ and $\bm{K}\subset \gtup$ is a  $\Gat$-convex set containing $0.$
	If the matrix
	convex hull of $\Gat(\bm{K}) \subset \mathbb{S}^\rv$ is closed and if 
	$Y\in \mathbb S_\ell^\gv \setminus K_\ell$, then there is a monic
	$\Gat$-pencil $\LG$ of size $\ell$ such that $\LG(X)$ is positive semidefinite for all $X$ in $\bm{K},$ but
	$\LG(Y)$ is not positive semidefinite.
\end{theorem}

\begin{proof}
	By Proposition \ref{prop:jp}, we have $\PhiG(Y)\notin \matco(\PhiG(\bm{K}))_\ell.$ Since $\matco(\PhiG(\bm{K}))$ is, 
	 by assumption, closed, it is 
	a closed matrix convex subset of $\mathbb{S}^\rv$ containing $0.$ Thus, 
	by Theorem \ref{t:EW},
	 there is a monic linear pencil 
	\[
	L(z) = I_\ell +\sum_{j=1}^\rv A_jz_j
	\]
	of size $\ell$  such that $L(Z)\succeq 0$ for all 
	$Z\in \matco(\PhiG(\bm{K}))$, but
	$L(\PhiG(Y))\not\succeq 0$. Now  $\LG=L\circ \PhiG$ is  a monic $\Gat$-pencil
	of size $\ell$ that is positive semidefinite on $\bm{K},$ but 
    is not positive semidefinite at $Y.$
\end{proof}

\begin{remark}\label{rem:Kprime}
	Theorem \ref{t:thm:jp-intro} remains valid if 
  $\overline{\matco}(\Gat(\bm K))$ 
	 is replaced by any {closed} matrix convex 
	set $\bm{J}$ containing   $\Gat(\bm K)$  such that
	\[\pushQED{\qed}
	\Gat\mhyphen\conv (\bm{K}) =\PhiG^{-1}(\bm{J} \cap \range(\PhiG)).\qedhere \popQED
	\]
 \end{remark}

\subsection{Main results and guide to the paper} 
\label{sec:main}
This  paper is structured around three main themes.
First, it introduces a categorical duality in the $\Gat$-convex setting, 
relating 
$\Gat$-convex sets to 
$\Gat$-operator systems via the concept of 
$\Gat$-ucp maps.
Second, it explores the concept of extreme points in $\Gat$-convex sets $\bm{K}$, including a Krein-Milman-type theorem 
showing that under natural assumptions, the $\Gat$-extreme points of $\bm{K}$ span $\bm{K}$. Finally, it presents a free analog of the Lasserre-Parrilo lifts for $\Gat$-convex sets, constructing a sequence of free $\Gat$-spectrahedra whose projections offer increasingly better approximations and clamp down on the $\Gat$-convex hull of the free semialgebraic set $\cD_p=\{X\mid p(X)\succeq0\}$.

\subsubsection{Duality in the \texorpdfstring{$\Gat$}{Gamma}-convex setting}
Section \ref{sec dual} 
begins by revisiting 
the categorical duality between matrix convex sets and operator systems \cite[Proposition 3.5]{WW}. Subsection \ref{subsec:gama-dual} then introduces concrete $\Gat$-operator systems as well as $\Gat$-ucp maps, the $\Gat$ analogs of unital completely positive (ucp) maps.

For any tuple of self-adjoint operators  $A=(A_1,\dots,A_\gv) \in \mathbb{S}_\cH^\gv,$ let $\cR=\cR^\Gamma_A$
 denote the operator system \index{$\cR^\Gat_A$}
\begin{equation}\label{eq:gamaopsys1}
	\spann\{\Ii_\cH,\gamma_j(A) \ | \ j=1,\ldots,\rv\} \subseteq \cB(\cH).
\end{equation}
In particular, the dimension of $\cR,$ as a vector subspace of $\cB(\cH),$  is at most $\rv+1$.
In Proposition \ref{prop:bded} we show that if the \Hilby spectrahedron $\cD_{\Gat(A)}\subseteq \mathbb{S}^{\rv}$ is bounded, 
then the operator system $\cR$ in \eqref{eq:gamaopsys1}
has dimension $\rv+1.$ 

\begin{repdef}{def:g-opsys}
	Suppose  $A=(A_1,\dots,A_\gv) \in \mathbb{S}_\cH^\gv$. 
	\begin{enumerate}[(a)]
		\item The operator system $\cR^\Gat_A=$ span$\{\Ii_\cH,\gamma_j(A) \ | \ j=1,\ldots,\rv\}$ is called a \df{$\Gat$-operator system}.
		\item Let $\cR = \cR^\Gat_A$ be a $\Gat$-operator system and suppose $\cK$ is a Hilbert space. A  linear map $\varphi:\cR\to \cB(\cK)$ is a \df{$\Gamma$-ucp map}
		if it is ucp, and, for all $i,$
		\[ \pushQED{\qed}
		\varphi(\gamma_i(A))=\gamma_i(\varphi(A_1),\ldots,\varphi(A_\gv)). \qedhere \popQED
		\]
	\end{enumerate}
\end{repdef}
An \df{isomorphism of $\Gat$-operator systems} is, by definition,  a bijective $\Gat$-ucp map whose inverse is also $\Gat$-ucp.
To mimic the Webster-Winkler duality, we associate to any $\Gat$-operator system $\cR$ as in \eqref{eq:gamaopsys1} the graded set $\widecheck{\cR} = W^\Gat(A) = (W_n^\Gat(A))_n$ with \index{$\widecheck{\cR}$} \index{$W_n^\Gat(A)$} 
\begin{equation*}
	W_n^\Gat(A) = \{(\varphi(A_1),\ldots,\varphi(A_\gv)) \in \mathbb{S}^\gv_n \ | \ \varphi: \cR \to \Mm_n \text{ is $\Gat$-ucp}\}.
\end{equation*}
Proposition \ref{p:GammaUCP-convex} establishes that 
$\widecheck{\cR}$ is a compact $\Gat$-convex set.
Conversely, if  $\bm{K}$ is a compact $\Gat$-convex set, we set \index{$\widehat{Y}$} 
\begin{equation}\label{eq:not26}
\widehat{Y} = \bigoplus_{Y \in \bm{K}}Y \quad \text{ and }  \quad I = \bigoplus_{Y \in \bm{K}} \Ii_{\siz(Y)}
\end{equation}
and associate to $\bm{K}$ the $\Gat$-operator system \index{$\widehat{\bm K}$}
 $\widehat{\bm{K}}= \spann \{I,\gamma_1(\widehat{Y}),\ldots,\gamma_\rv(\widehat{Y})\},$
leading  to the following  duality between operations \, $\widehat{}$ \, and \, $\widecheck{}$ \,  in the $\Gat$-convex setting.
 A tuple $A\in \mathbb{S}_\cH^\gv$ is \df{semi-finite} if it is an at most countable direct sum of matrix tuples; that is, $A=\oplus_{j\in J} A^{(j)}$, where $J\subseteq\N$,
  and for each $j\in J$ there is a positive integer $n_j$ such that $A^{(j)}\in \mathbb{S}^\gv_{n_j}.$

\begin{repthm}{th:g-dual}[\textbf{Duality for $\Gat$-convex sets}]  The above operations  \, $\widehat{}$ \,
		and \, $\widecheck{}$ \, are dual to one another: 
	\begin{enumerate}[\rm (a)]
		\item 
		Suppose $A = (A_1,\ldots,A_\gv)\in \mathbb{S}_\cH^\gv$ is semi-finite  and let $\cR$ denote the span of $\{\Ii_\cH,\gamma_1(A),\ldots,\gamma_\rv(A)\}.$  If  {\CCBB 
         $I_\cH,\gat_1(A),\dots,\gat_r(A)$ is linearly independent,}
		then $\cR$ and $\widehat{\widecheck{\cR}}$ are isomorphic $\Gat$-operator systems. 
		
		\item   Suppose $\Gamma(0)=0$ and let $\bm{K}\subset \gtup$ denote a  \bdcpt  $\Gat$-convex set with  $0\in K_1.$ 
		{If $\bm K =\Gat^{-1}(\overline{\matco}(\Gat(\bm K))),$} 
		then 
		$\bm{K} = \widecheck{\widehat{\bm{K}}} = W^\Gat(\widehat{Y}) $
		for $\widehat{Y}$ defined as in equation~\eqref{eq:not26}.
	\end{enumerate}
\end{repthm}

{\CCBB The functorial, morphism-level formulation of the duality
of Theorem~\ref{th:g-dual} is recorded in
Appendix~\ref{app:categorical-duality} as Theorem \ref{th:gamma-categorical-duality}.}

We note that while the tuple $\widehat{Y}$ is not semi-finite, it is isomorphic, as a $\Gat$-operator system to
a $\Gat$-operator  system defined by a semi-finite tuple. That is, there exists  a semi-finite tuple $\widehat{E}$ acting on
a separable Hilbert space   $\mathscr{E}$
   such that $\spann\{I, \widehat{Y}\}$ and  $\spann\{I_{\mathscr{E}}, \widehat{E}\}$ are isomorphic
   as operator systems via the unital map that sends $\widehat{Y}_j$ to $\widehat{E}_j.$ (See Remark~\ref{r:relax:sep}.)

\subsubsection{Extreme points of \texorpdfstring{$\Gat$}{Gamma}-convex sets}
Section \ref{sec: g-ext} extends the concept of {free extreme points} of matrix convex sets to the  $\Gat$-convex setting.   Ensuring the  existence of such extreme points requires modifying the  definition of a matrix (or $\Gat$-)convex set to an \textit{operator \textnormal{(}$\Gat$-\textnormal{)}convex set} $\bm{K} = (K_n)_{n \in \N \cup \{\infty\}}$ to include an infinite-dimensional level $K_\infty \subseteq \cB(\cH)_{\sa}^\gv,$ where $\cH$ is an infinite-dimensional separable Hilbert space.
An operator ($\Gat$-)convex set is thus closed under generalized matrix convex combinations with operator coefficients.
The following definition is an adaptation of the concept of an \textit{nc extreme point}  from \cite{DK}, a notion closely related to that of a free extreme point, tailored to the study of operator convex sets explored here.

\begin{repdef}{def-gext}
	Let $\bm{K} = (K_n)_{n \in \N \cup \{\infty\}}$ be an operator $\Gat$-convex set (that is not assumed closed).  A tuple $X \in K_n$ for $n \in \N \cup \{\infty\}$ is a \df{$\Gat$-extreme point} if {in} any expression
	of $X$ as an operator  $\Gat$-convex combination
	\begin{equation*}
		X = \sum_{i=1}^k V_i^\ast X^{(i)} V_i,
	\end{equation*}
	where $X^{(i)} \in K_{n_i},$ the $V_i \in {M}_{n_i,n}$ are all nonzero with $\sum_{i=1}^k V_i^\ast V_i = {I}_n$, and 
 \[
    (X,V) = (\oplus_iX^{(i)},\,\text{col}(V_1,\ldots,V_k))
\] 
        is a $\Gat$-pair,
	{it is the case} that for each $i,$ the matrix $V_i$ is a  scalar multiple of an isometry $W_i \in {M}_{n_i,n}$ such
	that $(X^{(i)},W_i)$ is a $\Gat$-pair 
	and, with respect to the range of $V_i,$ 
	$$X^{(i)} = Y^{(i)} \oplus Z^{(i)}$$
	for some $Y^{(i)},Z^{(i)} \in \bm{K}$ with $Y^{(i)}$ unitarily equivalent to $X.$
	Denote the set of $\Gat$-extreme points of $\bm{K}$ by \df{$\Gat\mhyphen\ext(\bm{K})$}.\qed
\end{repdef}

Proposition \ref{tr22} is key to establishing the existence of $\Gat$-extreme points.
 It asserts that under mild assumptions, if $\bm{K}$ is an operator $\Gat$-convex set and
$\Gat(X)$ is a free extreme point of $\opco(\Gat(\bm{K})),$ the operator convex hull of  $\Gat(\bm{K}),$ then $X$ is a  $\Gat$-extreme point of  $\bm{K}.$
As in the case of matrix convexity, the (closed) operator convex hull of a graded set $\bm{S}=(S_n)_n \subseteq \mathbb{S}^\gv$ is, by definition,  
 the smallest (closed) operator convex set that contains $\bm{S}.$
Language such as the {\it $\tau$-closed} operator convex hull indicates that the closure is taken with respect to the topology $\tau,$ typically
either the strong operator topology (SOT) or the \wstar-topology (w${}^*$). 
 We use the notation $\ovopco^{\tau}(\bm S).$ {\CCBB We use the notation $\Gat\mhyphen \ovopco_{\tau}(\bm{S})$\footnote{See Remark~\ref{r:sub:sot:explained}
 for the rationale for 
 $\tau$ as a subscript, and not as a superscript.} for the smallest $\tau$-closed $\Gamma$-convex 
 set containing $\bm S.$}
Finally, we establish a Krein-Milman type theorem for $\Gat$-convex sets, whose proof relies critically on \cite[Theorem~6.4.2]{DK}.

 \begin{repthm}{th-gkm}[\textbf{Krein-Milman theorem for $\Gat$-convex sets}] Suppose $\Gat(0)=0.$
 	If  $\bm{K}$ is an SOT-closed and bounded operator $\Gat$-convex set that contains $0,$ 
 	then $\Gat\mhyphen\ext(\bm{K}) \neq \emptyset$ and 
 	\begin{equation*}
 		\bm{K} = 
 		\Gat\mhyphen \ovopco_{\SOT} (\Gat\mhyphen\ext(\bm{K})).
 	\end{equation*}
 \end{repthm}

\subsubsection{Lasserre-Parrilo lifts for \texorpdfstring{$\Gat$}{Gamma}-convex sets}
Section \ref{sec 7} provides a $\Gat$-analog of the results in \cite{HKM16}, namely a free version of the Lasserre-Parrilo construction in the $\Gat$-convex setting.
Fixing  a  symmetric matrix-valued noncommutative polynomial
$p \in M_{\mu}(\C\langle x, y \rangle),$ 
\begin{equation*}
	p(x) = \sum_{|\gamma| \leq \delta} p_\gamma \gamma, 
\end{equation*}
of degree $\leq \delta$ in $\gv$ variables $x,$ where  $p_\alpha \in \Mm_\mu,$ 
we give 
a construction of a sequence $\mathcal{D}^\Gat_{A^{(d)}}$ of free $\Gat$-spectrahedra
in increasingly many auxiliary variables
whose projections give better and better approximations to the  $\Gat$-convex hull of the operator positivity domain $\mathcal{D}_p^\infty$ of $p,$
\[
\mathcal{D}_p^\infty = 
(\{X \in \mathbb{S}_n^\texttt{g} \ | \ p(X) \succeq 0\})_{n\in\N\cup\{\infty\}}.
\]

Firstly, by leveraging the theory of moments, in Theorem \ref{th-lift} we employ free Hankel matrices to explicitly present
the operator $\Gat$-convex hull of a bounded\footnote{More precisely, for Archimedean $p$, see \eqref{eq:archpoly}}
$\cD_p^\infty$ as the projection of a \Hilby{} $\Gat$-spectrahedron. Second, truncating this construction yields the clamping down Theorem \ref{th-cd}: the operator $\Gat$-convex hull of 
$\cD_p^\infty$ is the decreasing intersection of an explicit countable family of projections of free $\Gat$-spectrahedra.

Finally, Subsection \ref{ssec:stop} gives two criteria for when this intersection stabilizes in finitely many steps (in which case the operator $\Gat$-convex hull of 
$\cD_p^\infty$ is the projection of a free $\Gat$-spectrahedron, a so-called $\Gat$-spectrahedrop).

\subsection*{Acknowledgment}
\CB{We} thank Ken Davidson for patiently corresponding with us during the preparation of this article.  
{We thank the three anonymous referees for their careful reading and thoughtful suggestions that improved the paper.}
 
\section{Categorical duality} \label{sec dual}

The categorical duality between compact matrix convex sets and operator systems was established in \cite{WW}. 
We refer to \cite[$\S 13$]{Pa} for a treatment of the  abstract characterization of operator systems via the  Choi-Effros axioms.
An \df{abstract operator system} $\cR$ is a matrix ordered $*$-vector space with an Archimedean matrix order unit. This means firstly, that $\cR$ is a $*$-ordered vector space; that is,  $\cR$ is a vector space equipped with an involution $*$ giving rise to a real subspace $\cR_{\sa} \subseteq \cR$ of self-adjoint elements. {\CCBB  The latter is equipped with a partial order $\le_\cR$, equivalently a pointed cone $\cR^+$ of positive elements ($x \in \cR^+$  if and only if $0 \le_\cR x$), as well as an Archimedean order unit
$e.$ That $e$ is an order unit means, for every $x \in \cR$, there exists a scalar $\lambda > 0$ such that 
\[
-\lambda e \le_\cR x \le_\cR \lambda e;
\]
 and the Archimedean modifier means if $x\in \cR$ and $x \le_\cR \epsilon e$ for all $\epsilon > 0$ then $x \le_\cR 0$. Next, a matrix ordering
 consists of a $*$-ordering on each $M_n(\cR)$ that makes $M_n(\cR)$
 a $*$-ordered vector space with Archimedean order unit $1_n\otimes e$
 such that the $*$-orderings at the various levels are further  compatible
 in that $a^* x a\in \cR_m^+$
for positive integers $m,$  for all positive integers $n$ and  $x\in \cR_n^+$ and all $n\times m$ scalar matrices $a.$}  Thus,  when dealing with $\cR$ we always also consider its associated family of matrix spaces $(M_n(\cR))_n=\MM(\cR).$

There is also the notion of a \df{concrete operator system}, namely a closed self-adjoint unital subspace of $\cB(\cH)$ for some Hilbert space  $\cH.$ By a result due to Choi-Effros (see \cite[Theorem 13.1]{Pa}), every abstract operator system $\cR$ has a concrete realization\footnote{There is a concrete operator system $\cS$ such that
$\cR$ and $\cS$ are isomorphic as operator systems as described below.}.
 On the other hand, it is easy to check that every concrete operator system satisfies the Choi-Effros axioms of an abstract operator system, since 
 $\cB(\cH)$ is an operator system by way of the identification of  $M_n(\cB(\cH))$ with $\cB(\oplus^n \cH).$

There is a duality between operator systems and matrix convex sets that we describe in Subsection~\ref{s:cat-dual-general}. In the case of matrix convex sets over
 the vector space $\C^\gv$ and finite-dimensional operator systems $\cR,$ certain topological considerations in the general setting disappear  and
 there are rather transparent proofs of the Webster-Winkler duality between convex sets and operator systems that we present in 
 Subsection~\ref{s:finite-dim-case}.  The duality for $\Gat$-convex sets and $\Gat$-operator systems appears   in Subsection~\ref{subsec:gama-dual}.

\subsection{Duality: matrix affine and ucp maps} 
\label{s:cat-dual-general}
{\CCBB In this subsection we describe the duality between matrix convex sets and operator systems
that mirrors the commutative theory of the Kadison duality between convex sets and function systems \cite{Alf71}.
As we expand on below with full details available in \cite{WW},}
 given a matrix convex set $\bm K,$  the set of matrix affine maps $\rA(\bm K)$ is an operator system; and given
 an operator system $\cR,$ the set $\UCP (\cR)$ of unital completely positive (ucp)  maps from $\cR$  into matrix algebras  is a matrix convex set.
 The operations $\bm K \mapsto \rA(\bm K)$ and $\cR\mapsto \UCP (\cR)$ \cite{WW} are dual to one another.

Given  (locally convex Hausdorff) vector spaces $\vs$ and $\vs^\prime$ and a matrix convex set $\bm K\subseteq \MM(\vs),$ 
 a \df{matrix affine map} $\theta:\bm K \to \MM(\vs^\prime)$ is a sequence of maps $\theta_n:K_n\to M_n(\vs^\prime)$ satisfying,
 \[
    \theta_n(\sum_{i=1}^m  V_i^* A^{(i)} V_i) = \sum_{i=1}^m V_i^* \theta_{n_i}\left(A^{(i)}\right) V_i,
 \]
 whenever $A^{(i)}\in K_{n_i}$ and $V_i\in M_{n_i,n}$ satisfy $\sum_{j=1}^m V_j^*V_j=\Ii_n.$
  Matrix convex sets  $\bm K \subseteq \MM(\vs)$ and $\bm K^\prime \subseteq \MM(\vs^\prime)$
   are \df{isomorphic as matrix convex sets}  if there exists a matrix affine homeomorphism $F:\bm K \to \bm K^\prime.$
    Thus $F$ is a bijective matrix affine map and each $F_n:K_n \to K^\prime_n$
     is a homeomorphism; that is,  $F=(F_n)_n$ is a sequence of homeomorphisms 
  satisfying 
 \[
  F_n(\sum_{i=1}^m V_i^* A^{(i)} V_i) = \sum_{i=1}^m V^*_i F_{n_i}\left(A^{(i)}\right) V_i,
 \]
 whenever $A^{(i)}\in K_{n_i}$ and $V_i\in M_{n_i,n}$ satisfy $\sum_{i=1}^m V_i^*V_i=\Ii_n.$

  Assuming  $\bm K$ is compact, meaning for each $n$ the set $K_n\subseteq M_n(\vs)$ is compact in the product topology, 
  let $\rA(\bm K,M_r)$ denote the continuous matrix affine mappings $\theta:\bm K\to \MM(M_r).$ In the case $r=1,$  so
   that $M_r=M_1=\C,$   let $\rA(\bm K)=\rA(\bm K,\C)$.  \index{$\rA(\bm K)$}

    The space
	$$
	\rA(\bm{K}) = \{\theta = (\theta_n: K_n \to \Mm_n)_{n \in \mathbb{N}} \ | \ \theta \text{ continuous matrix affine}\}
	$$ 
	is an abstract operator system. Indeed, the involution ${}^*$ on $\rA(\bm{K})$ is pointwise conjugation,
	\begin{equation}\label{eq:con}
	    \theta^*_n(X)= \theta_n(X)^*	
	\end{equation}
	for all $n$ and $X \in K_n$, and the positive cone $C_1$ consists of all $\theta \in \rA(\bm{K})$ such that
 $$
 \theta(X) \succeq 0
 $$
 for all $X \in \bm{K}.$  It remains to describe the cones $C_r$ for $M_r(\rA(\bm{K}))$ and an  Archimedean matrix order unit for $\MM(\rA(\bm K)).$ 
   After   identifying 
	$$
	M_r(\rA(\bm{K})) \cong \rA(\bm{K}, \Mm_r),
	$$
    the involution ${}^\ast$ on $M_r(\rA(\bm{K}))$ is defined as pointwise conjugation as in \eqref{eq:con}
 and
	the positive cone $C_r$ is defined to consist of all the maps $\theta \in \rA(\bm{K},\Mm_r)_{\sa} $ that are pointwise positive semidefinite,
	$$
	\theta_m(X) \succeq 0
	$$
	for all $m$ and $X \in K_m.$ It is easy to check that the cones $C_n$ are pointed; that is,  $C_r \cap (-C_r) = \{0\}$ holds for all $r,$ and they satisfy the compatibility condition $\alpha^* C_r\, \alpha \subseteq C_s$ for all $r,s$ and $r \times s$ complex matrices $\alpha.$ Further, letting  $e=(e_n)$ denote the continuous matrix
	 affine map  given by the sequence 
	 of constant functions $e_n:K_n\to M_n$ defined by $e_n(v)=\Ii_n,$ 
 the sequence of continuous affine maps   $\mathbb{I} = (e\otimes 1_r)_r$ 
  {from}  $(M_r(\rA(\bm K)))_r \cong (\rA(\bm K,M_r))_r$  
  is an  Archimedean matrix order unit.  If $\bm K$ and $\bm K^\prime$ are isomorphic as matrix convex sets via an affine homeomorphism $F,$ then $ \rA({\bm K}^\prime)\ni f\mapsto f\circ F \in \rA({\bm K})$ determines a complete order isomorphism. Thus $\rA({\bm K})$ and $\rA({\bm K}^\prime)$ are isomorphic as
  operator systems. 
 The first part of \cite[Proposition 3.5]{WW} states that every abstract operator system is of the form $\rA(\bm{K})$ for a compact matrix convex set $\bm{K}.$\looseness=-1

We  now turn to the space \df{$\UCP(\cR)$}  of ucp matrix-valued maps on an operator system $\cR.$
A linear map $\Phi:\cR \to \Mm_n$ is $k$-positive if its $k$-th ampliation $\Phi_k = \Phi \otimes \Ii_k : M_k(\cR) \to M_k(\Mm_n)$ is positive. Here  for any $B = (B_{i,j}) \in M_k(\cR),$
	$$\Phi_k(B)  = \big(\Phi(B_{i,j})\big)$$
and $\Phi_k$ is positive if $\Phi_k(B)$ is positive semidefinite whenever $B$ is.
 The map $\Phi$ is \df{completely positive} if it is $k$-positive for all $k.$ 
The \df{matrix state space} of $\mathcal{R}$ is the sequence $\UCP (\mathcal{R}) = (\UCP_n(\mathcal{R}))_n,$ where
	$$
	\UCP_n(\mathcal{R})  = \{\Phi:\mathcal{R} \to \Mm_n \ | \ \Phi \text{ \text{unital completely positive}}\},
	$$
     canonically identified as a subset of $M_n(\mathcal{R}^*).$  Given $\varphi\in \UCP_t(\cR)$ and an isometry $V:\C^s\to \C^t,$ the mapping
      $V^*\varphi V: \cR\to \Mm_s$ defined by $V^* \varphi V(R) = V^* \varphi(R) V$ for $R\in \cR$ is ucp. In this way, $\UCP(\cR)$ 
      is a \wstar-compact matrix convex set in $\MM(\cR^*).$
      Part (b) of \cite[Proposition 3.5]{WW} states that every compact matrix convex set is the matrix state space of an operator system.

Two operator  systems  $\cR$ and $\cR^\prime$ are \df{isomorphic as operator systems}  if there exists a linear isomorphism $G:\cR\to\cR^\prime$
  such that both $G$ and $G^{-1}$ are \CB{unital} completely positive. 
  Such a $G$ is a \df{complete order isomorphism}.

To any compact matrix convex set $\bm{K}$ we associate the operator system $\rA(\bm{K})$ and to any (abstract) operator system $\cR$ we associate the matrix state space $\UCP(\cR),$ which is a (\wstar) compact matrix convex set. By \cite[Proposition 3.5]{WW}, the operations UCP and $A$ are dual to each other: $\cR$ and $\rA(\UCP(\cR))$ are isomorphic operator systems and, on the other hand, $\bm{K}$ and $\UCP (\rA(\bm{K}))$ are matrix affinely homeomorphic compact matrix convex sets.\looseness=-1

\subsection{The finite-dimensional case}
\label{s:finite-dim-case} 
 By a \df{finite-dimensional matrix convex set}  we mean a matrix convex set $\bm K \subseteq \MM(\vs),$ where 
  $\vs$ is a finite-dimensional (locally convex Hausdorff) vector space. Often $\vs=\C^\gv$ so that $\bm K\subseteq \MM^\gv.$ Likewise, a \df{finite-dimensional
 operator system} $\cR$ is an operator system that is finite-dimensional as a vector space.
 Since $\cR$ is finite-dimensional, $\cR^*$ is finite-dimensional and locally convex.
  In finite dimensions topological
 considerations are trivial, since the topology on a finite-dimensional locally convex space is unique and determined
 by a norm.  See, for instance,  \cite[Theorem 1.21]{R}. Thus, there is no ambiguity when using terms such as closed and compact
  that apply level-wise to a matrix convex set $\bm K.$
 In the finite-dimensional setting of this section,   there are relatively simple proofs of the Webster-Winkler  duality described in the previous
 subsection, Subsection~\ref{s:cat-dual-general}, that we outline here in preparation for establishing similar  results
 in the $\Gat$-convex  setting.\looseness=-1
 
 Before proceeding we record a lemma. If $\bm K$ is a compact matrix convex set, then each $K_n$ is compact and hence bounded. More is true when $\bm K$
  is a subset of a finite-dimensional vector space.

 \begin{lemma}[\protect{\cite[Proposition 4.3 and Lemma 4.2]{HKM17}}]
 \label{l:level:bounded}
     If $\bm K\subseteq \mathbb{S}^\gv$ is a matrix convex set, then $\bm K$ is bounded
   if and only if $K_1$ is bounded and in this case, if $R$ is a bound for $K_{1},$ then 
    $\|Y_j\|\le R$ for all $n$ and $Y=(Y_1,\dots,Y_\gv)\in K_n,$ and 
     $\sqrt{\gv} \, R$ is a bound for $\bm K.$
 \end{lemma}
 
\begin{proof} 
   Suppose $R$ is a bound for $K_1$ and let $n$ and $Y=(Y_1,\dots,Y_\gv)\in K_n$ be given.
   Given a unit vector $h\in \C^n$ the map $V_h: \C\to \C^n$ defined by $V_h \, c=c\, h$ is an isometry. Thus 
\[
 V^*YV =(h^*Y_1h,\dots, h^*Y_\gv h) \in K_1
\]
 and consequently 
\[
   | h^* Y_j h |^2  \le \sum_{j=1}^\gv h^* Y_j h h^* Y_j h  = \|V^* Y V\| \le R^2.
\]
 Thus $|h^*Y_jh|\le R$ for each $j.$ Since $Y_j$ is self-adjoint, it follows that $\|Y_j\|\le R.$
 Finally, 
\[
 \sum Y_j^2 \le g R^2
\]
 and thus $\|Y\| \le \sqrt{\gv} R$ as claimed.
\end{proof}

 \subsubsection{Polar duals} 
 \label{sss:PD-MSS-MR}
   Recall, given a tuple $B=(B_1,\dots,B_\gv)\in \mathbb{S}_\cH^\gv,$  
\begin{equation*}
 L_B(x) =\Ii_\cH + \sum_{j=1}^\gv B_j x_j
\end{equation*}
 denotes the resulting (operator) \df{monic linear pencil}. \index{$L_B$}
 In particular, for a tuple $X\in \Mm_n^\gv,$
\[
 L_B(X) = \Ii_\cH \otimes \Ii_n + \sum_{j=1}^{\gv} B_j \otimes X_j
\]
 and the positivity domain of $L_B,$  \index{$\cD_B$} denoted $\mathcal{D}_B = \big(\mathcal{D}_{B}(n)\big)_n\subseteq\mathbb{S}^\gv,$  is defined as
\[
\mathcal{D}_{B}(n) = \{X \in \mathbb{S}^\gv_n \ | \ L_B(X) \succeq 0 \}.
\]

Given a tuple  $A=(A_1,\ldots, A_\gv)\in \mathbb{S}_\cH^\gv,$  the \df{polar dual} of $\cD_A$ (see \cite{HKM17}) is 
 the graded set $\cD_A^\circ = (\cD_A^\circ(n))_{n \in \N},$ where \index{$\cD_A^\circ$}
$$
\cD_A^\circ(n) =  \{B \in \mathbb{S}_n^\gv \ | \ 
L_B(X) \succeq 0 \text{ for all }  X \in \cD_A \}.
$$
More generally, the \df{polar dual} of a graded set  $\emptyset \ne \bm{S} =(S_n)_n \subseteq \gtup,$  is the graded set $\bm{S}^\circ = (S_n^\circ)_n,$ where
$$
S_n^\circ = \{B \in \gtupn \ | \ L_B(X) = \Ii_n \otimes \Ii_{\siz(X)} +  \sum_{i=1}^\gv B_i \otimes X_i \succeq 0 \text{ for all } X \in \bm{S}\}.
$$
It is  well known that the polar dual $\bm{S}^\circ$ is matrix convex as we now demonstrate.  Suppose $B^{(i)} = (B^{(i)}_1,\ldots,B^{(i)}_\gv) \in S_{n_i}^\circ$ and $V_i \in \Mm_{n_i,n}$ for $i=1,\ldots,k$ and  $\sum_{i=1}^k V_i^*V_i = \Ii_n.$ 
 To  show $\widetilde{B}=\sum_{i=1}^k V_i^*B^{(i)}V_i$ lies in $S_n^\circ,$  let $X \in S_r$ be given
and,  using the assumption that $L_{B^i}(X) \succeq 0$ for all $i,$  compute
\begin{align*}
	L_{\widetilde{B}}(X)  =& \Ii_n \otimes \Ii_r + \sum_{j=1}^\gv \left(\sum_{i=1}^k V_i^*B^{(i)}_jV_i \right) \otimes X_j\\
	=& \Ii_n \otimes \Ii_r + \sum_{i=1}^k \sum_{j=1}^\gv V_i^*B^{(i)}_jV_i  \otimes X_j\\
	=& \sum_{i=1}^k V_i^*V_i \otimes \Ii_r + \sum_{i=1}^k (V_i^* \otimes \Ii_n)\left( \sum_{j=1}^\gv B^{(i)}_j\otimes X_j\right)(V_i \otimes \Ii_n)\\
	=& \sum_{i=1}^k (V_i \otimes \Ii_n)^* L_{B^{(i)}}(X)(V_i \otimes \Ii_n) \succeq 0.
\end{align*}
It is now immediate that $\bm{S}^\circ$ is a closed matrix convex set containing $0.$  
For future use, we state the following result
summarizing some basic facts about the relations between $\bm S$ and $\bm S^\circ$ in the finite-dimensional setting.

\begin{lemma}[cf.~\protect{\cite[Proposition 4.3]{HKM17}}]
 \label{l:HKM17:4.3}
 Suppose $\bm S\subseteq \mathbb{S}^\gv.$ 
 \begin{enumerate}[\rm(1)]
  \item \label{i:4.3:i} $\bm S^\circ$ is a closed matrix convex set containing $0;$
  \item 
    If $0$ is in the interior of $S_1,$ then $\bm S^\circ$ is bounded;
  \item \label{i:4.3:iii} 
   If $S_1$ is bounded, then $0$ is in the interior of $\bm S^\circ;$
  \item \label{i:4.3:iv} If $\bm S$ is a closed matrix convex set and $0\in S_1,$ then $\bm S=\bm S^{\circ\circ};$
  \item \label{i:4.3:v}  Assuming $\bm S$ is a matrix convex set containing $0,$ if $0$ is in the interior of
  $(S^\circ)_1,$ then $\bm S$ is bounded.
 \end{enumerate}
\end{lemma}

\begin{proof}
  A proof of  item~\ref{i:4.3:i}  appeared before the statement of the lemma.  By \cite[Lemma~4.2]{HKM17}, if $0\in \R^\gv$ is in the interior
   of $S_1,$ then $0$ is in the interior of $\bm S.$ Hence by item~(2) of \cite[Proposition 4.3]{HKM17}, if $0$ is in the interior of $S_1,$
   then $\bm S^\circ$ is bounded.  If $S_1$ is bounded, then, by Lemma~\ref{l:level:bounded}, $\bm S$ is bounded and thus by
   item~(4) of \cite[Proposition 4.3]{HKM17}, $0$ is in the interior of $\bm S^\circ.$  Item~\ref{i:4.3:iv} is a consequence 
   of the Effros-Winkler Bipolar Theorem. Indeed, 
    it is immediate that $\bm S \subseteq \bm{S}^{\circ\circ}.$ On the other hand,
    if $B\in \bm{S}^{\circ\circ}\setminus \bm S,$  then, by Theorem~\ref{t:EW}, 
     there exists a tuple $X$ such that $L_X(\bm S)\succeq 0,$ but $L_X(B)\not\succeq0.$
    But $L_X(\bm S)\succeq 0$ is equivalent to $X\in \bm{S}^\circ$ and then $L_X(B)\not\succeq 0$ leads to the contradiction that $B\notin \bm{S}^{\circ\circ}.$
     Finally item~\ref{i:4.3:v} flows from items~\ref{i:4.3:iv} and \ref{i:4.3:iii}.
\end{proof}
     
     \begin{remark}\label{r:4.3}
    Since $0$ is in the interior of $\cD_A,$ the closed matrix convex set $\cD_A^\circ$ is bounded and hence (level-wise) compact.
   
   When $\bm{K}$ is a matrix convex set in a finite-dimensional space, 
   by restricting 
   $\bm{K}$ to the affine span of $K_1$ followed by a possible translation (see \cite[Theorem II.2.4]{Ba}),
    it can be assumed, without loss of generality,  that $0$ is in the interior of $\bm K.$
    \qed
\end{remark}

\subsubsection{The matrix state space and matrix range}
 By Remark~\ref{r:4.3}, a  finite-dimensional operator system $$\cR=\spann\{I_\cH,A_1,\dots,A_\gv\}$$ determined by the tuple $A\in\mathbb{S}^\gv_{\cH}$
 gives rise to the compact matrix convex set $\cD_A^\circ.$
 In Theorem~\ref{th:find} below we see that $\cD_A^\circ$ is isomorphic, as a matrix convex set, to $\UCP(\cR).$ Let $\widecheck{\cR}=\cD_A^\circ.$
 \index{$\widecheck{\cR}$} 

 The matrix state space $\UCP(\cR)$ of the operator system $\cR=\spann\{I_\cH,A_1,\dots,A_\gv\}$ is naturally identified with  the (joint) 
  \df{matrix range}  $W(A) =(W_n(A))_n$ of 
  the  tuple $A=(A_1,\dots,A_{\gv})\in \mathbb{S}_\cH^{\gv}$ \CB{(cf. \cite{A2, Far93, DDOSS17, LLPS18})} defined by
 \[
   W_n(A)  = \{(\varphi(A_1),\dots,\varphi(A_\gv)) \in \mathbb{S}_n^{\gv} : \varphi \in \UCP_n(\cR)\} \subseteq \mathbb{S}_n^\gv.  \index{$W(A)$}
 \]
  The matrix range $W(A)$ is easily seen to be matrix convex. In fact,  assuming $\cD_A$ is bounded,  the  matrix range is another description 
  of the polar dual $\cD_A^\circ$ of $\cD_A.$
  For the reader's convenience the presentation that follows includes a proof of this fact. See Proposition~\ref{p:Rcheck-WA}.
 It follows from the work of Arveson \cite{A2} that any compact matrix convex set in $\C$ is matrix affinely homeomorphic to the matrix range of a bounded operator on a separable Hilbert space. {\CCBB  An extension of this result to the case of a compact matrix convex set in $\C^\gv$ for $\gv \in \N$ is \cite[Proposition 3.5]{DDOSS17}.}\looseness=-1

 \subsubsection{Duality}

  Using Lemma~\ref{l:level:bounded}, given a compact matrix convex set $\bm K \subseteq \mathbb{S}^\gv$ and   setting $Y_0=I_{\siz(Y)},$ the
 direct sums
 \[
   \widehat{Y}_j =  \bigoplus_{Y \in \bm{K}}Y_j,
 \]
  produce a tuple of bounded operators and we  associate to $\bm K$ the operator system 
 \[
  \widehat{\bm K} = \spann \{\widehat{Y}_j: 0\le j \le \gv\}. 
 \]

Theorem~\ref{th:find}  below is a version of \cite[Proposition~3.5]{WW} tailored to the present finite-dimensional setting and incorporating the
spectrahedral point of view. It  states that the operations \, $\widehat{}$ \, and \, $\widecheck{}$ \, are dual to each other.
 A tuple $A\in \mathbb{S}_\cH^\gv$ is \df{semi-finite} if it is an at most countable direct sum of matrix tuples; that is, $A=\oplus_{j\in J} A^{(j)}$, where $J\subseteq\N$,
  and for each $j\in J$ there is a positive integer $n_j$ such that $A^{(j)}\in \mathbb{S}^\gv_{n_j}.$

\begin{theorem}\label{th:find} 
  Suppose $A_1,\dots,A_\gv\in \mathbb{S}_\cH$
  and let $\cR = \spann\{\Ii_\cH,A_1,\ldots,A_\gv\} \subseteq \cB(\cH)$ denote the resulting finite-dimensional operator system
   and $\cD_A$ the corresponding \Hilby  free spectrahedron.
 \begin{enumerate}[\rm (1)]\itemsep=5pt
   \item \label{i:bmK:i}  If $\cD_A$ is bounded, then $\cR$  and $\rA(\cD_A^\circ)$ are isomorphic operator systems;
   \item \label{i:bmK:ii}   $\UCP(\cR)$ and $W(A)$ are matrix affinely homeomorphic (isomorphic as matrix convex sets)
     and, if $\cD_A$ is bounded, then $\widecheck{\cR}:=\cD_A^\circ = W(A);$ and
   \item \label{i:bmK:iii}  If $\cD_A$ is bounded, then $\cR$ and $\widehat{\widecheck{\cR}}$   are isomorphic operator systems.
 \end{enumerate}
 
    Suppose $\bm{K}\subseteq \mathbb{S}^\gv$ is a matrix convex set.  
  \begin{enumerate}[\rm (a)]\itemsep=5pt
    \item \label{i:bmK:a}  If $\bm K$ is closed 
    and  $K_1$ contains a neighborhood of $0,$ 
    then there exists a semi-finite tuple $A=(A_1,\ldots,A_{\gv})\in \mathbb{S}_\cH$ 
      such that $\bm K$ is the \Hilby free spectrahedron $\cD_A;$ 
   \item \label{i:bmK:b} If $\bm K$ is compact  and  $K_1$ contains a neighborhood of $0,$
        then  $\rA(\bm K)$ and $\widehat{\bm K}$ are isomorphic operator systems; and
   \item \label{i:bmK:c}  If $\bm K$ is compact and $0\in K_1,$ then $\widecheck{\widehat{\bm K}}:=\cD^\circ_{\widehat{Y}}=\bm K.$
  \end{enumerate}
\end{theorem}

\begin{remark}
\label{r:relax:sep}
 As it stands, the operators $\widehat{Y}$ generically act on a non-separable Hilbert space $\mathscr{Y}.$
 However, the resulting operator system admits a representation using a semi-finite tuple; 
  that is, there is a separable Hilbert space $\mathscr{E}$ and a semi-finite tuple $\widehat{E}$ acting on $\mathscr{E}$
   such that $\spann\{I_{\mathscr{Y}}, \widehat{Y}\}$ and  $\spann\{I_{\mathscr{E}}, \widehat{E}\}$ are isomorphic
   as operator systems via the unital map that sends $\widehat{Y}_j$ to $\widehat{E}_j.$ Indeed, simply choose a countable
    dense graded set $\bm F\subseteq \bm K$ (meaning $F_m$ is countable and dense in $K_m$
     for each $m$) and set $\widehat{E} = \oplus_{E\in \bm F} E.$ A variation on this construction, along with Theorem~\ref{t:EW},
    gives item~\ref{i:bmK:a} of Theorem \ref{th:find}, proved as Lemma~\ref{l:semi-finite} below.
 \qed
\end{remark}

\begin{remark}
  \label{r:weird-for-us}
     From item~\ref{i:bmK:a}, 
      there exists a separable Hilbert space $\cH$ and a tuple $A=(A_1,\dots,A_\gv)\in \mathbb{S}_\cH^\gv$ such that $\bm K =\cD_A.$
     However, to identify the operator system $\rA(\bm K)$ concretely,  one first uses item~\ref{i:bmK:c}
      of Theorem~\ref{th:find}, which gives $\bm K = \cD_{\widehat{Y}}^\circ.$ From Lemma~\ref{l:HKM17:4.3}\ref{i:4.3:v},
      the assumption that $0$ is in the interior of \CB{$K_1=\cD_{\widehat{Y}}^\circ(1)$} implies $\cD_{\widehat{Y}}$ is bounded  so that, by 
       item~\ref{i:bmK:i}, 
      $\rA(\bm K)$ is isomorphic, as an operator system, to $\cS=\spann\{\Ii,\widehat{Y}_1,\dots,\widehat{Y}_\gv\}.$
       This circuitous route to the duality expressed in item~\ref{i:bmK:b} offers insight into the perspective obtained by viewing matrix convex sets
        as \Hilby  spectrahedra.
        
       There are various strategies for establishing Theorem~\ref{th:find} using Webster-Winkler duality.
        For instance, using  items~\ref{i:bmK:b} and \ref{i:bmK:ii} and the Webster-Winkler duality between $\cR$ and $\rA(\UCP(\cR))$ gives 
        $\widehat{\widecheck{\cR}}\cong  \rA(\widecheck{\cR}) \cong  \rA(\UCP(\cR))\cong \cR.$    We will prove
         Theorem~\ref{th:find} without invoking duality and thus provide an alternate proof of Webster-Winkler
         duality in this finite-dimensional setting. 
     \qed
 \end{remark}

The proof of Theorem~\ref{th:find} occupies the remainder of this subsection.  After proving item~\ref{i:bmK:a} (see Lemma~\ref{l:semi-finite}), 
we state  Propositions~\ref{prop:bded} and \ref{prop:cp}, which connect completely positive maps between finite-dimensional 
 operator systems with \Hilby free spectrahedra. 
  Items~\ref{i:bmK:iii}  and \ref{i:bmK:c} are then established as Lemmas~\ref{l:bmK:iii} and \ref{prop:setdu} respectively,
   followed by  
  items~\ref{i:bmK:ii} and \ref{i:bmK:i} as Proposition~ \ref{p:Rcheck-WA} 
  and  \ref{p:DA-double-check}.
  As was noted in Remark~\ref{r:weird-for-us}, item~\ref{i:bmK:b} follows from 
   item~\ref{i:bmK:c} ($\bm K = \cD_{\widehat{Y}}^\circ$)
   and item~\ref{i:bmK:i}, which together give 
\[
   \rA(\bm K) \cong \rA(\cD_{\widehat{Y}}^\circ) \cong \spann \{\Ii, \widehat{Y}_1,\ldots, \widehat{Y}_\gv\}.
\]

\begin{lemma}
 \label{l:semi-finite}
   If $\bm K$ is a closed matrix convex set  such that $0$ is in the interior of $K_1,$ then there is a semi-finite tuple $A=(A_1,\dots,A_\gv)\in \mathbb{S}_\cH$
    such that $\bm K =\cD_A.$
\end{lemma}

\begin{proof}
      Fix a positive integer $n$ and let $F_n=\mathbb{S}^\gv_n \setminus K_n.$ The
     set $F_n$ is an open subset of Euclidean space and hence every open cover of $F_n$ admits an at most countable subcover, since every open set in Euclidean space is a countable union of compact sets.
     For each $G\in F_n,$ there is, by Theorem~\ref{t:EW}, a tuple
      $A_G\in \mathbb{S}_n^\gv$ such that $L_{A_G}(\bm K)\succeq 0,$ but $L_{A_G}(G)\not\succeq 0.$ Let
       $U_G=\{Y\in \mathbb{S}_n^\gv: L_{A_G}(Y)\not\succeq 0\}\subseteq F_n$ and note $U_G$ is open.
        Since $F_n = \cup_{G\in F_n}  U_G,$ there is a countable set $G_n\subseteq F_n$ such that 
         $F_n =\cup_{G\in G_n} U_G.$  Let $\bm G=\cup G_n$ and observe, if $Y\in \mathbb{S}^\gv_\ell\setminus K_\ell,$
       then there is a $G\in G_\ell$ such that $Y\in U_G$ so that $L_{A_G}(Y)\not\succeq 0.$ 
       
        Since $0$ is in the interior of $K_1,$ the polar dual $K^\circ$ is bounded by Lemma~\ref{l:HKM17:4.3}.  Thus there is an $M$ such that $\|A_G\|\le M$
         for all $G\in \bm G.$  Hence the  operator  $A=\oplus_\ell \oplus_{G\in G_\ell} A_G$ acting on the separable Hilbert space 
        $\oplus_\ell \oplus_{G\in G_\ell} \C^\ell$ is bounded.  By construction, $\bm K =\cD_A.$   Indeed, if $Y\notin K_\ell,$ then there is
        $G\in F_\ell$ such that $Y\in U_G$ so that $L_{A_G}(Y)\not\succeq 0$ and thus $L_{A}(Y)\not\succeq 0.$ Hence $Y\notin \cD_A.$
         On the other hand, if $Y\in \bm K,$ then  $L_{A_G}(Y)\succeq 0$ for all $n$ and $G\in G_n$ and therefore $L_A(Y)\succeq 0.$
         Hence $Y\in \cD_A.$ 
\end{proof}

\begin{proposition}[{{\cite[Proposition 2.6]{HKM13}}}]\label{prop:bded}
	Let $A=(A_1,\dots,A_\gv) \in \mathbb{S}_\cH^\gv.$ 
	If the \Hilby free spectrahedron $\cD_A$ is bounded, then $\Ii_\cH,A_1,\ldots,A_\gv$ are linearly independent. 
\end{proposition}

     Proposition~\ref{prop:bded}  is proved,
	 under the assumption that $\cH$ is finite-dimensional as \cite[Proposition 2.6]{HKM13}, but the argument works just
	 as well in the case that $\cH$ is infinite-dimensional and the $A_j$ are self-adjoint operators instead of  matrices.

\begin{proposition}
  \label{prop:cp}
   Let $H,K$ denote  Hilbert spaces and 
    suppose  $A=(A_1,\ldots,A_\gv) \in \mathbb{S}_H^{\gv}$ and $B = (B_1,\ldots,B_\gv)\in \mathbb{S}_K^\gv.$ 
     Let $\cS_A = \spann\{I,A_1,\ldots,A_\gv\}$ and $\cS_B = \spann\{I,B_1,\ldots,B_\gv\}.$    
      If  $\cD_A$ is bounded, then the unital linear map $\tau: \cS_A \to \cS_B$ sending
	$$
	A_i \mapsto B_i
	$$
	for all $i$
	is completely positive if and only if $\cD_A \subseteq \cD_B.$
\end{proposition}

Proposition ~\ref{prop:cp} is stated with  the assumption that $H,K$ are finite-dimensional
as \cite[Theorem~3.5]{HKM13},  though the argument works
similarly in infinite dimensions,
cf.~\cite[Theorem~2.5]{zalar} or \cite[Theorem~5.13(2)]{DDOSS17}.
Note that the boundedness
hypotheses together with Proposition~\ref{prop:bded} implies that the map determined by $A_i\to B_i$
 is well defined.

For $A = (A_1,\ldots,A_\gv) \in \mathbb{S}_\cH^{\gv},$ the Hilbertian spectrahedron  $\cD_A$
is a closed matrix convex set and $0$ is in the interior of $\cD_A(1).$ Hence its polar
dual $\cD_A^\circ$ is bounded by Lemma~\ref{l:HKM17:4.3}. 
This observation is
implicit in the statement of Lemma~\ref{l:bmK:iii} below.

\begin{lemma}
\label{l:bmK:iii}
 Fix $A = (A_1,\ldots,A_\gv) \in \mathbb{S}_\cH^{\gv},$  let $\cR = \spann\{\Ii_\cH,A_1,\ldots,A_\gv\}$ and let
 \[
	\widehat{B_j} = \bigoplus_{B \in \cD_A^\circ} B_j \quad \text{ and }\quad  I = \bigoplus_{B \in \cD_A^\circ} \Ii_{\siz(B)}
\]
 acting on the Hilbert space $\oplus_{B\in \cD_A^\circ} \C^{\siz(B)}.$
 
 If the \Hilby  free spectrahedron $\cD_A$ is bounded, 
   then $\cR$ and $\widehat{\widecheck{\cR}} = \spann\{I,\widehat{B}_1,\ldots,\widehat{B}_\gv\}$ are isomorphic operator systems.
\end{lemma}

\begin{proof}
     Since  $\cD_A$ is assumed bounded,  Proposition~\ref{prop:bded} implies 
    the operators $\Ii_\cH,A_1,\ldots,A_\gv$ are linearly independent. Hence
    there is a linear map  $\tau: \cR \to \widehat{\widecheck{\cR}}$ determined by
	\begin{align*}
		\Ii_n &\mapsto I\\
		 A_i &\mapsto \widehat{B}_i.
	\end{align*}
 Let $X\in \cD_A$ be given. By the definition of the polar dual, $L_B(X)\succeq 0$ for all $B\in \cD_A^\circ.$  Hence $L_{\widehat{B}}(X)\succeq 0$
 and thus $\cD_A\subseteq \cD_{\widehat{B}}.$ By Proposition~\ref{prop:cp}, the map $\tau$ is ucp.  On the other hand, if $X\in \cD_{\widehat{B}},$
 then $L_B(X)\succeq 0$ for all $B\in \cD_A^\circ$ and therefore $X\in \cD_A^{\circ\circ}=\cD_A$ by Lemma~\ref{l:HKM17:4.3}.
 Hence $\cD_{\widehat{B}} \subseteq \cD_A.$  Consequently, $\cD_{\widehat{B}}$ is bounded and thus $\Ii,\widehat{B}_1,\dots,\widehat{B}_\gv$ is an independent set by Proposition~\ref{prop:bded}. Hence $\tau$ is onto and  $\tau^{-1}$ is also ucp and therefore
  $\cR$ and $\widehat{\widecheck{\cR}}$ are isomorphic operator systems.
\end{proof}

\begin{lemma}\label{prop:setdu}
	If  $\bm{K} \subseteq \gtup$ is a compact matrix convex set and $0 \in K_1,$  then  $\bm K = \cD^\circ_{\widehat{Y}},$ where
	$$
	\widehat{Y} = \bigoplus_{Y \in \bm{K}}Y.%
	$$
 \end{lemma}
 
 Note that  the compactness assumption ensures that $\widehat{Y}$ is a bounded operator.
 
\begin{proof}  
         By Lemma~\ref{l:HKM17:4.3},  $\bm{K}$ equals its bipolar $\bm{K}^{\circ\circ}.$ We will show that $\bm{K}^\circ$ equals $\cD_{\widehat{Y}}.$
	Letting $\mathscr{Y}$ denote the Hilbert space that $\widehat{Y}$ acts on, note that by definition, for any $n$ and $B \in {K}^\circ_n,$ 
	$$
	L_B(\widehat{Y})=I_n \otimes I_{\mathscr{Y}}  +  \sum_{i=1}^\gv B_i \otimes \widehat{Y}_i \succeq 0,
	$$
	which (after applying the canonical shuffle) implies that $B$ lies in $\cD_{\widehat{Y}}.$ 
 Hence, $\bm K^\circ \subseteq \cD_{\widehat{Y}}.$ 
 
 Now suppose  $B \in \cD_{\widehat{Y}}(n)\setminus {K}^\circ_n.$ Since $\bm{K}^\circ$ is closed, by the Hahn-Banach Theorem \ref{t:EW}, there is a monic linear pencil $L_A$ defined by a tuple $A\in\mathbb{S}_n^\gv$ such that $L_A$ is positive semidefinite on $\bm K^\circ,$ but not positive semidefinite at  $B.$ Since $L_A$ is positive semidefinite on $\bm K^\circ,$ the tuple $A$ belongs to $\bm K^{\circ\circ} = \bm K.$ But then $A$ is a direct summand in $\widehat{Y},$ so $L_A(B) \succeq 0,$ which is a contradiction. Hence, $\bm K^\circ =\cD_{\widehat{Y}}$ and $\bm K= \bm K^{\circ\circ} = \cD_{\widehat{Y}}^\circ.$
\end{proof}

Proposition~\ref{p:Rcheck-WA} below interprets Proposition~\ref{prop:cp} as follows \CB{(see also \cite[Example 1.9]{Kir24})}. 
  
 \begin{proposition}
  \label{p:Rcheck-WA}
   For $A=(A_1,\dots,A_\gv)\in \mathbb{S}_\cH^\gv,$ 
   the (graded) sets  $W(A)$ and $\UCP(\cR)$ are isomorphic as matrix convex sets.
    Moreover, if $\cD_A$ is bounded, then $W(A)=\cD_A^\circ.$ 
  \end{proposition}

 \begin{proof}
    For notational ease, let $A_0=I_\cH.$
    If $\varphi\in \UCP_t(\cR)$ and $V: \C^s\to \C^t$ is an isometry, then 
    $V^*\varphi V: \cR\to M_s$ defined by $V^*\varphi V(R) = V^* \varphi(R)V$  for $R\in \cR$ is easily
     seen to be unital and completely positive. Indeed, for $R\in M_m(\cR)$ given by
  \[
     R = \sum Y_j\otimes A_j, 
  \]
   where $(Y_0,Y_1,\ldots, Y_\gv)\in \mathbb{S}_m^{\gv},$ 
 \[
  1_m\otimes (V^*\varphi V) (R) = \sum Y_j \otimes V^* \varphi(A_j)V
   = (\Ii_m \otimes V)^* (1_m\otimes \varphi)(R)  (\Ii_m\otimes V).
 \]
  Thus, if $R\succeq 0,$ then, since $\varphi$ is ucp, $1_m\otimes \varphi (R)$ is positive semidefinite 
   and hence so is $1_m\otimes (V^*\varphi V)(R).$ Therefore, $V^*\varphi V\in \UCP_s(\cR)$
   and $\UCP(\cR)$ is matrix convex. It now follows that  $W(A)$ is matrix convex
    since $V^* \varphi(A_j) V= V^*\varphi V(A_j)$ for each $1\le j\le \gv.$
    
   The sets $\UCP_t(\cR)$ are bounded since ucp maps have norm one. It is straightforward
   to see the (norm) limit of ucp maps $\varphi_k:\cR\to M_t$ is again ucp.  Hence $\UCP_t(\cR)$
    is closed and hence compact.
       
     Given a positive integer $t,$ define $\Psi_t:\UCP_t(\cR)\to W_t(A)$ as follows. For
     $\varphi\in \UCP_t(\cR)$, let\looseness=-1
  \[ 
   \Psi_t(\varphi) = \big (\varphi(A_1),\ldots,\varphi(A_{\gv}) \big ) \in W_t(A).
  \]
   From the definitions, $\Psi_t$ is bijective.  It is immediate that $\Psi_t$ is continuous
   and thus its range,  $W_t(A),$ is  compact and $\Psi_t$ is a homeomorphism.
   
    To complete the proof that $\UCP(\cR)$ and $W(A)$ are isomorphic as matrix convex
    sets, it remains to show that $\Psi$ is matrix affine. To this end, suppose $r$ 
    is a positive integer  and finitely many $V_i:\C^{s_i}\to \C^t$
     and $\varphi_i \in \UCP_{s_i}(\cR)$ such that $\sum V_i^* V_i=I$ 
     are given and observe
   \[
    \Psi_t(\sum V_i^* \varphi_i V)  = [\sum V_i^*\varphi_i V] (A) = \sum V_i^* \varphi_i(A) V_i =\sum V_i^* \Psi_{s_i}(\varphi_i) V_i.
   \]
   Thus $\Psi_t$ is matrix affine.

   Now suppose $\cD_A$ is bounded.  Using Proposition~\ref{prop:cp}, $B\in \cD_A^\circ(t)$  if and only if
  \[
    L_B(X)=\Ii_t\otimes \Ii_n +\sum_{j=1}^{\gv} B_j\otimes X_j \succeq 0
  \]
   for all $n$ and $X\in \cD_A(n)$  if and only if  $\cD_A\subseteq \cD_B$ if and only if 
    the unital map $\tau:\cR\to M_t$ sending $A_i$ to $B_i$ is completely positive if and only if 
  \[
   \big (\tau(A_1),\ldots,\tau(A_{\gv}) \big )  = (B_1,\dots,B_\gv) \in W_t(A).
  \]
 Hence $W(A)=\cD_A^\circ.$
  \end{proof}

\begin{proposition}
\label{p:DA-double-check}
	Let $A = (A_1,\ldots,A_\gv) \in \mathbb{S}_\cH^{\gv}$ and define $\cR = \spann\{\Ii_\cH,A_1,\ldots,A_\gv\}.$ 
   If $\cD_A$ is bounded, then $\cR$ and the space $\rA(\cD_A^\circ)$ of matrix affine maps on the polar dual of $\cD_A$ are isomorphic operator systems.
\end{proposition}

\begin{proof}
  The boundedness assumption on $\cD_A$ implies that $\Ii_\cH,A_1,\dots,A_\gv$ are linearly independent by Proposition~\ref{prop:bded}. 
   Thus we may construct a map 
     $\Phi: \cR \to \rA(\cD_A^\circ)$ sending a $\tilde{b}\in \cR$ to a matrix affine map 
      $\Phi[{\tilde{b}}]=(\Phi[{\tilde{b}}]_t)_{t \in \N}:\cD_A^\circ \to \MM$ as follows.    Given 
  \[
 \tilde{b}=\lambda\, \Ii_\cH + \sum_{i=1}^\gv b_i A_i \in \cR, 
 \] 
  a positive integer $t$ and $X \in \cD_A^\circ(t) \subseteq \mathbb{S}_t^\gv,$  let
 	\begin{equation*}
		\Phi[{\tilde{b}}]_t (X) 
		    = \lambda \, I_t  + \sum_{i=1}^\gv b_i X_i \in \Mm_t.
	\end{equation*}
  Given 
  finitely many $V_j:\C^{s_i}\to \C^t$ and $X^{(j)}\in \cD_A^\circ(s_j)$
   such that $\sum V_j^* V_j =I_t,$ 
  \[
   \Phi[\tilde{b}]_t(\sum V_j^* X^{(j)} V_j) = \lambda I_t +\sum_{i,j}  b_i  V_j^* X^{(j)}_i V_j =\sum_j V_j^* (\lambda I_t +\sum b_i X^{(j)}_i ) V_j
    = \sum V_j^* \Phi[\tilde{b}]_{s_j}(X) V_j.
  \]
   Thus $\Phi[\tilde{b}]$ is matrix affine. From the construction, $\Phi$ is continuous.

	To prove injectivity of $\Phi$ suppose $\Phi[{\tilde{b}}]$ is the zero map for some $\tilde{b}=\lambda + \sum_{i=1}^\gv b_i A_i.$ 
	In particular, $$\lambda + \sum_{i=1}^\gv b_i x_i=0$$ for all $x \in \cD_A^\circ(1).$ 
	Plugging in $x=0$ we obtain that $\lambda=0.$ Since $\cD_A$ is bounded, we have by 
	 Lemma~\ref{l:HKM17:4.3} that $0$ is in the interior of $\cD_A^\circ(1).$ 
	 Hence $b_i=0$ for $1\le i\le \gv$ and thus $\tilde{b}=0.$

	Next, we show that $\Phi$ is bijective. {\CCBB By the discussion preceding the proof of the duality}  \cite[Proposition 3.5]{WW}, the  operator system $\rA(\cD_A^\circ)$ is isomorphic (as a vector space) to the space of continuous affine maps on $\cD_A^\circ (1).$ Hence, $\rA(\cD_A^\circ)$ and $\cR$  have the same dimension $\gv + 1,$ which proves that the injective map $\Phi$ is bijective. 
	
	We now prove that $\Phi$ is completely positive.  Accordingly, fix a positive integer $m$
	 and consider $\Phi_m:M_m(\cR)\to M_m(\rA(\cD_A^\circ))\cong \rA(\cD_A^\circ,M_m).$  A self-adjoint  $\widetilde{B}\in M_m(\cR)$ 
	 	 has (after applying the canonical shuffle) a representation of  the form
\[
   \widetilde{B}= B_0 \otimes \Ii_\cH + \sum_{i=1}^\gv B_i \otimes A_i
 \]
      for $B_0,\ldots,B_\gv \in \mathbb{S}_m.$  Applying $\Phi_m$ to $\widetilde{B},$ produces the  
        matrix affine mapping  
    \[
     (\Phi_m[\widetilde{B}]_t)_{t \in \N} = \Phi_m[{\widetilde{B}}] : \cD_A^\circ \to \MM(\Mm_m),
    \]
       where, for $X\in \cD^\circ_A(t),$ 
 \[
	\Phi_m[{\widetilde{B}}]_t (X) = B_0 \otimes \Ii_t + \sum_{i=1}^\gv B_i \otimes X_i.
 \]
  Now suppose  $\widetilde{B}\succeq 0.$  By \cite[Lemma 3.6]{HKM17}, which requires the assumption that
  $\cD_A(1)$ is bounded,  $B_0 \succeq 0.$ We include the argument to keep the presentation self contained. 
   If $B_0 \nsucceq 0,$ then there is a vector $v$ such that $\langle B_0 v, v\rangle < 0.$ 
   Letting $V:\C^m \otimes \C v\to \C^m\otimes \cH$
   denote the inclusion,
	$$
	\langle B_0 v, v\rangle \otimes \Ii_\cH + \sum_{i=1}^\gv \langle B_i v, v\rangle \otimes A_i =
	(V \otimes \Ii_\cH)^* \Big(B_0 \otimes \Ii_\cH + \sum_{i=1}^\gv B_i \otimes A_i\Big)(V \otimes \Ii_\cH)^* \succeq 0,
	$$
	which implies that $\sum_{i=1}^\gv \langle B_i v, v\rangle \otimes A_i  \succ 0.$
	Thus the (nonzero) point $t(\langle B_i v, v\rangle)_{i=1}^\gv \in \R^\gv$ lies in $\cD_A(1)$ for all $t>0,$ contradicting  the boundedness of $\cD_A(1).$
	Hence, $B_0 \succeq 0.$

    For $\epsilon>0,$ note that   $\CB{B_\epsilon} := (B_0+\epsilon) \otimes \Ii + \sum_{i=1}^\gv B_i \otimes A_i \succ 0.$ Since $B_{0,\epsilon}:=B_0+\epsilon\succ 0,$ its positive square root $B_{0,\epsilon}^{\frac12}$ is invertible and therefore,
	\begin{align*}
		\Ii_m \otimes \Ii + \sum_{i=1}^\gv B_{0,\epsilon}^{-1/2} \, B_i \, 
         B_{0,\epsilon}^{-1/2} \otimes A_i & = 
		(B_{0,\epsilon}^{-1/2} \otimes \Ii)\Big(B_{0,\epsilon} \otimes \Ii + \sum_{i=1}^\gv B_i \otimes A_i\Big)(B_{0,\epsilon}^{-1/2} \otimes \Ii)\\ & \succeq 0.
	\end{align*}
	Thus the tuple $\overline{B}_\epsilon := \big(B_{0,\epsilon}^{-1/2} \, B_i \, B_{0,\epsilon}^{-1/2}\big)_{i=1}^\gv$ lies in $\cD_A(m)$ and hence, by definition of the polar dual, $\Phi_m[{\overline{B}_\epsilon}]_t(X) \succeq 0$ for all $X \in \cD^\circ_A.$ Thus
 \[
\Phi_m[B_\epsilon]_t(X) = (B_{0,\epsilon}^{1/2} \otimes \Ii_t) \, \Phi_m [{\overline{B}_\epsilon}]_t(X)\, (B_{0,\epsilon}^{1/2} \otimes \Ii_t) \succeq 0
 \]
	for all $X \in \cD^\circ_A.$ Letting $\epsilon>0$ tend to $0$ gives $\Phi_m[{{\CB{\widetilde{B}}}}]_t(X)\succeq 0,$ proving that $\Phi$ is indeed completely positive.

	It remains to prove that the inverse of $\Phi$ is also completely positive. So assume $\Phi_m[{\widetilde{B}}]: \cD_A^\circ \to \MM(M_m)$
	 is positive. Equivalently, 
	$$
	 \Phi_m[{{\widetilde{B}}}]_t(X)= B_0 \otimes \Ii_t + \sum_{i=1}^\gv B_i \otimes X_i \succeq 0
	$$
	 for all $t$ and $X\in \cD_A^\circ(t).$
	 It follows from $0\in \cD_A^\circ$ that $B_0$ is positive semidefinite.
	By a similar argument as before, the matrix $B_0$ can be assumed positive definite so that we have
	$$
	\Ii_m \otimes \Ii_t + \sum_{i=1}^\gv B_0^{-1/2}B_iB_0^{-1/2} \otimes X_i \succeq 0
	$$
	for all $X \in \cD^\circ_A.$
	Hence, the tuple $\overline{B}= \big(B_0^{-1/2}B_iB_0^{-1/2}\big)_{i=1}^\gv$ lies in the polar dual of $\cD^\circ_A.$ But then, by
	 Lemma~\ref{l:HKM17:4.3}\ref{i:4.3:iv} (the Bipolar Theorem \cite[Corollary 5.5]{EW}), $\overline{B}$ lies in $\cD_A;$ that is, 
\[
   0\preceq L_A(\overline{B}) = \Ii_\cH \otimes\Ii_m + \sum_{i=1}^\gv A_i\otimes \overline{B}_i
    \cong \Ii_m \otimes \Ii_\cH + \sum_{i=1}^\gv B_0^{-1/2} B_i B_0^{-1/2} \otimes A_i
\]
     and thus
	$$
	B_0 \otimes \Ii_t + \sum_{i=1}^\gv B_i \otimes A_i = 
	(B_0^{1/2} \otimes \Ii_n)\Big(\Ii_m \otimes \Ii_n + \sum_{i=1}^\gv B_0^{-1/2}B_iB_0^{-1/2} \otimes A_i \Big)(B_0^{1/2} \otimes \Ii_n) \succeq 0, 
	$$
Hence $0\preceq \widetilde{B}  \in M_m(\cR)$ as desired. 
\end{proof}

\subsection{Duality in the \texorpdfstring{$\Gat$}{Gamma}-convex setting}\label{subsec:gama-dual}
Let again $\Gat=(\gat_1,\dots,\gat_\rv)$ be a tuple of symmetric noncommutative polynomials
with $\gat_j=x_j$ for $1\le j\le \gv\le \rv.$  In this section we introduce $\Gat$-analogs of operator systems and ucp maps
 and  prove a  
duality resembling \cite[Proposition 3.5]{WW} and Theorem \ref{th:find}. 
{\CCBB The full morphism-level categorical formulation is
given in 
Theorem \ref{th:gamma-categorical-duality} in
Appendix~\ref{app:categorical-duality}.}

\subsubsection{The matrix state space,  matrix ranges and ucp maps in the \texorpdfstring{$\Gat$}{Gamma}-convex  setting} 
Given a tuple $A=(A_1,\dots,A_\gv) \in \mathbb{S}_\cH^\gv,$ let $\cR=\cR^\Gamma_A$ denote the operator system
\begin{equation}\label{eq:gamaopsys}
	\spann\{\gamma_j(A) \ | \ j=0,1,\ldots,\rv\},
\end{equation}
where, for notational purposes, $\gamma_0(A)=\Ii_{\cH}.$ \index{$\gamma_0$}
We call $\cR=\cR^\Gamma_A$ 
a \df{$\Gat$-operator system}. \index{$\cR^\Gat_A$}

\begin{definition}\label{def:g-opsys}
	Given a Hilbert space $\cK,$ 
  a linear map $\varphi:\cR^\Gamma_A \to \cB(\cK)$ is a \df{$\Gamma$-concomitant} provided,
  \[
		\varphi(\gamma_i(A))=\gamma_i(\wvphi{A}), 
  \]
  for each  $1\le i\le \rv,$
  where \index{$\wvphi{A}$}
  \begin{equation*}
   \gamma_i(A) =\gamma_i(A_1,\dots,A_\gv), \ \ \
   \wvphi{A} =(\varphi(A_1),\ldots,\varphi(A_\gv)). 
  \end{equation*}

 {\CCBB 
 A linear map $\varphi:\cR^\Gat_A \to \cB(\cK)$} is  a \df{$\Gamma$-ucp map} if it is ucp and a $\Gamma$-concomitant.
{\CCBB  For $\Gamma(\wvphi{A})= \big (\gamma_1(\wvphi{A}),\ldots, \gamma_\rv(\wvphi{A}) \big )$ we have} 
\[ \pushQED{\qed}
  L_{\Gamma(\wvphi{A})}(y) = \Ii_{\cK} + \sum_{i=1}^\rv \varphi(\gamma_i(A)) \, y_i. \qedhere \popQED
\]
 \end{definition}

 \begin{remark}\label{rem:GatUCP}
  Suppose $A=(A_1,\ldots,A_\gv)\in \mathbb{S}^\gv_\cH.$ 
  Assuming, $\{\gamma_j(A):0\le j\le \rv\}$ is linearly independent, which, by Proposition~\ref{prop:bded}, is the case
 when $\cD_{\Gat(A)}$ is bounded,  and given $(B_1,\dots,B_\gv)\in \mathbb{S}_{\cK}^\gv$ there exists a $\Gat$-concomitant uniquely
  determined by  $\gamma_j(A) \mapsto \gamma_j(B).$
  
	The connection between spectrahedral inclusions $\cD_A\subseteq \cD_B$ and ucp maps exposed in Proposition~\ref{prop:cp}
	extends to $\Gat$-ucp maps.
	{\CCBB  Suppose $\varphi:\cR^\Gamma_A\to M_m$ is a $\Gat$-concomitant. If $\varphi$ is $\Gat$-ucp,
    then}
	 $\cD_{\Gat(A)} \subseteq \cD_{\Gamma(\wvphi{A})}.$ 
     {\CCBB Conversely, if $\cD_{\Gat(A)}$ is bounded
     and the inclusion $\cD_{\Gat(A)} \subseteq \cD_{\Gamma(\wvphi{A})}$ holds,  then  $\varphi$
     is $\Gat$-ucp. In particular, assuming $\cD_{\Gat(A)}$ is bounded,} $\varphi$ is ucp if and only if  $Y \in \mathbb{S}^\rv_m$ and 
     \[
	L_{\Gamma(A)}(Y) =\Ii_d\otimes \Ii_m + \sum_{i=1}^{\rv} \gamma_i(A) \otimes Y_i  \succeq 0,
	\]
	implies
	\[ \pushQED{\qed}
	L_{\Gamma(\wvphi{A})} (Y)=
	\Ii_n\otimes \Ii_m + \sum_{i=1}^{\rv} \varphi(\gamma_i(A))\otimes Y_i\succeq 0.  \qedhere \popQED
	\]
\end{remark}

Since they preserve the unit and the positivity structure of $\Gat$-operator systems, $\Gat$-ucp maps are the  \df{morphisms} of $\Gat$-operator systems.
An \df{isomorphism of $\Gat$-operator systems} is a bijective $\Gat$-ucp map whose inverse is also $\Gat$-ucp.

 Given a $\Gat$-operator system $\cR$ as in \eqref{eq:gamaopsys} and a positive integer $n,$ \index{$\UCP^\Gat$}
 let $\UCP^\Gat_n(\cR)$ denote the $\Gat$-ucp maps $\varphi:\cR\to M_n$ and let $\UCP^\Gat(\cR) = (\UCP^\Gat_n(\cR))_n.$
 A pair $(\varphi,V),$ where $\varphi\in \UCP_t^\Gat(\cR)$ and $V:\C^s\to \C^t$ is an isometry, is a \df{$\Gat$-pair} provided
  \[
  V^* \gamma_i (\wvphi{A}) V = \gamma_i (V^* \wvphi{A} V),
 \]
  for all $1\le i\le \rv.$
  It is shown in Proposition~\ref{p:GammaUCP-convex} below that $\UCP^\Gat(\cR)$ is $\Gat$-convex. That is, if $(\varphi,V)$ is
  a $\Gat$-pair, then $V^*\varphi V: \cR\to M_s$ defined by $V^* \varphi V (R) = V^* \varphi(R)V$ is $\Gat$-ucp.

 The  \df{$\Gat$-matrix range} of a tuple $A=(A_1,\dots,A_\gv)\in \mathbb{S}_\cH^\gv$ is 
 $\widecheck{\cR} = W^\Gat(A) = (W_n^\Gat(A))_n,$ where  \index{$\widecheck{\cR}$}
\begin{equation}\label{eq:range}
	W_n^\Gat(A) = \{(\varphi(A_1),\ldots,\varphi(A_\gv)) \in \mathbb{S}^\gv_n \ | \ \varphi \in \UCP^\Gat_n(\cR^\Gat_A)\}.
\end{equation}
   Since a ucp map $\varphi$ is bounded with norm one, it follows that $W^\Gat(A)$ is bounded.

 In the case that $\gv=\rv$ and thus $\Gat(x)=x,$ the graded set $W(A)=W^\Gat(A)$ is the {matrix range} of $A$ from Subsection~\ref{sss:PD-MSS-MR}.
  A pair $(B,V),$ where $B\in W^\Gat_t(A)$ and $V:\C^s\to\C^t$ is an isometry, is a \df{$\Gat$-pair} for $W^\Gat(A)$ if $B= \wvphi{A},$
   where $\varphi\in \UCP^\Gat_t(\cR)$ and $(\varphi,V)$ is a $\Gat$-pair for $\UCP^\Gat(\cR).$

Given $\Gat$-convex sets $\bm K$ and $\bm K^\prime$
 a sequence  $\Phi=(\Phi_n)$ of maps $\Phi_n:K_n\to K^\prime_n$ is  \df{matrix $\Gat$-affine} if 
 $V^*\Phi_n(X) V = \Phi_m(V^*XV)$ for each $\Gat$-pair $(X,V)$ with $X\in K_n$ and $V:\C^m\to \C^n.$ 
The  $\Gat$-convex sets $\bm K$ and $\bm K^\prime$  are {\bf isomorphic} if there exists
 a matrix $\Gat$-affine map $\Phi$ such that each $\Phi_n$ is a homeomorphism.

\begin{proposition}
	\label{p:GammaUCP-convex}
If   $(A_1,\dots,A_\gv)\in \mathbb{S}_\cH^\gv,$   
	then   $W^\Gat(A)$  and $\UCP^\Gat(\cR^\Gat_A)$ are isomorphic \bdcpt $\Gamma$-convex sets.
\end{proposition}

\begin{proof}
  \CB{Let $\cR=\cR^\Gat_A.$}
 To prove \CB{$\UCP^\Gat(\cR)$} is $\Gat$-convex,
  suppose \CB{$\varphi \in \UCP^\Gat_t(\cR)$} and $V:\C^s\to\C^t$ is an isometry such that $(\varphi,V)$ is a $\Gat$-pair.  It is immediate,
   from the definition of $\Gat$-pair, 
   that $\psi=V^* \varphi V$ is a $\Gat$-concomitant.
  Indeed, by definition of a $\Gat$-pair,   
  $V^* \gamma_i(\wvphi{A})V = \gamma_i(V^*\wvphi{A}V)$ 
  for $0\le i\le \rv.$
  Hence, as $\varphi$ is a $\Gat$-concomitant, 
 \[
    \psi(\gamma_i(A)) = V^* \varphi({\gamma}_i(A)) V 
     = V^* \gamma_i(\wvphi{A})V = \gamma_i(V^* \wvphi{A}V)
     =\gamma_i({\psi}(A)).
 \]
  
  To see that $\psi$ is ucp, let $R=\sum Y_j\otimes \gamma_j(A)\in M_m(\cR)$ be given and observe,
 \[
  \psi_m(R) = (\Ii_m \otimes V)^* \varphi_m(R)(\Ii_m \otimes V).
 \]
 Hence if $R$ is positive semidefinite, then so is $\psi_m(R).$ Thus $\psi$ is ucp. Hence $\UCP^\Gat(\cR)$ is $\Gat$-convex.
  That the maps  $\Psi_s:\UCP^\Gat_s(\cR) \to W^\Gat_s(A)$ given by 
   $\Psi_s(\varphi) =\wvphi{A}$ are affine bijections 
  follows from the definitions. The details are omitted.

  The graded  set $\UCP^\Gat(\cR)$ is bounded as ucp maps have norm one.
   It is straightforward to check
   that $\UCP^\Gat_s(\cR)$ is closed so that $\UCP^\Gat_s(\cR)$ is compact.   
    Since $\Psi_s$ is evidently continuous, by compactness of its domain,
    $W^\Gat(A)$ is also compact and $\Psi_s$ is a homeomorphism. 
  \end{proof}

Given a  tuple $B\in \mathbb{S}_\cH^{\rv},$ let
\[
 \cD^\Gat_B =
 \big(\{X\in \mathbb{S}^\gv_m: I\otimes I +\sum_{j=1}^\rv  B_j\otimes \gamma_j(X) \succeq 0\}\big)_{m\in \N}.
\]
 We call $\cD^\Gat_B$ a \df{\Hilby  $\Gat$-spectrahedron}. \index{$\cD^\Gat_B$}

\begin{proposition}
\label{p:not:vac}
 The graded set $\cD_B^\Gat$ is $\Gat$-convex.

 Suppose $\bm K \subseteq\mathbb{S}^\gv$ is closed and  $\Gat$-convex and $\bm J \subseteq \mathbb{S}^\rv$
 is a closed matrix convex set that contains $\Gat(\bm K).$ If 
\begin{enumerate}[\rm (a)]
 \item  $0\in K_1;$
  \item $\Gamma(0)=0;$ 
   \item \label{i:not:vac:ii+} $0$ is in the interior of $J_1;$ and
  \item \label{i:not:vac:iii} 
     $X\in \mathbb{S}^\gv$ and $\Gat(X)\in \bm J$
    implies $X\in \bm K,$ 
 \end{enumerate}
  then  there exists a semi-finite tuple $A\in \mathbb{S}_\cH^\rv$ such that $\bm K =\cD^\Gat_A.$
  \end{proposition}

\begin{remark}\rm
\label{r:not:vac}
  (1) Note that $\bm J =\overline{\matco}(\Gat(\bm K))$ contains $\Gat(\bm K).$ In any case, the condition of item~\ref{i:not:vac:iii}
   of Proposition~\ref{p:not:vac} is, given the assumption that $\Gat(\bm K)\subseteq \bm J,$ equivalent to $\bm K =\Gat^{-1}(\bm J).$ 

  {\CCBB 
   (2) Proposition~\ref{p:not:vac} provides a characterization
   of \Hilby $\Gat$-spectrahedra assuming $\Gat(0)=0$
   stated in terms of the existence of a $\bm J$ satisfying
   the hypotheses of the proposition. 
   Indeed, for the converse, suppose $\bm K=\cD_A^\Gat$
   for some (possibly operator) tuple $A,$ set
   $\bm J=\cD_A$ and observe that $\Gat^{-1}(\bm J)=\bm K.$
   That $J_1$ contains a neighborhood of $0$ and
   $0\in K_1$ are automatic. 

   (3) If $\bm K$ is a \Hilby $\Gat$-spectrahedron,
    then each $K_n$ has non-empty interior. 
    The set $\bm K=\{(X,Y) \in \bbS^2: XY-YX=0, \, \|X\|, \, \|Y\| \le 1\}$
    is a closed and bounded $\Gat=(x,y,x^2,y^2)$
    convex set. See Example~\ref{ex:x^2y^2}. Since
    $K_2$ has empty interior, it is not a 
    \Hilby $\Gat$-spectrahedron. So not every $\Gat$
    convex set that contains $0$ (or even $0$ in its
    interior at level $1$) is a \Hilby $\Gat$-spectrahedron.
   }
   \qed
\end{remark}

\CB{Before proving  Proposition \ref{p:not:vac}, we give a few examples.}

\begin{example}
In \cite[Theorem 2.6]{JKMMP} a condition on a $\Gat$-convex set $\bm K$ is given under which $\matco (\Gat(\bm K))$ contains a neighborhood of $0.$ Namely, if $\Gat(0)=0$ and $\bm{K}$ is a  $\Gat$-convex set containing $0,$ then $0$ is in the interior of $\matco (\Gat(\bm K))(1)$ if and only if the real span of $\{\gamma_j \ | \ 1 \leq j \leq \rv\}$ does not contain a polynomial $q \in \C\langle x \rangle$ such that $q(X) \succeq 0$ for all $X \in \bm{K}.$

This certificate applied to the case of  $xy$-convexity, where $\Gat=(x,y,xy+yx,i(xy-yx)),$ 
implies if $\bm K$ is $xy$-convex, $K_1$ contains a neighborhood of $0$ and, for some $n,$
there is a pair $(X,Y)\in K_n$ such that $XY\ne YX,$  then  $\matco (\Gat(\bm K))(1)$ contains a neighborhood of $0.$
In particular, the result holds if $K_2$ also contains a neighborhood of $0$ as we now show. 

 Suppose $a,b,c\in\R$ and $ax+by + c xy  \ge 0$ for all $(x,y)$ in the non-empty interior of $K_1$.
 For $(x,y)$ sufficiently small in $K_1,$  we also have $(-x,-y), (x,-y)$ and $(-x,y)\in K_1.$ Thus $-ax-by+c\, xy \ge 0.$
 Hence $c\, xy\ge 0$ for all $(x,y)$ sufficiently small. Thus $c=0.$ A similar argument  gives $a=b=0$ too.   Thus,
 if we start with $a,b,c,d\in\R$ such that  $q(X,Y) = aX + bY + c\, (XY+ YX) +i\, d(XY-YX)  \succeq 0$ on $\bm K$ and $K_1$ has an interior, then
 we conclude that $a=b= c =0.$    Thus $q(x,y) =  i\,d(xy-yx),$ for some real number $d$ and now $q(x,y)\succeq 0$
  on $\bm K$ implies if $(x,y)\in K_n$ then $i\,d (xy-yx)\succeq 0.$ If there is an $n$ such that $K_n$ contains a pair $(X,Y)$
  such that $XY-YX\ne 0,$ then $i\,(XY-YX)$ has both positive and negative eigenvalues and
  thus $d=0.$ \qed
 \end{example}

\begin{example}\rm
\label{e:not:vac}
 Let  $K_n =\{(X,Y)\in \mathbb{S}_n^2: I-X^2-Y^4\succeq 0\}$ and $\bm K = (K_n)_n.$ It is straightforward
  to verify that $\bm K$ is $\Gat$-convex, where $\Gat(x,y)=(x,y,y^2);$ that is,  $\bm K$ is $y^2$-convex.
  
  Let $\bm J$ denote the closure of the matrix convex hull of $\Gamma(\bm K)$ and the point $(0,0,-1)\in \mathbb{S}^3_1=\mathbb{R}^3.$
  Let $E_k = (0,0,I_k)$ for positive integers $k.$ In particular, $-E_k\in J_k.$  Since $J_1$ contains the points $(\pm 1,0,0),$ $(0,\pm 1,1)$
  and $(0,0,-1)$ and is convex, it contains a neighborhood of $0.$  A routine argument shows that elements of $\bm J$ have
  the form
  \[
     \begin{pmatrix} V^* & W^* \end{pmatrix}  \, \begin{pmatrix} \Gat(X,Y) & 0 \\ 0 & -E_k \end{pmatrix}
     \,  \begin{pmatrix} V\\ W \end{pmatrix},
  \]
  for some positive integers $k,\ell,$ tuple $(X,Y)\in K_\ell,$ and maps $V$ and $W$ such that $V^*V+W^*W=I.$

 Suppose $\Gat(X,Y)=(X,Y,Y^2)$ is in  $\bm J.$ 
 Thus there exists a sequence 
  $(X_n,Y_n)\in K_{m_n}$ and $E_{k_n}\in J_{k_n}$ and isometries 
\[
   \cV_n =\begin{pmatrix} V_n \\ W_n \end{pmatrix}
\]
 such that 
 \[
     \cV_n^*  \begin{pmatrix} \Gat(X_n,Y_n) & 0 \\ 0  & -E_{k_n} \end{pmatrix} \cV_n
       = (V_n^*X_n V_n, V_n^* Y_n V_n, V_n^* Y_n^2 V_n - W_n^* W_n)
\]
   converges to $(X,Y,Y^2).$  Set 
 \[
  Z_n = \cV_n^*  \begin{pmatrix} Y_n^2 & 0 \\ 0  & -I_{k_n} \end{pmatrix} \cV_n = V_n^* Y_n^2 V_n - W_n^* W_n.
 \]
  Thus $Z_n$ converges to $Y^2$ and therefore $Z_n^2$ converges to $Y^4.$ 
  Moreover,
 \[
 \begin{split}
    Z_n^2  & \preceq \cV_n^* \,  \begin{pmatrix} Y_n^2 & 0 \\ 0  & -I_{k_n} \end{pmatrix}^2 \, \cV_n
     \\ & = \cV_n^*  \begin{pmatrix} Y_n^4 & 0 \\ 0  & I_{k_n} \end{pmatrix} \cV_n
    \\ & = V_n^* Y_n^4 V_n + W_n^* W_n.
 \end{split}
 \]  
  Thus, as $(X_n,Y_n)\in \bm K$ so that $X_n^2+Y_n^4 \preceq I,$
 \[
   (V_n^* X_n V_n)^2 +  Z_n^2  \preceq  V_n^* \big ( X_n^2 + Y_n^4)V_n  +W_n^*W_n 
   \preceq V_n^* V_n +W_n^*W_n =I.
 \]
 Hence, taking the limit on $n$ gives $X^2+Y^4\preceq I.$ Since $\bm K$ is closed,  $(X,Y)\in \bm K$
 and consequently $\bm J$ and $\Gat$ satisfy the hypotheses of Proposition~\ref{p:not:vac}.
 It follows that there exists a semi-finite tuple $A$ such that $\bm K =\cD^\Gat_A.$
 
 Of course, as is easily verified, $\bm K$ is determined by the $\Gat$-linear matrix inequality,
 \[
   I + B_0 x + B_1 y +B_2 y^2 =   \begin{pmatrix} 1 & x & y^2 \\ x & 1 & 0 \\ y^2 & 0 & 1 \end{pmatrix}
 \]
  so that $\bm K = \cD_B^\Gat.$
  
  As a further remark, the graded set $\bm J =\overline{\matco}(\Gat(K))$ satisfies all the hypotheses of Proposition~\ref{p:not:vac},
  save for item~\ref{i:not:vac:ii+} since if  $(x,y,z)\in  \matco(\Gat(\bm K))(1),$ then $z\ge 0.$
\qed
\end{example}

\begin{proof}[Proof of Proposition~\ref{p:not:vac}]
  If $(X,V)$ is a $\Gat$-pair and $X\in \cD_B^\Gat,$ then
 \[
\begin{split}
  I\otimes I  + \sum_{j=1}^\rv  B_j\otimes \gamma_j(V^*XV)  
   &  = I\otimes I + \sum_{j=1}^\rv  B_j\otimes V^* \gamma_j(X)V
   \\ & = [I\otimes V]^* \left (I\otimes I + \sum_{j=1}^\rv   B_j\otimes \gamma_j(X)  \right ) [I\otimes V] \succeq 0.
 \end{split}
 \]
 Thus $V^*XV\in \cD_B^\Gat$ and thus $\cD_B^\Gat$ is $\Gat$-convex.

  The hypotheses  imply that $\bm J$ is a closed matrix convex set with $0$ in its
    interior, certifying an application of Lemma~\ref{l:semi-finite}. 
    Hence there exists a semi-finite tuple $A\in \mathbb S^\rv_{\mathcal H}$ 
    for some separable Hilbert space $\mathcal H,$ such that $\bm J = \cD_A.$   If $X\in\bm K,$
   then $\Gat(X)\in \matco(\Gat(\bm K))\subseteq \bm J=\cD_A$ and thus $X\in \cD_A^\Gat;$ i.e.,
 \[
    I + \sum  A_j\otimes  \gat_j(X)  \succeq 0. 
 \]
  Conversely,  if   $X\in \mathbb{S}^\gv\setminus \bm K,$ 
  then  the hypothesis of item~\ref{i:not:vac:iii} 
   implies $\Gat(X)\notin \bm J = \cD_A$ and so $X\notin \cD_A^\Gat.$ Thus 
    $\bm K =\cD_A^\Gat$ for some semi-finite tuple $A$ as claimed.
\end{proof}

\subsubsection{Duality} 
For any \bdcpt $\Gat$-convex set \CB{$\bm{K}\subset \gtup,$} let
\begin{equation}
 \label{e:wideY}
\widehat{Y} = \bigoplus_{Y \in \bm{K}}Y \quad\text{ and }\quad   I = \bigoplus_{Y \in \bm{K}} \Ii_{\siz(Y)},
\end{equation}
and associate to $\bm{K}$ the $\Gat$-operator system $\widehat{\bm{K}}= $ span$\{I=\gamma_0(\widehat{Y}),\gamma_1(\widehat{Y}),\ldots,\gamma_\rv(\widehat{Y})\}.$
 \index{$\widehat{Y}$}
The assumption that $\bm K$ is bounded ensures each $\widehat{Y}_j$ is a bounded operator.

\begin{theorem}[\textbf{Webster-Winkler duality for $\Gat$-convex sets}]
 \label{th:g-dual} 
 The operations  \, $\widehat{}$ \,
of \eqref{e:wideY}
  and \, $\widecheck{}$ \, of \eqref{eq:range} are dual to one another: 
 \begin{enumerate}[\rm (a)]
  \item \label{i:g-dual:a}
    Suppose $A = (A_1,\ldots,A_\gv)\in \mathbb{S}_\cH^\gv$ is semi-finite  and let $\cR$ denote the span of  $\{\Ii_\cH,\gamma_1(A),\ldots,\gamma_\rv(A)\}.$  If 
{\CCBB $\Ii_\cH,\gamma_1(A),\ldots,\gamma_\rv(A)$ are linearly independent}
(e.g.,     $\cD_{\Gat(A)}$ is bounded), 
 then $\cR$ and $\widehat{\widecheck{\cR}}$ are isomorphic $\Gat$-operator systems. 

 \item \label{i:g-dual:b}  Suppose $\Gamma(0)=0$ and let $\bm{K}\subset \gtup$ denote a  \bdcpt  $\Gat$-convex set with  $0\in K_1.$  
 {If $\bm K =\Gat^{-1}(\overline{\matco}(\Gat(\bm K))),$}
 then 
 $\bm{K} = \widecheck{\widehat{\bm{K}}} = W^\Gat(\widehat{Y}) $
for $\widehat{Y}$ defined as in equation~\eqref{e:wideY}.
\end{enumerate}
\end{theorem}

{\CCBB Theorem~\ref{th:g-dual} gives the object-level duality.  The corresponding
morphism-level statement, yielding a contravariant equivalence of categories, is
given as Theorem \ref{th:gamma-categorical-duality} in Appendix~\ref{app:categorical-duality}.}

\begin{remark}\rm
\label{r:221}
\mbox{}\par
{\CCBB (1) In item~\ref{i:g-dual:a} of Theorem~\ref{th:g-dual} the semi-finite assumption is needed. Consider $\Gamma=(x,y,x^2,y^2)$ (as in Example \ref{ex:x^2y^2}).
Let $D$ denote a diagonal operator on $\ell^2=\ell^2(\N_0)$ with diagonal entries
$d_0=1$ and $(d_n)_n$ a decreasing strictly monotonically to $0.$ Thus $D$ is compact. Let $S$ denote the shift.
Let $A_1=D$ and $A_2=S^*D+DS.$ Thus $A_1,A_2$ are both compact and self-adjoint.  
An exercise shows that the $C^*$-algebra
generated by $(A_1,A_2)$ is the algebra of compact operators, $\mathcal K(\ell^2(\mathbb N_0)).$  
Indeed, by the functional calculus, the rank-one projection $P_n$ onto
$\mathbb C e_n$ belongs to $C^*(D)$ for every $n$. Also,
\[
   P_{n+1}A_2P_n=d_{n+1}E_{n+1,n}.
\]
Since $d_{n+1}\neq 0$, it follows that $E_{n+1,n}\in C^*(A_1,A_2)$.
Taking adjoints gives $E_{n,n+1}$, and products of these adjacent matrix
units give all $E_{ij}$. Hence $C^*(A_1,A_2)$ contains the finite-rank
operators. Since $A_1,A_2$ are compact, the reverse inclusion is immediate,
and therefore $C^*(A_1,A_2)=\mathcal K(\ell^2(\mathbb N_0))$.

Suppose $\varphi\in \UCP^\Gat_n(\cR^\Gat_A).$
By the Stinespring-Arveson theorem (\cite[Theorem 4.1 and Theorem 7.5]{Pa}),
 there is a representation
$\pi$ {of $B(\ell^2)$ on a Hilbert space $\mathscr{K}$ and an isometry $V:\C^n\to \mathscr{K}$} such that $\varphi(\cdot)
= V^* \pi(\cdot)V.$
As in Example \ref{ex:x^2y^2}, 
the range of $V$ reduces the pair of self-adjoint 
 operators $\pi(A_1),\pi(A_2).$  
 Thus the map $\psi$
 on the $C^*$-algebra generated by $(A_1,A_2)$ defined
 by $\psi(X)=V^* \pi(X)V$ is a representation
 $\mathcal K(\ell^2)\to M_n(\C)$. 
Since {the $C^*$-algebra
$\mathcal K(\ell^2)$  is simple} it  has no nonzero finite-dimensional
representations, whence $\psi=0$. In particular,
$
   \varphi(A_1)=\varphi(A_2)=0.
$
Hence
$
   W_n^\Gat(A)=\{(0,0)\}
$
for every $n$, and so
\[
   \widehat{\widecheck{\cR}}=\mathbb C I.
\]
On the other hand,
\[
  \cR= \cR_A^\Gat
   =
   \spann\{I,A_1,A_2,A_1^2,A_2^2\}
\]
has dimension $5$. Thus $\cR$ and $\widehat{\widecheck{\cR}}$ are not isomorphic
$\Gamma$-operator systems.}

{\CCBB We also point out that in the present context
of $\Gat(x,y)=(x,y,x^2,y^2)$-convexity, if $\phi:\cR^\Gat_A \to M_n$ is 
a $\Gamma$-ucp map, then $\phi(A_j^2)=\phi(A_j)^2.$ Thus, identifying $\phi$
with its extension to 
the $C^*$-algebra
$C^*(A)$ 
generated by $A,$  both $A_1$ and $A_2$ are in the 
multiplicative domain of $\phi.$ Since both $A_1$ and $A_2$
are self-adjoint, it follows that $\phi$ is a representation
of $C^*(A)$  \cite[Theorem~3.18]{Pa}.}

\smallskip

(2)
 In item~\ref{i:g-dual:b} of Theorem~\ref{th:g-dual}, if the matrix convex hull of $\Gat(K)$ is closed, then, by Theorem~\ref{t:thm:jp-intro}, it is the case
 that $\Gat(X)$ is in the closure of the matrix convex hull of $\Gat(\bm K)$ if and only if $X\in \bm K.$ 

 \smallskip
 
 (3)
 {\CCBB 
    The isomorphism $\tau$ between $\cR$ and $\widehat{\widecheck{\cR}}$ 
constructed in the proof of Theorem~\ref{th:g-dual}
item~\ref{i:g-dual:a} below
    has the additional
    property that if $\varphi:\widehat{\widecheck{\cR}}\to M_n$ is $\Gat$-ucp, then so is $\varphi\circ\tau.$
    }
  \qed
\end{remark}

\begin{proof}[Proof of Theorem~\ref{th:g-dual}]
Set
\[
	\widehat{B} = \bigoplus_{B \in W^\Gat(A)} B \quad \text{ and }\quad  I = \bigoplus_{B \in W^\Gat(A)} \Ii_{\siz(B)}
\]
 and let $\cK$ denote the space that the $\widehat{B}_j$ act on. 
From the definitions,  
\[
  \widehat{\widecheck{\cR}} = \spann\{\gamma_0(\widehat{B})=\Ii_\cK,\gamma_1(\widehat{B}), \ldots, \gamma_\rv(\widehat{B})\}.
 \]
For notational convenience, let $A_0=\Ii_\cH$ and
$\widehat{B}_0=\Ii_\cK.$
Since the operators
$\Ii_\cH,\gamma_1(A),\ldots,\gamma_\rv(A)$ are linearly independent,
 there is a unital map  $\tau: \cR \to \widehat{\widecheck{\cR}}$ determined by $\tau(\gamma_j(A)) = \gamma_j(\widehat{B}).$
By construction, $\tau$ is bijective and a $\Gat$-concomitant. {\CCBB A bit more is true as noted
in (3) of Remark~\ref{r:221}. If $\varphi:\widehat{\widecheck{\cR}}\to M_n$ is $\Gat$-ucp,
then
\[
(\varphi\circ \tau) \circ\gamma_j(A) 
 = \varphi(\gamma_j(\widehat{B})) 
 = \gamma_j(\varphi(\widehat{B}))
 = \gamma_j \circ (\varphi \circ \tau)(A).
\]
Hence $\varphi\circ \tau$ is a $\Gat$-concomitant
and hence $\Gat$-ucp.
}

Suppose $R=\sum X_j \otimes \gamma_j(A)\in M_m(\cR)$ is positive semidefinite. Given $B\in W^\Gat(A),$ there exists a $\Gat$-ucp map
 $\varphi$ such that $B=\wvphi{A}.$ Hence, 
 \begin{equation*}
 0\preceq  \varphi(R) = \sum X_j \otimes \varphi(\gamma_j(A)) = \sum X_j \otimes \gamma_j(\wvphi{A}) =\sum X_j\otimes \gamma_j(B).
 \end{equation*}
Therefore,
\[
 0\preceq \sum  X_j\otimes \gamma_j(\widehat{B}) =\tau(R)
\]
and it follows that $\tau$ is completely positive. On the other hand, $A$ is an at most countable direct sum $A=\oplus A^{(i)}$
 where $A^{(i)}\in \mathbb{S}_{n_i}^\gv.$  Thus if $R\not\succeq 0,$ then there is some $i$ such that $R^\prime =\sum X_j\otimes \gamma_j(A^{(i)}) \not\succeq 0.$
 Since the unital map $\psi$  sending  $A$ to $A^{(i)}$ is $\Gat$-ucp, it follows that $\psi(A)=A^{(i)}\in W^\Gat(A).$
 {\CCBB 
   Therefore $A^{(i)}$ is a summand of $\widehat{B}$ and consequently $\tau(R)\not\succeq 0;$ that is,
  $R\not\succeq 0$ implies $\tau(R)\not\succeq 0.$ }
  Hence the inverse of $\tau$ is also completely positive and the proof of item~\ref{i:g-dual:a} is complete.

 Turning to item~\ref{i:g-dual:b},  recall that $\widehat{\bm{K}}= $ span$\{I,\gamma_1(\widehat{Y}),\ldots,\gamma_\rv(\widehat{Y})\}.$
	For any $n \in \N$ and $B \in K_n$ define the linear map $\varphi : \widehat{\bm{K}} \to \Mm_n$ by
	\begin{align*}
		I_{\mathscr{Y}} &\mapsto \Ii_n\\
		\gamma_i(\widehat{Y}) &\mapsto \gamma_i(B),
	\end{align*}
where $\mathscr{Y}$ is the space that the $\widehat{Y}_j$ act on. 
By construction of $\widehat{Y},$ there is an isometry $V$ reducing $\widehat{Y}$ such that $B = V^*\widehat{Y}V.$ Thus $(\widehat{Y}, V)$ is a $\Gat$-pair and $\varphi$ is a $\Gat$-ucp with image $B.$  Thus $B\in W^\Gat(\widehat{Y}).$

Conversely, let $B \in W_\ell^\Gat(\widehat{Y})$ be the image of $\widehat{Y}$ by a $\Gat$-ucp map $\varphi.$ 
 Arguing by contradiction, suppose $B \notin K_\ell.$  
  The hypothesis that $\bm K =\Gat^{-1}(\overline{\matco}(\Gat(\bm K)))$ 
   implies  $\Gat(B)$ is not in the closed matrix convex set  $\bm J=\overline{\matco}(\Gat(\bm K));$ that is $\Gat(B)\notin J_\ell.$
   By Theorem~\ref{t:EW}, there exists a tuple $C=(C_1,\dots,C_\rv)\in \mathbb{S}^\rv_\ell$ such that 
  $L_C(\Gat(Y)) \succeq 0$ for all $Y\in \bm K,$ but $L_C(\Gat(B))\not\succeq 0.$  
Setting
\[
 \LG_C(X) = \Ii_\ell \otimes \Ii_{\siz(X)}  + \sum_{i=1}^\rv C_i \otimes \gamma_i(X)
 \]
 and applying the canonical shuffle gives, 
\[
  L_{\Gat(Y)} (C) \uapprox L_C(\Gat(Y)) =  \LG_C(Y) \succeq 0
  \]
   for all $Y \in \bm{K},$ but 
\[
 L_{\Gat(B)}(C) \uapprox  L_C(\Gat(B)) =  \LG_C(B)   \nsucceq 0,
\]
 where  \df{$\uapprox $} indicates unitary equivalence. 
It follows that $L_{\Gat(\widehat{Y})} (C) \succeq 0,$ but $L_{\Gat(B)}(C) \nsucceq 0,$ 
showing that 
$\cD_{\Gat(\widehat{Y})}\not\subseteq \cD_{\Gat(B)}.$ 
  Thus $\varphi$ is not $\Gat$-ucp  (see Remark~\ref{rem:GatUCP}). 
  Hence, if $B\in W^\Gat_\ell(\widehat{Y}),$ then $B\in K_\ell$ completing
   the proof that $\bm K =W^\Gat(\widehat{Y}).$
\end{proof}

\subsubsection{\texorpdfstring{$\Gat$}{Gamma}-polar dual} 
  Given the tuple $A=(A_1,\dots,A_\gv)\in \mathbb{S}^\gv_\cH,$ the \df{$\Gat$-polar dual}  of $\cD_{\Gat(A)}$ is, by definition,
   the sequence $\cD_{A}^{\Gat\mhyphen\circ} = (\cD_{A}^{\Gat\mhyphen\circ}(n))_n,$ where
 \[
  \cD_{A}^{\Gat\mhyphen\circ}(n) =\{X \in {\mathbb{S}^\gv}: L_B(\Gat(X))= I_{\siz{B}} \otimes I_n +\sum_{j=1}^\rv B_j\otimes \gamma_j(X) \succeq 0 
     \mbox{ for all } B\in \cD_{\Gat(A)} \}.  \index{$\cD_{A}^{\Gat\mhyphen\circ}$}
 \]
{\CCBB  Evidently this definition reduces to the usual
 polar dual $\cD_A^\circ$ in the case $\Gat(x)=x.$}
 
\begin{proposition}
\label{p:Gat-polar-dual:1}
 With notations above, if $\cD_{\Gat(A)}$ is bounded, then   $\cD_{A}^{\Gat\mhyphen\circ} = W^\Gamma(A).$
 \end{proposition}

\begin{proof}
   Observe, using Proposition~\ref{prop:cp}, Remark~\ref{rem:GatUCP} and the boundedness assumption on $\cD_{\Gat(A)},$
    that  $X\in \cD_{A}^{\Gat\mhyphen\circ}$ if and only if $L_{B}(\Gat(X))\succeq 0$ for all $B\in \cD_{\Gat(A)}$
   if and only if the unital $\Gat$-concomitant map  $\varphi$ that sends $\gamma_i(A)$ to $\gamma_i(X)$
    is $\Gat$-ucp  if and only if 
    $X=\wvphi{A}\in W^\Gamma(A).$
\end{proof}

Adding a hypothesis produces a bipolar result.

\begin{proposition}
\label{p:Gat-polar-dual:2}
  With notations above, if $A$ is semi-finite and $\cD_\Gat(A)$ is bounded,  then 
\begin{equation}\label{eq:gatbipolar}
 \{B\in \mathbb{S}^{\rv}: L_B(\Gat(X))\succeq 0 \text{ for all } X\in \cD_{A}^{\Gat\mhyphen\circ} \} = \cD_{\Gat(A)}.
\end{equation}
\end{proposition}

\begin{proof}
 Let $S$ denote the left hand side of \eqref{eq:gatbipolar}.
 Suppose $B\in \cD_{\Gat(A)}.$ For $X\in \cD_{A}^{\Gat\mhyphen\circ},$ we have  $L_C(\Gat(X))\succeq 0$ for all $C\in \cD_{\Gat(A)}.$
 Choosing  $C=B$ gives $B\in S$.
  
  Conversely, suppose $B\in S.$ Since $A$ is semi-finite, $A=\oplus_{j=1}^\infty A^{(j)}$ for some tuples $A^{(j)}\in \mathbb{S}_{m_j}^\gv.$ Moreover, $L_B(\Gat(A^{(j)}))\succeq 0$
   since each $A^{(j)}\in \cD_{A}^{\Gat\mhyphen\circ}$ and $B\in S.$  Thus $L_B(\Gat(A))\succeq 0$ and $B\in \cD_{\Gat(A)}.$
\end{proof}

{\CCBB
 This subsection concludes with two examples that
  explain the need for the boundedness hypotheses
  in Propositions~\ref{p:Gat-polar-dual:1} and
  \ref{p:Gat-polar-dual:2}. Namely, Lemma~\ref{l:level:bounded} does not necessarily
  hold for $\Gat$-convex sets.
  Compactness (level-wise) of a $\Gat$-convex set $\bm K$ does not necessarily imply it is bounded (uniformly).

\begin{example}
Let
$
   \Gamma(x,y)=(x,y,x^2,y^2).
$
For $d\in\mathbb N$, let
$
   D_d=\operatorname{diag}(1,2,\ldots,d)\in \bbS_d
$
and let
\[
   T_d=\sum_{j=1}^{d-1}(E_{j,j+1}+E_{j+1,j})\in \bbS_d.
\]
Set
$
   A^{(d)}=(D_d,T_d)\in \bbS_d^2.
$
Then $A^{(d)}$ is irreducible., i.e., $C^*(D_d,T_d)=M_d(\C)$.

Define a graded set $\bm K=(K_n)_n\subseteq \bbS^2$ as follows. A tuple
$X\in K_n$ if and only if $X$ is unitarily equivalent to a finite direct sum
\begin{equation}\label{eq:bigDS}
   \bigoplus_{d=1}^n \bigl(A^{(d)}\bigr)^{\oplus m_d},
   \qquad
   m_d\in\mathbb N_0,
   \qquad
   \sum_{d=1}^n d\,m_d=n.
\end{equation}
Then $\bm K$ is a free set. It is closed under direct sums by construction and
closed under simultaneous unitary conjugation by definition.

By Example \ref{ex:x^2y^2},  a $\Gamma$-pair is precisely a compression to
a common reducing subspace. Since each $A^{(d)}$ is irreducible, and since
the $A^{(d)}$'s have different sizes, a reducing subspace of a finite direct
sum as in \eqref{eq:bigDS}
is unitarily equivalent to a direct sum
\[
   \bigoplus_{d=1}^n \bigl(A^{(d)}\bigr)^{\oplus k_d},
   \qquad
   0\leq k_d\leq m_d.
\]
Hence $\bm K$ is closed under all $\Gamma$-compressions, and is therefore 
$\Gamma$-convex.

For each fixed $n$, the set $K_n$ is compact. Indeed, there are only
finitely many multiplicity vectors
\[
   (m_1,\ldots,m_n)\in\mathbb N_0^n
   \quad\text{with}\quad
   \sum_{d=1}^n d\,m_d=n.
\]
For each such vector, the corresponding unitary orbit is compact, and $K_n$
is a finite union of such compact unitary orbits.
On the other hand, $\bm K$ is not uniformly bounded. Indeed,
$
   A^{(n)}\in K_n
$
and
$
   \|A^{(n)}\|\geq \|D_n\|=n,
$
so the norms of elements of $\bm K$ are unbounded over the levels. Thus $\bm K$ is
level-wise compact, but not bounded.
\end{example}

\begin{example}
Let
$
   \Gamma(x,y)=(x,y,y^2).
$
For a self-adjoint matrix $Y$, write its spectral decomposition as
\[
   Y=\sum_{\lambda\in\sigma(Y)}\lambda P_\lambda .
\]
For each $n$, define $K_n\subseteq \bbS_n^2$ to consist of all pairs
$(X,Y)$ such that
\[
   0\preceq Y\preceq I
\]
and, with respect to the spectral projections 
$P_\lambda$
of $Y$,
\begin{equation}\label{eq:compressPl}
\begin{split}
   P_\lambda XP_\lambda& =0
   \qquad(\lambda\in\sigma(Y)) \\
   \|P_\lambda XP_\mu\|& \leq |\lambda-\mu|
   \qquad(\lambda\neq\mu).
\end{split}
\end{equation}
Set
$   \bm K=(K_n)_{n\in\mathbb N}.
$
The defining conditions are invariant under simultaneous unitary conjugation.
They are also preserved under direct sums, since the spectral projection of
$Y\oplus Y'$ corresponding to $\lambda$ is $P_\lambda\oplus P'_\lambda$.
Thus $\bm K$ is a free set.

We next show that $\bm K$ is $\Gamma$-convex. Since
$\Gamma=(x,y,y^2)$, an isometry $V:\C^m\to\C^n$ is a
$\Gamma$-pair for $(X,Y)$ precisely when
$\range(V)$ reduces $Y$. Thus a
$\Gamma$-compression of $(X,Y)$ is exactly a compression to a reducing
subspace for $Y$.

Let $M$ be a reducing subspace for $Y$. Then
\[
   M=\bigoplus_{\lambda\in\sigma(Y)} M_\lambda,
   \qquad
   M_\lambda\subseteq P_\lambda\C^n.
\]
The compression of $Y$ to $M$ has spectral projections the compressions of
the $P_\lambda$'s to the subspaces $M_\lambda$. 
Clearly, the conditions \eqref{eq:compressPl} hold after compression to $M$. Hence $\bm K$ is closed
under $\Gamma$-compressions. That is, $\bm K$ is $\Gamma$-convex.

For each fixed $n$, the set $K_n$ is compact. Indeed, 
 $0\preceq Y\preceq I$,
and the defining inequalities  \eqref{eq:compressPl} imply
$
   \|X\|\leq n.
$
Thus $K_n$ is bounded. It is also closed: if $(X_j,Y_j)\to (X,Y)$, then the
spectral block inequalities pass to the limit by the finite-dimensional
spectral calculus. In particular, when eigenvalues coalesce, the inequality of \eqref{eq:compressPl}
forces the corresponding off-diagonal blocks to vanish in the limit. Hence
$(X,Y)\in K_n$. Therefore each $K_n$ is compact.

However, $\bm K$ is not uniformly bounded. For $d\in\mathbb N$, set
\[
   Y_d=\frac1d\operatorname{diag}\left(1,2,\ldots,d\right)
\]
and define $X_d\in \bbS_d$ by
\[
  (X_d)_{i,j} =\frac1d |i-j|.
\]
Then $(X_d,Y_d)\in K_d$ since the spectral projections of $Y_d$ are the
coordinate projections.
On the other hand, choosing
\[
   \mathbf 1_d=(1,\ldots,1)^T\in\C^d,
\]
gives
\[
   \frac{\langle X_d\mathbf 1_d,\mathbf 1_d\rangle}{\|\mathbf 1_d\|^2}
   =
   \frac1{d^2}\sum_{i,j=1}^d |i-j|
   =
   \frac{d-\frac1d}{3}.
\]
Therefore
$
   \|X_d\|\geq \frac{d-\frac1d}{3}.
$
Thus $\bm K$ is level-wise compact but not bounded.
\end{example}
}

\section{Extreme points of \texorpdfstring{$\Gat$}{Gamma}-convex sets} \label{sec: g-ext}
In this section we introduce a notion of an extreme point 
and establish an analog of the Krein-Milman Theorem for  a $\Gamma$-convex set $\bm K =(K_n).$  Doing so
requires the  addition of a countably infinite level $K_\infty,$ even in  the case of a matrix convex set ($\Gamma(x)=x$)
\cite{Kr, Evert-no-ext, DK}.  The need for this additional level  is discussed in further detail 
 as  Remark~\ref{r:need:for:inf} at the end of Subsection~\ref{ssec:FKM}.  The notion of extreme point here is closely related to
 that of an Arveson boundary point \cite{A1,A2,Ham79}. It was introduced by Kleski \cite{Kle14} 
 and was recently studied and adapted  in \cite{EHKM, EH, DK, EEHK} for examples.

\subsection{Pascoe's SOT Bolzano-Weierstraß modulo unitary similarity theorem} 
To tackle the problem of potential insufficiency of free extreme points at finite levels we consider matrix convex sets with an added infinite-dimensional component
and take advantage of \cite[Theorem~6.4.2]{DK}.  This additional level is a subset of $\mathbb S_{\cHi}^\gv=\cB(\cHi)_{\sa}^\gv,$ the set of $\gv$-tuples  of self-adjoint operators on an infinite-dimensional separable complex Hilbert space $\cHi.$  
 For convenience, in this section $\cH_k$ will denote a Hilbert space of finite dimension $k \in \mathbb{N}.$

 Pascoe's SOT Bolzano-Weierstraß Modulo Unitary Similarity Theorem (SOT-BW-MUST), will play an important role in what follows. 
  A version of the result  appears in \cite{Man} (see Lemma~4.5 and Remark~4.6) and is used in \cite{JKMMP}.  
    The  result naturally
  extends to sequences $(X_j)_j$ {\CCBB \cite{Pascoe}} as opposed to finite tuples {\CCBB as stated here. On the
  other hand, it does not extend to uncountably many $j,$
  and thus to non-separable Hilbert space, 
  as an example in \cite{Pascoe} shows.} We use the usual abbreviation \df{SOT} for the \df{strong operator topology}.  
  
\begin{theorem}[Pascoe's SOT-BW-MUST] 
 \label{t:sot-bw-must}
   Suppose $\cH$ is an infinite-dimensional separable Hilbert space and  $(X^{(j)})_j$ is  a sequence from $\cB(\cH)^\gv.$ 
   If  the sequence $(X^{(j)})_j$ is bounded, then there is a subsequence $(Y^{(n)})_n$ of $(X^{(j)})_j$ 
   and a sequence of unitary mappings $(U_n)_n$ on $\cH$ such that $U_n^*$ converges SOT to an isometry
   and $U_n^*Y^{(n)} U_n$ converges SOT. 
   
    In particular, if $S\subseteq \cB(\cH)^\gv$ is SOT-closed,  bounded and closed under isometric  conjugation (meaning if 
     $Y\in S$ and $W\in \cB(\cH)$ is an isometry, then $W^* YW\in S$), then $S$ is  WOT-compact.
\end{theorem}

\begin{remark}\label{rem : top}
    Because of Theorem~\ref{t:sot-bw-must},  the strong operator topology on free subsets of $\mathbb{S}^{\gv}$ is often most natural for the purposes here.
       However,  both the  \df{\wstar  topology}  (synonymously \df{ultra-weak}) and \df{weak operator topology} (\df{WOT}) play a role.
        On  norm-bounded sets, the WOT and \wstar topology coincide. The SOT is stronger than the WOT and therefore,
         a (norm) bounded {\CCBB 
         WOT-closed set is SOT-closed.}
         On the other hand, the SOT and WOT have the  same closed convex sets
        (since they have the same continuous linear functionals).  
        Thus a norm bounded convex set that is closed in any one of the WOT, SOT or \wstar topology is closed in all three. 
         In particular, the closure of a norm bounded convex set is the same in all three of these topologies.
	  
	 For a separable Hilbert space, as is the case here, all three topologies are metrizable on bounded sets. 
\qed
\end{remark}

\subsection{Preliminaries on free extreme points} 
As a convention in this section, let $\mathbb{S}^\gv$ denote the graded set $(\mathbb{S}^\gv_n)_{n\in \N \cup \{\infty\}},$ 
where $\mathbb{S}^\gv_n$ is identified
with $\cB(\cH_n)^\gv_{\sa}$ for $n\in \N$ and $\cH_\infty =\cH.$ 
A  graded set $\bm{K} = (K_n)_{n \in \N \cup \{\infty\}} \subseteq \mathbb{S}^\gv$ 
is \df{closed} if each $K_n$ is closed in the specified  topology. We call $\bm{K}$
\df{bounded} if 
there is $R\in\R_{\geq0}$ such that $\|X\|\leq R$ (see Definition~\ref{def gamma})  for all
$X\in\bm K$.
Under mild and natural assumptions (e.g., $\bm K$
is closed under  countable direct sums
(cf.~Remark \ref{rem:infSums}) or $K_\infty$ is bounded),  $\bm K$ is bounded.

\begin{definition}
 \label{d:opco}
   Suppose $\bm S, \, \bm K\subseteq \mathbb{S}^\gv.$ Thus $\bm K = (K_n)_{n\in\N\cup \{\infty\}}$ and $K_n\subseteq \mathbb{S}^\gv_n$ for each $n$
    and similarly for $\bm S.$
\begin{enumerate}[(a)]	
        \item  The graded set $\bm S$ is a \df{free set} if it is closed under unitary similarity and at most countable direct sums  of norm-bounded families.
        \item A free set $\bm S$ is \df{fully free} if it is closed under restrictions to reducing subspaces.
	\item  Given  $n,n_i\in \mathbb{N}\cup\{\infty\}$ and $A^{(1)},\ldots, A^{(k)} \in \bbS^{\gv}$ with $A^{(i)} \in \bbS^{\gv}_{n_i}$ for $1\le i \le k,$   an expression of the form
	\begin{equation}\label{eq-op comb}
		\sum_{i=1}^k V_i^\ast A^{(i)} V_i,
	\end{equation}
	where  $V_i \in {M}_{n_i,n}$ satisfy $\sum_{i=1}^k V_i^\ast V_i = {I}_n$, is an \df{operator convex combination} of  $A^{(1)},\ldots, A^{(k)}$. 
     \item The \textbf{(fully) free hull}  of  a graded set $\bm S$ is the smallest (fully) 
       free set containing $\bm S;$ \index{fully free hull} \index{free hull} \index{closed fully free hull}
      the  \textbf{closed} (in a specified topology) \textbf{fully free hull} is the smallest closed
       fully free set 
       containing $\bm S.$ \index{closed free hull}
      \item  A free set  $\bm{K}$ is  \textbf{operator convex} if it is closed under  
	operator convex combinations.
	Observe that each operator convex set is fully free.
	\item  The \df{operator convex hull} of a graded set $\bm S\subseteq \mathbb{S}^\gv$ is the intersection of all operator convex sets containing $\bm S,$
	 denoted $\opco(\bm S).$ Its  (level-wise) closure is denoted by  $\ovopco(\bm S)$, and a superscript such as SOT or $w^*$ may be added to specify the topology if it is not clear from the context.\qed
\end{enumerate}
\end{definition} 

\begin{remark}\rm
 For $\infty> R>0,$ the ball $\{X\in \mathbb{S}^\gv: \|X\|\le R\}$ is bounded, 
 fully free,  closed (in all the relevant topologies), and 
 operator convex. 
 It is clear that any graded set $\bm S$ is contained in a closed fully free
 operator convex set. 
 Since an intersection of (fully) free sets is (fully) free, the (fully) free hull of a graded set
 is also (fully) free. The same statement holds with operator convex hull instead of free hull.
 Thus these hulls exist and are bounded if $\bm S$ is.

 Proposition~\ref{l:cl:cvx} below contains some further initial observations.\qed
\end{remark}

\begin{proposition}
 \label{l:cl:cvx}
   A free set $\bm K$ is operator convex if and only if for each $m,n\in \N\cup\{\infty\},$ each  $X\in K_n$ and each isometry $V:\cH_m\to \cH_n,$
    the tuple $V^*XV\in K_m.$
 
   For a graded set $\bm S\subseteq \mathbb{S}^\gv,$ and  $\tau$ either the SOT, the WOT or the \wstar topology, the 
    graded set $\ovopco^{\tau}(\bm S)$ is operator convex and it is the smallest
    $\tau$-closed operator convex set containing $\bm S;$ that is, $\ovopco^{\tau}(\bm S)$ is the $\tau$-closed convex hull
     of $\bm S.$ 
     
     If $\bm F$ is a free set, then 
      \[ 
       \opco(\bm F) = \{V^* FV: F\in \bm F,\ V \textrm{ is an isometry}\}
    \]
     and if moreover, $\bm F$ is bounded, then so is $\opco(\bm F).$ 
    Finally, if $\bm F$ is a free set that is bounded and SOT-closed, then 
      $\opco(\bm F)$ 
     is bounded  SOT-closed and hence WOT
     and \wstar closed, since $\opco(\bm F)$  is norm bounded and convex.

     The level-wise SOT closure of a norm bounded fully free set is fully free. In particular, 
    if $\bm S$ is a bounded graded set, then its SOT-closed fully free hull is bounded and \CB{is}
     the level-wise closure of its fully free hull.
  \end{proposition}

\begin{proof}
The proof of the first statement is routine. Suppose $\bm K$ is free   and let  an operator convex combination as in \eqref{eq-op comb}  be given.
 Since $\bm K$ is closed under direct sums, $X=\oplus A^{(i)}$ is in $\bm K.$ Now the operator  $V=\col(V_1,\ldots, V_k)$
 is an isometry and $\sum V_i^*A^{(i)}V_i=V^*XV.$

For the second statement, let $X\in \ovopco^{\tau}(\bm S)$, and choose a net $(X^{(i)})_i$ in $\opco(\bm S)$ converging to $X$
{with respect to the topology $\tau.$} 
Given an isometry $V,$
the net $(V^*X^{(i)}V)_i$ is in $\opco(\bm S)$ and converges to $V^*XV\in\ovopco^{\tau}(\bm S)$.

The identity of the third  statement is a consequence of the first statement {using the assumption that $\bm F$
is free.}  If $\bm F$ is bounded, then $\bm F$ is contained in 
a bounded operator convex set and hence $\opco(\bm F)$ is bounded. 
To show that $\opco(\bm F)$ is SOT-closed assuming $\bm F$ is bounded and SOT-closed, 
we proceed as follows.
First note that, since $\bm F$ is bounded, so is $\bm G= \opco(\bm F).$ Since
each $G_m$ is a bounded subset of operators on a separable Hilbert space,
the SOT topology on $G_m$ is metrizable. It thus suffices to show,
if $Y^{(i)}$ is a sequence from $G_m$ that converges to some $Y,$ then $Y\in G_m.$
The proof of this fact will use the following observations for bounded  sequences
$(S_n)$ and $(T_n)$ of operators on Hilbert space.  If $(S_n)$ converges WOT and
$(T_n)$ converges SOT to $S$ and $T$ respectively, then $(S_n T_n)$ converges WOT to $ST;$
if both sequences converge SOT, then $(S_n T_n)$ converges SOT to $ST;$ 
and if $(S_n)$ converges SOT to $S,$ then $(S_n^*)$ converges {WOT} to $S^*.$

Without loss of generality, assume $\bm F\neq\emptyset$.
 Suppose $Y$ is in the SOT-closure of $\opco(\bm F)$ and let $\cY$ denote the  separable Hilbert space that $Y$ acts upon. 
 By assumption, there is a  {sequence} 
  $Y^{(i)}=V_i^*F^{(i)}V_i\in \opco(\bm F)$ that converges SOT to $Y,$  where   $F^{(i)}\in \bm F$ and $V_i$ are isometries.
  Let $\cE_j$ denote the separable Hilbert space that $F^{(j)}$ acts upon. Since the free set $\bm F$ is not empty,
 there is an $E\in F_\infty.$ Let   $\cE_0$ denote the separable infinite-dimensional Hilbert space that $E$ acts upon.
 The assumption
that $\bm F$ is closed under countable direct sums  justifies replacing $F^{(j)}$  and $V_j$ with 
\[
  \begin{pmatrix} F_j   & 0 \\ 0 & E\end{pmatrix}, \  \  \   \begin{pmatrix} V_j \\ 0 \end{pmatrix}
\]
acting on $\cE_j\oplus\cE_0$ and $\cY$ respectively;  
 and then closure with respect to unitary similarity justifies assuming that the $F^{(j)}$ all act upon the same infinite-dimensional
Hilbert space $\cH.$ Finally, once again using the assumption that $\bm F$ is  closed under unitary similarity,  the fact that the dimensions of the ranges of $V_j$ in $\cH$
are all the same justifies assuming that  $V_j=V:\cY\to \cH$ for some fixed isometry $V$ and all $j.$

By Theorem~\ref{t:sot-bw-must}, passing to a subsequence if needed, there exists a sequence 
 $(U_n)_n$ of unitary operators on $\cH$ such that $(U_n^* F^{(n)} U_n)_n$ and $(U_n^*)_n$  each converge SOT to some $F$ and
  isometry $U^*$.
  Moreover,  $U_n^*V$ converges SOT to an isometry $W$ and $(U_n^*V)^* = V^* U_n$ converges WOT to $W^*.$
 Finally, $[(U_n^* F^{(n)} U_n) U_n^*V]$ is the product of SOT convergent sequences and hence converges SOT to $FW.$
   Therefore, 
 \[
  \begin{split}
    Y & = \rmSOT\mhyphen \lim_n V^* F^{(n)}V
    \\ & = \rmSOT\mhyphen \lim_n V^* U_n U_n^* F^{(n)} U_nU_n^* V
    \\ & = \rmWOT\mhyphen \lim_n V^* U_n [(U_n^* F^{(n)} U_n) U_n^*V]
    \\ & = W^* FW. 
  \end{split}
 \]
  Hence $Y\in \opco(\bm F)$ as claimed.

  Since the SOT-closed set $\opco(\bm{F})$ is bounded and convex, it is WOT and  \wstar closed too.

To prove the fourth, and final statement, let $\bm F$ denote a fully free set and $\overline{\bm F}$
denote its level-wise SOT closure. Since the SOT topology is metrizable on bounded
sets in the case of separable Hilbert space,  given $F\in \overline{F_n}$ and a unitary $U$ of
 size $n,$ there exists
 a sequence $(F^{(m)})$ from $F_n$ that converges SOT to $F.$ Hence $(U^* F^{(m)} U)$ is a sequence from
 $\bm F$ that converges SOT to $U^*FU.$ So $U^*FU \in \overline{\bm F}$ and thus $\overline{\bm F}$
 is closed under unitary similarity.  

  Now suppose $\bm F$ is fully free and bounded by $R.$ 
  Given a sequence $(F^{(m)})$  from $\overline{\bm F},$ for each $m$ there is a sequence
  $(F^{(m,\ell)})_\ell$ that converges SOT to $F^{(m)},$ since $\bm F$ is a norm
   bounded subset of operators on a separable Hilbert space (so that the relative SOT
   topology on $\bm F$ is metrizable). Let $x=\oplus x_m$ denote a vector from
  the space that $\oplus F^{(m)}$ acts on and observe, for each $M,$ 
\[
 \| (\oplus_m (F^{(m,\ell)} - \oplus F^{(m)} )x \|^2 
  \le \sum_{m=1}^M \| (F^{(m,\ell)} -F^{(m)} )x_m\|^2 + 2 R \sum_{m=M+1}^\infty \|x_m\|^2,
\]
from which it readily follows that $\oplus_m F^{(m,\ell)}$ converges SOT to $\oplus_m F^{(m)}.$

 If $\bm S\subseteq \mathbb{S}^\gv$ is bounded by $R,$ then $\bm S$ is a subset of the fully free
  set $\bm R =\{X\in \mathbb{S}^\gv : \|X\|\le R\}.$ Hence the fully free hull of $\bm S$ is bounded
   by $R.$
\end{proof}

\begin{remark}\label{rem:infSums}
  If $\bm K\subseteq\mathbb{S}^\gv$ is closed under arbitrary (i.e., not only norm-bounded) countable direct sums, then $\bm K$ is bounded. Indeed, arguing the contrapositive,
  if there does not exist a $C$ such that $\|X\|\le C$ for all $X\in \bm K,$ then for each $n$ there exists $X^{(n)}\in \bm K$ with $\|X^{(n)}\|\ge n,$
   in which case $\oplus_n X^{(n)}$ is not a bounded operator and thus not in $\bm K.$ Conversely, assuming $\bm K$ is closed
   with respect to finite direct sums and is bounded along with some additional closure property implies $\bm K$ is closed with
    respect to countable direct sums.  We leave details to the interested reader and instead consider a similar result for operator
    convex combinations.
    
        A {\wstar} closed operator convex set $\bm K$ that contains $0$  is closed under arbitrary (i.e., not necessarily finite) convex combinations of any norm-bounded family. 
	Recall that if $0 \in K_1,$ then $\bm{K}$ is closed under conjugation by contractions (see, e.g., \cite[Lemma 2.3]{HKM16}, which remains valid in the infinite
	dimensional setting  with the same proof). Using notation as in equation~\eqref{eq-op comb}, but now with arbitrarily  many indices $i\in J,$
	 where $J$ is an index set,  
	 for any finite subset of indices $I\subseteq J$  we have $\sum_{i\in I}  V_{i}^*  V_{i} \preceq I$ and the sum  
\[
	a_I = \sum_{i \in I} V_{i}^* A^{(i)} V_{i}
\]
	 lies in $\bm K.$
	The net $(a_I)_I$ is \wstar convergent as we now explain. By the norm
	boundedness assumption on the family $(A^{(i)})_i,$ there is an $R>0$ such that $\|A^{(i)}\| \leq R$ for all $i.$ 
    In particular, $R-A^{(i)}_j \ge 0$ for each $i$ and $1\le j\le \gv.$  
	Consequently, with $R-A^{(i)}$ denoting the tuple with $j$-th entry $R-A^{(i)}_j,$  the net $(a^\prime_I)_I$ defined by
	$$
	a^\prime_I = \sum_{i \in I} V_{i}^* (R - A^{(i)}) V_{i}
	$$
	is an entrywise increasing 
	bounded net of self-adjoint operators and hence WOT-convergent by Vigier's theorem \cite[Theorem 4.1.1]{Mur90}.  
	Since $\sum_i V_i^* V_i = I,$ the net $(a_I)_I$ is also WOT-convergent. Now Remark \ref{rem : top} implies that $(a_I)_I$ is \wstar convergent 
	and hence its limit, denoted by $\sum_i V_i^* A^{(i)} V_i,$ lies in $\bm{K}.$
 
 By \cite[Proposition~3.7(b)]{JKMMP},  an SOT-closed operator convex set $\bm K$ containing $0$ is the SOT-closure of its matricial levels, $(K_n)_{n\in\N}.$
  \qed
\end{remark}

\subsection{Free extreme points and the Krein-Milman theorem for operator convex sets}
\label{ssec:FKM}
Free extreme points are intuitively those points of a matrix or operator convex set $\bm{K}$ that cannot be expressed as a nontrivial matrix convex combination of any finite subset of $\bm{K}.$  A notion of an nc convex set is introduced in \cite{DK} and  it  turns out that a \bdcpt operator convex set given its \wstar topology
is a  relatively simple example of a compact nc convex set over the dual operator algebra $\sR_\gv,$ where $\sR_\gv$ is the operator space known
as $\gv$-dimensional row Hilbert space.   See Appendix~\ref{sec:app}. 
The \cite{DK} notion of an \df{extreme point} of an nc convex set,  specialized to operator convex sets follows.
 For consistency with earlier work, cf.~\cite{EPS,EEHK}, we use the terminology \df{free extreme point}. 
 In finite dimensions (the matrix convex setting) free extreme points also go by absolute extreme points \cite{EHKM,EH,Kr}.

\begin{definition}[{\cite[Definition 6.1.1]{DK}}]
	\label{def:fext}
	Let $\bm{K} = (K_n)_{n \in \N  \cup \{\infty\}}$ be an operator convex set.  A tuple $X \in K_n,$ where $n \in \mathbb{N} \cup \{\infty\},$ 
	is a \df{free extreme point} (of $\bm K$)  if any expression
	of $X$ as an  operator convex combination 
	\begin{equation}\label{eq:fextdef}
		X = \sum_{i=1}^k V_i^\ast X^{(i)} V_i,
	\end{equation}
	where $X^{(i)} \in K_{n_i}$ and the $V_i \in {M}_{n_i,n}$ are all nonzero {and satisfy $\sum_{i=1}^k V_i^\ast V_i = {I}_n,$}
	 it is the case that, for each $i,$ the matrix $V_i$ is a 
	 scalar multiple of an isometry
	   $W_i \in \mathcal{M}_{n_i,n}$ satisfying $W_i^*X^{(i)}W_i = X$ and with respect to the range of $V_i,$ 
  \[
   X^{(i)} = Y^{(i)} \oplus Z^{(i)}
  \]
	 for some $Y^{(i)},Z^{(i)} \in \bm{K},$ where $Y^{(i)}$ is unitarily equivalent to $X.$ Let $\ext(\bm K)$ denote the graded set of free extreme points of $\bm K$.
\qed
\end{definition}

The next lemma follows from the definition of a free extreme point.

\begin{lemma}\label{lema:red}
	Suppose $\bm{K}$ is an operator convex set and  $X \in \bm{K}.$  If $X$ is a free extreme point, expressed as in \eqref{eq:fextdef}, then $V = \col(V_1,\ldots,V_k)=\col(V_i)$ reduces $\oplus_iX^{(i)}.$
 If $X\in K_n$ for $n$ finite, then $X$ is a free extreme point if and only if it satisfies the conditions of Definition~\ref{def:fext} with $n_i$ also finite.
\end{lemma}

\begin{proof}
  Using notation in Definition~\ref{def:fext}, note that $W_i W_i^*$ is the projection onto the
   range of $V_i$ and  there are constants $\lambda_i$ such
   that
  $V_i=\lambda_i W_i.$ Hence $ X^{(i)} W_i = W_i W_i^* X^{(i)} W_i = W_i X$ and 
\begin{equation}
\label{e:red}
   [\oplus_i X^{(i)}] \,  V  = \col{(\lambda_i X^{(i)}W_i)}  = \col(\lambda_i W_i X) = V\, X.
\end{equation}
  Hence the range of $V$ is invariant, and thus reducing, for $\oplus_i X^{(i)}.$

  Another routine argument based on equation~\eqref{e:red} proves the second statement.  
\end{proof}

\begin{remark}\rm
Any free extreme point $X \in K_n$ with $n \in \N$ is irreducible. Indeed, if $X$ is not irreducible, let $V_1$ denote the inclusion into $\cH_n$  of 
a proper nontrivial invariant subspace $\cV\subseteq \cH_n$ of $X$ and $V_2$  the inclusion of the orthogonal complement of $\cV$ into $\cH_n.$
 Thus $V_1$ and $V_2$ are isometries and $\sum V_j V_j^* =I_n.$
 Setting $X^{(j)} = V_j^*  X V_j\in \bm K,$  
\[
	X =  V_1 X^{(1)} V_1^*  + V_2 X^{(2)} V_2^*.
\]
Since $X$ is free extreme and the size of $X^{(1)}$ is at most the size of $X,$ it follows that $X$ is unitarily equivalent to $X^{(1)}.$ 
But this contradicts the fact that $\cV$ is a proper nontrivial invariant subspace of $\C^n.$ \qed
\end{remark}

\begin{remark}
	If $\bm{K}$ is a matrix convex set, then the condition in Definition \ref{def:fext} that $V_i$ must be a  scalar multiple of an isometry can be dropped.
	In fact, in that case a point $X \in \bm{K}$ is free extreme if and only if \eqref{eq:fextdef} implies that for each $i$ either $n_i = n$ and $X^{(i)}$ is unitarily equivalent to $X$ or $n_i > n$ and there is a tuple $Y^{(i)} \in \bm{K}$  such that $X^{(i)}$ is unitarily equivalent to $X \oplus Y^{(i)}.$ To see the equivalence of the two definitions (in the 
	absence of an infinite level) note that, for each $i,$ we have as above that $X^{(i)} \uapprox  X \oplus Y^{(i,1)}.$  Writing $V_i = \text{col}(V_{i,1}, V_{i,2})$ gives
	$$
	X = \sum_{i=1}^k V_i^\ast X^{(i)} V_i = \sum_{i=1}^k V_{i,1}^\ast X V_{i,1} + \sum_{i=1}^k V_{i,2}^\ast Y^{(i,1)} V_{i,2},
	$$
	which is again an expression of $X$ as a matrix convex combination with nonzero $V_{i,j}.$
	Since $X$ is free extreme, each of the $Y^{(i,1)}$ must be unitarily equivalent to $X \oplus Y^{(i,2)}.$ We proceed by induction until, for some $m,$ the size of $Y^{(i,m)}$ is 
  at most $n,$ forcing  { $Y^{(i,m)} = U_{i,m}^*XU_{i,m}$ for unitary matrices $U_{i,m}.$ Hence
   there are unitary matrices $U_{i,j}$ such that}
	$$
	X =\sum_{j=1}^m \sum_{i=1}^k V_{i,j}^*   U_{i,j}^*XU_{i,j} V_{i,j}
	$$
	with $V_i = \text{col}(V_{i,1}\ldots,V_{i,m}).$
	Now by \cite[Proposition 4.6]{EHKM}, for each $i,j,$  the matrix $U_{i,j}V_{i,j}$ is a scalar multiple of the identity, say $U_{i,j}V_{i,j} = t_{i,j} \Ii_n,$ which implies that each $V_i$ is a scalar multiple of an isometry as
	$$
	V_i^* V_i = \sum_{j=1}^m V_{i,j}^* V_{i,j}  = \sum_{j=1}^m V_{i,j}^*   U_{i,j}^* U_{i,j} V_{i,j}  = \sum_{j=1}^m t_{i,j}^2 \,\Ii_n.
	$$
	It is immediate that $V_i^*X_iV_i = \sum_{j=1}^m t_{i,j}^2 \,X.$ \qed
\end{remark}

The following result is \cite[Theorem~6.4.2]{DK} specialized to operator convex sets.

\begin{theorem}[\textbf{Krein-Milman theorem for free extreme points}]\label{th: nckm}
	A  \wstar \bdcpt  operator convex set $\bm{K}$ is the \wstar closed operator convex hull of its free extreme points; that is, every \wstar \bdcpt operator convex set that contains the extreme point of $\bm K$ contains 
  $\bm K.$
\end{theorem}

An important class of matrix convex sets where free extreme points at finite levels do span the whole set (even without taking closures) are real free spectrahedra \cite{EH};
  i.e., spectrahedra closed under complex conjugation. 
 Given a tuple $A\in \mathbb{S}^\gv_d,$ the spectrahedron $\cD_A$ is \df{closed under complex conjugation}
 if $X=(X_1,\dots,X_\gv) \in \cD_A$ implies $\overline{X} \in \cD_A,$ where $\overline{X}$ is the tuple obtained from $X$ by 
 entry-wise conjugating each $X_j.$  Let $\bm K$ denote the operator convex set
  \[
   \bm K = \{X\in \mathbb{S}^\gv: L_A(X) \succeq 0\}.
 \]
 Thus $K_n=\cD_A(n)$ for $n\in \N.$  
 In this case, the main result  (Theorem~1.3) of \cite{EH} yields a stronger conclusion than that of Theorem~\ref{th: nckm} \CB{(see \cite{Pas22}, who shows that the conclusion of \cite[Theorem 1.3]{EH} can fail for complex spectrahedra that are not closed under complex conjugation.)}
 
\begin{proposition}
 \label{prop:EH}
   With notations as above and assuming $\cD_A$ is closed under complex conjugation, if $\bm K$ is bounded, then
    $\bm K$ is the SOT-closure of the operator convex hull of $\ext(\bm K)\cap \cD_A.$ 
\end{proposition}

\begin{proof}
An application of \cite[Theorem 1.3]{EH} together with the second part of Lemma~\ref{lema:red} yields that $\cD_A$ (the matricial levels of $\bm{K})$ is the  matrix convex hull of $\ext(\bm{K}) \cap \cD_A.$  

By \cite[Proposition~3.7(b)]{JKMMP}, the SOT-closure of the
fully free set $\bm S$ generated by $\cD_A$ 
is $\bm{K}.$ 
  \end{proof}

\subsection{Weak converses to the Krein-Milman theorem}
 This subsection consists of  several variations on weak converses to Theorem~\ref{th: nckm} of independent interest.
We begin with a partial converse inspired by Agler's abstract approach to model theory \cite{Ag}.\looseness=-1

The analog of a boundary for an Agler family of operators adapted to operator convex sets reads as follows.
A \df{boundary}
 of an operator convex set $\bm K$ is 
 an SOT-closed 
 fully free set $\bm B\subseteq \bm K$ (so closed under  norm-bounded countable direct sums, unitary similarities and
restrictions to reducing subspaces {and level-wise SOT-closed}) such that $\opco(\bm B) =\bm K.$  
Note that by Proposition~\ref{l:cl:cvx} $\opco(\bm B)$ is automatically SOT-closed and bounded.
Also note that $\bm K$ is a boundary for itself.   In particular, the set
 $\mathscr{B},$ of  boundaries for $\bm K$ is not empty  and  we call the  set
\begin{equation*}
 \partial^A \bm K =\bigcap \{\bm B: \bm B\in\mathscr{B}\}
\end{equation*}
 the \df{Agler boundary} for $\bm K$ by analogy with the notion of the Agler boundary for 
 a family of operators \cite{Ag}.

 Adapting Agler's notion of an extremal element of a family of operators \cite{Ag} 
 yields the following 
 definition.    A point $X\in K_n$ is an \df{\agler extreme point} of $\bm K,$ if $Y\in K_m$
  and $V:\cH_n\to \cH_m$ is an isometry such that $X=V^* YV,$ then the range of $V$ reduces $Y.$  Let $\ext^A(\bm K)$ denote
   the \agler extreme points of $\bm K.$   
 By Lemma~\ref{lema:red}, $\ext(\bm K)\subseteq\ext^A(\bm K).$ 
Theorem~\ref{th: nckm} says that the SOT-closed \df{fully free hull}  of
the free extreme points of a \bdcpt operator convex set $\bm K$ is a boundary for $\bm K.$

\begin{proposition}
\label{prop:milman:agler-style}
  If $\bm B$ is a boundary of {a bounded SOT-closed}  operator convex set $\bm K,$  then $\ext^A(\bm K) \subseteq \bm B.$ 
  In particular, $\ext(\bm K) \subseteq \bm B.$
 
  The  SOT-closed fully free hull of $\ext(\bm K)$  
   is a boundary that is contained in every other boundary of $\bm K;$  {\CCBB {i.e.}},
   the Agler  boundary for $\bm K$
  is a boundary that is contained in  all other boundaries.
 \end{proposition}

\begin{proof}  Let $\bm B$ be a boundary of $\bm K.$  
   By definition  $\opco(\bm B) =\bm K$ and \CB{$\bm B$} is SOT-closed and fully free. 
     So given $X\in \ext^A(\bm K),$ 
   there exists a $B\in \bm B$ and an isometry $V$ such that $X=V^*BV.$  Since $X$ is \CB{Agler} extreme in $\bm K,$ the range of $V$ reduces $B$ by {\CCBB {definition}}. 
   Since, being fully free, $\bm B$ is closed with respect to reducing subspaces, $X\in \bm B.$  
   {Thus, if $\bm B$ is a boundary for $\bm K,$ then} {\CCBB 
    $\ext^A(\bm K)\subseteq \bm B.$}

   {Now let $\bm E$ denote the SOT-closed fully free hull of $\ext(\bm K).$   Since a boundary  $\bm B$ for $\bm K$ is SOT-closed
    and contains $\ext(\bm K),$ as was just proved,  it follows that $\bm E\subseteq \bm B.$ Since $\bm K$ is a boundary, $\bm E \subseteq \bm K$
    and thus $\opco(\bm E)\subseteq \bm K.$ 
    By Proposition~\ref{l:cl:cvx},  $\opco(\bm E)$ is bounded and SOT-closed. 
    Since $\opco(\bm E)$ is bounded, SOT-closed and operator convex, it is also \wstar closed.  By
   the Krein Milman Theorem for operator convex sets,
     Theorem~\ref{th: nckm},  $\bm K\subseteq \opco(\bm E).$ Thus $\opco(\bm E) =\bm K$ and therefore
      $\bm E$ is a boundary for $\bm K$ that is contained in every other boundary for $\bm K.$}
   \end{proof}

Theorem~6.4.3 from \cite{DK} specialized to the present setting gives the following result.
A graded set $\bm S\subseteq \bbS^{\gv}$ is \df{closed under compressions} if $V^*XV\in \bm S$
whenever $X\in \bm S$ and $V$ is an isometry. We call $V^*XV$ a \df{compression} of $X.$

\begin{proposition}[{Special case of \cite[Theorem~6.4.3]{DK}}]
 \label{prop:milman:DK}
   Suppose $\bm K$ is a \bdcpt operator convex set. If $\bm S\subseteq \bm K$ is \wstar closed and closed under compressions and if the closed operator
   convex hull of $\bm S$ is  $\bm K,$ then $\ext(\bm K)\subseteq \bm S.$
\end{proposition}

 In Proposition~\ref{prop:milman:agler-style}, $\bm B$ is assumed closed with respect to only reducing subspaces, and not necessarily compressions like
  $\bm S$ from Proposition~\ref{prop:milman:DK}. On the other hand, $\bm B$ is assumed to be a free set, in particular closed with respect to norm-bounded
  countable direct sums  unlike $\bm S.$  Proposition~\ref{prop:milman}  below gives a couple
   of variants of Proposition~\ref{prop:milman:DK} obtained by imposing  slightly different hypotheses on $\bm S.$

\begin{proposition}
\label{prop:milman}
Let $\bm{K}$ be an operator convex set.
\begin{enumerate}[\rm (a)]
\item \label{i:milman:a}
Suppose $\bm S \subseteq \bm K$ is a subset of irreducible tuples that is closed under unitary conjugation. If the operator convex hull of $\bm S$ equals $\bm{K},$ then $\bm S$ contains all the free extreme points of $\bm{K}.$
\item \label{i:milman:b}
Suppose $\bm S \subseteq \bm K$  is closed under compressions. If the operator convex hull of $\bm S$ equals $\bm{K},$ then $\bm S$ contains all the free extreme points of $\bm{K}$.
\end{enumerate}
\end{proposition}

\begin{proof}
Suppose $X \in \bm{K}$ is a free extreme point.   Since $X \in \bm{K},$ by assumption, it can  be written as an operator convex combination \eqref{eq-op comb},
	$$
	X = \sum_{i=1}^k V_i^\ast X^{(i)} V_i,
	$$
	of points $X^{(i)}$ from $\bm S$. 

In the case of item~\ref{i:milman:a}, since $X$ is free extreme and the tuples $X^{(i)}$ are irreducible, we deduce that $X^{(i)}$ is unitarily equivalent to $X$ for each $i.$ As $\bm S$ is closed under unitary conjugation, $X$ lies in $\bm S$.

For item~\ref{i:milman:b}, 
since $X$ is free extreme, each $V_i$ is a scalar multiple of an isometry $W_i$ and $X=W_i^*X^{(i)}W_i$. As $\bm S$ 
 is closed under compressions, this implies $X\in \bm S$ as desired.
\end{proof}

\begin{remark}\rm
 \label{r:need:for:inf}
   We now return to the need for the addition of an infinite level, alluded to at the outset of
   this section, to obtain a Krein-Milman Theorem for operator convex sets.  Suppose $\bm K$
   is a matrix convex set. What is desired is a  notion of extreme point with 
   the property that (1) the closed matrix convex hull of the graded set $\bm E$ of extreme points
   of $\bm K$ is $\bm K;$ and (2) if $\bm S$ is any closed graded set whose closed matrix convex
   hull is $\bm K,$ then $\bm S$ contains $\bm E.$ In this sense, $\bm E$ (or its closure)
   is {\it the smallest spanning set}
   for $\bm K.$  From what we have seen, the only possibility for $\bm E$  is the graded set
   of free extreme points. On the other hand,  it can happen that a matrix convex set does not contain any free extreme points. 
The simplest example is exhibited by the Cuntz \cite{Cun77} isometries, cf.~\cite[Example 6.30]{Kr}; an alternate self-contained construction 
is given in \cite{Evert-no-ext}. Namely,
given a tuple $A$ of compact self-adjoint operators with no nontrivial finite-dimensional reducing subspace, 
the matrix convex set generated by $\{A\}$ has no (finite-dimensional) free extreme points.

There is a notion of, and a version of the Krein-Milman for, matrix extreme points \cite{WW,Morentz}. See also \cite{FHL,Fischer,F}
for some additional references. 
However,  the graded set of matrix extreme points of a matrix convex set is not necessarily a minimal spanning set, because a matrix extreme point of a matrix convex set could be expressible as a compression of a matrix extreme point at a higher level. That is, the graded set of matrix extreme points satisfies condition (1), but not condition (2), above
 -- see for instance \cite{EEHK} where it is shown (2) can fail even for a free spectrahedron. \qed
\end{remark}

\subsection{Operator \texorpdfstring{$\Gat$}{Gamma}-convex sets}
The notion of an operator convex set naturally extends to that of an \df{operator $\Gat$-convex set}
by appending a countably infinite-dimensional level $K_\infty \subseteq \cB(\cHi)_{\sa}^\gv=\mathbb S_\infty^\gv$ for an infinite-dimensional separable Hilbert space $\cHi.$
For a tuple $\Gat=(\gat_1,\dots,\gat_\rv)$ of symmetric noncommutative polynomials
with $\gat_j=x_j$ for $1\le j\le \gv\le \rv,$ the mapping $\PhiG:\mathbb{S}^\gv \to \mathbb{S}^\rv$ naturally extends to the operator level.

\begin{lemma}
 \label{l:Gat:SOT-continuous}
    The mapping $\Gat$ is (level-wise) SOT-SOT continuous on bounded sets.
    Moreover, if $\bm B \subseteq \mathbb{S}^\gv$ is bounded and SOT-closed, then 
     so is $\Gamma(\bm B).$
\end{lemma}

\begin{proof}
 Continuity is immediate at finite levels. Suppose $B\subseteq \cB(\cHi)$ is bounded. 
 Thus the  (product)  SOT topology on $B$ is
 metrizable. Hence SOT-SOT-continuity of $\Gat\vert_B$ is equivalent to sequential SOT-SOT-continuity,
 which is immediate since products and sums of bounded SOT convergent sequences converge SOT.

 Now suppose $\bm B$ is closed and bounded.  It is immediate that $\Gat(\bm B)$ 
 is bounded. To prove $\Gat(\bm B)$ is closed, it suffices to show $\Gat(\bm B)$ is SOT-sequentially closed.
 Given an SOT-convergent sequence $\Gat(X^{(n)})=(X^{(n)},\gamma(X^{(n)})),$ there is some $X$ 
  such that $X^{(n)}$ converges SOT to $X.$ Since $\bm B$ is SOT-closed, $X\in\bm B.$  By continuity of $\Gat$ on bounded sets, $\Gat(X^{(n)})$ converges SOT to $\Gat(X)\in \Gat(\bm B).$
\end{proof}
 
 The notion of a \df{$\Gat$-pair} is extended to include any pair $(X,V)$ satisfying $V^* \PhiG(X)V= \PhiG(V^*XV),$
 where $X\in \mathbb{S}^\gv_n$ and $V:\cH_m\to\cH_n$ is an isometry (with $\cH_\infty=\cHi$).
 As before, denote the graded set of all $\Gat$-pairs by $\Coup.$ \index{$\Coup$}

\begin{definition}
  Let $\bm S, \bm K\subseteq \mathbb{S}^\gv$ denote  graded sets.
   Thus $S_n, K_n\subseteq \mathbb{S}^\gv_n$ for each $n\in \N \cup\{\infty\}.$
  \begin{enumerate}[\rm (1)]\itemsep=6pt
  \item  The graded set  $\bm{K} = (K_n)_{n \in \N \cup \{\infty\}}$ 
  is  \df{operator $\Gat$-convex} if it is free and  $V^*XV \in \bm K$ whenever $X\in \bm K$ and $(X,V)\in \Coup.$
	
	\item The \df{operator $\Gat$-convex hull} of   $\bm{S}$ is the  intersection of all operator $\Gat$-convex sets containing $\bm{S},$ 
	denoted  by $\Gat\mhyphen \opco(\bm{S}).$ 
For a topology $\tau$, the 
$\tau$-closed operator $\Gat$-convex hull{\footnote {\CCBB The universe is closed (in any topology)
and $\Gat$-convex. See Example~\ref{ex:matco-all}} }of $\bm S$ is 
the intersection of all $\tau$-closed operator $\Gat$-convex sets containing $\bm S,$ denoted by $\Gat\mhyphen\ovopco_{\tau}(\bm{S})$.\qed
 \end{enumerate}
\end{definition}

\begin{remark}\rm
\label{r:sub:sot:explained}
\mbox{}\par
(1)
It is easy to see that a free set $\bm{K}$ is operator $\Gat$-convex if and only if it is closed under \df{operator $\Gat$-convex combinations}; that is, operator convex combinations of the form\looseness=-1
\begin{equation*}
	\sum_{i=1}^k V_i^\ast X^{(i)} V_i
\end{equation*}
as in \eqref{eq-op comb}, where  $(X, V)$ with
$$
X = \oplus_{i=1}^k X^{(i)} \quad
\text{ and }
\quad
V = \text{col}(V_1,\ldots,V_k) =
\begin{pmatrix}
	V_1 \\
	\vdots \\
	V_k	
\end{pmatrix},
$$
is a $\Gamma$-pair.

\smallskip

(2)
Note, in the case  $\Gamma(x)=x,$ that $\Gat\mhyphen\ovopco_{\tau}(\bm S) = \ovopco^{\tau}(\bm S)$
    where $\tau$ is either the SOT, WOT  or \wstar topology; 
    that is, in these cases the closure of $\opco(\bm K)$ is the 
    smallest closed operator convex set containing $\bm S$  by Lemma~\ref{l:cl:cvx}.  Thus we use
  the notation $\Gat\mhyphen \ovopco_{\tau},$ instead of $\Gat\mhyphen\ovopco^{\tau},$ to distinguish between the smallest closed 
   $\Gat$-convex set containing $\bm S$ and the
   {\CCBB level-wise} closure of $\Gat\mhyphen\opco(\bm S).$

\smallskip

{\CCBB

(3)
If $\bm S$ is bounded, SOT-closed, and free, then 
\cite[Theorem 3.3]{JKMMP} implies that $\Gat\mhyphen\opco(\bm S)$ is already
SOT-closed. Hence in this case
\[
   \Gat\mhyphen \ovopco^{\rm sot}(\bm S)
   =
   \Gat\mhyphen \opco(\bm S)
   =
   \Gat\mhyphen \ovopco_{\rm sot}(\bm S).
\]
In general, however, these two closure constructions do not coincide as we show next.

\smallskip

(4)
Let
$
   \Gamma(x,y)=(x,y,y^2).
$
For $n\in\mathbb N$, set
\[
X=
   \begin{pmatrix}
      0 & 1\\
      1 & 0
   \end{pmatrix},
   \qquad
   Y_n=
   \begin{pmatrix}
      0 & 0\\
      0 & \frac1n
   \end{pmatrix},
\]
and let
\[
   A_n=(X,Y_n)\in \bbS_2^2.
\]
Let $\bm S\subseteq\bbS^2$ be the graded set with
\[
   S_2=\{A_n:n\in\mathbb N\},
   \qquad
   S_m=\emptyset\quad(m\neq 2).
\]

We claim that
\[
\Gat\mhyphen \ovopco^{\rm sot}(\bm S)
   \neq
   \Gat\mhyphen \ovopco_{\rm sot}(\bm S).
\]
Indeed,
$
   A_n\to A=(X, 0_2)
$
in norm, hence also in the SOT. Since $A_n\in \bm S\subseteq
\Gat\mhyphen\opco(\bm S)$, it follows that
\[
   A\in 
\Gat\mhyphen \ovopco^{\rm sot}(\bm S).
\]
Consequently $A$ belongs to every SOT-closed operator $\Gamma$-convex set
containing $\bm S$, and hence
$
   \Gat\mhyphen \ovopco_{\rm sot}(\bm S).
$

Now the second coordinate of $A$ is $0_2$. Therefore every subspace reduces
the second coordinate. In particular, if
\[
   v=\frac1{\sqrt2}
   \begin{pmatrix}
      1\\
      1
   \end{pmatrix}
\]
and $V:\mathbb C\to\mathbb C^2$ is the isometry $V1=v$, then
 $(A,V)$ is a $\Gamma$-pair. Since
$\Gat\mhyphen \ovopco_{\rm sot}(\bm S)$ is $\Gamma$-convex, it follows
that
$
   V^*AV
   =
   \bigl(V^*XV, 0\bigr)
   =
   (1,0) \in \Gat\mhyphen \ovopco_{\rm sot}(\bm S).
$

On the other hand,
\[
   (1,0)\notin
\Gat\mhyphen \ovopco^{\rm sot}(\bm S).
\]
To see this, we show that every scalar point in
$\Gat\mhyphen\opco(\bm S)_1$ has first coordinate  $0$.

Let $B=(b_1,b_2)\in \Gat\mhyphen\opco(\bm S)_1$. By
Proposition~\ref{prop:newjp} {below}, $B$ is obtained by compressing
a tuple $F$ in the free hull of $\bm S$. Thus
$
   B=V^*FV
$
for some unit vector $V:\C\to \cH_F$, and 
$\Gat$-pair
$(F,V)$.
Write $F=(F_1,F_2)$. Since
$\Gamma=(x,y,y^2)$, the condition that $(F,V)$ is a $\Gamma$-pair implies $V1$ is an eigenvector of
$F_2$.
But $F$ is a direct sum, up to unitary equivalence, of copies of the tuples
$A_n$. For each $A_n=(X,Y_n)$, the compression of $X$ to any eigenspace
of $Y_n$ is zero.
Therefore the same holds for $F$. Since $V1$ lies in an eigenspace of
$F_2$, we get
$
   b_1=V^*F_1V=0.
$
Thus
\[
   \Gat\mhyphen\opco(\bm S)_1\subseteq \{0\} \times \R.
\]
Taking SOT closure at level $1$, which is just ordinary Euclidean closure,
still gives only points with first coordinate $0$. Hence
$(1,0)\not\in\Gat\mhyphen \ovopco^{\rm sot}(\bm S).$

Therefore
\[
\Gat\mhyphen \ovopco^{\rm sot}(\bm S) 
\subsetneq 
\Gat\mhyphen \ovopco_{\rm sot}(\bm S).
\]
By finite-dimensionality, the same example also separates the corresponding weak$^*$-closed notions.
}
\qed \end{remark}

A Hahn-Banach separation theorem for operator $\Gat$-convex subsets of  $\bbS^{\gv}$  was proved in \cite{JKMMP}. Assuming $0 \in \bm{K},$ this theorem holds in our setting as well after identifying the finite levels $K_n$ of $\bm{K}$ with $K_n \oplus 0 \subseteq K_\infty.$

\begin{theorem}{\cite[Theorem 3.8]{JKMMP}} \label{hb-op}
	Suppose $\Gat(0) = 0$ and let $\bm{K}$ be an SOT-closed bounded operator $\Gat$-convex set such that $0 \in \bm{K}.$ If $X \notin \bm{K},$
	 then  there is an $n \in \N$ and a monic linear pencil $L$ of size $n$ such that the $\Gat$-pencil $\LG = L\circ \Gat$ is positive semidefinite on $\bm{K},$ but not at $X.$ \looseness=-1
\end{theorem}

Propositions~\ref{prop:gammahull} and \ref{prop:jp} generalize immediately to the operator setting. 
Recall the definition of the fully free hull of a graded set $\bm S \subseteq \mathbb{S}^\gv$
from Definition~\ref{d:opco}.

\begin{proposition}\label{prop:newjp} If $\bm{J} = (J_n)_{n \in \N \cup\{\infty\}}$ is not empty and $\bm F$ is its free hull, then
	\begin{enumerate}[\rm (a)]
	        \item  \label{item 0} $\Gat$-$\opco(\bm J)=\Gat$-$\opco(\bm F)$ and  $\opco(\Gat(\bm J)) =\opco(\Gat(\bm F));$
		\item  \label{item 1} $\Gat$-$\opco(\bm{J}) =\{V^*XV: X\in \bm{F}, \, (X,V)\in\Coup \};$
		\item \label{item 2} $\PhiG^{-1}(\opco(\PhiG(\bm{J}))) =\Gat\mhyphen\opco(\bm J);$ 
		\item \label{item 5} {$\PhiG^{-1}(\ovopco^{\SOT} (\PhiG(\bm{J})))$ is $\Gat$-convex;}
		\item \label{item 3} if $\Gat(0)=0,$ if  the graded set $\bm{J}$ is bounded and if  $0 \in J_1,$ then 
 \[
 \Gat\mhyphen\ovopco_{\SOT} (\bm{J}) =    \PhiG^{-1}(\ovopco^{\SOT} (\PhiG(\bm{J}))).	
\]
  In particular, if $\bm K$ is operator convex and  $X\in \bm K\setminus \Gat\mhyphen \ovopco_{\SOT}(\bm J),$ then
   $\Gat(X) \notin \ovopco^{\SOT}(\Gat(\bm J)).$
	\end{enumerate}
\end{proposition}

\begin{proof}
        It is immediate that $\Gat$-$\opco(\bm J)\subseteq \Gat$-$\opco(\bm F).$ On the other hand, 
        since $\Gat$-$\opco(\bm J)$ is a free set (since it is $\Gat$-convex)
        that contains $\bm J,$ it follows that $\bm F \subseteq \Gat$-$\opco(\bm J).$ Hence $\Gat$-$\opco(\bm F) \subseteq \Gat$-$\opco(\bm J).$
        Similarly, since $\Gat(\bm J)\subseteq \Gat(\bm F),$ it follows that $\opco(\Gat(\bm J)) \subseteq \opco(\Gat(\bm F)).$
         On the other hand, it is readily checked that         $\opco(\Gat(\bm J))$ is a free set that contains 
         $\Gat(\bm F)$ giving the reverse inclusion
         and completing the proof of item~\ref{item 0}.  
        
        From item~\ref{item 0}, it suffices to prove the remaining items with $\bm J$ replaced by a free set $\bm F.$
        For such an $\bm F,$  items~\ref{item 1} and \ref{item 2} are readily
        verified (cf.~Propositions \ref{prop:gammahull} and  \ref{prop:jp}).   For instance, for item~\ref{item 1},  it suffices to show that the right hand side is operator $\Gat$-convex.
        To show that the right-hand side is a free set, i.e., closed under countable direct sums, we proceed as follows. Suppose $(Y^{(j)})_j$ is a sequence from 
        the right-hand side of \ref{item 1}.
         Thus there exists a sequence $(X^{(j)})_j$ from $\bm F$ and isometries $V_j$ such that $Y^{(j)}=V_j^* X^{(j)} V_j$ 
          and $(X^{(j)},V_j)\in \Coup.$ 
          Since $\bm F$ is closed with respect to countable direct sums, $X=\oplus X^{(j)} \in \bm F.$ Let $V=\oplus V_j.$ 
           Thus $V$ is an isometry,
           $(X,V)\in \Coup$ 
           and $V^* X V =\oplus_j Y^{(j)}.$ Hence 
           $\oplus_j Y^{(j)}$ is in the right-hand side of \ref{item 1}.  Similar arguments 
        show the right hand side is closed under unitary similarity and operator $\Gat$-convex combinations.

	To show that  $\bm L = \Gat^{-1}( \ovopco^{\SOT}(\Gat(\bm J)))$ is $\Gat$-convex, suppose $X\in \bm L$ and $(X,V)\in \Coup.$ 
	Thus $\Gat(V^*XV)=V^*\Gat(X)V$ and $\Gat(X)\in  \ovopco^{\SOT}(\Gat(\bm J)).$  Since $\ovopco^{\SOT}(\Gat(\bm J))$ is
	 operator convex by Proposition~\ref{l:cl:cvx}, it follows that $\Gat(V^*XV) =V^*\Gat(X)V\in \ovopco^{\SOT}(\Gat(\bm J)).$ Thus $V^*XV\in \bm L$ and hence
	  $\bm L$ is $\Gat$-convex and item~\ref{item 5} is proved.

	Turning to item \ref{item 3} with its hypothesis that
    $\bm J,$ and hence $\bm F,$ is bounded,  from item~\ref{item 2}, 
   $\Gat\mhyphen\opco(\bm F)=\Gat^{-1}(\opco(\Gat(\bm F))).$ 
	By the
	   SOT-continuity of  $\Gat$ on bounded sets, Lemma~\ref{l:Gat:SOT-continuous} implies that  
	 the graded set $\Gat^{-1}( \ovopco^{\SOT}(\Gat(\bm F)))$ is SOT-closed.   From item~\ref{item 5} it is $\Gat$-convex.
  Therefore,
\[
  \Gat\mhyphen\ovopco_{\SOT}(\bm F) \subseteq \Gat^{-1}( \ovopco^{\SOT}(\Gat(\bm F))). 
 \]
To prove the reverse inclusion,  suppose $W\notin \Gat\mhyphen\ovopco_{\SOT}(\bm F).$ By 
		 Theorem \ref{hb-op} applied to $\Gat\mhyphen \ovopco_{\SOT}(\bm F),$ there is an $n \in \N$ and a $\Gat$-pencil $\LG = L \circ \Gat,$ where $L$ is a monic linear pencil of size $n,$
		 such that $\LG(W) \nsucceq 0$ and $\LG(Y) = L(\Gat(Y)) \succeq 0$ for all $Y\in \Gat$-$\ovopco_{\SOT}(\bm F).$   
		 If $Z\in \opco(\Gat(\bm F)),$ then  
		 $Z=V^*\Gat(Y)V$ for some  $Y\in \bm F$ and isometry $V.$  
		   Hence, as $Y\in \bm F \subseteq \Gat\mhyphen \ovopco_{\SOT}(\bm F),$ it follows that
       $L^\Gat(Y)\succeq 0$ and therefore,
  \[
      L(Z) = (I\otimes V)^*L(\Gat(Y))(I\otimes V) = (I\otimes V)^* L^\Gat(Y) (I\otimes V) \succeq 0.
   \]
	Thus $L\succeq 0$ on $\opco(\Gat(\bm F))$ and thus on $\ovopco^{\SOT}(\Gat(\bm F)).$   
	 Consequently, $\Gat(W)\notin\ovopco^{\SOT}(\Gat(\bm F)).$ Equivalently,  $W\notin \Gat^{-1}(\ovopco^{\SOT}(\Gat(\bm F))$ and the proof 
         is complete.
\end{proof}

\begin{proposition}\label{prop:candb}
	Suppose $\bm{F} = (F_n)_{n \in \N \cup \{\infty\}}$ is a free set. If $\bm F$  
	is SOT-closed and bounded, then so  is $\opco(\Gat(\bm{F})).$
	In particular, $\opco(\Gat(\bm{F}))$ is also \wstar closed.
\end{proposition}

Proposition \ref{prop:candb} extends \cite[Theorem 3.3]{JKMMP},  which gave a similar statement for subsets of tuples of operators $\cB(\cHi)_{\sa}^\gv$. 

\begin{proof}
  Without loss of generality, we assume throughout the proof that $\bm F$ is nonempty.
  It is immediate that $\opco(\Gat(\bm F))$ is bounded.
  
  Since $\bm F$ is bounded and SOT-closed, $\Gat(\bm F)$ is bounded and SOT-closed 
   by Lemma~\ref{l:Gat:SOT-continuous}. 
   Since $\bm F$ is free, so is $\Gat(\bm F).$
 An application of  Proposition~\ref{l:cl:cvx}  with $\Gat(\bm  F)$ in place of $\bm F$ shows
$\opco(\Gat(\bm F))$ is SOT  and \wstar closed.
\end{proof}

\subsection{\texorpdfstring{$\Gat$}{Gamma}-extreme points}
We are now ready to define the $\Gat$-analog of a free extreme point.

\begin{definition}\label{def-gext}
	Let $\bm{K} = (K_n)_{n \in \N \cup \{\infty\}}$ be an operator $\Gat$-convex set (that is not assumed closed).  A tuple $X \in K_n$ for $n \in \N \cup \{\infty\}$ is a \df{$\Gat$-extreme point} if any expression
	of $X$ as an operator  $\Gat$-convex combination
	\begin{equation*}
		X = \sum_{i=1}^k V_i^\ast X^{(i)} V_i,
	\end{equation*}
	where $X^{(i)} \in K_{n_i},$ the $V_i \in {M}_{n_i,n}$ are all nonzero and $\sum_{i=1}^k V_i^\ast V_i = {I}_n$ and $(X,V),$ where $X=\oplus_iX^{(i)}$ and $V=\col(V_1,\ldots,V_k),$ is a $\Gat$-pair,
	implies that for each $i,$ the matrix $V_i$ is a  scalar multiple of an isometry $W_i \in {M}_{n_i,n}$ such
	 that $(X^{(i)},W_i)$ is a $\Gat$-pair 
	 and, with respect to the range of $V_i,$ 
	$$X^{(i)} = Y^{(i)} \oplus Z^{(i)}$$
	for some $Y^{(i)},Z^{(i)} \in \bm{K}$ with $Y^{(i)}$ unitarily equivalent to $X.$
	Denote the graded set of $\Gat$-extreme points of $\bm{K}$ by $\Gat\mhyphen\ext(\bm{K}).$ \qed
\end{definition}

We first establish the existence and then a spanning property, i.e., a Krein-Milman type theorem, for $\Gat$-extreme points. For that we first show that all the free extreme points of $\opco(\Gat(\bm{K}))$ lie in $\Gat(\bm{K})$ and explain the correspondence between $\Gat$-extreme points of $\bm{K}$ and free extreme points of $\opco(\PhiG(\bm{K})).$ 

\begin{lemma}\label{lema:fext}
If  ${\bm F}\subseteq \bbS^{\gv}$ is  a fully free set,
then every free extreme point of $\opco(\Gat(\bm{F}))$ is of the form $\Gat(X)$ for some $X \in \bm{F};$ that is, 
 $\ext(\opco(\Gat(\bm F)))\subseteq \Gat(\bm F).$
\end{lemma}

\begin{proof}
	Let $Y$ be a free extreme point of $\opco(\Gat(\bm{F})).$ Since $Y \in \opco(\Gat(\bm{F})),$
     by Lemma~\ref{l:cl:cvx} 
    there exists $X^{(i)} \in \bm F$ and $V_i$ such that $Y$ 
     can be expressed as the operator convex combination 
	\begin{equation*} 
		Y =  \sum_{i=1}^{k} V_i^\ast \Gat(X^{(i)}) V_i = V^* \Gat(X) V,
	\end{equation*}
	where $X=\oplus X^{(i)}\in \bm F$ and the operator
    $V = \text{col}(V_1,\ldots,V_k)$ is an isometry. Since $Y$ is a free extreme point of 
    the operator convex set $\opco(\Gat(\bm F)),$   Lemma~\ref{lema:red} applied
     to $\opco(\Gat(\bm F))$ implies that $V$ reduces $\Gat(X)$ and hence $X,$ 
    since $\bm F$ is fully free and hence closed with respect to restriction to reducing subspaces.
    Thus $V^*XV\in \bm F$ 
	   and
       $Y= V^* \Gat(X)V = \Gat(V^*XV) \in \Gat(\bm{F}).$
\end{proof}

\subsection{Krein-Milman for operator \texorpdfstring{$\Gat$}{Gamma}-convex sets}

\begin{proposition}\label{tr22}
	Let $\bm{K}$ be an operator $\Gat$-convex set. 
	 If  $Z$ is a free extreme point of $\opco(\Gat(\bm K)),$ then there
         is a $\Gat$-extreme point $X\in \bm K$ such that $Z=\Gat(X).$ Thus,
    \[
      \ext(\opco(\Gat(\bm K))) \subseteq \Gat(\Gat\mhyphen\ext(\bm K)).
    \]
In particular, if $X\in \mathbb{S}^\gv$ and $\Gat(X)$
          is a free extreme point of $\opco(\Gat(\bm K)),$ then $X$ is a $\Gat$-extreme point of $\bm K.$
 \end{proposition}

\begin{proof} Let $T$ be a free extreme point of $\opco(\Gat(\bm{K})).$ By Lemma \ref{lema:fext}, there is an $n \in \N \cup \{\infty\}$ and $X \in K_n$ such that $T = \Gat(X).$  To prove $X$ is a $\Gat$-extreme point of $\bm K,$ suppose
	\begin{equation}\label{eq3}
		X = \sum_{i=1}^{k} V_i^\ast X^{(i)} V_i
	\end{equation}
	for a $\Gat$-pair $(\oplus_i X^{(i)},  \,\text{col}(V_1,\cdots,V_k))$
	with $X^{(i)} \in K_{n_i}$ and $n_i \in \N \cup \{\infty\},$ and nonzero $V_i \in {M}_{n_i,n}$ satisfying $\sum_{i=1}^k V_i^\ast V_i = {I}_n.$ 
	Applying $\Gat$ on both sides of \eqref{eq3} and using the defining property of a $\Gat$-pair gives
	$$
	 \Gat(X) =  \sum_{i=1}^{k} V_i^\ast \Gat(X^{(i)}) V_i.
	$$
	Since $ \Gat(X) $ is a free extreme point of $\opco(\Gat(\bm K)),$ each of the $V_i$ is a  scalar multiple of an isometry $W_i \in {M}_{n_i,n}$ satisfying $W_i^*\Gat(X^{(i)})W_i = \Gat(X)$
	 and  
 with respect to the range of $V_i,$
\[
  \Gat(X^{(i)}) = Y^{(i)}\oplus Z^{(i)},
\]
for some $Y^{(i)},Z^{(i)}\in \opco(\Gat(\bm K)),$ where $Y^{(i)}$ is unitarily equivalent to $\Gat(X).$  From the identity $W_i^* \Gat(X^{(i)})W_i = \Gat(X),$ it follows
 that $W_i^*X^{(i)}W_i=X$ and thus $W_i^*\Gat(X^{(i)})W_i =\Gat(W_i^*X^{(i)}W_i);$ that is,  $(X^{(i)},W_i)$ is a $\Gat$-pair. Further, 
setting
$\cY^{(i)} = (Y^{(i)}_1,\dots,Y^{(i)}_\gv)$ and $\cZ^{(i)}=(Z^{(i)}_1,\dots,Z^{(i)}_\gv),$ it follows that $X^{(i)} = \cY^{(i)}\oplus \cZ^{(i)}$
with respect to the range of $V_i$ and $\cY^{(i)}$ is unitarily equivalent to $X.$  In particular,  since  $X^{(i)}\in \bm K$ and $\bm K$ is $\Gat$-convex and therefore fully free,
 both $\cY^{(i)}$ and $\cZ^{(i)}$ 
 are in $\bm K.$  Hence $X$ is a $\Gat$-extreme point of $\bm K.$
 \end{proof}

 We offer the following partial converse to Proposition~\ref{tr22},
 which does not figure in future developments.
 
\begin{proposition}
\label{tr22:con}
 Let $\bm K$  be an SOT-closed and bounded operator $\Gat$-convex set. 
  If \\ $\opco(\ext(\opco(\Gat(\bm K))))$ is SOT-closed and $X\in \Gat\mhyphen \ext(\bm K),$ then 
         $\Gat(X)\in  \ext(\opco(\Gat(\bm K))).$
\end{proposition}

\begin{wrapfigure}{r}{0.5\linewidth}
\centering
\resizebox{\linewidth}{!}{%
\begin{tikzpicture}[scale=2.5, >=Stealth]

    \foreach \n in {1, 2, 3, 4, 5, 6, 8, 11, 16, 25} {
        \pgfmathsetmacro{\ang}{180 / \n}
        \pgfmathsetmacro{\px}{0.45 * sin(\ang)}
        \pgfmathsetmacro{\py}{1.2 - 1.2 * cos(\ang)}
        
        \coordinate (F\n) at (\px, \py);
        \coordinate (B\n) at (-\px, \py);
    }

    \foreach \n [remember=\n as \lastn (initially 1)] in {2, 3, 4, 5, 6, 8, 11, 16, 25} {
        \fill[gray!10, opacity=0.4] (-3,0) -- (B\lastn) -- (B\n) -- cycle;
        \fill[gray!15, opacity=0.4] (3,0) -- (B\lastn) -- (B\n) -- cycle;
        \draw[gray!60, dashed, thick] (B\lastn) -- (B\n);
    }
    \draw[gray!60, dashed, thick] (B25) -- (0,0);

    \foreach \n in {2, 3, 4, 5, 6, 8, 11, 16, 25} {
        \draw[black!60, dashed, thick] (-3,0) -- (B\n); 
        \draw[black!60, dashed, thick] (3,0) -- (B\n);  
    }

    \foreach \n [remember=\n as \lastn (initially 1)] in {2, 3, 4, 5, 6, 8, 11, 16, 25} {
        \fill[blue!18, opacity=0.65] (-3,0) -- (F\lastn) -- (F\n) -- cycle;
        \fill[blue!18, opacity=0.65] (3,0) -- (F\lastn) -- (F\n) -- cycle;
        
        \draw[black, thick] (F\lastn) -- (F\n);
    }
    \draw[black, thick] (F25) -- (0,0);

    \fill[blue!25, opacity=0.55]
        (F1) -- (F2) -- (F3) -- (F4) -- (F5) -- (F6)
             -- (F8) -- (F11) -- (F16) -- (F25) -- (0,0)
             -- (B25) -- (B16) -- (B11) -- (B8)
             -- (B6) -- (B5) -- (B4) -- (B3) -- (B2) -- (B1)
             -- cycle;

    \draw[black, thick]
        (F1) -- (F2) -- (F3) -- (F4) -- (F5) -- (F6)
             -- (F8) -- (F11) -- (F16) -- (F25) -- (0,0);

    \draw[gray!60, dashed, thick]
        (B1) -- (B2) -- (B3) -- (B4) -- (B5) -- (B6)
             -- (B8) -- (B11) -- (B16) -- (B25) -- (0,0);

    \foreach \n in {1, 2, 3, 4, 5, 6, 8, 11, 16} {
        \draw[black!50, thin] (-3,0) -- (F\n);
        \draw[black!50, thin] (3,0) -- (F\n);
    }

    \draw[thick, black] (-3,0) -- (0,2.4) -- (3,0); 
    \draw[thick, black] (-3,0) -- (3,0);           

    \foreach \n in {1, 2, 3, 4, 5, 6, 8} {
        \fill[red] (F\n) circle (1pt);
        \fill[red!60!black!60] (B\n) circle (0.8pt);
    }
    
    \node[right, font=\tiny, red] at (F1) {$Q_1$};
    \node[right, font=\tiny, red] at (F2) {$Q_2$};
    \node[right, font=\tiny, red] at (B2) {$Q_3$};
    \node[right, font=\tiny, red] at (F3) {$Q_4$};
    
    \fill[red, draw=black] (-3,0) circle (1.5pt) node[below left, black, font=\small] {$P_-$};
    \fill[red, draw=black] (3,0) circle (1.5pt) node[below right, black, font=\small] {$P_+$};

    \fill[white, draw=black, thick] (0,0) circle (1.5pt);

    \draw[<-] (0,0) -- (0.5,-0.5) node[right, font=\small\sffamily, align=left] {
        {Cluster Point} $Q_\infty$ \\ 
        Limit of vertices $Q_n$, but lying on \\ 
        the segment $[P_-, P_+]$, so {not extreme}.
    };
\end{tikzpicture}

}
\caption{\textcolor{black}{Compact convex set $C\subseteq\R^3$ with non-closed $\ext(C)$}}
\label{fig-nonclosedext}
\end{wrapfigure}

{\CCBB
The hypothesis in Proposition~\ref{tr22:con} that
$
   \opco\big(\ext(\opco(\Gat(K)))\big)
$
is SOT-closed does not imply that the set of extreme points itself is closed.
This phenomena already occurs in the case $\Gamma(x)=x$, with
$x=(x_1,x_2,x_3)$.

Let $C\subseteq \R^3$ be the compact convex set 
of Figure \ref{fig-nonclosedext}
whose set of
ordinary extreme points is not closed, but each extreme point is exposed. 
Let $E=\ext(C)=\{P_-,P+\}\cup\{Q_n\mid n\in\N\}$.

Now define
$
   \bm K=\opco(C),
$
i.e., $\bm K$ is the smallest operator convex set containing $C$.
Since $C=\conv(E)$, we have
$
   \bm K=\opco(C)=\opco(E).
$

We claim $\bm K$ is SOT-closed. To verify this claim, put
$
   \Omega=E\cup\{Q_\infty\}
$
and let
\[
   \mathcal S=\spann\{1,x_1,x_2,x_3\}\subseteq \cC(\Omega).
\]
Consider the operator matrix range
\[
   L_m
   =
   \{(\phi(x_1),\phi(x_2),\phi(x_3)):
      \phi:\mathcal S\to B(\cH_m)\text{ is ucp}\}.
\]
Here $\cH_m=\C^m$ for $m<\infty$ and $\cH_\infty=\cH$ is an infinite-dimensional separable Hilbert space. Thus $\bm L$
is the operator matrix range of the coordinate tuple
$(x_1,x_2,x_3)\in\cC(\Omega)^3$.

We claim that
$
   \bm L=\opco(E)=\bm K.
$
Let $X\in L_m$. Thus
   $X=(\phi(x_1),\phi(x_2),\phi(x_3))$
for a ucp map
$
   \phi:\mathcal S\to \cB(\cH_m).
$
Extend $\phi$ to a ucp map on $\cC(\Omega)$. By Stinespring's theorem,
there exist a representation $\pi:\cC(\Omega)\to \cB(\cH_\pi)$ and an
isometry $V:\cH_m\to\cH_\pi$ such that
\[
   \phi(f)=V^*\pi(f)V,\qquad f\in \cC(\Omega).
\]
Since $\Omega$ is countable, we may write
\[
   \cH_\pi=\bigoplus_{\omega\in\Omega}\cH_\omega, 
   \qquad
   \pi(f)=\bigoplus_{\omega\in\Omega} f(\omega)I_{\cH_\omega}.
\]
Hence
\[
   X
   =
   V^*
   \left(
      \bigoplus_{\omega\in\Omega}\omega I_{\cH_\omega}
   \right)
   V.
\]

The only summand not indexed by $E$ is the $Q_\infty$-summand. Since
  $ Q_\infty=\frac12 (P_-+P_+),$
replace the summand
$
   Q_\infty I_{\cH_{Q_\infty}}
$
by the compression of
\[
   P_- I_{\cH_{Q_\infty}}
   \oplus
   P_+ I_{\cH_{Q_\infty}}
\]
by the isometry
\[
\cH_{Q_\infty} \to \cH_{Q_\infty} \oplus \cH_{Q_\infty}, \qquad
   h\mapsto \frac{1}{\sqrt 2}h\oplus \frac{1}{\sqrt 2}h.
\]
Thus the tuple
\[
   \bigoplus_{\omega\in\Omega}\omega I_{\cH_\omega}
\]
is a compression of a countable direct sum of points of $E$. Since
$\opco(E)$ is an operator convex free set containing $E$, it contains this
countable direct sum and all of its compressions. Therefore
$
   X\in\opco(E).
$
Thus
$
   \bm L\subseteq\opco(E).
$
The reverse inclusion is immediate from the definition of $\bm L$, and hence 
$
   \bm L=\opco(E)=\bm K.
$

Since $\mathcal S$ is finite-dimensional, point-SOT limits of ucp maps
$\cS\to \cB(\cK)$ are again ucp. Therefore $\bm K$ is SOT-closed.

It remains to observe that the free extreme points of $\bm K$ are not closed.
Let $e\in E$. As shown above, $e$ is exposed in $C$. Choose an affine
linear functional $\ell_e$ such that
\[
   \ell_e(e)=0,
   \qquad
   \ell_e(d)>0\quad\text{for all }d\in E\setminus\{e\},
   \qquad
   \ell_e\geq 0\quad\text{on }C.
\]
Suppose $Y\in K_m$ and $V:\C\to \cH_m$ is an isometry such that
$
   V^*YV=e.
$
Using the representation of $Y$ as an operator convex combination of points
of $E$, write
\[
   Y=\sum_{d\in E}d\otimes P_d,
   \qquad
   P_d\succeq0,
   \qquad
   \sum_{d\in E}P_d=I.
\]
Then
\[
   0
   =
   \ell_e(e)
   =
   V^*\ell_e(Y)V
   =
   \sum_{d\in E}\ell_e(d)\,V^*P_dV.
\]
Each summand is positive. Since $\ell_e(d)>0$ for $d\neq e$, it follows
that
\[
   P_dV=0\qquad(d\neq e).
\]
Hence $P_eV=V$, so $\range(V)$ reduces $Y$. Thus $e$ is an
Arveson boundary point of $\bm K$, and therefore a free extreme point.

Consequently,
$
   E\subseteq \ext(\bm K).
$
But
$
   Q_n\in E\subseteq \ext(\bm K),
$ and $Q_n\to Q_\infty,$
while $Q_\infty$ is not even an ordinary extreme point of
$
   K_1=C.
$
Therefore $Q_\infty\notin\ext(\bm K)$, and $\ext(\bm K)$ is not closed.

Finally, since
$
   \bm K=\opco(E)
$ and $
   E\subseteq\ext(\bm K),
$
we have
\[
   \bm K
   =
   \opco(E)
   \subseteq
   \opco(\ext(\bm K))
   \subseteq
   \bm K,
\]
whence
$
   \opco(\ext(\bm K))=\bm K,
$
which is SOT-closed.  Thus the hypothesis of
Proposition~\ref{tr22:con} can hold even though the set of free
extreme points is not closed.
}

\begin{proof}[Proof of Proposition \ref{tr22:con}]
	Suppose  $n \in \N \cup \{\infty\}$ and $X \in K_n$  is a $\Gat$-extreme point
  of $\bm K.$ To prove
	 $\Gat(X)$ is a free extreme point in $\opco(\Gat (\bm K)),$ let   $n_i \in \N \cup \{\infty\},$ tuples
	  $Y^{(i)} \in \opco(\Gat(\bm{K}))_{n_i}$  and nonzero $V_i \in {M}_{n_i,n}$ 
	  such that $\sum_{i=1}^k V_i^\ast V_i = {I}_n$ and
	\begin{equation*}
		\Gat(X) = \sum_{i=1}^{k} V_i^\ast Y^{(i)} V_i
	\end{equation*}
	 be given.   By Proposition~\ref{prop:candb},  $\bm J:=\opco(\Gat(\bm K))$ is  \wstar closed.
	 Thus,  by the Krein-Milman Theorem \ref{th: nckm},  $\bm J$ is the \wstar closed hull of its
	  free extreme points. Hence $\bm J = \ovopco^{w^*} (\ext(\bm J)) = \opco(\ext(\bm J)),$
     since it is assumed that $\opco(\ext(\bm J))$ is SOT-closed and the SOT and \wstar 
      closures of the bounded convex set $\opco(\ext(\bm J))$ are the same.
    Thus each of the $Y^{(i)}$ is an operator convex combination of free extreme points of $\opco\,(\Gat(\bm{K})).$ By Lemma \ref{lema:fext}, every free extreme point of $\opco\,(\Gat(\bm{K}))$ is of the form $\Gat(X)$ for some $X \in \bm{K}.$ So each $Y^{(i)}$ is an operator convex combination
 \[
	Y^{(i)} = \sum_{j=1}^{m_i} W_{i,j}^* \Gat(X^{(i,j)})W_{i,j},
 \]
	where $\Gat(X^{(i,j)}) \in \mathbb{S}^\rv_{r_i}$ is a free extreme point of $\opco\,(\Gat(\textbf{K})),$ and
	\begin{equation}\label{eq2}
		\Gat(X) = \sum_{i=1}^{k} \sum_{j=1}^{m_i}   V_i^\ast W_{i,j}^* \Gat(X^{(i,j)})W_{i,j} V_i.
	\end{equation}
	Now a comparison of the first $\gv$ coordinates in \eqref{eq2} implies that $X$ is expressed as an operator convex combination
	$$
	X =  \sum_{i=1}^{k} \sum_{j=1}^{m_i} V_i^\ast W_{i,j}^*  X^{(i,j)} W_{i,j} V_i.
	$$
	Plugging this expression into \eqref{eq2} we see that $(\oplus_{i,j} X^{(i,j)},\, \text{col}( W_{i,j} V_i))$ is a $\Gat$-pair.
	As $X$ is $\Gat$-extreme,  
	for each $i,j,$ every $ W_{i,j} V_i = t_{i,j} Q_{i,j}$ is a  scalar multiple of an isometry $Q_{i,j} \in {M}_{r_i,n}$  with $Q_{i,j}^*\Gat(X^{(i,j)})Q_{i,j} = \Gat(X).$
	Moreover, for each $i,j,$ there are tuples $A^{(i,j)}, B^{(i,j)} \in \bm{K}$  such that with respect to the range of $Q_{i,j},$ the tuple $X^{(i,j)}$ decomposes as $A^{(i,j)} \oplus B^{(i,j)}$  with $A^{(i,j)} \uapprox  X.$ 
	Now
	$$
	V_i^* V_i =  V_i^\ast \sum_{j=1}^{m_i} W_{i,j}^*  W_{i,j} V_i  = \sum_{j=1}^{m_i}V_i^\ast W_{i,j}^* W_{i,j} V_i = \sum_{j=1}^{m_i} t_{i,j}^2 \,Q_{i,j}^* Q_{i,j} = \sum_{j=1}^{m_i} t_{i,j}^2 \,\Ii_n,
	$$
	so $V_i$ is a scalar multiple of an isometry.
	Since $Q_{i,j}^*\Gat(X^{(i,j)})Q_{i,j} = \Gat(X)$ for all $i,j,$ we have that
	$$
	V_i^* Y_i V_i = \sum_{j=1}^{m_i} V_i^\ast W_{i,j}^*  \Gat(X^{(i,j)}) W_{i,j} V_i =
	\sum_{j=1}^{m_i}t_{i,j}^2\, Q_{i,j}^* \Gat(X^{(i,j)}) Q_{i,j}
	$$
	is a scalar multiple of $\Gat(X).$
	Since $\Gat$ respects unitary conjugation, the decomposition of $X^{(i,j)}$ as $A^{(i,j)} \oplus B^{(i,j)}$ for each $i,j$ implies 
	$$\Gat(X^{(i,j)}) = \Gat(A^{(i,j)} \oplus B^{(i,j)}) = \Gat(A^{(i,j)}) \oplus \Gat(B^{(i,j)})$$ with $\Gat(A^{(i,j)})$ unitarily equivalent to $\Gat(X)$
proving that $\Gat(X)$ is free extreme.
\end{proof}

The previous proposition implies that the existence of $\Gat$-extreme points is implied by the existence of free extreme points. 
 Proposition~\ref{prop:candb} gives a natural condition under which $\opco(\Gat(\bm{K}))$ is closed and bounded and hence
 ensures that the latter exist.

\begin{theorem}[\textbf{Krein-Milman theorem for $\Gat$-convex sets}]\label{th-gkm} Suppose $\Gat(0)=0.$
	If $\bm{K}$ is an SOT-closed and bounded operator $\Gat$-convex set that  contains $0,$  
      then $\Gat\mhyphen\ext(\bm{K}) \neq \emptyset$ and 
	\begin{equation}\label{eq:gkm}
		\bm{K} = 
		\Gat\mhyphen \ovopco_{\SOT} (\Gat\mhyphen\ext(\bm{K})).
		\end{equation}
\end{theorem}

\begin{remark}\rm
\label{r:more:gkm}
\pushQED{\qed}
\mbox{}\par
{(1)}
  Assuming $\bm K$ is bounded and \wstar closed it follows that $\bm K$ is WOT-closed and hence SOT-closed. 
  (See  Remark~\ref{rem : top}.) 
 Thus, if $\bm K$ is \wstar closed and  bounded, operator $\Gat$-convex and contains  $0,$ then Theorem~\ref{th-gkm} applies with
 the conclusion $\bm K=\Gat\mhyphen \ovopco_{\SOT}(\Gat\mhyphen\ext( \bm K)).$
 Since $\Gat\mhyphen\opco(\Gat\mhyphen\ext(\bm K)) \subseteq \bm K$ and $\bm K$ is \wstar closed,
\[
  \Gat\mhyphen \ovopco_{w^*}(\Gat\mhyphen\ext(\bm K)) \subseteq \bm K  = \Gat\mhyphen \ovopco_{\SOT} (\Gat\mhyphen\ext(\bm{K}))
   \subseteq  \Gat\mhyphen \ovopco_{w^*}(\Gat\mhyphen\ext(\bm K)),
\] 
 {where the last inclusion follows from the fact that 
 \wstar closed sets are SOT-closed.}
Thus, when $\bm K$ is \wstar closed,
\[
 \bm K =  \Gat\mhyphen \ovopco_{w^*}(\Gat\mhyphen\ext(\bm K)) =\Gat\mhyphen \ovopco_{\SOT}(\Gat\mhyphen\ext(\bm K)).
\]

\smallskip

{\CCBB 
(2) Suppose $\bm K$ is \wstar closed and satisfies the hypotheses of Theorem~\ref{th-gkm}.
If $\bm B$  is the \wstar closed fully free hull of $\Gat\mhyphen \ext(\bm K)$, then
\[
   \Gat\mhyphen \ovopco_{\SOT}(\bm B) = \Gat\mhyphen \ovopco_{\SOT}(\Gat\mhyphen\ext(\bm K)) = \bm K. \qedhere \popQED
\]
}
\end{remark}

\begin{proof}[Proof of Theorem~\ref{th-gkm}]
Since $\bm{K}$ is  SOT-closed and bounded, by Proposition \ref{prop:candb} the operator convex set $\opco\,(\Gat(\bm{K}))$ is also SOT-closed and bounded. Since $\opco(\Gat(\bm K))$ is convex, Remark \ref{rem : top} implies that it is also \wstar closed.
	 Hence, an application of the Banach-Alaoglu theorem shows that it is \wstar compact.
	By Theorem \ref{th: nckm}, the {\wstar} compact operator convex set $\opco\,(\Gat(\bm{K}))$ has a free extreme point $Y.$ 
	By Lemma \ref{lema:fext}, there is an $X \in \bm{K}$ such that $Y=\Gat(X)$ and thus,
	 by Proposition \ref{tr22},  $X$ is a $\Gat$-extreme point of $\bm{K}.$

{To prove \eqref{eq:gkm} assume there is an $X \in \bm{K}$ not lying in $\Gat\mhyphen \ovopco_{\SOT}(\Gat\mhyphen\ext (\bm K)).$
	 An application of  item \ref{item 3} in Proposition \ref{prop:newjp}  to $\bm J = \Gat\mhyphen\ext (\bm K)$  
	  gives 
\[
  \Gat(X) \notin \ovopco^{\SOT}(\Gat(\bm J)).
\] 
 Since $\opco(\Gat(\bm K))$ is SOT-closed, and thus \wstar closed,  Theorem~\ref{th: nckm} gives
 \[
  \ovopco^{w^*}(\ext(\opco(\Gat(\bm K))) = \opco(\Gat(\bm K)) = \ovopco^{w^*}(\Gat(\bm K)) .
 \]
  On the other hand,
    Proposition~\ref{tr22} implies    $\ext(\opco(\Gat (\bm K)) \subseteq \Gat(\bm J).$ 
  Thus  
\[
 \opco(\Gat(\bm K)) = \ovopco^{w^*}(\Gat(\bm K)) \subseteq \ovopco^{w^*}(\Gat(\bm J))  \subseteq \ovopco^{w^*}(\Gat(\bm K)) =\opco(\Gat(\bm K)),
\]
as, for bounded convex sets,   SOT and \wstar closures are the same. 
 Thus
\[
 \opco(\Gat(\bm K)) = \ovopco^{\SOT}(\Gat (\bm J)).
\]
Since $X\in \bm K,$ we obtain the contradiction, 
  $\Gat(X) \in \opco(\Gat(\bm K)) =\ovopco^{\SOT}(\Gat(\bm J))\not\ni \Gat(X).$
 Hence $\Gat\mhyphen \ovopco^{\SOT}(\bm J) = \bm K$ as claimed.
 }
\end{proof}

 We close this subsection with a weak converse to the $\Gat$-version of the Krein-Milman Theorem,
Theorem~\ref{th-gkm}.

The analog of a boundary for an Agler family of operators adapted to operator $\Gat$-convex sets reads as follows. A \df{$\Gat$-boundary}
 of an operator $\Gat$-convex set $\bm K$ is a SOT-closed 
 fully free set $\bm B\subseteq \bm K$ such that $\Gat\mhyphen\ovopco_{\SOT}(\bm B) =\bm K.$  Let
 $\BG$ denote the set of $\Gat$-boundaries for $\bm K$ and set
\[
 \partial_\Gat^A \bm K =\bigcap \{\bm B: \bm B\in\BG\}.
\]

A point $X\in K_n$ is a \df{$\Gat$-\agler extreme point} of $\bm K,$ if $Y\in K_m$
  and $V:\cH_n\to \cH_m$ is an isometry such that 
    $(Y,V)\in \Coup$ and $X=V^*YV,$ 
   then the range of $V$ reduces $Y.$ 
   Let $\Gat\mhyphen\ext^A(\bm K)$ denote
   the $\Gat$-\agler extreme points of $\bm K.$  
Evidently, $\Gat\mhyphen\ext(\bm K)\subseteq \Gat\mhyphen\ext^A(\bm K).$ 
Theorem~\ref{th-gkm} (also see Remark~\ref{r:more:gkm})  implies  that the SOT-closed {\it fully free hull}  of
the free $\Gat$-extreme points of a \bdcpt operator $\Gat$-convex set $\bm K$ is a 
 $\Gat$-boundary for $\bm K.$ 
\CB{
\begin{theorem}
\textcolor{black}{
 If  $\bm B$ is a $\Gat$-boundary of 
  a bounded and SOT-closed operator $\Gat$-convex set $\bm K,$ then 
  $\Gat\mhyphen \ext^A(\bm K) \subseteq \bm B.$ Hence
  $\Gat\mhyphen\ext (\bm K)\subseteq \bm B$ and therefore the SOT-closed fully free hull of $\Gat\mhyphen \ext(K)$
   is $\partial^A_\Gat K$ and is a boundary for $\bm K.$ In particular, $\partial^A_\Gat \bm K$ is the smallest boundary
   for 
   $\bm K.$}
\end{theorem}}

\begin{proof}
 Suppose $\bm B\subseteq \bm K$ is fully free SOT-closed and $\Gat\mhyphen\ovopco_{\SOT}(\bm B)=\bm K.$
  By Proposition~\ref{prop:candb}, $\opco(\Gat(\bm B))$ is SOT-closed and bounded.  By Proposition~\ref{prop:newjp}\ref{item 3},
 \begin{equation}
 \label{e:magic:think}
  \bm K = \Gat\mhyphen\ovopco_{\SOT}(\bm B) = \Gat^{-1}(\ovopco^{\SOT}(\Gat(\bm B))) = \Gat^{-1}(\opco(\Gat(\bm B))).
 \end{equation}
  Now suppose $X\in \Gat\mhyphen\ext^A(\bm K).$ Since $X\in \bm K,$ equation~\eqref{e:magic:think} implies
   $\Gat(X)\in \opco(\Gat(\bm B)).$ Hence there is an isometry $V$ and $B\in \bm B$ such that $\Gat(X)=V^* \Gat(B)V$
   and thus $X=V^*BV$ and  $(B,V)\in \Coup$.   Since $X$ is Agler extreme, it follows that the range of $V$ reduces
    $B$ and since $\bm B$ is closed under restrictions to reducing subspaces, $X\in \bm B.$ 
\end{proof}

\subsection{The free parabola}
 We close this section with the example of the free parabola.
 
\begin{example}
    \label{eg:parabola}
    Let $\bm{K}=(K_n)_{n\in\N\cup\{\infty\}}$ denote  the  operator $\Gamma(x,y)=(x,y,y^2)$-convex ($y^2$-convex) set 
\[
    \bm K = \{ (X,Y) \ | \ -(Y^2+I) \preceq X \preceq Y^2 + I, \ Y^2 \preceq I\}
\]
and let $\bm K^{\fin} = (K_n)_{n\in\N}$ denote the matrix convex set consisting of the finite levels
  of $\bm K.$
	\begin{figure}[ht]
	\centering
	 \begin{tikzpicture}[scale=1.5]
			\draw[very thin,color=gray] (-2.3,-1.3) grid (2.3,1.3);
			\draw[->] (-2.2, 0) -- (2.2, 0) node[right] {$x$};
			\draw[->] (0, -1.3) -- (0, 1.5) node[above] {$y$};
			
			\draw[scale=1, domain=-1:1, smooth, variable=\x, black] plot ({1+\x*\x},{\x});
			\draw[scale=1, domain=-1:1, smooth, variable=\x, black] plot ({-1-\x*\x},{\x});
			\draw[scale=1, domain=-2:2, smooth, variable=\x, black] plot ({\x},1);
			\draw[scale=1, domain=-2:2, smooth, variable=\x, black] plot ({\x},-1);
			
			\fill[blue!30, opacity=0.5] plot[domain=-1:1, variable=\x] ({1+\x*\x},{\x})
			--	plot[domain=-1:1, variable=\x] ({-1-\x*\x},{\x}) 
			-- plot[domain=-2:2, variable=\y] ({\y},1) 
			-- plot[domain=-2:2, variable=\y] ({\y},-1) --  cycle;
		\filldraw[color=black] (1,0) circle (0.02) node[below right] {$1$};
		\filldraw[color=black] (2,0) circle (0.02) node[below] {$2$};
\filldraw[color=black] (0,1) circle (0.02) node[above left] {$1$};
		\end{tikzpicture}
\caption{The first level component $K_1$ of $\bm K$.}
\label{fig-parabola}
\end{figure}
It is straightforward to check that $\bm K,$ and hence $\bm K^{\fin},$  is  bounded and 
  SOT-compact
   and $\Gat$-convex, \CB{where $\Gat = (x,y,y^2)$}.  By Lemmas~\ref{l:Gat:SOT-continuous} and \ref{l:cl:cvx}, 
     $\ovopco^{w^*}(\Gat(\bm K))=\opco(\Gat(\bm K)).$
  Below we show that (1) $\ovopco^{w^*}(\Gat(\bm K))=\opco(\Gat(\bm K))$ is described by a linear matrix inequality; (2)
 $\matco(\Gat(\bm K^{\fin}))$ is a spectrahedron and (3) it is the matrix
  convex hull of its finite level free extreme points;  (4)  $\bm K^{\fin}$ is the $\Gat$-convex hull of its 
  {finite level} $\Gat$-extreme points.

 Let $\bm J$ denote the bounded SOT-closed operator convex set consisting of all tuples $(X,Y,Z)$ 
 described by the linear matrix inequality
\begin{align*}
    -(Z+I) &\preceq X \preceq Z + I \\
    0 &\preceq Z \preceq I \\
   0 &\preceq  \begin{pmatrix}
        I & Y \\
        Y & Z
    \end{pmatrix},
\end{align*}
and note it contains $\Gat(\bm K).$
 The affine linear transformation $(X,Y,Z)\mapsto (X,Y,W=Z-\frac12 I)$ sends $\bm J$ to the
  free spectrahedron $\cD_A,$ where
\[
 L_A(x,y,w) = \begin{pmatrix}\frac32 +w - x \end{pmatrix} \oplus \begin{pmatrix}\frac32 +x +w \end{pmatrix} \oplus \begin{pmatrix} 1-2w \end{pmatrix} \oplus\begin{pmatrix} 1+2w
  \end{pmatrix} \oplus \begin{pmatrix} 1 &\sqrt{2}y\\ \sqrt{2}y & 1+2w \end{pmatrix}.
\]
Since  $\cD_A$ is closed under complex conjugation,  it is spanned by its  finite level free extreme points 
by \cite[Theorem~1.3]{EH} (see Proposition~\ref{prop:EH}). Thus so is   $\bm J^{\fin};$  that is,
\begin{equation}
\label{e:parabola:1}
  \matco(\ext(\bm J)\cap \bm J^{\fin}) = \bm J^{\fin}.
\end{equation}

{\it Claim:} If $(X,Y,Z)\in \bm J^{\fin}$ is a free extreme point of $\bm J,$ then $(X,Y,Z)=(X,Y,Y^2)\in \Gat(\bm K^{\fin}).$

      Before proceeding, we collect
   some consequences of this claim. First, 
\begin{equation}
\label{e:parabola:2}
  \ext(\bm J)\cap {\bm J^{\fin}} \subseteq \Gat(\bm K^{\fin}) 
   \subseteq \matco(\Gat(\bm K^{\fin})) \subseteq \bm J^{\fin}.
\end{equation}
 Combining the identities of equations~\eqref{e:parabola:1} and \eqref{e:parabola:2} gives
\[
 \matco(\Gat(\bm K^{\fin})) = {\bm J}^{\fin}.
\]
  Thus $\matco(\Gat(\bm K^{\fin}))$ is a spectrahedron described by
  a linear matrix inequality so that (2) holds. Moreover, {by 
  \cite[Proposition~3.7(b)]{JKMMP}, $\ovopco^{\SOT}(\bm J^{\fin}) = \bm J$. Since 
   for convex sets SOT and \wstar  closures coincide, 
    $\ovopco^{w^*}(\bm J^{\fin}) =\bm J$ and we conclude that 
    $\ovopco^{w^*}(\Gat(\bm K^{\fin})) = \bm J.$} Since also 
\[
  \bm J = \ovopco^{w^*}(\Gat(\bm K^{\fin})) \subseteq \ovopco^{w^*}(\Gat(\bm K)) 
    =\opco(\Gat(\bm K)) \subseteq \bm J,
\]
 it follows $\opco(\Gat(\bm K)) =\bm J$ so  that (1) holds. Further by equation~\eqref{e:parabola:1},
 (3) holds. 
  
   By Proposition~\ref{tr22} and the Claim, if $(X,Y,Z)\in \ext(\bm J)\cap \bm J^{\fin},$
   then $(X,Y,Z)=\Gat(X,Y)$ and $(X,Y)\in \Gat\mhyphen \ext(\bm K).$

    Now let $(X,Y)\in \bm K^{\fin}$ be given.  Since $\Gat(X,Y) \in \bm J^{\mat},$
    by another application of \cite[Theorem~1.3]{EH}, there exist
    free extreme points $(X_i,Y_i,Z_i)\in \bm J^{\fin}$ 
   of $\bm J$ and $V_i$ such that $\sum_i V_i^* V_i=I$ and 
    $\sum V_i^* (X_i,Y_i,Z_i)V=\Gat(X,Y).$ Since $(X_i,Y_i,Z_i)\in \ext(\bm J)\cap \bm J^{\fin},$
    we have $(X_i,Y_i,Z_i)=\Gat(X_i,Y_i)$ and $(X_i,Y_i)\in \Gat\mhyphen \ext(\bm K).$ 
Finally,  setting $V^* = \begin{pmatrix} V_1^* & \ldots & V_N^* \end{pmatrix}$ and
 $(\widetilde{X},\widetilde{Y}) =\oplus (X_i,Y_i),$ 
\[
  \Gat(X,Y) 
   = V^* \Gat(\widetilde{X},\widetilde{Y}) V.
\]
 Thus $((\widetilde{X},\widetilde{Y}),V)$ is a $\Gat$-pair and $(X,Y) =V^*(\widetilde{X},\widetilde{Y}) V.$  So $(X,Y)\in \Gat\mhyphen\conv(\Gat\mhyphen\ext(\bm K) \cap \bm K^{\fin})$ and therefore
  $\bm K^{\fin} = \Gat\mhyphen\conv(\Gat\mhyphen\ext(\bm K) \cap \bm K^{\fin})$ as
   claimed; that is, (4) holds.

  Turning to the proof of the Claim, we first argue that if  $(X,Y,Z)\in \bm J^{\fin}$ is a free extreme point of $\bm J,$
 then  $Y$ and $Z$ commute.  To do so, it suffices to show, given 
 $(X,Y,Z)\in\bm J^{\fin},$ that either  $Y$ and $Z$ commute or there exists a tuple $(\varphi,\alpha,u)\ne 0$
 and $\beta,w, \psi$ such that 
\[
 X_* =\begin{pmatrix} X & \varphi \\ \varphi^* & \psi \end{pmatrix}
  \ \ \
 Y_*= \begin{pmatrix} Y & \alpha \\ \alpha^* & \beta \end{pmatrix},
 \ \ \
 Z_* =\begin{pmatrix} Z& u\\ u^* & w \end{pmatrix}
\]
  satisfies $I\succeq Z_* \succeq Y_*^2$ and $-(Z_*+I) \preceq X_* \preceq Z_* +I;$
  that is, $(X_*,Y_*,Z_*)\in \bm J^{\fin}.$

 Suppose $Y$ and $Z$ do not commute and let $\cH$ denote the finite-dimensional space $(X,Y,Z)$ that 
  	acts upon. Let $\cK=\ker(I-Z)\cap \ker(Z-Y^2).$ If $\cK=\cH,$ then $Z$ is the identity.
   Thus $\cK^\perp \ne \{0\}.$

  From the definition of $\cK,$
 \begin{equation}
  \label{e:span:condition}
    \range(I-Z)+\range(Z-Y^2)=\cK^\perp \ne \{0\}.
  \end{equation}
 Let $\Delta^2=Z-Y^2.$   If  $Y\Delta^2=0,$ then $YZ=Y^3$ and hence $YZ=Y^3=(YZ)^*=ZY.$
 Thus $Y$ and $Z$  commute, contradicting our earlier assumption. It follows that   $Y\Delta^2\ne 0.$
 Choose $h$ such that
 $Y\Delta^2 h \ne 0$ and let $\alpha =\Delta^2 h.$  From the spanning condition of equation~\eqref{e:span:condition}, 
 there exists $f,g\in \cK^\bot$ such that
\begin{equation*}
 (I-Z)f+ \Delta^2 g = Y \Delta^2 h = Y \alpha.
\end{equation*}
Let 
\[
 u= (I-Z)f. 
\]
In particular,
\begin{equation}
 \label{e:span:2}
      u-Y\alpha =  (1-Z)f - Y\alpha  
  = - \Delta^2 g.
\end{equation}
 Since $u\in \range(I-Z),$ there exist  $s_0,w>0$ such that
\[
   \begin{pmatrix} I-Z & -s u \\ -s u^* & 1-w \end{pmatrix} \succeq 0.
\]
 With this $w$ fixed and $0<s\le s_0,$ let
\[
 Y_*(s)= \begin{pmatrix} Y& s\alpha \\ s\alpha^* & 0
 \end{pmatrix}, 
 \ \ \
 Z_*(s) =\begin{pmatrix} Z& s u\\ s u^* & w \end{pmatrix}
\]
and observe  from equation~\eqref{e:span:2} that
\begin{equation} 
 \label{e:span:3}
\begin{split}
 Z_*(s)-Y_*^2(s)
 & = \begin{pmatrix}  Z- [Y^2+s^2\alpha\alpha^*] & s[u-Y\alpha]
 \\ * & w-s^2 \alpha^*\alpha  \end{pmatrix}
 = \begin{pmatrix} \Delta^2 - s^2 \alpha \alpha^*  & -s \Delta^2 g \\ * & w-s^2 \alpha^* \alpha \end{pmatrix}
\\ & = \begin{pmatrix} \Delta & 0 \\0 &1 \end{pmatrix} \, \begin{pmatrix} I-s^2 \Delta hh^* \Delta & -s \Delta g \\
                    * & w-s^2\alpha^*\alpha \end{pmatrix} \, \begin{pmatrix} \Delta & 0\\0 & w-s^2\alpha^*\alpha \end{pmatrix}.
\end{split}
\end{equation}
 Without loss of generality, assume $0<s_0$ is small enough so that 
\[
 Z_*(s)-Y_*^2(s) \succeq  \begin{pmatrix}  \Delta^2  &0\\0 & w \end{pmatrix}
  - \begin{pmatrix} -s^2 \Delta hh^* \Delta & -s \Delta g \\ -s g^* \Delta  & -s^2\alpha^*\alpha \end{pmatrix} \succeq 0,
\]
which is possible since the matrix in the middle of the right hand side of equation~\eqref{e:span:3} is positive definite for $s_0$ small.

 Finally, we construct a suitable dilation $X_*$ of $X$. 
 If  $v \in \ker (Z+I-X) \cap \ker(Z+I+X),$ then  $Xv = (Z+I)v = -(Z+I)v$ and hence $v$ lies in the kernel of $Z+I.$ Now $Z\succeq 0$ implies $v=0.$ So $\ker (Z+I-X) \cap \ker(Z+I+X)=\{0\},$ or alternatively, $\range(Z+I-X)+\range(Z+I+X) = \cH,$ which implies there are vectors $u_1, u_2 \in \cH$ such that 
 \[
  u = (Z+I-X)u_1 + (Z+I+X)u_2.
 \]
  Now set 
 \[
  \phi = (Z+I-X)u_1 - (Z+I+X)u_2
 \]
  and note that $u-\phi = 2(Z+I+X)u_2 \in \range(Z+I+X)$ and $u+\phi = 2(Z+I-X)u_1 \in \range(Z+I-X).$ Hence, setting
\[
 X_*(s)  =\begin{pmatrix} X & s \phi \\ s \phi^* & 0 \end{pmatrix}
\]
 and choosing $0<s\le s_0$ sufficiently small, 
 \[
  	Z_*(s)+I \pm X_*(s)  = \begin{pmatrix}Z+I \pm X & s(u\pm \phi) \\ s(u\pm \phi)^* & w+1 \end{pmatrix} 
 \]
are both positive semidefinite.   We conclude, if $Y$ and $Z$ do not commute, then $(X,Y,Z)$ is not an \agler extreme point, 
 and hence by Proposition~\ref{prop:milman:agler-style},  not a  free extreme point 
 of $\bm J^{\fin}.$
 
 Now suppose $Y$ and $Z$ commute.  To show either $Z=Y^2$ or $(X,Y,Z)$ is not an extreme point of $\bm J,$ 
  first note that as $Y$ and $Z$ commute and they are both self-adjoint, they are, without loss of generality, (simultaneously) diagonal.
  Let $Y_{j,j}$ and $Z_{j,j}$ denote their diagonal entries and note that $Y^2\preceq Z$ is equivalent to  $|Y_{j,j}| \le \sqrt{Z_{j,j}}.$
  Suppose $Y^2\ne Z.$ In this case, without loss of generality,  $|Y_{1,1}| <\sqrt{Z_{1,1}}.$  For notational
   convenience, let $y=Y_{1,1}$ and $z=Z_{1,1}.$ Hence there exists a
   $0<\lambda <1$ such that $y=\lambda \sqrt{z} - (1-\lambda)\sqrt{z}$. Let $Y^\pm$ denote the diagonal matrices with 
    $Y^\pm _{1,1}= \pm \sqrt{z}$ and $Y^\pm _{j,j}= Y_{j,j}$ for $j>1.$ 
   Thus $(Y^\pm)^2 \preceq Z$ and therefore $(X,Y^\pm,Z)\in \bm K.$  Moreover, 
 \[
   (X,Y,Z) =  \lambda(X,Y^+,Z) +(1-\lambda)(X,Y^-,Z),
\]
 and hence $(X,Y,Z)$ is not free extreme.  Thus, if $(X,Y,Z)$ is extreme in $\bm J,$ then $Z=Y^2$ as claimed.
\qed
\end{example}

\section{Free analog of the Lasserre-Parrilo construction in the \texorpdfstring{$\Gat$}{Gamma}-convex setting}
\label{sec 7} 

In the commutative setting, a construction due to Lasserre \cite{Las09} (see also Parrilo \cite{Par06}) assigns to a  semialgebraic set $D_p = \{x \in \R^g \ | \ p(x)\geq 0\}$ a sequence of spectrahedra whose projections give a decreasing family of convex semialgebraic sets, called \textit{relaxations}, approximating the convex hull of $D_p.$ Under mild assumptions \cite{HN09,HN10}, such an approximation scheme is \textit{exact}, in which case the convex hull of $D_p$ is presented as a projection of a spectrahedron, called \textit{spectrahedrop} \cite{BPR13}.

\subsection{Introduction and basic notation}  
A free analog of the Lasserre-Parrilo relaxations was introduced and studied in \cite{HKM16}.
By adapting those methods, we give a construction of a sequence $\mathcal{D}^\Gat_{A^{(d)}}$ of free $\Gat$-spectrahedra
{in increasingly
many auxiliary variables}
whose projections give better and better approximations of the operator $\Gat$-convex hull of the
 positivity domain $\mathcal D_p$ of a symmetric matrix-valued noncommutative polynomial $p$,
 \[
 \mathcal{D}_p = 
( \mathcal{D}_p(n))_n = (\{X \in \mathbb{S}_n^\texttt{g} \ | \ p(X) \succeq 0\})_n.
 \]

 {\CCBB In what follows we work with matrix-valued $\Gamma$-pencils
with auxiliary variables.
 See equation~\eqref{eq: gen-sp}. There are two related reasons for working with matrix-valued pencils $L^\Gamma .$ For computational purposes one is forced to cut down to finite dimensions; and, as explained in Remark \ref{r:not:vac}(2), under a bit of an additional hypothesis a $\Gamma $-convex
set is a \Hilby $\Gamma $-spectrahedron. No approximation {(with projections of such sets)} is needed.}
{\CCBB Similarly,  the algorithms require
verifying that localizing matrices $H^{\Uparrow}_p(Y)$
(see Definition~\ref{def:Hankel}, equation~\eqref{eq:hankeldeftrunc}) are positive semidefinite and, in practice, doing that for an operator-valued $p$ would lead to
finite dimensional (matrix) approximations.}

 Fix a symmetric matrix-valued noncommutative polynomial $p \in M_{\mu}(\C\langle x \rangle)$
 of degree $\leq \delta$ in $\gv$ variables.  Thus
 $p$ takes the form
 \begin{equation}\label{eq-p}
 	p(x) = \sum_{|\alpha| \leq \delta} p_\alpha \alpha, 
 \end{equation}
   where  $p_\alpha^{*} = p_{\alpha^*}.$

In the first part of this section free analogs of moment sequences and Hankel matrices adapted to the $\Gat$-convex setting are used to construct an infinite free $\Gat$-spectrahedron $\mathfrak{L}_p$ and a canonical projection of $\mathfrak{L}_p$ onto the matricial levels of the operator $\Gat$-convex hull of $\mathcal{D}_p.$ In the final part of this section we explain how $\mathfrak{L}_p$ naturally determines a sequence of finite free $\Gat$-spectrahedra whose projections, i.e., \textit{free $\Gat$-spectrahedrops}, are increasingly finer outer approximations to the matricial levels of the operator $\Gat$-convex hull of $\mathcal{D}_p.$

\subsubsection{Free \texorpdfstring{$\Gat$}{Gamma}-spectrahedra and their projections}
	Fix $\Gat(x) = (\gamma_1, \ldots, \gamma_\rv)$ with $\gamma_i = x_i$ for $i=1, \ldots, \gv.$ We extend the notion of a $\Gat$-pencil to any matrix-valued polynomial of the form
	\begin{equation}\label{eq: gen-sp}
		\LG(x,y) = A_0 + \sum_{i=1}^\gv A_i\,x_i + \sum_{j=\gv+1}^\rv A_j \gamma_j(x) 
     \, +  \, \sum_{k=1}^\hv B_k y_k
	\end{equation}
	for some $\hv \in \N.$ Here the $y_k$ are new {symmetric}  noncommutative variables (that do not appear in any of the $\gamma_j$).
	In other words, we infer that the terms that determine $\Gat$ are $\gamma_i$ for $i=\gv+1, \ldots, \rv$ and refer to the polynomial obtained by adding new variables as linear terms to a $\Gat$-pencil as a $\Gat$-pencil as well.
	
	 Given a $\Gat$-pencil $\LG$ as in \eqref{eq: gen-sp}, the free set $\proj_x(\mathcal{D}_{L^\Gat}) = (\text{proj}_x(\mathcal{D}_{L^\Gat})_n)_n$ with
	 \begin{equation*}
	 \text{proj}_x(\mathcal{D}_{L^\Gat})_n =  \{X \in \mathbb{S}_n^\gv \ | \ \exists Y \in \mathbb{S}_n^{\hv} \text{ such that } 
	  L^\Gat(X,Y) \succeq 0\}
	 \end{equation*}
	 is called a \df{free $\Gat$-spectrahedrop}. Note that by definition, we project the positivity domain of $\LG$ onto the first $\gv$ variables in the linear part of $\LG.$
	 It is easy to check that a free $\Gat$-spectrahedron is a $\Gat$-convex set, or see Proposition \ref{p:not:vac}.

\subsubsection{Free operator semialgebraic sets} 
	Recall the definitions of an operator convex set (hull) and an operator $\Gat$-convex set (hull) from Section \ref{sec: g-ext}. 
	Similarly, we extend the matricial positivity domain $\cD_p$ of $p$ by including an  operator level in $\cB(\cH),$ where $\cH$ is an infinite-dimensional separable Hilbert space. 
Recall  from Section \ref{sec: g-ext} that 
$\mathbb{S}^\gv$ is the graded set $(\mathbb{S}^\gv_n)_{n\in \N \cup \{\infty\}}$ where $\mathbb{S}^\gv_n$ is identified
with $\cB(\cH_n)^\gv_{\sa}$ for an $n$-dimensional Hilbert space $\cH_n$ for $n\in\N$, and $\cH_\infty=\cH.$

	Given a symmetric free matrix-valued polynomial $p$ 
	let
	\begin{equation}\label{eq:operatorsa}
	\cD_p^\infty = \{X \in \mathbb{S}^\gv \ | \ p(X) \succeq 0\}.
	\end{equation}
   The set  $\cD_p^\infty$ is the \df{free operator semialgebraic set} defined by the polynomial $p.$
    We extend the notion of a $\Gat$-convex hull as in Definition \ref{def gamma} to accept \CB{as inputs}  free operator semialgebraic sets as follows.
	The \df{$\Gat$-convex hull} of $\cD_p^\infty$ is the graded set  \index{$\Gatconv$}
  $$\Gatconv(\cD_p^\infty) =\big(\Gatconv(\cD_p^\infty)(n)\big)_n,$$
  {where $n\in\N$. (Most of what follows also  works
  for \textit{operator} $\Gat$-convex hulls, $\Gat\mhyphen\opco(\cD_p^\infty).$ See Subsection \ref{ssec:moreop}.)}
	Here a tuple  $X \in \gtupn$ lies in $\Gatconv(\cD_p^\infty)(n)$ if there exists a tuple $Y \in \cD_p^\infty$ acting on some 
(finite-dimensional or separable)
	Hilbert space $\cH$ and an isometry $V: \C^n \to \cH$ such that $(Y,V)$ is a $\Gat$-pair and $X = V^*YV.$
  By construction this $\Gat$-convex hull is the smallest 
  $\Gat$-convex  set {\it spanned}  by $\cD_p^\infty.$

\begin{example}\rm
From  \cite{BP65} (see also \cite{Tao}), if $H$ is an infinite-dimensional separable Hilbert space, then
the operator 
\[
 C=\begin{pmatrix} I & I \\0 & I \end{pmatrix}
\]
 acting on $K=H\oplus H,$ where $I$ is the identity on $H,$ is a commutator; that is, there exist
 operators $X,Y$ on $K$ such that $[X,Y]=XY-YX=  C.$  Let $X_1=\frac12 (X+X^*)$ and $X_2=\frac{1}{2i} (X-X^*)$ denote  the real and imaginary parts of $X$ and similarly for $Y.$ In particular, $X=X_1+iX_2$ and likewise for $Y.$ Observe,
  \[
 \begin{split}
4\left ( [X,Y]+[X,Y]^* \right ) & = 4 \left ( [X,Y]- [X^*,Y^*]  \right ) = [X_1 +i X_2, Y_1 +i Y_2] -  [X_1-iX_2, Y_1-iY_2]
\\ & = 2 i \left ([X_2,Y_1] +  [X_1,Y_2] \right ). 
\end{split}
\]
Hence, {the polynomial} $p(x_1,x_2,y_1,y_2) = 2i([x_2,y_1] + [x_1,y_2]) -1 $ 
{is symmetric and further}
\[
  p(X_1,X_2,Y_1,Y_2)  = \begin{pmatrix} 3 & 2\\ 2 & 3\end{pmatrix} \succ0.
\]
On the other hand, for any tuple $(W_1,W_2,Z_1,Z_2)$ of self-adjoint matrices acting on a  Hilbert space of finite dimension $n$ we have
$p(W_1,W_2,Z_1,Z_2)$ has trace $-n$  and thus is not positive semidefinite.   Thus 
\begin{equation}\label{eq:infnonzero}
\cD_p^\infty\neq\emptyset=\cD_p.
\end{equation}
By direct summing $N\pm x_i$ and $N\pm y_i$ for an appropriate $N\in\N$, we can obtain an Archimedean polynomial $p$ satisfying \eqref{eq:infnonzero}.

In the setting of $xy$-convexity, where $\Gat$ consists of
(the real and imaginary parts of) $x_iy_j$, the constructed
polynomial $p$ is a $\Gat$-pencil. Thus $\cD_p^\infty$ is
an operator $\Gat$-convex set with empty finite levels $\cD_p$.
\qed
\end{example}

\subsection{Free Hankel matrices and moment sequences}
The main tool for constructing the spectrahedral lifts and the 
$\Gat$-convex
projections approximating $\cD_p$ are free analogs of Hankel matrices as in \cite{HKM16}. In the  classical matrix-valued single variable theory, a Hankel matrix $H = (H_{i,j})$ with entries from $M_n$ for some $n$ has the property of being constant on anti-diagonals, that is, there is a \textit{moment sequence} $(A_k)_k$ of self-adjoint matrices from $M_n$ such that $H_{i,j}= A_{i+j}.$

\begin{definition}\label{def:Hankel}
	Let $n$ be a positive integer and $Y = (Y_\alpha)_\alpha$ a sequence of $n \times n$ matrices indexed by words $\alpha$ in the free symmetric variables $x_1, \ldots, x_{\gv}.$  Let $\Gat=(\gat_1,\dots,\gat_\rv)$ be a tuple of symmetric free polynomials  with $\gat_j=x_j$ for $1\le j\le \gv\le \rv$ and monomial expansions $\gamma_j = \sum_k \gamma_{j,k}m_{j,k},$ where $\gamma_{j,k} \in \C$ and $m_{j,k}$ are words in $x$.
	\begin{enumerate}[\rm (a)]\itemsep=10pt
		\item The sequence $Y$ is a \df{$\Gat$-moment sequence} if  $Y_{\emptyset} = I,$ if, for each $\alpha,$ 
		\begin{equation}\label{eq-ms1}
			Y_{\alpha^\ast} = (Y_\alpha)^\ast,
		\end{equation}
	 and for each $j=1, \ldots, \rv,$ 
		\begin{equation}\label{eq:gmom}
		\sum_{k} \gamma_{j,k}\, Y_{m_{j,k}(x)}
		= 	\sum_{k} \gamma_{j,k} \,
		m_{j,k}(Y_{x_1}, \ldots, Y_{x_{\gv}}) = \gamma_j(Y_{x_1}, \ldots, Y_{x_{\gv}}).
		\end{equation}
		In particular, property \eqref{eq-ms1} implies each of the $Y_{x_j}$ is self-adjoint. 
		
		\item  Let $\mathfrak{M}^\Gat(n)$ denote the set of $\Gat$-moment sequences $(Y_\alpha)_\alpha$ with $Y_\alpha \in M_n.$

		\item The \df{free $\Gat$-Hankel matrix} associated to $Y\in \mathfrak{M}^\Gat$ is defined as 
		\begin{equation}\label{eq:hankeldef}
		H(Y) = \big(Y_{\alpha^\ast \beta}\big)_{\alpha, \beta},
		\end{equation}
		while for a positive integer $d,$ the corresponding \df{truncated free $\Gat$-Hankel matrix} is \looseness=-1
		\begin{equation}\label{eq:hankeldeftrunc}
		H_d(Y) = \big(Y_{\alpha^\ast \beta}\big)_{|\alpha|, |\beta| \leq d}.
		\end{equation}
		These definitions are taken from \cite{HKM16}, but are here only applied to $\Gat$-moment sequences $Y$.
		
		\item For a $\mu \times \mu$ {symmetric} matrix-valued polynomial $p$ as in  equation~\eqref{eq-p} and $Y\in\mathfrak{M}^\Gat(n),$
		 the \df{$p$-localizing matrix} $H^{\Uparrow}_p(Y)=(H^\Uparrow(Y)_{\alpha,\beta})_{\alpha,\beta}$ 
		 with  $n\mu \times n\mu$ matrix  block entry at position $(\alpha, \beta)$ 
		$$
		H^{\Uparrow}_p(Y)_{\alpha, \beta} = \sum_{|\gamma| \leq \delta} p_\gamma \otimes Y_{\alpha^\ast \gamma \beta}.
		$$
		For a positive integer $d,$ the \df{$d$-truncated localizing matrix} of $p$ is
		\[ \pushQED{\qed}
		H^{\Uparrow}_{p, d}(Y) = \big(H^{\Uparrow}_p(Y)_{\alpha, \beta} \big)_{|\alpha|, |\beta| \leq d}. \qedhere\popQED
		\]
	\end{enumerate}
\end{definition}

\begin{example}
	To illustrate the structure of a  $\Gat$-Hankel matrix, let $\Gat(x_1, x_2) = (x_1, x_2, x_2^2).$  To shorten the notation we write {$Y_j$} instead of $Y_{x_j}$ and  substitute the index $\alpha$ in $Y_\alpha$ by the indices of the variables in $\alpha,$ e.g., we write $Y_{121}$ instead of $Y_{x_1 x_2 x_1}.$  Equation \eqref{eq:gmom} then says that $Y_{22} = X_2^2$
	and the truncated $\Gat$-Hankel matrix for $d=2$ equals
	\begin{equation*}
		H_2(Y)=
		\begin{pmatrix} 
			1 & X_1 & X_2 & Y_{11} & Y_{12} & Y_{21} & X_2^2 \\[.1cm]
			X_1 & Y_{11} & Y_{12} & Y_{111} & Y_{112} & Y_{121} & Y_{122} \\[.1cm]
			X_2 & Y_{21} & X_2^2 & Y_{211} & Y_{212} & Y_{221} & Y_{222} \\[.1cm]
			Y_{11} & Y_{111} & Y_{112} & Y_{1111} & Y_{1112} & Y_{1121} & Y_{1122} \\[.1cm]
			Y_{21} & Y_{211} & Y_{212} & Y_{2111} & Y_{2112} & Y_{2121} & Y_{2122} \\[.1cm]
			Y_{12} & Y_{121} & Y_{122} & Y_{1211} & Y_{1212} & Y_{1221} & Y_{1222} \\[.1cm]
			X_2^2 & Y_{221} & Y_{222} & Y_{2211} & Y_{2212} & Y_{2221} & Y_{2222} 
		\end{pmatrix},
	\end{equation*}
	which is an $x_2^2$-pencil. \qed
\end{example}

\begin{remark}
 \label{r:Gat:pair:moment}
	Any $Z \in \mathcal{D}_p^\infty(n)$,
where $n\in\N\cup\{\infty\}$,
	 together with an isometry $V \in \Mm_{n,m}$ such that $(Z, V)$ is a $\Gat$-pair, determines a $\Gat$-moment sequence
	\begin{equation}\label{eq-ms}
		Y_\alpha = V^\ast Z^\alpha V\in\Mm_m,
	\end{equation}
	where $Z^\alpha = \alpha (Z).$ For example, if $\alpha =  x_1 x_2^2\, x_1,$ then
	$Y_\alpha = V^\ast Z_1 Z_2^2 Z_1 V.$ It is easy to verify that $(Y_\alpha)_\alpha$ is indeed a $\Gat$-moment sequence and direct computation shows
\begin{equation} \label{e:pos:con} \pushQED{\qed}
	H (Y) \succeq 0 \quad \text{ and } \quad \ H^{\Uparrow}_p(Y) \succeq 0. \qedhere\popQED
\end{equation}
\end{remark}

\subsection{Step 1: construction of the lift}\label{sec-lift}
{Fix a symmetric matrix-valued free polynomial $p$.}
Let $\mathfrak{L}^\Gat_p = (\mathfrak{L}^\Gat_p(n))_n$ be defined by 
\begin{equation*}
\mathfrak{L}^\Gat_p(n) = \{Y = (Y_{\alpha})_\alpha \in \mathfrak{M}^\Gat(n)  \ | \  \ H (Y) \succeq 0, \ H^{\Uparrow}_p(Y) \succeq 0\}.
\end{equation*}
 Given $Y \in \mathfrak{L}^\Gat_p(n)$ let
\begin{equation}
\label{d:Yhat}
\hat{Y} = (Y_{x_1}, Y_{x_2}, \ldots, Y_{x_{\gv}}) \in \gtup_n
\end{equation}
and denote
$$
\hat{\mathfrak{L}_p^\Gat} = \{\hat{Y} \ | \ Y \in \mathfrak{L}^\Gat_p\}.
$$
\begin{remark}\label{rem lpdp}
With the notation just introduced, Remark \ref{r:Gat:pair:moment}
states that  the $\Gat$-moment sequence $Y$ of equation~\eqref{eq-ms} belongs to $\mathfrak{L}^\Gat_p$ if $Z$ is in $\mathcal{D}_p^\infty.$ \qed
\end{remark}

We now introduce a boundedness assumption on the polynomial $p$ 
that will ensure our construction of the $\Gat$-convex lifts clamps down on the $\Gat$-convex hull of $\cD_p^\infty$.

\begin{definition}
	A matrix-valued polynomial $p$ is \df{Archimedean} if there is a $k>0$ and finitely many matrix-valued polynomials $s_j$ and $t_j$ such that
	\begin{equation}\label{eq:archpoly}
	k^2 - \sum_{i=1}^\gv x_i^2 = \sum_j s_j^\ast s_j + \sum_j t_j^\ast p \,t_j.
	\end{equation}
 In this case we say that $p$ is \df{$k$-Archimedean}.
 Observe that \eqref{eq:archpoly} implies each $X\in\cD_p^\infty$ satisfies
 $\|X\|\leq k$. \qed
\end{definition}

We now prove the main result of this subsection:\ a spectrahedral realization of the  $\Gat$-convex hull of the free operator semialgebraic set defined by an Archimedean polynomial. It is the $\Gat$-analog to \cite[Theorem 5.4]{HKM16}. 

\begin{theorem}\label{th-lift}
	If $p$ is Archimedean, then
	$$
	\Gat\mhyphen\conv(\mathcal{D}_p^\infty) = \hat{\mathfrak{L}_p^\Gat}.
	$$
	
	Moreover, if $p$ is $k$-Archimedean, then $\hat{\mathfrak{L}_p^\Gat}$ is bounded by $k$ and if
	 $Y\in \mathfrak{L}^\Gat_p,$ then $\|Y_\alpha \| \le k^{|\alpha|}$ for  each word $\alpha.$ 
\end{theorem}

The proof of Theorem~\ref{th-lift} will use the following proposition.

\begin{proposition}
\label{prop:lift}
   If $p$ is Archimedean, then for each $Y\in \mathfrak{L}^\Gat_p(n)$ there exists a $\Gat$-pair $(T,V)$ with  $T\in \cD_p^\infty$ and $V$ an isometry
  such that $Y_\alpha=V^* T^\alpha V$ for all $\alpha.$ 
\end{proposition}

\begin{proof}
 Let $Y\in \mathfrak{L}^\Gat_p(n)$ be given.  Define a sesquilinear form $[ \cdot, \cdot]$ on
	$\mathcal{V} = \C\langle x\rangle \otimes \C^n$ by
	\begin{equation*}
		[ s,t ]_Y = \sum_{\alpha,\beta}  \langle Y_{\beta^* \alpha} s_\alpha, t_\beta \rangle
	\end{equation*}
	where $s=\sum \alpha \otimes s_\alpha$ and $t=\sum \beta\otimes t_\beta$. 
	The assumption $H (Y)\succeq 0$ implies $[s,s]_Y \geq 0$ for all $s \in \mathcal{V},$ so that this sesquilinear  form is positive semidefinite. 
	A standard argument (using the Cauchy-Schwarz inequality) shows that
	\[
	\cN = \{s : [s,s]_Y = 0\}
	\]
	is a subspace of $\mathcal{V}.$ Modding out $\cN$ produces a well-defined and positive semidefinite form
	$$
	[ s,t ]_Y = [ s + \cN,t+\cN ]_Y
	$$
	on the quotient $\mathcal{W}$ of $\mathcal{V}$ by $\cN,$ where, as is standard practice, $s\in \mathcal{V}$ is identified
	with its image $s+\cN$ in the quotient.
	
	Similar to the standard GNS construction we now show that $\cN$ is a left $\C\ax$-submodule: if $s\in \cN$ and $1\le j\le \gv$, then 
	$r=x_j s \in\cN.$
	Indeed, if $s\in\cN$, then 
	$H (Y) \succeq 0$ implies
	\[
	\sum_\alpha Y_{\beta^* \alpha} s_\alpha =0
	\]
	for each $\beta$ (and conversely). Hence
	\[
	\sum_\gamma Y_{\beta^* \gamma} r_\gamma 
	=  \sum_{\alpha} Y_{\beta^* x_j\alpha} s_\alpha 
	=  \sum_{\gamma} Y_{(x_j\beta)^* \alpha} s_\alpha 
	=  0
	\]
	so that $r\in\cN$.  It now follows that the multiplication mapping $Z_j$
	sending $s$ to $x_j s$ is well defined on $\mathcal{W}.$ The computation above also shows that (even without the condition $s\in \mathcal N$), 
	\begin{equation}
	\label{e:lift:1}
	  [ x_j s,t ]_Y = [ s, x_j t ]_Y.
	\end{equation}
	Further, for each word $\gamma$ we obtain an operator $Z^\gamma = \gamma(Z)$ on $\mathcal{W}$ satisfying
	$Z^\gamma s= ws = \sum \gamma \alpha\otimes s_\alpha.$ 
		 
	To prove $p(Z)=\sum p_\eta \otimes Z^\eta$ is positive semidefinite on $\mathcal{W},$ let
	$s=\sum e_j\otimes \alpha \otimes s_{\alpha,j}$, where $\{e_1,\dots,e_\mu\}$
	is the standard orthonormal basis for $\C^\mu$ (with $p_\eta \in \Mm_\mu$)
	and $s_{\alpha,j}\in\C^n$ and observe
	\begin{equation}
	\label{e:lift:2}
	\begin{split}
		\langle p(Z)s,s \rangle & =  \sum_{\alpha,\beta,\eta,j,k} \langle (p_\eta \otimes Z^\eta) e_j\otimes \alpha\otimes s_{\alpha,j}, e_k\otimes \beta \otimes s_{\beta,k}\rangle \\
		&   =  \sum \langle p_\eta e_j,e_k\rangle \, \langle Z^\eta \alpha\otimes s_{\alpha,j}, \beta\otimes s_{\beta,k}\rangle \\
		& =  \sum \langle p_\eta e_j,e_k\rangle \, \langle Y_{\beta^* \eta \alpha} s_{\alpha,j},s_{\beta,k}\rangle \\
		& =  \sum_{\alpha,\beta}  \Big\langle \big(\sum_\eta p_\eta \otimes Y_{\beta^*\eta\alpha}\big) \sum_j e_j \otimes s_{\alpha,j}, \sum_k e_k \otimes s_{\beta,k}\Big\rangle \\
		& =  \langle H_p^{\Uparrow}(Y) \vec s,\vec s \, \rangle, 
	\end{split}
	\end{equation}
	where $\vec s = (s_\alpha)_\alpha$ is the vector with $s_\alpha = \sum_j e_j \otimes s_{\alpha,j}$. The assumption $H_p^{\Uparrow}(Y) \succeq 0$ 
	implies $p(Z) \succeq 0$ as desired.   
	
	Define $Q:\C^n \to \mathcal{W}$ by
	\[
	Qv = \emptyset \otimes v.
	\]
	It is straightforward to see that $Q$ is an isometry and by construction, 
	\begin{equation}\label{eq-qz}
		Q^* Z^\alpha Q= Y_\alpha
	\end{equation}
	for all $\alpha,$ where $Z^\alpha$ is defined as $\alpha(Z_1,\ldots,Z_\gv).$ Moreover, $(Z, Q)$ is a $\Gat$-pair, where $Z = (Z_1, \ldots, Z_{\gv}).$  Indeed, if $\gamma_j = \sum_k \gamma_{j,k}m_{j,k}$ with $\gamma_{j,k} \in \C$ and words $m_{j,k},$
	then \eqref{eq-qz} together with the fact that $Y$ is a $\Gat$-moment sequence implies that
	\begin{align} \label{e:QZQ}
		\gamma_j(Q^* Z Q) & = \gamma_j(\hat{Y})  = \sum_k \gamma_{j,k}m_{j,k}(\hat{Y}) = 
		\sum_k \gamma_{j,k} Y_{m_{j,k}} = \sum_k \gamma_{j,k} Q^\ast Z^{m_{j,k}} Q\\
		 & = \notag
		Q^* \left(  \sum_k \gamma_{j,k} Z^{m_{j,k}}\right)   Q = Q^* \gamma_j(Z) Q
	\end{align}
	 for every $j=1, \ldots, \rv.$

	Since $p$ is Archimedean, the $Z_j$ are bounded operators.  Indeed, 
	 since $p$ is Archimedean, 
	\[
	k^2-\sum_j x_j^2 = \sum_i f_i^*f_i + \sum_k g_k^* p\, g_k
	\]
	for some $k>0$ and  noncommutative polynomials $f_i,g_k;$ and  $p(Z) \succeq 0$  
by \eqref{e:lift:2}.
	Hence
	$$k^2-\sum_j Z_j^2\succeq0.$$ 
	So $\|Z\|\le k$ and in particular $\|Z_j\|^2\leq k^2.$ Thus  $Z_j$ is bounded for each $1\le j\le \gv.$
        It now follows that, for  each word $\gamma,$ the operator $Z^\gamma$ on $\mathcal{W}$ extends to a  bounded operator $\widetilde{Z}_\gamma$ on 
	 the completion $\mathcal{H}$ of $\mathcal{W}.$  Setting $T_j=\widetilde{Z}_{x_j},$  the  identity of equation~\eqref{e:lift:1} implies  $T_j^*=T_j.$
	In particular, if $(h_n)$ is a sequence from $\mathcal{W}$ that converges to $h\in \cH,$ then $(Z_j h_n)=(T_j h_n)$ 
	 is a sequence from $\mathcal{W}$ that converges
	 to $T_j h$ and thus, by the invariance of $\mathcal{W}$ under $Z_j,$ it follows that $(Z^w h_n)$ converges 
	 to $\widetilde{Z}_w h= T^w h.$ For instance, $(Z_1 Z_2 h_n = Z_1 [Z_2 h_n])$  converges to $T_1 [T_2 h]$
	 since $(Z_2h_n)$ converges to $\widetilde{Z}_{x_2} h.$ 
	 Consequently,  $q(Z)$ is the restriction of $q(T)$ to $\mathcal{W}$ for $q\in \C\langle x \rangle.$ 
	 
	 It now follows that $p(T)\succeq 0$ since $p(Z)\succeq 0.$  Hence $T\in \cD_p^\infty.$    Further, $Q$ is an isometry into $\cH$
	  such that  $Q^*\gamma_j(T)Q = Q^*\gamma_j(Z)Q.$ Hence 
	 $\gamma_j(Q^* TQ)= Q^*\gamma_j(T)Q$ from equation~\eqref{e:QZQ}.  
	  Thus $(T,Q)$ is a $\Gat$-pair.  Finally, equation~\eqref{eq-qz} gives
	 $Q^* T^\alpha Q = Y_\alpha$ for each word $\alpha.$  
\end{proof}

\begin{proof}[Proof of Theorem~\ref{th-lift}]
	To prove the inclusion $\Gat\mhyphen\conv(\mathcal{D}_p^\infty) \subseteq  \hat{\mathfrak{L}_p^\Gat},$ note that every $X \in \Gat\mhyphen\conv\, (\mathcal{D}_p^\infty)$ is of the form $V^\ast Z V$ for some $Z \in \mathcal{D}_p^\infty$ and isometry $V$ such that $(Z, V)$ is a $\Gat$-pair. Remark~\ref{r:Gat:pair:moment} now implies the moment sequence $Y=(Y_\alpha)_\alpha$ with $Y_\alpha = V^\ast Z^\alpha V$ as in \eqref{eq-ms} belongs to $\mathfrak{L}_p^\Gat,$ hence $X = \hat{Y} \in \hat{\mathfrak{L}_p^\Gat}.$
	
	To prove the converse,  let $\widehat{Y}\in \hat{\mathfrak{L}_p^\Gat}$ be given. 
  Thus there exists a $\Gat$-moment sequence $(Y_\alpha)_\alpha$ from $\mathfrak{L}^\Gat_p(n)$ such that  $\widehat{Y}$ 
  is given by equation~\eqref{d:Yhat}.  From Proposition~\ref{prop:lift}, there is a $\Gat$-pair $(T,V)$ such that $T\in \cD_p^\infty$
  and $V$ is an isometry such that $\widehat{Y} = V^* T V$ so that $\widehat{Y}\in \Gat\mhyphen\conv(\cD_p^\infty).$

	Finally, since $\|T_j\|\le k$ for each $j,$ it follows that $\|T^\alpha\|\le k^{|\alpha|}$ for each word $\alpha.$
	Hence  $\| Y_\alpha \| =\| V^* T^\alpha V\| \le \|T^{\alpha}\| \le k^{|\alpha|}.$ Similarly, $\|\hat{Y}\|\le k$ since $\|T\|\le k.$
\end{proof}

\begin{corollary} 
	If $p$ is Archimedean, then $\Gat\mhyphen\conv(\mathcal{D}_p^\infty)$ is closed and bounded. 
\end{corollary}

\begin{proof}
	Boundedness of $\Gat\mhyphen\conv(\mathcal{D}_p^\infty)$ follows from Theorem~\ref{th-lift}.

	To prove that it is closed, suppose  $(X^{(k)})_k$ is a sequence from  $\Gat\mhyphen\conv (\mathcal{D}_p^\infty)$ that  converges to some $X \in \gtupn.$ 
	For each $k$ there is an isometry $V_k$ and an element $Y^{(k)} \in \cD_p^\infty$ such that $(Y^{(k)},V_k)$ is a $\Gat$-pair and 
	$X^{(k)} = V_k^* Y^{(k)} V_k.$ As noted in 
	Remark~\ref{rem lpdp}, for each $k,$ the moment sequence $(Z^{(k)}_\alpha)_\alpha$ defined by
	$$
	Z^{(k)}_\alpha = V_k^* (Y^{(k)})^\alpha V_k
	$$
	lies in $\mathfrak{L}_p^\Gat.$ Now for any $\alpha,$ Theorem~\ref{th-lift}  implies that the sequence $(Z^{(k)}_\alpha)_k$ is bounded. 
	It thus has a convergent subsequence and by passing to a subsequence, we can assume that for each $\alpha,$ the sequence $(Z^{(k)}_\alpha)$ converges
	 to some $Z_\alpha.$ Since membership in $\mathfrak{L}_p^\Gat$ is determined by the positivity conditions of equation~\eqref{e:pos:con},
	  it follows that  $(Z_\alpha)_\alpha \in \mathfrak{L}_p^\Gat.$
	By construction, $X = (Z_{x_1},\ldots,Z_{x_\gv}) \in \hat{\mathfrak{L}_p^\Gat} = \Gat\mhyphen\conv (\mathcal{D}_p^\infty),$ where the last equality is guaranteed by Theorem \ref{th-lift}.
\end{proof}

\subsection{Step 2: truncated lifts and \texorpdfstring{$\Gat$}{Gamma}-moment sequences}
\label{ssec:step2}

Here we show that the degree-bound truncations of $\mathfrak{L}_p^\Gat$ form a sequence of finite free $\Gat$-spectrahedral lifts of $\cD_p^\infty$ whose projections give better and better outer approximations of the $\Gat$-convex hull of $\cD_p^\infty.$

\subsubsection{The clamping down theorem}  Let $\delta$ be the maximum degree of the polynomials $\gamma_j.$
For $n \in \N$ and $d \geq \delta$ denote the $n$-th level of the $d$-truncation of $\mathfrak{L}_p^\Gat$ by
\begin{equation}\label{eq:clamp1}
\begin{split}
\mathfrak{L}_p^\Gat(n, d) = \big\{Y = (Y_\alpha)_{|\alpha|\le 2d+\deg p +1}  &\mid   Y_\alpha \in M_n,  \ Y_\emptyset =I, \ Y_{\alpha^*}=Y_\alpha^*,\\ 
&\phantom{{}\mid{}} Y \text{ satisfies } \eqref{eq:gmom}, \  
H _{d + \left\lceil\frac12\deg p\right\rceil} (Y) \succeq 0, \ H^{\Uparrow}_{p,d}(Y)  \succeq  0 \big\}
\end{split}
\end{equation}
and let $\mathfrak{L}_p^\Gat(\cdot, d) = (\mathfrak{L}_p^\Gat(n, d))_n.$ As in the previous section define
\begin{equation}\label{eq:clamp2}
\hat{\mathfrak{L}_p^\Gat}(n, d) = \{\hat{Y} = (Y_{x_1}, Y_{x_2}, \ldots, Y_{x_{\gv}}) \in \gtupn \ | \ Y \in \mathfrak{L}_p^\Gat(n, d)\}.
\end{equation}

The next theorem asserts that the $d$-truncations $\hat{\mathfrak{L}_p^\Gat}(\cdot, d)$ are projections of free $\Gat$-spectrahedra and that they intersect precisely at $\hat{\mathfrak{L}_p^\Gat},$ which, by Theorem \ref{th-lift}, is $\Gat\mhyphen\conv(\cD_p^\infty).$

\begin{theorem}\label{th-cd}
	If $p$ is an Archimedean symmetric matrix-valued noncommutative polynomial, then 
	\begin{enumerate}[\rm (a)]
		\item \label{i:cd:a} 
     for every $n,$
		\begin{equation}\label{eq-cd}
			\bigcap_{d=0}^\infty \hat{\mathfrak{L}_p^\Gat}(n, d) = \hat{\mathfrak{L}_p^\Gat}(n) = \Gat\mhyphen\conv  (\cD_p^\infty)(n); 
		\end{equation}
		\item \label{i:cd:b}
    for every $d$ there is a $\Gat$-pencil $\LG_d$ given by a tuple $A^{(d)}$ such that $\hat{\mathfrak{L}_p^\Gat}(\cdot, d)$ is the projection of $\cD^\Gat_{A^{(d)}}.$ 
	\end{enumerate}
	Hence, the free $\Gat$-spectrahedrops $\hat{\mathfrak{L}_p^\Gat}(\cdot, d)$ give increasingly finer outer approximations of the $\Gat$-convex hull of $\cD_p^\infty.$
\end{theorem}

\begin{lemma}\label{lema:seqbded}
	Let $p$ be an Archimedean polynomial. Then there is a natural number $\nu$ and a constant $c>0$ such that for every $Y \in \mathfrak{L}_p^\Gat(n, d)$ and word $\alpha$ of length $|\alpha|\leq2(d-\nu),$
	$$
	\| Y_\alpha \| \leq c^{|\alpha|}.
	$$
\end{lemma}

The proof  follows along the lines of \cite[Lemma 6.5]{HKM16}, since the definition of our localizing matrices $H_p^{\Uparrow}(Y)$ coincides with the one there.

\begin{proof}[Proof of Theorem \ref{th-cd}] 
	To prove item~\ref{i:cd:a}  first observe that if $(Y_\alpha)_\alpha$ is a moment sequence, then 
	\begin{equation}\label{eq:111}
		H (Y) \succeq 0 \quad \text{ and } \quad \ H^{\Uparrow}_p(Y) \succeq 0
	\end{equation}
	if and only if for all $d,$
	\begin{equation}\label{eq:112}
		H _{d + \left\lceil\frac12\deg p\right\rceil} (Y) \succeq 0 \quad \text{ and } \quad \ H^{\Uparrow}_{p, d}(Y) \succeq 0.
	\end{equation}

	The second set equality in \eqref{eq-cd} is given in Theorem \ref{th-lift}.
	To prove the nontrivial inclusion $\subseteq$ in the first equality of \eqref{eq-cd}
	let $Z\in \bigcap_d\hat{\mathfrak{L}_p^\Gat}(n,d)$.
	For every $d$ there is a (truncated) moment sequence
	$Y^{(d)} = (Y^{(d)}_\alpha)\in {\mathfrak{L}_p^\Gat}(n;d)$ such that
	\[
	(Y^{(d)}_{x_1},\dots,Y^{(d)}_{x_{\gv} }) = Z.
	\]
	
	 By Lemma~\ref{lema:seqbded},  for a given word $\alpha,$ 
	the sequence $(Y^{(d)}_\alpha)_{|\alpha|\leq 2d+\deg(p)}$ is bounded. Since 
	 there are countably many such 
	sequences, there exists a moment sequence $(Y_\alpha)$ from $M_n$ and a subsequence $(d_k)_k$ of indices  such that
	$(Y^{(d_k)}_\alpha)$ converges to  $Y_\alpha \in M_n$ for each word $\alpha.$
	
	For each $k$ the sequence $Y^{(d_k)}$ satisfies \eqref{eq:112} with $d=d_k$ and hence for all $d\le d_k.$ Hence the limit $\Gat$-moment sequence 
	$Y = (Y_\alpha)_\alpha$ satisfies \eqref{eq:112} for all indices $d$  and thus $Y$ 
       satisfies \eqref{eq:111} and hence belongs to $\mathfrak{L}_p^\Gat(n).$ Thus, $Z\in \hat{\mathfrak{L}_p^\Gat}(n)$ since
	\[
	Z = (Y_{x_1},\dots,Y_{x_{\gv} }) = \hat{Y}. 
	\]
	
	Item~\ref{i:cd:b}  is proved in the next Subsection \ref{sec:gspec}, where we show that the $\hat{\mathfrak{L}_p^\Gat}(\cdot, d)$ are in  fact free $\Gat$-spectrahedrops.
\end{proof}

\subsubsection{Free \texorpdfstring{$\Gat$}{Gamma}-spectrahedral lifts of the \texorpdfstring{$\hat{\mathfrak{L}_p^\Gat}(\cdot, d)$}{LpGamma}}\label{sec:gspec}
By the Lasserre-Parrilo construction in Section \ref{sec-lift}, the $\hat{\mathfrak{L}_p^\Gat}(\cdot, d)$ are projections of the approximate lifts $\mathfrak{L}_p^\Gat(\cdot, d)$ and each of the latter is the positivity set of a matrix-valued polynomial of the form
\begin{equation*}
	\LG(x, y) = \cA(x) + \sum_{j=1}^{\hv} (B_j y_j + B_j^\ast y_j^\ast), 
\end{equation*}
where $\cA$ is a {(not necessarily monic) } $\Gat$-pencil of size $k$ and $B_j \in \Mm_k.$ The following procedure, described in \cite[Lemma 6.3]{HKM16}, 
replaces the non-self-adjoint coefficients $B_j$ and non-symmetric variables $y_j$ in $\LG(x, y)$ with (twice as many) self-adjoint coefficients and symmetric variables such that the resulting  $\Gat$-pencil, denoted $\widetilde{\LG},$  provides  a lift of $\hat{\mathfrak{L}_p^\Gat}(\cdot, d)$ to a $\Gat$-spectrahedron:
$$
{\rm proj}_x \cD^\Gat_{\widetilde{\LG}} = {\rm proj}_x \cD^\Gat_{\LG} = \hat{\mathfrak{L}_p^\Gat}(\cdot, d).
$$

Decomposing $B_j = C_j + \mathfrak{i}D_j$ and $y_j = w_j + \mathfrak{i}w_{-j}$ into self-adjoint matrices $C_j, D_j$ and free symmetric variables $w_{-\hv}, \ldots, w_{-1}, w_1, \ldots, w_\hv$  gives
\[
\begin{split}
	B_j y_j + B_j^* y_j^* & = 
	(C_j + \mathfrak{i} D_j) (  w_{j} + \mathfrak{i} w_{-j} ) + 
	(C_j - \mathfrak{i} D_j) (  w_{j} - \mathfrak{i} w_{-j} ) \\
	&= 2 ( C_j w_j - D_j w_{-j} )
\end{split}
\]	
and 
$$
\widetilde{\LG}(x,w) = \cA(x)
+ 2 \sum_{j=1}^h \big( C_j w_j - D_j w_{-j})  
$$
is a $\Gat$-pencil with self-adjoint coefficients and symmetric variables satisfying ${\rm proj}_x \cD^\Gat_{\widetilde{\LG}} = {\rm proj}_x \cD^\Gat_{\LG}.$

\subsection{Examples}
We next give a few examples of the Lasserre-Parrilo lifting construction in the $\Gat$-setting.

\begin{example}
	Denote the variables by $x, y$ instead of $x_1, x_2$ and let $\Gat = (x, y, y^2).$ Consider $p = (1-2 y^2+x^2) \oplus (1 - x^2 ).$ Then
	$$
	\cD_p = \{(X, Y) \ | \ 2Y^2  \preceq 1+X^2, \ X^2 \preceq I\}
	$$
	is not $y^2$-convex as seen from Figure \ref{fig-pr} representing $\cD_p(1).$ 
	\begin{figure}[ht]
	\centering
	 \begin{tikzpicture}[scale=2]
			\draw[very thin,color=gray] (-1.6,-1.6) grid (1.6,1.6);
			\draw[->] (-1.5, 0) -- (1.5, 0) node[right] {$x$};
			\draw[->] (0, -1.5) -- (0, 1.5) node[above] {$y$};
			
			\draw[scale=1, domain=-1:1, smooth, variable=\x, black] plot ({\x}, {sqrt((\x*\x+1)/2)});
			\draw[scale=1, domain=-1:1, smooth, variable=\x, black] plot ({\x}, {-sqrt((\x*\x+1)/2)});
			\draw[scale=1, domain=-1:1, smooth, variable=\x, black] plot (1,{\x});
			\draw[scale=1, domain=-1:1, smooth, variable=\x, black] plot (-1,{\x});
			
			\fill[blue!30, opacity=0.5] plot[domain=-1:1, variable=\x] ({\x}, {sqrt((\x*\x+1)/2)})
			--	plot[domain=-1:1, variable=\x] ({\x}, {-sqrt((\x*\x+1)/2)}) 
			-- plot[domain=-1:1, variable=\y] (1, {\y}) 
			-- plot[domain=-1:1, variable=\y] (-1, {\y}) --  cycle;
		\filldraw[color=black] (1,0) circle (0.02) node[below right] {$1$};
\filldraw[color=black] (0,1) circle (0.02) node[above left] {$1$};
		\end{tikzpicture}
\caption{The first level component $\cD_p(1)$ is clearly not convex in $x,$ hence $\cD_p$ is not $y^2$-convex.}
\label{fig-pr}
\end{figure}
	
	Let us prove that the $y^2$-convex hull of $\cD_p,$ i.e., the convex hull with respect to the $x$-coordinate, is
	\begin{equation}\label{eq-xhull}
		\Gat\mhyphen\conv (\cD_p) = \{(X, Y) \ | \ Y^2 \preceq I, \ X^2 \preceq I\}.
	\end{equation}
	Denoting the right hand side in \eqref{eq-xhull} by $\bm{K} = (K_n)_n,$ we immediately see that $\bm{K}$ is free and convex (in $X$ for each fixed $Y$). We now prove that, for each $n\in \N,$  any point $(X,Y) \in K_n $ is a $y^2$-convex combination of points from $\cD_p.$ 
	
	Note that by combining the two defining inequalities for $\cD_p,$ we obtain $Y^2 \preceq I$ on $\cD_p$. That is, $\cD_p\subseteq\bm K$.
	Let $(X,Y)\in\bm K$, i.e., $X$ and $Y$ are contractions.
	 Then $X$ can be diagonalized as $X = U^* DU,$ where $D$ is diagonal with diagonals $d_1,\ldots,d_n \in [-1,1]$ and $U$ is unitary. Write $d_1$ as a convex combination of the points $\pm 1$ and express $X = t X_1 + (1-t) X_2$ as a convex combination of matrices $X_1,X_2,$ where $X_i= U^* D_iU$ and $D_i$ has diagonals $(-1)^i,d_2,\ldots,d_n.$ Repeat this process on $X_1$ and $X_2$ to obtain an expression of $X$ as a convex combination 
	$$
	X = \sum_{i=1}^k t_i X_i
	$$ 
	of matrices $X_1,\ldots,X_k$ with $X_i^2 = I$. Clearly, for each $i,$ the tuple $(X_i,Y)$ lies in $\cD_p$ since
	\[
1+X_i^2 = 2 \succeq 2Y^2.
	\]
We can now write
\[
(X,Y)= \sum_{i=1}^k t_i (X_i,Y)\in\Gat\mhyphen\conv (\cD_p).
\]
We deduce that $\bm{K}$
	 must be contained in every other free set which is $y^2$-convex and contains $\cD_p.$
	 Hence $\bm{K}=\Gat\mhyphen\conv (\cD_p).$
	The first level of \eqref{eq-xhull} is the square $[-1,1]^2$.
	
	We now show that the projection of the first Lasserre-Parrilo lift $\mathfrak{L}_p^\Gat(\cdot, 0),$ given by 
	\begin{align*}
		H _1(Z) &=
		\begin{pmatrix} 
			I & X & Y \\[.1cm]
			X & Z_{11} & Z_{12} \\[.1cm]
			Y & Z_{21} & Y^2  \\[.1cm]
		\end{pmatrix}\succeq0,\\
		H_{p, 0}^{\Uparrow}(Z) &=  
		\begin{pmatrix}
			I-2Y^2+Z_{11} & 0 \\
			0 & I -Z_{11}
		\end{pmatrix}\succeq0,
	\end{align*}
	equals $\Gat\mhyphen\conv (\cD_p),$ i.e., the lift is exact. Indeed, for any $(X, Y) \in \Gat\mhyphen\conv (\cD_p),$ the canonical moment sequence \eqref{eq-ms} lies in $\mathfrak{L}_p^\Gat(\cdot, 0).$ For the other inclusion let $(X,Y) \in \hat{\mathfrak{L}_p^\Gat}(\cdot, 0).$ Then combining the inequalities 
	\begin{align*}
		I-2Y^2+Z_{11} \succeq 0 \quad \text{ and } \quad I -Z_{11}  \succeq 0
	\end{align*}
	gives $Y^2 \preceq I,$ 
	while taking Schur complements in $H_1(Z)$ implies
	$$
	\begin{pmatrix}
		Z_{11} & Z_{12}\\
		Z_{21} & Y^2
	\end{pmatrix}
- \begin{pmatrix}
	X \\ Y
\end{pmatrix}
\begin{pmatrix}
	X & Y
\end{pmatrix} =
\begin{pmatrix}
	Z_{11}-X^2 & Z_{12} - XY\\
	Z_{21}-YX & 0
\end{pmatrix}
\succeq 0.
	$$
	The latter forces $X^2 \preceq Z_{11}$ and hence
	$$
	X^2 \preceq Z_{11} \preceq I,
	$$
	so the tuple $(X,Y)$ lies in $\Gat\mhyphen\conv (\cD_p).$

	Finally, using Theorem~\ref{th-cd}, we can deduce
\begin{equation}\label{eq:notTV1}
\Gat\mhyphen\conv  (\cD_p) \subseteq
\Gat\mhyphen\conv  (\cD_p^\infty)
= 
	\bigcap_{d=0}^\infty \hat{\mathfrak{L}_p^\Gat}(\cdot, d) = \hat{\mathfrak{L}_p^\Gat} \subseteq 
	\hat{\mathfrak{L}_p^\Gat}(\cdot, 0) = \Gat\mhyphen\conv  (\cD_p),
\end{equation}
whence we have equalities throughout. \qed
\end{example}

\begin{example}
Consider the free semialgebraic set $\cD_p$ defined by 
\[p(x,y)=1-x^2-y^4,
\]
 the so-called TV screen \cite[Example 1.1]{JKMMP}. 
Figure \ref{fig-TV} depicts $\cD_p(1).$ 
\def\nos{100}
\begin{figure}[ht]
\begin{center}
\begin{tikzpicture}[domain=-1.6:1.6,scale=2] 
\draw[very thin,color=gray] (-1.6,-1.6) grid (1.6,1.6);
\draw[color=black, domain=1:-1, very thick,  samples=\nos] plot (\x,sqrt{(sqrt{(1-(\x)^2)})});
\draw[color=black, domain=1:-1, very thick,  samples=\nos] plot (\x,-sqrt{(sqrt{(1-(\x)^2)})});
\draw[color=black, domain=-1:1, very thick,  samples=\nos] plot (\x,sqrt{(sqrt{(1-(\x)^2)})});
\draw[color=black, domain=-1:1, very thick,  samples=\nos] plot (\x,-sqrt{(sqrt{(1-(\x)^2)})});
 \fill[color=blue!20, domain=1:-1, samples=\nos] plot (\x,sqrt{(sqrt{(1-(\x)^2)})}); 
 \fill[color=blue!20, domain=1:-1, samples=\nos] plot (\x,-sqrt{(sqrt{(1-(\x)^2)})});
 \fill[color=blue!20, domain=-1:1, samples=\nos] plot (\x,sqrt{(sqrt{(1-(\x)^2)})}); 
 \fill[color=blue!20, domain=-1:1, samples=\nos] plot (\x,-sqrt{(sqrt{(1-(\x)^2)})});
\fill[blue!20] (0,0) circle (1);
\draw[->] (-1.5,0) -- (1.5,0) node[right] {$x$}; \draw[->] (0,-1.5) -- (0,1.5) node[above] {$y$};
\filldraw[color=black] (1,0) circle (0.02) node[below right] {$1$};
\filldraw[color=black] (0,1) circle (0.02) node[above left] {$1$};
\end{tikzpicture} 
\end{center}~\caption{Bent TV screen $\cD_p(1)=\{ (x,y)\in\R^2 : 1-x^2-y^4\geq0\}$.}
			\label{fig-TV}
\end{figure}

 Since
 $p(X,Y)\succeq0$ if and only if
 \begin{equation*} 
\begin{pmatrix}
I & X & Y^2 \\ X& I & 0 \\ Y^2 & 0 & I
\end{pmatrix}\succeq0,
 \end{equation*}
 $\cD_p$ is $y^2$-convex. Let us show that the projection of the first Lasserre-Parrilo lift 
$\mathfrak{L}_p^\Gat(\cdot, 0),$ given by 
	\begin{equation}\label{eq:TV2}
	\begin{split}
		H _2(Z) &=
		\begin{pmatrix} 
		 I & X & Y & Z_{11} & Z_{12} & Z_{21} & Y^2 \\[1mm]
 X & Z_{11} & Z_{12} & Z_{111} & Z_{112} & Z_{121} & Z_{122} \\[1mm]
 Y & Z_{21} & Y^2 & Z_{211} & Z_{212} & Z_{221} & Z_{222} \\[1mm]
 Z_{11} & Z_{111} & Z_{112} & Z_{1111} & Z_{1112} & Z_{1121} & Z_{1122} \\[1mm]
 Z_{21} & Z_{211} & Z_{212} & Z_{2111} & Z_{2112} & Z_{2121} & Z_{2122} \\[1mm]
 Z_{12} & Z_{121} & Z_{122} & Z_{1211} & Z_{1212} & Z_{1221} & Z_{1222} \\[1mm]
 Y^2 & Z_{221} & Z_{222} & Z_{2211} & Z_{2212} & Z_{2221} & Z_{2222} 
		\end{pmatrix}\succeq0,\\
		H_{p, 0}^{\Uparrow}(Z) &=  
		\begin{pmatrix}
 			 I-Z_{11}-Z_{2222}
		\end{pmatrix}\succeq0,
	\end{split}
	\end{equation}
 of $\cD_p$ is exact, i.e., $\hat{\mathfrak{L}_p^\Gat}(\cdot, 0)=\cD_p$. 

Suppose $(X,Y,Z)$ satisfy \eqref{eq:TV2}.
By considering the top left $2\times2$ submatrix of $H_2(Z)$ we get 
\begin{equation}\label{eq:TV3}
Z_{11}\succeq X^2,
\end{equation}
and by considering the $2\times2$ principal submatrix on the first and last column we deduce
\begin{equation}\label{eq:TV4}
Z_{2222}\succeq Y^4.
\end{equation}
Using \eqref{eq:TV3} and \eqref{eq:TV4} in $H_{p, 0}^{\Uparrow}(Z) \succeq0$ yields
\[
0\preceq I-Z_{11}-Z_{2222} \preceq I-X^2-Y^4,
\]
whence $(X,Y)\in\cD_p$.
We can now conclude as in \eqref{eq:notTV1} that
\[ \pushQED{\qed}
\cD_p
= 
	\bigcap_{d=0}^\infty \hat{\mathfrak{L}_p^\Gat}(\cdot, d) = \hat{\mathfrak{L}_p^\Gat} =
	\hat{\mathfrak{L}_p^\Gat}(\cdot, 0). \qedhere\popQED
\]
\end{example}

\begin{example}
The situation changes if we consider the even more bent TV screen, that is, the free semialgebraic set $\cD_p$ defined by 
\[p(x,y)=1-x^2-y^6.
\]
Figure \ref{fig-TV6} depicts $\cD_p(1).$ 
\def\nos{100}
\begin{figure}[ht]
\begin{center}
\begin{tikzpicture}[domain=-1.6:1.6,scale=2] 
\draw[very thin,color=gray] (-1.6,-1.6) grid (1.6,1.6);
\draw[color=black, domain=1:-1, very thick,  samples=\nos] plot (\x,{((1 - (\x)^2)^(1/6))});
\draw[color=black, domain=1:-1, very thick,  samples=\nos] plot (\x,-{((1 - (\x)^2)^(1/6))});
\draw[color=black, domain=-1:1, very thick,  samples=\nos] plot (\x,{((1 - (\x)^2)^(1/6))});
\draw[color=black, domain=-1:1, very thick,  samples=\nos] plot (\x,-{((1 - (\x)^2)^(1/6))});
 \fill[color=blue!20, domain=1:-1, samples=\nos] plot (\x,{((1 - (\x)^2)^(1/6))}); 
 \fill[color=blue!20, domain=1:-1, samples=\nos] plot (\x,-{((1 - (\x)^2)^(1/6))});
 \fill[color=blue!20, domain=-1:1, samples=\nos] plot (\x,{((1 - (\x)^2)^(1/6))}); 
 \fill[color=blue!20, domain=-1:1, samples=\nos] plot (\x,-{((1 - (\x)^2)^(1/6))});
\fill[blue!20] (0,0) circle (1);
\draw[->] (-1.5,0) -- (1.5,0) node[right] {$x$}; \draw[->] (0,-1.5) -- (0,1.5) node[above] {$y$};
\filldraw[color=black] (1,0) circle (0.02) node[below right] {$1$};
\filldraw[color=black] (0,1) circle (0.02) node[above left] {$1$};
\end{tikzpicture} 
\end{center}~\caption{Bent TV screen $\cD_p(1)=\{ (x,y)\in\R^2 : 1-x^2-y^6\geq0\}$.}
			\label{fig-TV6}
\end{figure}
The graded set $\cD_p$ is again $y^2$-convex as is easily seen. In fact, it is even a $\Gat$-spectrahedron \cite[Proposition 4.2]{JKMMP}. 

However, in contrast to the above examples,
in this case the projection of the first 
Lasserre-Parrilo lift 
$\mathfrak{L}_p^\Gat(\cdot, 0)$ is not exact.
Consider 
\[
X=\begin{pmatrix} 1 \\ & 0\end{pmatrix},
\quad
Y=\frac12 \begin{pmatrix} 1 & 1 \\ 1 & 1\end{pmatrix}.
\]
Then 
\begin{equation}\label{eq:detTV}
\det( p(tX,tY))=
\det(I-t^2 X^2-t^6 Y^6)= \frac{t^8}{2}-t^6-t^2+1,
\end{equation}
so $t(X,Y)\in\cD_p$ for $|t|\leq t_0 \approx 0.861724$,
where $t_0$ is the smallest positive root of \eqref{eq:detTV}.

Now consider the first 
Lasserre-Parrilo lift $\mathfrak{L}_p^\Gat(\cdot, 0)$ described by {\CCBB the two inequalities}
	\begin{equation} 
	\begin{split}
		H_3(Z) &=
\resizebox{0.9\textwidth}{!}{$\left(
\begin{array}{ccccccccccccccc}
 1 & X & Y & Z_{11} & Z_{12} & Z_{21} & Y^2 & Z_{111} & Z_{112} &
   Z_{121} & Z_{122} & Z_{211} & Z_{212} & Z_{221} & Z_{222} \\
 X & Z_{11} & Z_{12} & Z_{111} & Z_{112} & Z_{121} & Z_{122} & Z_{1111} &
   Z_{1112} & Z_{1121} & Z_{1122} & Z_{1211} & Z_{1212} & Z_{1221} & Z_{1222}
   \\
 Y & Z_{21} & Y^2 & Z_{211} & Z_{212} & Z_{221} & Z_{222} & Z_{2111} &
   Z_{2112} & Z_{2121} & Z_{2122} & Z_{2211} & Z_{2212} & Z_{2221} & Z_{2222}
   \\
 Z_{11} & Z_{111} & Z_{112} & Z_{1111} & Z_{1112} & Z_{1121} & Z_{1122} &
   Z_{11111} & Z_{11112} & Z_{11121} & Z_{11122} & Z_{11211} & Z_{11212} &
   Z_{11221} & Z_{11222} \\
 Z_{21} & Z_{211} & Z_{212} & Z_{2111} & Z_{2112} & Z_{2121} & Z_{2122} &
   Z_{21111} & Z_{21112} & Z_{21121} & Z_{21122} & Z_{21211} & Z_{21212} &
   Z_{21221} & Z_{21222} \\
 Z_{12} & Z_{121} & Z_{122} & Z_{1211} & Z_{1212} & Z_{1221} & Z_{1222} &
   Z_{12111} & Z_{12112} & Z_{12121} & Z_{12122} & Z_{12211} & Z_{12212} &
   Z_{12221} & Z_{12222} \\
 Y^2 & Z_{221} & Z_{222} & Z_{2211} & Z_{2212} & Z_{2221} & Z_{2222} &
   Z_{22111} & Z_{22112} & Z_{22121} & Z_{22122} & Z_{22211} & Z_{22212} &
   Z_{22221} & Z_{22222} \\
 Z_{111} & Z_{1111} & Z_{1112} & Z_{11111} & Z_{11112} & Z_{11121} & Z_{11122}
   & Z_{111111} & Z_{111112} & Z_{111121} & Z_{111122} & Z_{111211} &
   Z_{111212} & Z_{111221} & Z_{111222} \\
 Z_{211} & Z_{2111} & Z_{2112} & Z_{21111} & Z_{21112} & Z_{21121} & Z_{21122}
   & Z_{211111} & Z_{211112} & Z_{211121} & Z_{211122} & Z_{211211} &
   Z_{211212} & Z_{211221} & Z_{211222} \\
 Z_{121} & Z_{1211} & Z_{1212} & Z_{12111} & Z_{12112} & Z_{12121} & Z_{12122}
   & Z_{121111} & Z_{121112} & Z_{121121} & Z_{121122} & Z_{121211} &
   Z_{121212} & Z_{121221} & Z_{121222} \\
 Z_{221} & Z_{2211} & Z_{2212} & Z_{22111} & Z_{22112} & Z_{22121} & Z_{22122}
   & Z_{221111} & Z_{221112} & Z_{221121} & Z_{221122} & Z_{221211} &
   Z_{221212} & Z_{221221} & Z_{221222} \\
 Z_{112} & Z_{1121} & Z_{1122} & Z_{11211} & Z_{11212} & Z_{11221} & Z_{11222}
   & Z_{112111} & Z_{112112} & Z_{112121} & Z_{112122} & Z_{112211} &
   Z_{112212} & Z_{112221} & Z_{112222} \\
 Z_{212} & Z_{2121} & Z_{2122} & Z_{21211} & Z_{21212} & Z_{21221} & Z_{21222}
   & Z_{212111} & Z_{212112} & Z_{212121} & Z_{212122} & Z_{212211} &
   Z_{212212} & Z_{212221} & Z_{212222} \\
 Z_{122} & Z_{1221} & Z_{1222} & Z_{12211} & Z_{12212} & Z_{12221} & Z_{12222}
   & Z_{122111} & Z_{122112} & Z_{122121} & Z_{122122} & Z_{122211} &
   Z_{122212} & Z_{122221} & Z_{122222} \\
 Z_{222} & Z_{2221} & Z_{2222} & Z_{22211} & Z_{22212} & Z_{22221} & Z_{22222}
   & Z_{222111} & Z_{222112} & Z_{222121} & Z_{222122} & Z_{222211} &
   Z_{222212} & Z_{222221} & Z_{222222} \\
\end{array}
\right)\succeq0,$}\\[5pt]
H_{p, 0}^{\Uparrow}(Z) &=  
		\begin{pmatrix}
 			 I-Z_{11}-Z_{222222}
		\end{pmatrix}\succeq0.
\end{split}
\label{eq:TV6}
	\end{equation}

Pick $t'=\frac{173}{200}=0.865 > t_0$. 
Then $p(t'X,t'Y)$ is not positive semidefinite; its smallest eigenvalue is
\[
\frac{53304846667911-29929 \sqrt{3362359178476091681}}{128000000000000} \approx -0.012306,
\]
whence $(t'X,t'Y)\not\in\cD_p$. However, we claim that
$(t'X,t'Y)\in \hat{\mathfrak{L}_p^\Gat}(\cdot, 0)$.
Replace $X,Y$ in \eqref{eq:TV6} by $t'X,t'Y$ leading to a pair of semidefinite constraints. We ran this SDP with the built-in solver in Wolfram Mathematica (with the trivial objective function $1$ which tends to produce high rank solutions) to obtain floating point solutions for the $2\times 2$ matrix variables $Z$. The smallest obtained eigenvalues for $H_3(Z)$ and $H_{p, 0}^{\Uparrow}(Z)$ were approximately $5\cdot 10^{-4}$ and $2\cdot 10^{-3}$, respectively. Since these are far enough from $0$, routine numerical analysis (see, e.g.~\cite{PP08,CKP15}) shows that a fine enough rationalization will lead to exact symbolic feasible points. Indeed, rationalizing leads to the following feasible point (as can be checked, e.g., with the Cholesky decomposition):\footnote{Since $H_3(Z)$ is positive semidefinite, it is self-adjoint. Hence $Z_{i_1i_2\cdots i_r}=Z_{i_ri_{r-1}\cdots i_1}^*$ for all indices $i_j$. 
We therefore omit one member of each such pair in the display below.}
{\scriptsize
\begin{align*} \pushQED{\qed}
\begin{autobreak} 
\phantom{=} 
(Z_{11})_{11}\to \frac{63}{82},\, \allowbreak
(Z_{11})_{12}\to -\frac{2}{39},\, \allowbreak
(Z_{11})_{22}\to \frac{53}{319},\, \allowbreak
(Z_{12})_{11}\to \frac{46}{123},\, \allowbreak
(Z_{12})_{12}\to \frac{55}{147},\, \allowbreak
(Z_{12})_{21}\to 0,\, \allowbreak
(Z_{12})_{22}\to 0,\, \allowbreak
(Z_{22})_{11}\to \frac{37}{69},\, \allowbreak
(Z_{22})_{12}\to \frac{5}{34},\, \allowbreak
(Z_{22})_{22}\to \frac{82}{117},\, \allowbreak
(Z_{111})_{11}\to \frac{112}{165},\, \allowbreak
(Z_{111})_{12}\to -\frac{4}{97},\, \allowbreak
(Z_{111})_{22}\to -\frac{1}{140},\, \allowbreak
(Z_{112})_{11}\to \frac{29}{94},\, \allowbreak
(Z_{112})_{12}\to \frac{24}{77},\, \allowbreak
(Z_{112})_{21}\to \frac{4}{75},\, \allowbreak
(Z_{112})_{22}\to \frac{2}{47},\, \allowbreak
(Z_{121})_{11}\to \frac{43}{133},\, \allowbreak
(Z_{121})_{12}\to 0,\, \allowbreak
(Z_{121})_{22}\to 0,\, \allowbreak
(Z_{122})_{11}\to \frac{32}{69},\, \allowbreak
(Z_{122})_{12}\to \frac{7}{55},\, \allowbreak
(Z_{122})_{21}\to 0,\, \allowbreak
(Z_{122})_{22}\to -\frac{1}{8348},\, \allowbreak
(Z_{212})_{11}\to \frac{16}{99},\, \allowbreak
(Z_{212})_{12}\to \frac{6}{37},\, \allowbreak
(Z_{212})_{22}\to \frac{16}{99},\, \allowbreak
(Z_{222})_{11}\to \frac{55}{166},\, \allowbreak
(Z_{222})_{12}\to \frac{18}{73},\, \allowbreak
(Z_{222})_{22}\to \frac{8}{15},\, \allowbreak
(Z_{1111})_{11}\to \frac{51}{32},\, \allowbreak
(Z_{1111})_{12}\to -\frac{65}{61},\, \allowbreak
(Z_{1111})_{22}\to \frac{571}{143},\, \allowbreak
(Z_{1112})_{11}\to \frac{43}{156},\, \allowbreak
(Z_{1112})_{12}\to \frac{22}{79},\, \allowbreak
(Z_{1112})_{21}\to -\frac{1}{47},\, \allowbreak
(Z_{1112})_{22}\to -\frac{1}{50},\, \allowbreak
(Z_{1121})_{11}\to \frac{67}{251},\, \allowbreak
(Z_{1121})_{12}\to -\frac{1}{982},\, \allowbreak
(Z_{1121})_{21}\to \frac{5}{114},\, \allowbreak
(Z_{1121})_{22}\to \frac{1}{232},\, \allowbreak
(Z_{1122})_{11}\to \frac{40}{99},\, \allowbreak
(Z_{1122})_{12}\to \frac{3}{38},\, \allowbreak
(Z_{1122})_{21}\to -\frac{1}{630},\, \allowbreak
(Z_{1122})_{22}\to \frac{7}{67},\, \allowbreak
(Z_{1212})_{11}\to \frac{13}{93},\, \allowbreak
(Z_{1212})_{12}\to \frac{7}{50},\, \allowbreak
(Z_{1212})_{21}\to 0,\, \allowbreak
(Z_{1212})_{22}\to \frac{1}{4614},\, \allowbreak
(Z_{1221})_{11}\to \frac{113}{81},\, \allowbreak
(Z_{1221})_{12}\to -\frac{67}{65},\, \allowbreak
(Z_{1221})_{22}\to \frac{639}{160},\, \allowbreak
(Z_{1222})_{11}\to \frac{43}{150},\, \allowbreak
(Z_{1222})_{12}\to \frac{19}{89},\, \allowbreak
(Z_{1222})_{21}\to 0,\, \allowbreak
(Z_{1222})_{22}\to 0,\, \allowbreak
(Z_{2112})_{11}\to \frac{130}{49},\, \allowbreak
(Z_{2112})_{12}\to -\frac{159}{61},\, \allowbreak
(Z_{2112})_{22}\to \frac{394}{85},\, \allowbreak
(Z_{2122})_{11}\to \frac{59}{294},\, \allowbreak
(Z_{2122})_{12}\to \frac{7}{127},\, \allowbreak
(Z_{2122})_{21}\to \frac{46}{229},\, \allowbreak
(Z_{2122})_{22}\to \frac{7}{128},\, \allowbreak
(Z_{2222})_{11}\to \frac{21}{64},\, \allowbreak
(Z_{2222})_{12}\to \frac{13}{95},\, \allowbreak
(Z_{2222})_{22}\to \frac{87}{136},\, \allowbreak
(Z_{11111})_{11}\to \frac{83}{69},\, \allowbreak
(Z_{11111})_{12}\to -\frac{101}{143},\, \allowbreak
(Z_{11111})_{22}\to -\frac{33}{97},\, \allowbreak
(Z_{11112})_{11}\to -\frac{1}{2668},\, \allowbreak
(Z_{11112})_{12}\to \frac{7}{102},\, \allowbreak
(Z_{11112})_{21}\to \frac{120}{109},\, \allowbreak
(Z_{11112})_{22}\to \frac{61}{80},\, \allowbreak
(Z_{11121})_{11}\to \frac{20}{91},\, \allowbreak
(Z_{11121})_{12}\to -\frac{1}{142},\, \allowbreak
(Z_{11121})_{21}\to -\frac{2}{97},\, \allowbreak
(Z_{11121})_{22}\to -\frac{1}{2946},\, \allowbreak
(Z_{11122})_{11}\to \frac{69}{193},\, \allowbreak
(Z_{11122})_{12}\to \frac{22}{309},\, \allowbreak
(Z_{11122})_{21}\to -\frac{2}{93},\, \allowbreak
(Z_{11122})_{22}\to -\frac{1}{75},\, \allowbreak
(Z_{11211})_{11}\to \frac{13}{58},\, \allowbreak
(Z_{11211})_{12}\to \frac{6}{175},\, \allowbreak
(Z_{11211})_{22}\to \frac{1}{645},\, \allowbreak
(Z_{11212})_{11}\to \frac{12}{103},\, \allowbreak
(Z_{11212})_{12}\to \frac{25}{213},\, \allowbreak
(Z_{11212})_{21}\to \frac{1}{47},\, \allowbreak
(Z_{11212})_{22}\to \frac{1}{51},\, \allowbreak
(Z_{11221})_{11}\to \frac{32}{47},\, \allowbreak
(Z_{11221})_{12}\to -\frac{49}{64},\, \allowbreak
(Z_{11221})_{21}\to \frac{5}{54},\, \allowbreak
(Z_{11221})_{22}\to -\frac{29}{168},\, \allowbreak
(Z_{11222})_{11}\to \frac{27}{112},\, \allowbreak
(Z_{11222})_{12}\to \frac{17}{104},\, \allowbreak
(Z_{11222})_{21}\to \frac{1}{38},\, \allowbreak
(Z_{11222})_{22}\to \frac{6}{89},\, \allowbreak
(Z_{12112})_{11}\to \frac{31}{17},\, \allowbreak
(Z_{12112})_{12}\to -\frac{111}{56},\, \allowbreak
(Z_{12112})_{21}\to \frac{24}{107},\, \allowbreak
(Z_{12112})_{22}\to -\frac{27}{113},\, \allowbreak
(Z_{12121})_{11}\to \frac{16}{129},\, \allowbreak
(Z_{12121})_{12}\to -\frac{1}{4415},\, \allowbreak
(Z_{12121})_{22}\to \frac{1}{3301},\, \allowbreak
(Z_{12122})_{11}\to \frac{17}{98},\, \allowbreak
(Z_{12122})_{12}\to \frac{1}{21},\, \allowbreak
(Z_{12122})_{21}\to 0,\, \allowbreak
(Z_{12122})_{22}\to \frac{1}{7768},\, \allowbreak
(Z_{12212})_{11}\to -\frac{5}{74},\, \allowbreak
(Z_{12212})_{12}\to -\frac{1}{104},\, \allowbreak
(Z_{12212})_{21}\to \frac{81}{73},\, \allowbreak
(Z_{12212})_{22}\to \frac{95}{117},\, \allowbreak
(Z_{12221})_{11}\to \frac{18}{95},\, \allowbreak
(Z_{12221})_{12}\to -\frac{1}{90},\, \allowbreak
(Z_{12221})_{22}\to -\frac{1}{449},\, \allowbreak
(Z_{12222})_{11}\to \frac{32}{113},\, \allowbreak
(Z_{12222})_{12}\to \frac{8}{67},\, \allowbreak
(Z_{12222})_{21}\to \frac{1}{6553},\, \allowbreak
(Z_{12222})_{22}\to -\frac{1}{3675},\, \allowbreak
(Z_{21112})_{11}\to \frac{14}{117},\, \allowbreak
(Z_{21112})_{12}\to \frac{22}{175},\, \allowbreak
(Z_{21112})_{22}\to \frac{11}{86},\, \allowbreak
(Z_{21122})_{11}\to \frac{221}{295},\, \allowbreak
(Z_{21122})_{12}\to -\frac{268}{153},\, \allowbreak
(Z_{21122})_{21}\to -\frac{43}{117},\, \allowbreak
(Z_{21122})_{22}\to \frac{169}{83},\, \allowbreak
(Z_{21212})_{11}\to \frac{5}{82},\, \allowbreak
(Z_{21212})_{12}\to \frac{3}{49},\, \allowbreak
(Z_{21212})_{22}\to \frac{7}{114},\, \allowbreak
(Z_{21222})_{11}\to \frac{12}{97},\, \allowbreak
(Z_{21222})_{12}\to \frac{4}{43},\, \allowbreak
(Z_{21222})_{21}\to \frac{15}{121},\, \allowbreak
(Z_{21222})_{22}\to \frac{8}{87},\, \allowbreak
(Z_{22122})_{11}\to \frac{49}{197},\, \allowbreak
(Z_{22122})_{12}\to \frac{5}{73},\, \allowbreak
(Z_{22122})_{22}\to \frac{1}{55},\, \allowbreak
(Z_{22222})_{11}\to \frac{16}{69},\, \allowbreak
(Z_{22222})_{12}\to \frac{19}{116},\, \allowbreak
(Z_{22222})_{22}\to \frac{23}{41},\, \allowbreak
(Z_{111111})_{11}\to \frac{19025}{73},\, \allowbreak
(Z_{111111})_{12}\to -\frac{150}{47},\, \allowbreak
(Z_{111111})_{22}\to \frac{22307}{82},\, \allowbreak
(Z_{111112})_{11}\to \frac{86}{115},\, \allowbreak
(Z_{111112})_{12}\to \frac{30}{49},\, \allowbreak
(Z_{111112})_{21}\to -\frac{101}{138},\, \allowbreak
(Z_{111112})_{22}\to -\frac{37}{63},\, \allowbreak
(Z_{111121})_{11}\to \frac{5}{72},\, \allowbreak
(Z_{111121})_{12}\to -\frac{1}{113},\, \allowbreak
(Z_{111121})_{21}\to \frac{62}{67},\, \allowbreak
(Z_{111121})_{22}\to \frac{23}{275},\, \allowbreak
(Z_{111122})_{11}\to \frac{51}{98},\, \allowbreak
(Z_{111122})_{12}\to -\frac{25}{113},\, \allowbreak
(Z_{111122})_{21}\to -\frac{3}{119},\, \allowbreak
(Z_{111122})_{22}\to \frac{251}{96},\, \allowbreak
(Z_{111211})_{11}\to \frac{38}{241},\, \allowbreak
(Z_{111211})_{12}\to -\frac{5}{64},\, \allowbreak
(Z_{111211})_{21}\to -\frac{4}{115},\, \allowbreak
(Z_{111211})_{22}\to \frac{5}{74},\, \allowbreak
(Z_{111212})_{11}\to \frac{4}{39},\, \allowbreak
(Z_{111212})_{12}\to \frac{11}{126},\, \allowbreak
(Z_{111212})_{21}\to -\frac{1}{112},\, \allowbreak
(Z_{111212})_{22}\to -\frac{1}{141},\, \allowbreak
(Z_{111221})_{11}\to \frac{11371}{127},\, \allowbreak
(Z_{111221})_{12}\to -\frac{406}{123},\, \allowbreak
(Z_{111221})_{21}\to -\frac{349}{138},\, \allowbreak
(Z_{111221})_{22}\to \frac{8227}{81},\, \allowbreak
(Z_{111222})_{11}\to \frac{17}{80},\, \allowbreak
(Z_{111222})_{12}\to \frac{21}{137},\, \allowbreak
(Z_{111222})_{21}\to -\frac{1}{68},\, \allowbreak
(Z_{111222})_{22}\to -\frac{1}{63},\, \allowbreak
(Z_{112112})_{11}\to \frac{201}{98},\, \allowbreak
(Z_{112112})_{12}\to -\frac{139}{70},\, \allowbreak
(Z_{112112})_{21}\to -\frac{13}{37},\, \allowbreak
(Z_{112112})_{22}\to \frac{40}{47},\, \allowbreak
(Z_{112121})_{11}\to \frac{13}{83},\, \allowbreak
(Z_{112121})_{12}\to \frac{1}{3352},\, \allowbreak
(Z_{112121})_{21}\to \frac{4}{105},\, \allowbreak
(Z_{112121})_{22}\to 0,\, \allowbreak
(Z_{112122})_{11}\to \frac{29}{119},\, \allowbreak
(Z_{112122})_{12}\to \frac{5}{71},\, \allowbreak
(Z_{112122})_{21}\to \frac{7}{109},\, \allowbreak
(Z_{112122})_{22}\to \frac{1}{44},\, \allowbreak
(Z_{112211})_{11}\to \frac{14181}{83},\, \allowbreak
(Z_{112211})_{12}\to \frac{3}{103},\, \allowbreak
(Z_{112211})_{22}\to \frac{14321}{84},\, \allowbreak
(Z_{112212})_{11}\to -\frac{3}{64},\, \allowbreak
(Z_{112212})_{12}\to \frac{1}{467},\, \allowbreak
(Z_{112212})_{21}\to -\frac{1}{33},\, \allowbreak
(Z_{112212})_{22}\to -\frac{1}{49},\, \allowbreak
(Z_{112221})_{11}\to \frac{7}{40},\, \allowbreak
(Z_{112221})_{12}\to -\frac{1}{158},\, \allowbreak
(Z_{112221})_{21}\to -\frac{7}{109},\, \allowbreak
(Z_{112221})_{22}\to \frac{27}{350},\, \allowbreak
(Z_{112222})_{11}\to \frac{21}{86},\, \allowbreak
(Z_{112222})_{12}\to \frac{17}{222},\, \allowbreak
(Z_{112222})_{21}\to \frac{1}{148},\, \allowbreak
(Z_{112222})_{22}\to \frac{8}{85},\, \allowbreak
(Z_{121112})_{11}\to \frac{2}{75},\, \allowbreak
(Z_{121112})_{12}\to \frac{11}{137},\, \allowbreak
(Z_{121112})_{21}\to \frac{1}{286},\, \allowbreak
(Z_{121112})_{22}\to \frac{1}{129},\, \allowbreak
(Z_{121121})_{11}\to \frac{19401}{113},\, \allowbreak
(Z_{121121})_{12}\to \frac{6}{43},\, \allowbreak
(Z_{121121})_{22}\to \frac{7672}{45},\, \allowbreak
(Z_{121122})_{11}\to \frac{26}{49},\, \allowbreak
(Z_{121122})_{12}\to -\frac{135}{106},\, \allowbreak
(Z_{121122})_{21}\to \frac{6}{91},\, \allowbreak
(Z_{121122})_{22}\to -\frac{3}{17},\, \allowbreak
(Z_{121212})_{11}\to \frac{7}{100},\, \allowbreak
(Z_{121212})_{12}\to \frac{5}{71},\, \allowbreak
(Z_{121212})_{21}\to 0,\, \allowbreak
(Z_{121212})_{22}\to 0,\, \allowbreak
(Z_{121221})_{11}\to -\frac{7}{102},\, \allowbreak
(Z_{121221})_{12}\to \frac{71}{75},\, \allowbreak
(Z_{121221})_{21}\to \frac{1}{281},\, \allowbreak
(Z_{121221})_{22}\to \frac{2}{23},\, \allowbreak
(Z_{121222})_{11}\to \frac{13}{121},\, \allowbreak
(Z_{121222})_{12}\to \frac{10}{127},\, \allowbreak
(Z_{121222})_{21}\to 0,\, \allowbreak
(Z_{121222})_{22}\to \frac{1}{3350},\, \allowbreak
(Z_{122112})_{11}\to \frac{131}{91},\, \allowbreak
(Z_{122112})_{12}\to \frac{8}{89},\, \allowbreak
(Z_{122112})_{21}\to -\frac{27}{59},\, \allowbreak
(Z_{122112})_{22}\to -\frac{47}{92},\, \allowbreak
(Z_{122122})_{11}\to \frac{37}{91},\, \allowbreak
(Z_{122122})_{12}\to -\frac{17}{67},\, \allowbreak
(Z_{122122})_{21}\to \frac{1}{754},\, \allowbreak
(Z_{122122})_{22}\to \frac{279}{106},\, \allowbreak
(Z_{122212})_{11}\to \frac{5}{64},\, \allowbreak
(Z_{122212})_{12}\to \frac{11}{177},\, \allowbreak
(Z_{122212})_{21}\to \frac{1}{436},\, \allowbreak
(Z_{122212})_{22}\to \frac{1}{286},\, \allowbreak
(Z_{122221})_{11}\to \frac{34973}{135},\, \allowbreak
(Z_{122221})_{12}\to -\frac{173}{63},\, \allowbreak
(Z_{122221})_{22}\to \frac{38353}{141},\, \allowbreak
(Z_{122222})_{11}\to \frac{25}{127},\, \allowbreak
(Z_{122222})_{12}\to \frac{16}{105},\, \allowbreak
(Z_{122222})_{21}\to \frac{1}{9811},\, \allowbreak
(Z_{122222})_{22}\to \frac{1}{3641},\, \allowbreak
(Z_{211112})_{11}\to \frac{22123}{85},\, \allowbreak
(Z_{211112})_{12}\to -\frac{217}{202},\, \allowbreak
(Z_{211112})_{22}\to \frac{28977}{112},\, \allowbreak
(Z_{211122})_{11}\to \frac{9}{115},\, \allowbreak
(Z_{211122})_{12}\to \frac{1}{43},\, \allowbreak
(Z_{211122})_{21}\to \frac{1}{19},\, \allowbreak
(Z_{211122})_{22}\to \frac{5}{131},\, \allowbreak
(Z_{211212})_{11}\to \frac{53}{56},\, \allowbreak
(Z_{211212})_{12}\to \frac{87}{92},\, \allowbreak
(Z_{211212})_{21}\to -\frac{73}{75},\, \allowbreak
(Z_{211212})_{22}\to -\frac{108}{113},\, \allowbreak
(Z_{211222})_{11}\to \frac{50}{47},\, \allowbreak
(Z_{211222})_{12}\to -\frac{31}{23},\, \allowbreak
(Z_{211222})_{21}\to -\frac{131}{92},\, \allowbreak
(Z_{211222})_{22}\to \frac{237}{49},\, \allowbreak
(Z_{212122})_{11}\to \frac{13}{111},\, \allowbreak
(Z_{212122})_{12}\to \frac{1}{34},\, \allowbreak
(Z_{212122})_{21}\to \frac{19}{161},\, \allowbreak
(Z_{212122})_{22}\to \frac{4}{135},\, \allowbreak
(Z_{212212})_{11}\to \frac{18284}{107},\, \allowbreak
(Z_{212212})_{12}\to \frac{54}{167},\, \allowbreak
(Z_{212212})_{22}\to \frac{2561}{15},\, \allowbreak
(Z_{212222})_{11}\to \frac{13}{106},\, \allowbreak
(Z_{212222})_{12}\to \frac{1}{19},\, \allowbreak
(Z_{212222})_{21}\to \frac{13}{106},\, \allowbreak
(Z_{212222})_{22}\to \frac{3}{58},\, \allowbreak
(Z_{221122})_{11}\to \frac{41456}{243},\, \allowbreak
(Z_{221122})_{12}\to -\frac{38}{85},\, \allowbreak
(Z_{221122})_{22}\to \frac{33873}{197},\, \allowbreak
(Z_{221222})_{11}\to \frac{2}{13},\, \allowbreak
(Z_{221222})_{12}\to \frac{5}{44},\, \allowbreak
(Z_{221222})_{21}\to \frac{3}{71},\, \allowbreak
(Z_{221222})_{22}\to \frac{3}{98},\, \allowbreak
(Z_{222222})_{11}\to \frac{15}{68},\, \allowbreak
(Z_{222222})_{12}\to \frac{7}{88},\, \allowbreak
(Z_{222222})_{22}\to \frac{54}{73}. \end{autobreak}
   \qedhere\popQED\end{align*}}

\end{example}

\subsection{Stopping criteria}\label{ssec:stop}
Finally, we present two  criteria for the stopping of the Lasserre-Parrilo lifting hierarchy.

\subsubsection{Positivstellensatz inspired stopping criterion}
The first criterion mirrors the analogous statement (\cite[Theorem 6.8]{HKM16}) in the $\Gat$-free context of matrix convex sets.

\def\al{\alpha}\def\be{\beta}
\def\FF{\mathbb C}
 Given $\alpha,\beta,\nuv\in\N$, and an $\ell\times\ell$ symmetric matrix-valued noncommutative polynomial $p$, set 
\begin{equation*}
  \QM_{\al, \beta}^{\nuv}(p):=
   \Sigma_{\al}^{\nuv}  +\Big\{ \sum_i^{\rm finite}  f_i^* pf_i :
   \  f_i \in \FF^{\ell\times \nuv}\ax_\beta \Big\}
  \  \subseteq \ \FF^{\nuv\times \nuv}\ax_{\max \{2\al, 2\beta +a\}},
\end{equation*}
 where $a=\deg(p)$ and $\Sigma_\al^\nuv$ denotes all $\nuv\times\nuv$ sums
of squares of degree at most $2\al$.
Clearly, if $f\in \QM_{\al,\beta}^{\nuv}(p)$ then $f|_{\cD_p}\succeq0$.
We call $\QM_{\al,\beta}^\nuv(p)$ the \df{truncated quadratic module} defined by $p$.
For notational convenience, we write 
{$\QM_{k}^\nuv
\subseteq  \FF^{\nuv\times \nuv}\ax_{2k}
$ for  $\QM_{k, \, \lfloor \frac{2k-a}{2}\rfloor}^\nuv.$}
We also introduce
\[
 \QM^\nuv(p):=\bigcup_{\al,\be} \QM_{\al,\beta}^\nuv(p),
\]
the \df{quadratic module} defined by $p$. 
If $\nuv=1$ we shall often omit the superscript $\nuv$.
Observe that $p$ is Archimedean if 
the convex cone $\QM^\nuv(p)$ has an order unit, i.e., for all symmetric $\nuv\times\nuv$ matrix-valued
polynomials $f$ there is $N\in\N$ with $N-f\in \QM^\nuv(p)$. (This notion is 
easily seen to be 
independent of $\nuv$, cf.~\cite[\S 6]{HKM13}.)

\begin{definition}
We say that $p$ has the \df{$\Gat$-positivity certificate property  ($\Gat$-PCP)}, if for some $N\in\N$, 
every 
$\nuv\in\N$ and every
$\Gat$-pencil $L^\Gat$ of size $\nuv\times\nuv$, we have
\[\pushQED{\qed}
L^\Gat|_{\cD_p}\succeq0 \quad\Rightarrow\quad L^\Gat \in \QM_{N}^\nuv(p).\qedhere\popQED
\]
\end{definition}

We refer the reader to \cite{Scw04,Nis07,LPR20,SL23} for the classical commutative study
of degree bounds needed in Positivstellensatz certificates. See
also \cite{HN09,HN10} for an application of these bounds to convexity and Lasserre-Parrilo lifts in the commutative.

The next theorem says if  $\Gat$-PCP holds, then one of the truncated
Lasserre--Parrilo lifts gives exactly the $\Gat$-convex hull of $\cD_p^\infty$.

\begin{proposition} \label{prop:PCP}
Suppose $\cD_p$ is bounded, 
{$p(0)\succeq 0$}, and $\Gat(0)=0$.
If $p$ has the $\Gat$-PCP, 
then \[\Gat\mhyphen\conv  (\cD_p^\infty)= \hat{\mathfrak{L}_p^\Gat}\Big(\cdot, \left\lceil \frac N2\right\rceil\Big).\]
\end{proposition}

\begin{proof}
Let $\eta= \lceil \frac N2\rceil$. 
Since $\cD_p$ is  bounded and $p$ has the $\Gat$-PCP, $p$ is Archimedean. 
By Theorem \ref{th-lift},
$\Gat\mhyphen\conv  (\cD_p^\infty)\subseteq\hat{\mathfrak{L}_p^\Gat}(\cdot,\eta)$. 

Now let $\nuv\in\N$ and  $Y\in\hat{\mathfrak{L}_p^\Gat}(\nuv,\eta)\setminus\Gat\mhyphen \conv  (\cD_p^\infty)$
 be given. Choose $W\in{\mathfrak{L}_p^\Gat}(\nuv,\eta)$ satisfying $\hat W=Y$.
 In particular, $H_\eta(W)\succeq0$ and $H^{\Uparrow}_{p,\eta}(W)\succeq0.$
 By the $\Gat-$Hahn-Banach Theorem \ref{t:thm:jp-intro} (cf.~Remark \ref{rem:Kprime}), there 
is a  $\Gat$-pencil $L^\Gat$ (of size  $\nuv$) with $L^\Gat|_{\Gat\mhyphen\conv  (\cD_p^\infty)}\succeq0$ and $L^\Gat(Y)\not\succeq0$. 
By the $\Gat$-PCP property for $p$, 
we have that $L^\Gat \in \QM_{N}^\nuv(p)$, i.e.,  
\begin{equation}\label{eq:symbol}
L^\Gat= \sum_k h_k^*h_k + \sum_{i=1}^r  f_i^* pf_i ,
\end{equation}
where $\deg(h_k)\leq \lfloor \frac N 2\rfloor$  and $2\deg (f_i)+\deg(p)\leq N$ for $i=1,\ldots,r$. 
Now apply the Riesz moment map 
\[
\Phi_W^\nuv:\C^{\nuv\times\nuv}\ax_{\leq N}\to M_{ \nuv^2}, 
\quad
\sum_{\al\in\ax} B_\al \al \mapsto \sum_{\al\in\ax} B_\al \otimes W_\al,
\]
 to \eqref{eq:symbol}:
\begin{equation}\label{eq:las2}
\Phi_W^\nuv(L^\Gat) = \sum_k \Phi_W^\nuv(h_k^*h_k) + \sum_{i} \Phi_W^\nuv( f_i^* pf_i ).
\end{equation}
Since {$H_\eta(W)\succeq0$ and $H^{\Uparrow}_{p,\eta}(W)\succeq0$}, the right hand side of \eqref{eq:las2} is positive semidefinite.
On the other hand, by linearity {in the $\gamma_j$ built into the definition of a $\Gat$-moment sequence,
equation~\eqref{eq:gmom} implies}
\[
  \Phi_W^\nuv(L^\Gat)= L^\Gat(\hat W)=L^\Gat(Y)\not\succeq0,
\] 
a contradiction.
\end{proof}

\subsubsection{Moment problem inspired stopping criterion}
\label{ssec:moment:stop}
Given $d,\eta\in\N$ with $d\leq\eta$, and a truncated moment sequence $Z=(Z_\al)_{|\al|\leq 2\eta}$, we say that the moment sequence $Z$ or its truncated Hankel matrix $H_\eta(Z)$ is \df{$d$-flat} if
$\rank H_{\eta-d}(Z)=\rank H_\eta(Z)$. 
Versions of this flatness condition are important in the classical {commutative} theory of moments  since
they imply the moment problem on $Z$ has a solution \cite{CF08,Lau09}.

\begin{proposition}\label{prop:MP}
Suppose $p$ is a  symmetric matrix-valued noncommutative polynomial and 
let $\delta=\deg\Gat$, 
$a=\deg p$ and $2\eta\ge \delta+2a.$
If each $Y\in\hat{\mathfrak{L}_p^\Gat}(\cdot,\eta)$ admits a lift $W\in{\mathfrak{L}_p^\Gat}(\cdot,\eta)$ that is $a$-flat, then 
\[
\hat{\mathfrak{L}_p^\Gat}(\cdot,\eta)=
\Gat\mhyphen\conv  (\cD_p) =
\Gat\mhyphen\conv  (\cD_p^\infty).
\]
\end{proposition}

We split off the key step in establishing the proposition into a lemma that may be of independent interest.

\begin{lemma}\label{lem:MP}
Let $\delta=\deg\Gat$,
$a=\deg p$ and $2\eta \ge \delta+2a.$ 
If $Y\in\hat{\mathfrak{L}_p^\Gat}(\cdot,\eta)$ admits a lift $W\in{\mathfrak{L}_p^\Gat}(\cdot,\eta)$ that is 
{$a$-flat,}
then $Y\in\Gat\mhyphen\conv(\cD_p)$.
\end{lemma}

\begin{proof}
The proof goes along the same lines as that of Theorem \ref{th-lift}, so we only give a sketch, leaving some of the details to the reader.

   Define a sesquilinear form $[ \cdot, \cdot]$ on
	$\mathcal{V} = \C\langle x\rangle_\eta  \otimes \C^n$ by
	\begin{equation*}
		[ s,t ]= \sum_{|\alpha|,|\beta|\le \eta}  \langle Y_{\beta^* \alpha} s_\alpha, t_\beta \rangle
	\end{equation*}
	where $s=\sum \alpha \otimes s_\alpha$ and $t=\sum \beta\otimes t_\beta$. 
	The assumption $H_\eta(W) \succeq 0$ implies $[s,s]\geq 0$ for all $s \in \mathcal{V},$ so that this sesquilinear  form is positive semidefinite.
	 As before, 
	\[
	\cN = \{s : [s,s]= 0\}
	\]
	is a subspace of $\mathcal{V}.$ Modding out $\cN$ produces a well-defined and positive semidefinite form
	$$
	\langle s,t \rangle = [ s + \cN,t+\cN ]_Y
	$$
	on the quotient $\mathcal{W}$ of $\mathcal{V}$ by $\cN,$ where, as is standard practice, $s\in \mathcal{V}$ is identified
	with its image $s+\cN$ in the quotient.  At this point, observe 
	$\langle \cdot,\cdot \rangle$ is a positive definite sesquilinear form on the finite-dimensional 
	 vector space $\mathcal{W}.$ Hence $\mathcal{W}$ is already a Hilbert space. 

      Since $W$ is $a$-flat, $\rank H_{\eta-1}(W)=\rank H_{\eta}(W).$  Hence $\mathcal{W}$ is spanned 
       by $\{\alpha\otimes h: |\alpha|\le \eta-1, \, h\in \C^n\}.$  If $s\in \cN$ has degree at most $\eta-1,$
        then $x_j s$ has degree at most $\eta$ and   $[x_j s,t]=0.$  It follows that $x_j$ determines
        a (well defined) operator $Z_j$ on $\mathcal{W}$. Since $\mathcal{W}$ is a finite-dimensional
         Hilbert space, $Z_j$ is a bounded operator. 
        If $s$ as in equation~\eqref{e:lift:2}
        has degree at most $\eta-a,$ then the computation of 
        equation~\eqref{e:lift:2} gives $\langle p(Z)s,s\rangle \ge 0.$
        Since $\rank H_{\eta-a}(W) =\rank H_\eta,$ polynomials of degree at most $\eta-a$ span
        $\mathcal{V}.$ Thus $p(Z)\succeq 0.$ Hence $Z\in \cD_p$ and, letting $V$ denote
        the usual isometry, $(Z,V)$ is a $\Gat$-pair and $V^*Z^\alpha V= W_\alpha$ for
        all $|\alpha|\le \eta.$ As a side remark, we can extend $W$ to a full $\Gat$-moment
        sequence via this formula.
\end{proof}

\begin{proof}[Proof of Proposition \ref{prop:MP}]
Simple corollary of Lemma \ref{lem:MP}.
	Using 
the obvious inclusion of \eqref{eq-cd} of
	Theorem~\ref{th-cd} (which does not require Archimedeanity), we can deduce
\begin{equation*}
\Gat\mhyphen\conv  (\cD_p) \subseteq
\Gat\mhyphen\conv  (\cD_p^\infty)
\subseteq
	\bigcap_{d=0}^\infty \hat{\mathfrak{L}_p^\Gat}(\cdot, d)  \subseteq 
	\hat{\mathfrak{L}_p^\Gat}(\cdot, \eta) \subseteq \Gat\mhyphen\conv  (\cD_p),
\end{equation*}
whence we have equalities throughout.
\end{proof}

\subsection{The operator \texorpdfstring{$\Gat$}{Gamma}-convex hulls} \label{ssec:moreop}

Most of the results in this section have their counterpart in the context of operator $\Gat$-convex sets as we now explain. 

Let $p \in M_{\mu}(\C\langle x \rangle)$ be 
a symmetric matrix-valued noncommutative polynomial 
 in $\gv$ variables, and consider its 
 free operator semialgebraic set $\cD_p^\infty,$ whose 
 operator $\Gat$-convex hull is denoted $\Gat\mhyphen\opco(\cD_p^\infty).$ Thus  $\Gat\mhyphen\opco(\cD_p^\infty)$
 is just $\Gat\mhyphen\conv(\cD_p^\infty)$ together with a new infinite level 
 $\Gat\mhyphen\opco(\cD_p^\infty)(\infty) ,$ where 
 a tuple  $X \in \gtupn$ lies in $\Gat\mhyphen\opco(\cD_p^\infty)(\infty) $ if there exists a tuple $Y \in \cD_p^\infty$ acting on some 
separable
	Hilbert space $\cH$ and an isometry $V: \cH_n \to \cH$ such that $(Y,V)$ is a $\Gat$-pair and $X = V^*YV.$
  By construction, this operator $\Gat$-convex hull is the smallest  operator 
  $\Gat$-convex  set containing $\cD_p^\infty.$

Since every $\Gat$-pencil $L^\Gat$ as in \eqref{eq: gen-sp} is a matrix-valued polynomial, 
 this definition \eqref{eq:operatorsa} applies to $L^\Gat$  yielding the \df{operator $\Gat$-spectrahedron}
$\cD_{L^\Gat}^\infty$ and then the 
\df{operator $\Gat$-spectrahedrop}
$\proj_x(\mathcal{D}_{L^\Gat}^\infty)^\infty$.

Further, the set of all $\Gat$-moment sequences
$\mathfrak{M}^\Gat$ is extended by {\CCBB  {allowing}} moment sequences of operators on a separable infinite-dimensional Hilbert space $\cH$ to yield the infinite level $\mathfrak{M}^\Gat(\infty)$.
To each $Y\in \mathfrak{M}^\Gat(\infty)$ we can assign, as in Definition \ref{def:Hankel},  the $\Gat$-Hankel ``matrix''
$H(Y)$ as in \eqref{eq:hankeldef}, the truncated 
$\Gat$-Hankel ``matrix'' $H_d(Y)$ as in \eqref{eq:hankeldeftrunc}, and their localizing counterparts.

Finally, 
$\mathfrak{L}^\Gat_p$ of Subsection \ref{sec-lift} is extended
to include the infinite level
\begin{equation*}
\mathfrak{L}^\Gat_p(\infty) = \{Y = (Y_{\alpha})_\alpha \in \mathfrak{M}^\Gat(\infty)  \ | \  \ H (Y) \succeq 0, \ H^{\Uparrow}_p(Y) \succeq 0\}.
\end{equation*}
The obtained lift is denoted $\mathfrak{L}^{\Gat,\infty}_p$.

 Given $Y \in \mathfrak{L}^{\Gat,\infty}_p(n)$ for $n\in\N\cup\{\infty\}$ let
\begin{equation*}
\hat{Y} = (Y_{x_1}, Y_{x_2}, \ldots, Y_{x_{\gv}}) \in \gtup_n
\end{equation*}
and denote
\[
\hat{\mathfrak{L}}_p^{\Gat,\infty} = \{\hat{Y} \ | \ Y \in \mathfrak{L}^{\Gat,\infty}_p\}.
\]

\begin{corollary} 
	If $p$ is Archimedean, then
	$$
	\Gat\mhyphen\opco(\mathcal{D}_p^\infty) = \hat{\mathfrak{L}}_p^{\Gat,\infty}.
	$$
\end{corollary}

\begin{proof}
The proof is the same as that of Theorem \ref{th-lift}, with obvious modifications. The main change is that in the GNS construction one now works with $\C\ax\otimes\cH$ instead of $\C\ax\otimes\C^n$, where $\cH$ denotes an infinite-dimensional Hilbert space. The routine details are left to the reader.
\end{proof}

Obvious modifications of the truncations \eqref{eq:clamp1} and projections \eqref{eq:clamp2} of Subsection \ref{ssec:step2} lead to the sets 
$\mathfrak{L}_p^{\Gat,\infty}(\cdot, d)$ and
$\hat{\mathfrak{L}}_p^{\Gat,\infty}(\cdot, d)$, respectively.
With this at hand, we can state our final corollary.

\begin{corollary}
	If $p$ is an Archimedean symmetric matrix-valued noncommutative polynomial, then 
	\begin{enumerate}[\rm (a)]
		\item 
     for every $n\in\N\cup\{\infty\},$
		\begin{equation}\label{eq-cdz}
			\bigcap_{d=0}^\infty \hat{\mathfrak{L}}_p^{\Gat,\infty}(n, d) = \hat{\mathfrak{L}}_p^{\Gat,\infty}(n) = \Gat\mhyphen\opco  (\cD_p^\infty)(n); 
		\end{equation}
		\item 
    for every $d$ there is a $\Gat$-pencil $\LG_d$ given by a tuple $A^{(d)}$ such that $\hat{\mathfrak{L}_p^{\Gat,\infty}}(\cdot, d)$ is the projection of $\cD^{\Gat,\infty}_{A^{(d)}}.$ 
	\end{enumerate}
	Hence, the operator $\Gat$-spectrahedrops $\hat{\mathfrak{L}}_p^{\Gat,\infty}(\cdot, d)$ give increasingly finer outer approximations of the operator $\Gat$-convex hull 
$\Gat\mhyphen\opco(\cD_p^\infty)$
	of $\cD_p^\infty.$
\end{corollary}

\begin{proof}
By Theorem \ref{th-cd} it suffices to treat the case
$n=\infty$.
To prove the nontrivial inclusion 
in the first equality of \eqref{eq-cdz},
	let $Z\in \bigcap_d\hat{\mathfrak{L}}_p^{\Gat,\infty}(\infty,d)$.
	For every $d$ there is a (truncated) moment sequence
	$Y^{(d)} = (Y^{(d)}_\alpha)\in {\mathfrak{L}}_p^{\Gat,\infty}(\infty;d)$ such that
	\[
	(Y^{(d)}_{x_1},\dots,Y^{(d)}_{x_{\gv} }) = Z.
	\]
	
	 By Lemma~\ref{lema:seqbded} (more precisely, its operator counterpart that is established with the same proof),  for a given word $\alpha,$ 
	the sequence $(Y^{(d)}_\alpha)_{|\alpha|\leq 2d+\deg(p)}$ is bounded. Since 
	 there are countably many such 
	sequences, there exists a moment sequence $(Y_\alpha)$ from 
$\cB(\cH)$
	and a subsequence $(d_k)_k$ of indices  such that
	$(Y^{(d_k)}_\alpha)$ WOT-converges to  $Y_\alpha \in 
\cB(\cH)$
	for each word $\alpha.$
	
	For each $k$ the sequence $Y^{(d_k)}$ satisfies \eqref{eq:112} with $d=d_k$ and hence for all $d\le d_k.$ Hence the limit $\Gat$-moment sequence 
	$Y = (Y_\alpha)_\alpha$ satisfies \eqref{eq:112} for all indices $d$  and thus $Y$ 
       satisfies \eqref{eq:111} and hence belongs to $\mathfrak{L}_p^{\Gat,\infty}(\infty).$ Thus, $Z\in \hat{\mathfrak{L}}_p^{\Gat,\infty}(\infty)$ since
	\[
	Z = (Y_{x_1},\dots,Y_{x_{\gv} }) = \hat{Y}.  \qedhere
	\]
\end{proof}

\begin{remark}
Finally, let us discuss the stopping criteria for the Lasserre-Parrilo lifts in the operator $\Gat$-convex context. 
The Positivstellensatz inspired version, Proposition~\ref{prop:PCP} extends to the operator $\Gat$-convex case mutatis mutandis
with the conclusion
\[
\Gat\mhyphen\opco  (\cD_p^\infty)= \hat{\mathfrak{L}}_p^{\Gat,\infty}\Big(\cdot, \left\lceil \frac N2\right\rceil\Big).
\]

A version of the moment inspired stopping criterion with a conclusion as in Proposition~\ref{prop:MP} requires 
 the addition of an Archimedean hypothesis on the polynomial $p$ and a reinterpretation of the flatness condition.
 Indeed, in this case, the vector space $\mathcal{W}$ with positive definite sesquilinear form $\langle \cdot,\cdot\rangle$
 as in the proof of Lemma~\ref{lem:MP}, will not (necessarily) be complete. It will be, in general, necessary to form the
 closure. Hence the Archimedean assumption is needed to guarantee that the operators $Z_j$ extend to bounded operators
 $T_j$ on the completion of $\mathcal{W}.$  
 
 The natural reformulation of the flatness condition of Subsection~\ref{ssec:moment:stop} in terms of a Schur
 complement extends readily to the operator (infinite-dimensional) setting.   Borrowing from the presentation in 
 \cite{Dritschel} (or see \cite[Chapter XVI]{FF90}),  for  a positive semidefinite  block operator matrix 
 \[
   D=\begin{pmatrix} A & B \\ B^* & C \end{pmatrix},
 \]
  the Schur complement of $A$ 
  is the largest positive semidefinite operator $X$ such that 
 \[
   \begin{pmatrix} A& B \\ B^* & C-X\end{pmatrix} \succeq 0. 
 \]
  Thus, if $0\preceq Y$ and 
 \[
   \begin{pmatrix} A& B \\ B^* & C-Y\end{pmatrix} \succeq 0,
 \]
  then $Y\preceq X.$  Factorizing the positive operator matrix  $D$ leads to 
  a contraction $G$ such that 
  { $B=A^{\nicefrac12} G C^{\nicefrac12}$}
  and then
   { $X=C^{\nicefrac12} (I-G^*G)C^{\nicefrac12}$} 
   is the desired Schur complement.
   
    There is another description of $X$ that is well suited to the present
    application. Namely, again assuming $D\succeq 0$ and letting $\cE$
     denote the space that $C$ acts on and $\cH$ the space that $A$ acts on,
 { \begin{equation}\label{eq:sesqui}
  \langle X g,g\rangle =\inf\left\{ \begin{pmatrix} A &  B \\ B^* & C \end{pmatrix} \, \begin{pmatrix} h\\ g\end{pmatrix},
      \begin{pmatrix} h\\ g\end{pmatrix} \mid h \in \cH\right\},
      \quad g\in\cE.
 \end{equation}
 }
  Consider the positive semidefinite sesquilinear form 
\[
 [ \begin{pmatrix} h\\ g\end{pmatrix}, \begin{pmatrix} h^\prime \\ g^\prime \end{pmatrix} ]
    = \langle D  \begin{pmatrix} h\\ g\end{pmatrix}, \begin{pmatrix} h^\prime \\ g^\prime \end{pmatrix} \rangle .
\]
From the description \eqref{eq:sesqui} it is clear that if $X=0,$ then 
 the equivalence classes of $h\in \cH \subseteq \cH\oplus \cE$ in the quotient are dense in the 
    Hilbert space obtained in the usual way by modding out null vectors and forming the completion. 
  
   Given $n,d,\eta\in\N$ with $d\leq\eta$, 
   a truncated moment sequence $Z=(Z_\al)_{|\al|\leq 2\eta}$ with $Z_\alpha\in M_n$ 
      is $d$-flat 
    if and only if, setting $D=H_\eta(Z)$ and $A=H_{\eta-d}(Z),$ the Schur complement of
    $A$ (in $D$) is $0.$  Thus, we make this Schur complement condition the definition
    of $d$-flat when the $Z_\alpha$ are allowed to be operators on an infinite-dimensional Hilbert space.
  With these hypotheses,  the conclusion of Proposition \ref{prop:MP} 
in the operator $\Gat$-convex setting is
\[ \pushQED{\qed}
\hat{\mathfrak{L}}_p^{\Gat,\infty}(\cdot,\eta)=
\Gat\mhyphen\opco  (\cD_p) =
\Gat\mhyphen\opco  (\cD_p^\infty).
\qedhere\popQED
\]
\end{remark}

\newpage

\appendix

\section{An operator convex set as an nc convex set}
\label{sec:app}
   This appendix  explains how a weak$^*$-closed operator convex set can be viewed as an nc convex set.
     Let $\sR=\sR_\gv$ denote  $\gv$-dimensional \df{row Hilbert space}. Thus $\sR\subseteq \cB(\C^\gv)=M_{\gv}$ is the operator space consisting
     of $\gv\times \gv$ matrices with zero entries except along the first row.  It is common to view $\sR$ as $M_{1,\gv}.$ 
      Likewise, let $\sC\subseteq \cB(\C^\gv)=M_{\gv}$ denote $\gv$-dimensional   \df{column Hilbert space}.   As an 
       operator space, $\sR$ is endowed with the following matrix norm structure it inherits from the 
       space $M_n(\cB(\C^\gv)) = \cB(\C^n\otimes \C^\gv).$   An element $X$ in $M_n(\sR)$ is a 
       $1\times \gv$ block matrix with entries from $M_n,$
   \[
      X=\begin{pmatrix} X_1 & \cdots & X_{\gv} \end{pmatrix}, 
  \]
   and      $\|X\|=\|X\|_r,$ 
   the norm of $X,$ is the norm of $X$ as an operator  $\C^\gv \otimes \C^n\to \C^n.$ In particular,
         $\|X\|\le 1$ if and only if $\sum X_j X_j^* \preceq I_n.$  Similar considerations apply for the column space $\sC$. 
         
   In general, the (standard) dual $E^*$ of an operator space $E$ is the operator space $\Cb(E,\C)$  of completely 
    bounded maps from $E$ to $\C.$ 
    The matrix norm structure on $E^*$ is determined via the {(isometric)} identification
    of $M_n(E^*)$ with $\Cb(E,M_n),$ the space of completely bounded maps from $E$ into $M_n;$
    that is,  for $F=(f_{j,k}) \in M_n(E^*),$
  \[
   \|F\| =\| (f_{j,k})\| = \sup \Big\{ \| \big [ f_{j,k} (Y(a,b)) \big ]_{a,b} \|  \, : \,  Y=(Y(a,b)) \in M_m(E), \, \|Y\|=1,\, m\in\N\Big\}.
  \]
     One finds (see \cite[Section I.3.4]{ER00} or \cite[Proposition 14.9]{Pa})
     that $\sC^* =\sR,$ the operator space dual,  completely isometrically via the
     canonical  mapping $\gamma:\sR\to \sC^*,$ where, for 
 \[
  x= \begin{pmatrix} x_1  & \ldots & x_\gv\end{pmatrix} \in \sR,
 \]
   $\gamma[x]: \sC\to \C$ is given by $\gamma[x](y)   = \sum_j x_j \, y_j.$
 In particular,   for  $X\in M_n(\sR)$ and $Y\in M_m(\sC),$
\[
   1_m\otimes (1_n\otimes \gamma)[X](Y)   =\begin{pmatrix} \gamma[X(j,k)](Y(a,b)) \end{pmatrix}= 
    \left ( \sum_\ell X_\ell(j,k) \, Y_\ell(a,b) \right ) =  \sum X_\ell \otimes Y_\ell,
\]
 where $X(j,k) = \begin{pmatrix} X_1(j,k) & \ldots & X_\gv(j,k) \end{pmatrix}.$
  Further, 
\begin{equation}
\label{e:dual:norm}
 \begin{split}
 \|X\|  & =\|\sum X_\ell \otimes e_\ell^*\|_{M_{n,n\gv}} =\sup\{ \| \sum X_\ell \otimes Y_\ell \| : m\in\N, \, Y\in M_m(\sC), \,  \|Y\|=1\}
  \\ &  =\| 1_n\otimes \gamma[X]\|_{\Cb(\sC,M_n)}.
 \end{split}
\end{equation}
 To obtain the second equality in equation~\eqref{e:dual:norm}, observe that choosing $Y_\ell=e_1 e_\ell^*$ shows the supremum is
  at least $\|X\|.$ For the reverse inequality, note given $Y\in M_m(\sC)$ with $\|Y\|=1,$  that 
\[
 \widehat{X} = \sum_a X_a \otimes e_a^* \otimes I_m, \  \ \  \widehat{Y}=\sum_b  I_n \otimes e_b  \otimes Y_b
\]
 have norm $\|X\|$ and $\|Y\|=1$ respectively, since, for instance $\widehat{Y}^* \widehat{Y} = I_n \otimes \sum Y_j^* Y_j \preceq I \otimes \|Y\|^2.$
 It follows that
\[
 \|X\| =\|X\|\, \|Y\| =\|\widehat{X}\|\, \|\widehat{Y}\| \ge \| \widehat{X} \, \widehat{Y}\| = \|\sum X_\ell \otimes Y_\ell\|.
\]
Thus  the supremum is at most $\|X\|.$  Hence $\sR$ is a dual operator space with distinguished predual $\sC.$

    The convention in \cite{DK}, specialized to  our setting  where  $\bm K \subseteq \mathbb{S}^\gv,$  is to endow
     $K_m\subseteq \mathbb{S}^\gv_m\subseteq M_m(\sR)$ with the relative topology inherited from
      $M_m(\sR),$ where the latter is identified with $\Cb(\sC,M_m)$ given the \df{point-\wstar} topology.
      The point-weak topology is the  topology in which a net $\varphi_\lambda$ converging  to $\varphi$  means $\varphi_\lambda(y)$ converges
       to $\varphi(y)$ \wstar for each $y\in \sC.$     
        When $m\in \N,$ both $\sC$ and $M_m$
      are finite-dimensional, so this topological consideration is trivial.  The case $m=\infty$ requires
      some explanation.   By definition (see for instance \cite{DK,ER00}),  $M_\infty$ (resp. $M_\infty(\sR)$) 
      consists of the infinite matrices $X=(x_{j,k})_{j,k=1}^\infty$  with $x_{j,k}\in \C$
     (resp.~$x_{j,k}\in \sR$) for which there is a uniform bound on the norms of all finite submatrices.  As is readily checked,
      $M_\infty=\cB(\ell^2),$ and we identify $\cB(\ell^2)$  with $\cB(\cHi)$ (recall the convention that $\cHi$ is a separable infinite-dimensional Hilbert space).
        Likewise, $$M_\infty(\sR) = \sR\otimes \cB(\cHi)= \cB(\cHi^{\gv},\cHi).\footnote{Note that an $A\in \sR\otimes \cB(\cHi)$ is identified with 
        	the row operator  $A=\begin{pmatrix} A_1 &  \ldots & A_\gv\end{pmatrix},$ where $A_j\in \cB(\cH).$ By contrast a $B\in \sC\otimes \cB(\cHi)$
        	is identified with a column operator with entries from $\cB(\cHi)$ so that $\sC\otimes \cB(\cHi)=\cB(\cHi,\cHi^{\gv}).$}$$
        The topology on $M_\infty(\sR)$ is then the  \df{point-\wstar} topology on $\Cb(\sC,M_\infty) = \Cb(\sC,\cB(\cHi)).$
        Thus,  a net $\varphi_\lambda$ from $\Cb(\sC,\cB(\cHi))$ converges in this topology
         if and only if  for each $y\in \sC,$ 
         every $(Y_{\lambda,j})_\lambda$ converges \wstar to some $Y_j\in \cB(\cHi)$ (with $\cB(\cHi)$ dual to trace class),
          where
   \[
       \begin{pmatrix} Y_{\lambda,1} & \ldots & Y_{\lambda,\gv} \end{pmatrix}  
          =\begin{pmatrix} \varphi_\lambda(y)[e_1] & \ldots \varphi_\lambda(y)[e_\gv]\end{pmatrix}
          =         \varphi_\lambda(y).
   \] 
    Thus, the {point-\wstar}  topology on  $M_\infty(\sR)=\Cb(\sC,\cB(\cH)),$  is the same as the  \wstar topology on $\cB(\cHi^{\gv},\cHi)$
    defined via trace class duality, which we note  is uniquely determined since $\cB(\cHi)$ is a von Neumann algebra.

    The definition of a compact \df{nc convex} set  with cardinal upper bound $\aleph_0$ from \cite[Definition~2.2.1]{DK} reads as follows.  
    Given an operator space $E,$ a graded set $\bm K =(K_n)_{1\le n\le\infty}$ 
    with $K_n\subseteq M_n(E)$ is nc convex if it is closed under countable direct sums of norm bounded families 
    and isometric compressions. More concretely, given $1\le m,n\le\infty,$ if $X_j\in \bm K$ for $j\in \N,$ then $\oplus X_j \in \bm K;$ and if
     $X\in K_n$ and $V:\cH_m\to \cH_n$ is an isometry, then $V^* XV\in K_m.$  Finally, assuming  $E$ is a dual operator space
      with distinguished predual $E_*,$ the nc convex set $\bm K$ is compact if
      each $K_n$ is compact (in the point-\wstar topology). Here, 
$\Cb(E_*,M_n)$ is endowed with the point-\wstar topology, and
$K_n\subseteq M_n(E)$ is endowed with the relative topology after identifying $M_n(E)$ with $\Cb(E_*,M_n)$.
         Thus, a \wstar closed and (norm) bounded  operator convex set $\bm{K}\subseteq \mathbb{S}^\gv \subseteq M(\sR)$ is a compact  nc convex set over
       the operator space $\sR.$  See also \cite[Example~16.4(3)]{Dbook} for a closely related example.

 \section{Proof of Pascoe's SOT-BW-MUST}
 
 As the title of this section suggests and for the reader's convenience, this section contains a 
 streamlined  proof of Pascoe's SOT-BW-MUST
(\cite{Pascoe}) tailored to the version of that result  presented here, Theorem~\ref{t:sot-bw-must}. 
  Accordingly, suppose $\cH$ is a separable infinite-dimensional  Hilbert space and $(X^{(j)})_j$ is  a norm bounded sequence 
 from $\cB(\cH)^\gv.$  The first objective is to prove that there exists a subsequence $(Y^{(n)})_n$ of $(X^{(j)})_j$
 and a sequence of unitary operators $(U_n)$ such that $(U_n^*)_n$ and $(U_n^* Y^{(n)} U_n)_n$ both SOT converge.
 
 Let $\{e_1,e_2,\dots\}$ denote an orthonormal basis for $\cH.$ Let $\cH_k=\spann\{e_1,\dots,e_k\}.$ 
 Thus $\cH_0=\cup_{k=1}^\infty \cH_k$ is dense in $\cH.$ Recall the evaluations $Y^\alpha$ of  a word $\alpha$ in $\gv$
   noncommuting variables  at a $Y\in B(\cH)^\gv$ 
   described in Subsection~\ref{subsec polys}. 
 For each $k,j\in \N,$ let
\[
 \cH_k(Y)  = \spann \{Y^\alpha  \, e_n:  1\le n \le k, \, 0\le |\alpha|\le k\},
\] 
 where $|\alpha|$ is the length of the word $\alpha.$ For instance,
\[
 \cH_{1}(Y)  =\spann \{e_1, Y_1 e_1, \dots Y_\gv  e_1\}.
\]
By construction, $\cH_k\subseteq \cH_k(Y)$ and 
\[
 Y_\ell   \cH_k(Y)  \subseteq \cH_{k+1}(Y) 
\]
 for all $k\in \N$ and $1\le \ell\le \gv.$ Moreover,
 \begin{equation*}
   \dim \cH_{k-1}(Y) < \dim \cH_{k}(Y) \le d_k :=k\sum_{j=0}^k \gv^j
 \end{equation*}
 and 
 \[
\dim \cH_k(Y)-\dim \cH_{k-1}(Y)
\leq
(k-1)g^k+\sum_{j=0}^k g^j
=d_k-d_{k-1},
\]
independent of $Y.$

 Let  $\cH_Y=\cup_{k=1}^\infty \cH_k(Y).$ 
   From the observations in the previous paragraph,
   we can define, inductively, an isometry $W:\cH_Y\to \cH_0$ satisfying
 \begin{equation}
  \label{e:JEP:0}
  \cH_k \subseteq W \, \cH_k(Y) \subseteq \cH_{d_k},
 \end{equation}
  which then extends to an isometry  $W$ from  $\cH$ onto $\cH.$ 
   Indeed, define $W$ on $\cH_1(Y)$ by $We_1=e_1$ and then
   extend $W$ to an isometry  from $\cH_1(Y)$ to  $\cH_{d_1}=\cH_{\gv+1}$ by choosing any isometry
   from $\cH_1(Y) \ominus \cH_1$ to $\cH_{\gv+1}\ominus \cH_1,$
   which is possible since $\dim \cH_1(Y) -\dim \cH_1 \le  \gv = \dim \cH_{\gv+1} -\dim \cH_1.$
   Now suppose  $W:\cH_k(Y)\to \cH_{d_k}$ has been defined satisfying the inclusions
   in equation~\eqref{e:JEP:0}. If 
   $e_{k+1}$ is in $W\cH_k(Y),$ then extend $W$ to $\cH_{k+1}(Y)$ by choosing any isometry from the
   at most $d_{k+1}-d_k$  dimensional space $\cH_{k+1}(Y) \ominus \cH_k(Y)$
   to the at least  $d_{k+1}-d_k$ dimensional space $\cH_{d_{k+1}} \ominus W\cH_k(Y).$ Otherwise, 
    let $P$ denote the projection onto  $W\cH_k(Y)\subseteq \cH_{d_k}$  and let
    $e$ denote the unit vector in the direction of $e_{k+1}-Pe_{k+1}.$  Thus $e\in [W\cH_k(Y)]^\bot.$
    On the other hand, 
    since $e_{k+1}, \, P e_{k+1} \in \cH_{d_{k+1}}$  it follows that $e$ is in $\cH_{d_{k+1}}$ and
     thus $e\in \cH_{d_{k+1}}\ominus W\cH_k(Y).$ 
    Because $\dim \cH_k(Y) <\dim \cH_{k+1}(Y),$
    there exists a unit vector $f\in \cH_{k+1}(Y) \ominus \cH_k(Y).$ 
     Now choose any isometry from $\cH_{k+1}(Y) \ominus \cH_k(Y)$ 
     into $\cH_{d_{k+1}} \ominus W\cH_{k}(Y)$ that sends $f$ to $e$ to extend $W$
      to an isometry $W:\cH_{k+1}(Y) \to \cH_{d_{k+1}}.$ To complete the induction
    argument, note that both $Wf=e$ and $Pe_{k+1}$ are in $W\cH_{k+1}(Y)$  
    and hence  $e_{k+1} \in W\cH_{k+1}(Y)$ so that the inclusion $\cH_{k+1} \subseteq W \cH_{k+1}(Y)$
     also holds.

  Letting $V$ denote the unitary $W^*,$ from equation~\eqref{e:JEP:0} we  obtain,
\begin{equation}
\label{e:JEP:1}
 \cH_k \subseteq V^* \,  \cH_k(Y) \subseteq \cH_{d_k},
\end{equation}
  for all $k\in \N.$  It follows that 
\begin{equation*}
  V \cH_k \subseteq \cH_k(Y) \subseteq V\cH_{d_k},
\end{equation*}
  and also 
\begin{equation*}
  V^* \cH_k \subseteq  V^* \cH_k(Y) \subseteq \cH_{d_k},
 \end{equation*}
 since $\cH_k \subseteq \cH_k(Y).$  Finally,
\begin{equation}
\label{e:JEP:4}
 V^* Y V \cH_k \subseteq V^* Y \cH_k(Y) \subseteq V^* \cH_{k+1}(Y) \subseteq \cH_{d_{k+1}}.
\end{equation}
  
 Returning to the sequence $(X^{(j)}),$ for each $j$ there exists a unitary $V_j\in \cB(\cH)$ such that
  equations~\eqref{e:JEP:1} through \eqref{e:JEP:4} hold with $X^{(j)}$ and $V_j$
   in place of $Y$ and $V.$  
 Let $S^{(j)} = V_j^* X^{(j)} V_j$ and  let $J_k:\cH_k\to \cH_{d_{k+1}}$ denote the inclusion. 
  The sequence $(S^{(j)}J_k, V_j^* J_k)_j$ is a sequence of $(\gv+1)$-tuples of operators between the finite-dimensional Hilbert
   spaces $\cH_k$ and $\cH_{d_{k+1}}$ and thus has a (norm) convergent  subsequence. By a standard diagonalization argument with respect to the
    parameter $k,$
    there exists a subsequence $(T^{(n)},U_n^*)_n$ of $(S^{(j)}, V_j^*)_j$ such that $(U_n^* h)_n$ and $(T^{(n)}_\ell \, h)$ converges in $\cH$
    for each $h\in \cup_k \cH_k =\cH_0$ and $1\le \ell\le \gv.$   Since $(S^{(j)}, V_j^*)_j$ is a bounded sequence, so is $(T^{(n)}, U_n^*)_n.$
    Thus for each $\ell,$ the sequence $(T^{(n)}_\ell)_n$ is a norm-bounded sequence that is  pointwise
    Cauchy on the dense subset $\cH_0$ of $\cH$,  and similarly for $(U_n^*)_n.$  Hence    $(T^{(n)}_\ell)_n$ converges SOT to some bounded operator $T\in \cB(\cH)$
    and $U_n^*$ converges SOT to some bounded operator $U^*.$ Since $\left ( (U_n^*)^*=U_n\right )_n$ converges WOT to $U,$ it follows
     that $(U_n U_n^*)$ converges WOT to $U U^*$ and also to the identity. Thus $U^*$ is an  isometry.  Finally, $(Y_n= U_n T^{(n)} U_n^*)_n$
      is a subsequence of $(X^{(j)})_j$ and $(U_n^* Y^{(n)} U_n)$ converges SOT to $T,$ establishing the first part of the result.
      
   To prove the last statement of the proposition, suppose $S\subseteq \cB(\cH)^\gv$ is SOT-closed and bounded and closed under isometric
   conjugation.  Since $S$ is assumed bounded and $\cH$ is separable, to prove $S$ is WOT compact, it suffices to prove that $S$ is WOT sequentially compact.
    Given a sequence $(X^{(j)})_j$ from $S$ there exists a subsequence $(Y^{n})_n$ of $(X^{(j)})_j$ and a sequence $(U_n)_n$ of unitary operators 
     on $\cH$ such that $(U_n^*)_n$ converges SOT to an isometry $V$ and $(U_n^* Y^{(n)} U_n)$ converges SOT to some operator $T.$
     Since $S$ is closed under isometric conjugation, each $U_n^* Y^{(n)} U_n$ is in $S$ and therefore, since $S$ is SOT-closed, $T$  and $V^*TV$
     are in $S.$  Moreover,
     $(U_n \left [ U_n^* Y^{(n)}  U_n \right] U_n^*)_n$ converges WOT to $V^* TV\in S.$   Thus $(Y^{(n)})_n$ converges WOT to $V^* TV\in S$
      and the proof is complete.
 \qed
 
 {\CCBB
\section{Categorical duality}\label{app:categorical-duality}

In this section, we establish the full categorical duality associated with Theorem \ref{th:g-dual}; see Theorem \ref{th:gamma-categorical-duality} below. Specifically, we prove that this correspondence {between $\Gat$-convex sets and $\Gat$-operator systems} preserves and reflects structural maps, showing that two $\Gamma$-convex sets $\bm{K}$ and $\bm{L}$ are isomorphic if and only if their dual counterparts $\widehat{\bm{K}}$ and $\widehat{\bm{L}}$ are isomorphic  under the appropriate notion of isomorphism.

\subsection{Morphisms in the \texorpdfstring{$\Gat$}{Gamma}-convex setting}

Throughout this section 
let $\bm{K},\bm{L}\subseteq \mathbb{S}^g$ be compact $\Gamma$-convex sets satisfying the assumptions of Theorem \ref{th:g-dual}(b); that is,  $0\in K_1$ and  
      {$\bm K =\Gat^{-1}(\overline{\matco}(\Gat(\bm K))),$ and likewise for $\bm L$}. 
Set
\[
   \widehat{Y}_{\bm{K}}=\bigoplus_{X\in \bm{K}}X,
   \qquad
   \widehat{Y}_{\bm{L}}=\bigoplus_{Y\in \bm{L}}Y.
\]
Recall the associated $\Gamma$-operator systems
\[
   \widehat{\bm{K}}
   =
   \operatorname{span}\{I,\gamma_1(\widehat{Y}_{\bm{K}}),\ldots,\gamma_r(\widehat{Y}_{\bm{K}})\},
   \qquad
   \widehat{\bm{L}}
   =
   \operatorname{span}\{I,\gamma_1(\widehat{Y}_{\bm{L}}),\ldots,\gamma_r(\widehat{Y}_{\bm{L}})\}.
\]

\begin{definition} 
Let $\bm{K},\bm{L}\subseteq \mathbb{S}^g$ be $\Gamma$-convex sets.  A level-preserving map
$\Phi:\bm{K}\to \bm{L}$
is \df{$\Gamma$-coordinate affine} if there are scalars
$\lambda_{j0},\lambda_{j1},\ldots,\lambda_{jr}$, for  $1\leq j\leq r$, such that
for every $X\in \bm{K}$,
\[
   \gamma_j(\Phi(X))
   =
   \lambda_{j0}I+\sum_{\ell=1}^r \lambda_{j\ell}\gamma_\ell(X),
   \qquad 1\leq j\leq r.
\]
Equivalently, there is an affine linear map $\Lambda:\mathbb{S}^r\to \mathbb{S}^r$
such that
\[
   \Gamma(\Phi(X))=\Lambda(\Gamma(X)),
   \qquad X\in \bm{K}.
\]
\end{definition}

\begin{remark}\rm
Each $\Gamma$-coordinate affine map {$\Phi$}  is 
{matrix $\Gamma$-affine}; i.e., 
\begin{equation}\label{eq:Gammaaffine}
   \Phi(V^*XV)=V^*\Phi(X)V
\end{equation}
whenever $X\in \bm{K}$ and $(X,V)$ is a $\Gamma$-pair.

Indeed, 
\[
 \Gat(\Phi(V^*XV))= \Lambda(\Gamma(V^*XV)) = \Lambda(V^* \Gamma(X)V) = V^* \Lambda(\Gamma(X))V = V^* \Gamma(\Phi(X)) V.
\]
Since $\gamma_i=x_i$ for $1\leq i\leq \gv$, comparison of the first $\gv$
coordinates yields \eqref{eq:Gammaaffine}.
\hfill \qed
 \end{remark}

On the other side, one needs a mild composition-stability condition.

\begin{definition}
Let $\Psi:\cR_A^\Gamma\to \cR_B^\Gamma$
be a $\Gamma$-ucp map.  We say that $\Psi$ is \df{admissible} if for every
$n\in\mathbb N$ and every $\Gamma$-ucp map
$\phi:\cR_B^\Gamma\to M_n,$
the composition
\[
   \phi\circ\Psi:\cR_A^\Gamma\to M_n
\]
is again $\Gamma$-ucp.
\end{definition}

{By item (3) of Remark~\ref{r:221}, 
the map that implements the isomorphism 
between $\mathcal{R}$ and  $\widehat{\widecheck{\cR}}$
in item~\ref{i:g-dual:a} of Theorem \ref{th:g-dual} 
is admissible.}

\begin{example}\rm
A $\Gamma$-ucp map need not be admissible.  Let
$
   \Gamma=(x,y,xy+yx,i(xy-yx)).
$
Let
\[
   S=\begin{pmatrix}1&0\\0&-1\end{pmatrix},
   \qquad
   T=\begin{pmatrix}0&1\\1&0\end{pmatrix}.
\]
Set $B=(S,T)$.  Then
\[
   ST+TS=0,
   \qquad
   i(ST-TS)\neq 0,
\]
and
\[
   \mathcal R_B^\Gamma
   =
   \operatorname{span}\{I,S,T,ST+TS,i(ST-TS)\}
   =
   M_2.
\]

Let $\tau:M_2\to \mathbb C$ denote the normalized trace.  Since
\[
   \tau(S)=\tau(T)=\tau(ST+TS)=\tau(i(ST-TS))=0,
\]
we have
\[
   \tau(\gamma_j(B))
   =
   \gamma_j(\tau(S),\tau(T)),
   \qquad 1\leq j\leq 4.
\]
Thus $\tau:\mathcal R_B^\Gamma\to\mathbb C$ is $\Gamma$-ucp.

Now let
\[
   A_1=S \oplus 0 = \begin{pmatrix}
      1&0&0\\
      0&-1&0\\
      0&0&0
   \end{pmatrix},
   \qquad
   A_2=S \oplus 1 = \begin{pmatrix}
      1&0&0\\
      0&-1&0\\
      0&0&1
   \end{pmatrix}.
\]
Then 
$I,A_1,A_2$ are linearly independent.  Moreover,
\[
   A_1A_2+A_2A_1
   =
   \begin{pmatrix}
      2&0&0\\
      0&2&0\\
      0&0&0
   \end{pmatrix},
   \qquad
   i(A_1A_2-A_2A_1)=0.
\]

Define a unital $*$-homomorphism
\[
   \Theta:C^*(A_1,A_2)\to M_2
\]
by compressing to the first two diagonal coordinates:
\[
   \Theta\!
      \begin{pmatrix}
         d_1&0&0\\
         0&d_2&0\\
         0&0&d_3
      \end{pmatrix}
   =
   \begin{pmatrix}
      d_1&0\\
      0&d_2
   \end{pmatrix}.
\]
Let
\[
   \Psi=\Theta|_{\mathcal R_A^\Gamma}:
   \mathcal R_A^\Gamma\to \mathcal R_B^\Gamma.
\]
Then $\Psi$ is ucp.  
 Moreover,
\[
\begin{split}
   \Psi(A_1)& =S,\\ \Psi(A_2)&=S, \\
   \Psi(A_1A_2+A_2A_1)&=2I
   =
   SS+SS, \\
   \Psi(i(A_1A_2-A_2A_1))& =0
   =
   i(SS-SS).
\end{split}
\]
Hence $\Psi$ is $\Gamma$-ucp.

However, the composition $\tau\circ\Psi:\mathcal R_A^\Gamma\to\mathbb C$
is not $\Gamma$-ucp.  Indeed,
\[
\begin{split}
   (\tau\circ\Psi)(A_1)
   & =
   (\tau\circ\Psi)(A_2)
   =
   0,\\
   (\tau\circ\Psi)(A_1A_2+A_2A_1) &
   =
   \tau(2I)
   =
   2,\\
   \gamma_3\bigl((\tau\circ\Psi)(A_1),(\tau\circ\Psi)(A_2)\bigr)
 &  =
   0\cdot 0+0\cdot 0
   =
   0.
\end{split}
\]
Thus $\tau\circ\Psi$ fails the $\Gamma$-concomitant identity.  Therefore
$\Psi$ is $\Gamma$-ucp but not admissible.
\end{example}

\begin{proposition}\label{prop:morph-sets}
\label{prop:gamma-affine-induces-gamma-ucp} 
If
$\Phi:\bm{K}\to \bm{L}$
is $\Gamma$-coordinate affine, 
then $\Phi$ induces an admissible $\Gamma$-ucp map
$\Psi_\Phi:\widehat{\bm{L}}\to \widehat{\bm{K}}$
given by
\[
   \Psi_\Phi(\gamma_j(\widehat{Y}_{\bm{L}}))
   =
   \gamma_j(\Phi(\widehat{Y}_{\bm{K}})),
   \qquad 1\leq j\leq r.
\]
\end{proposition}

\begin{proof}
Since $\Phi$ is $\Gamma$-coordinate affine, each
$\gamma_j(\Phi(\widehat{Y}_{\bm{K}}))$ lies in $\widehat{\bm{K}}$. To check that $\Psi_\Phi$ is well-defined  suppose
\[
   a_0I+\sum_{j=1}^r a_j\gamma_j(\widehat{Y}_{\bm{L}})=0.
\]
Thus, 
\[
   a_0I+\sum_{j=1}^r a_j\gamma_j(Y)=0,
\]
for all  $Y\in \bm{L}.$
Since $\Phi(X)\in \bm{L}$ for every $X\in \bm{K}$, it follows that, for all $X\in \bm K,$
\[
   a_0I+\sum_{j=1}^r a_j\gamma_j(\Phi(X))=0.
\]
Taking the direct sum over $X\in \bm{K}$ gives
\[
   a_0I+\sum_{j=1}^r a_j\gamma_j(\Phi(\widehat{Y}_{\bm{K}}))=0.
\]
Thus $\Psi_\Phi$ is well-defined.

We next show that $\Psi_\Phi$ is completely positive.  Let
\[
   M=C_0\otimes I+\sum_{j=1}^r C_j\otimes \gamma_j(\widehat{Y}_{\bm{L}})
   \in M_d(\widehat{\bm{L}})
\]
be positive, which means that
\[
   C_0\otimes I+\sum_{j=1}^r C_j\otimes \gamma_j(Y)\succeq 0
\]
for every $Y\in \bm{L}$.  Hence, for every $X\in \bm{K}$,
\[
   C_0\otimes I+\sum_{j=1}^r C_j\otimes \gamma_j(\Phi(X))\succeq 0.
\]
Taking the direct sum over all $X\in \bm{K}$, we obtain
\[
   \Psi_\Phi^{(d)}(M)
   =
   C_0\otimes I+\sum_{j=1}^r C_j\otimes \gamma_j(\Phi(\widehat{Y}_{\bm{K}}))
   \succeq 0.
\]
Thus $\Psi_\Phi$ is ucp.

Next, since $\gamma_i=x_i$ for $1\leq i\leq \gv$,
\[
   \Psi_\Phi(\widehat{Y}_{\bm{L}})
   =
   \bigl(
      \Psi_\Phi(\widehat{Y}_{\bm{L},1}),\ldots,\Psi_\Phi(\widehat{Y}_{\bm{L},g})
   \bigr)
   =
   \Phi(\widehat{Y}_{\bm{K}}).
\]
Therefore,
\[
   \Psi_\Phi(\gamma_j(\widehat{Y}_{\bm{L}}))
   =
   \gamma_j(\Phi(\widehat{Y}_{\bm{K}}))
   =
   \gamma_j(\Psi_\Phi(\widehat{Y}_{\bm{L}})),
\]
so $\Psi_\Phi$ is a $\Gamma$-concomitant.  Hence $\Psi_\Phi$ is
$\Gamma$-ucp.

{\CCBB
It remains to prove that $\Psi_\Phi$ is admissible. Let
$\phi:\widehat{\bm K}\to M_n$ be a $\Gamma$-ucp map and set
\[
   Z=\phi(\widehat Y_{\bm K})
   =
   \bigl(\phi(\widehat Y_{\bm K,1}),\ldots,
          \phi(\widehat Y_{\bm K,\gv})\bigr).
\]
Then
$
   Z\in W^\Gamma(\widehat{\bm K})_n= K_n
$
by the standing assumptions on $\bm K$. Hence $\Phi(Z)$ is defined.

We show that $\phi\circ\Psi_\Phi$ is $\Gamma$-ucp. Since
$\Phi$ is $\Gamma$-coordinate affine, for each $1\leq j\leq r$ there
are scalars $\lambda_{j0},\lambda_{j1},\ldots,\lambda_{jr}$ such that
\[
   \gamma_j(\Phi(X))
   =
   \lambda_{j0}I+\sum_{\ell=1}^r
   \lambda_{j\ell}\gamma_\ell(X),
   \qquad X\in \bm K.
\]
Thus
\[
\begin{aligned}
   (\phi\circ\Psi_\Phi)(\gamma_j(\widehat Y_{\bm L}))
   &=
   \phi\bigl(\gamma_j(\Phi(\widehat Y_{\bm K}))\bigr)  \\
   &=
   \lambda_{j0}I_n+
   \sum_{\ell=1}^r
   \lambda_{j\ell}\phi(\gamma_\ell(\widehat Y_{\bm K})) \\
   &=
   \lambda_{j0}I_n+
   \sum_{\ell=1}^r
   \lambda_{j\ell}\gamma_\ell(Z) \\
   &=
   \gamma_j(\Phi(Z)).
\end{aligned}
\]
On the other hand, since $\gamma_i=x_i$ for $1\leq i\leq \gv$,
\[
   (\phi\circ\Psi_\Phi)(\widehat Y_{\bm L})
   =
   \phi(\Phi(\widehat Y_{\bm K}))
   =
   \Phi(Z).
\]
Therefore, for all $1\leq j \leq r,$
\[
   (\phi\circ\Psi_\Phi)(\gamma_j(\widehat Y_{\bm L}))
   =
   \gamma_j\bigl((\phi\circ\Psi_\Phi)(\widehat Y_{\bm L})\bigr).
\]
Hence $\phi\circ\Psi_\Phi$ is $\Gamma$-ucp. Since $\phi$ was arbitrary,
$\Psi_\Phi$ is admissible.}
\end{proof}

\begin{proposition}\label{prop:morph-opsys}
Let $A\in \mathbb{S}_\mathcal{H}^g$ and $B\in \mathbb{S}_\mathcal{K}^g$.  If
$\Psi:\cR_A^\Gamma\to \cR_B^\Gamma$
is an admissible $\Gamma$-ucp map, then $\Psi$ induces a continuous 
$\Gamma$-coordinate affine map
\[
   \Phi_\Psi:W^\Gamma(B)\to W^\Gamma(A)
\]
defined as follows. For 
\[
   Y=\phi(B)
   =
   \bigl(\phi(B_1),\ldots,\phi(B_{\gv} )\bigr)
   \in W_n^\Gamma(B),
\]
where $\phi:\cR_B^\Gamma\to M_n$ is $\Gamma$-ucp, set
\[
   \Phi_\Psi(Y)
   =
{\phi\circ\Psi}(A)
   =
   \bigl(
      \phi(\Psi(A_1)),\ldots,\phi(\Psi(A_{\gv} ))
   \bigr).
\]
\end{proposition}

\begin{proof}
The map is well-defined.  Indeed, a $\Gamma$-ucp map
$\phi:\cR_B^\Gamma\to M_n$ is determined by the tuple
$Y=\phi(B),$
because
\[
   \phi(\gamma_j(B))=\gamma_j(Y),
   \qquad 1\leq j\leq r.
\] 
Thus the formula for $\Phi_\Psi(Y)$ depends only on $Y$, not on the chosen
$\phi$.
Since $\Psi$ is admissible, the composition $\phi\circ\Psi$ is
$\Gamma$-ucp.  Therefore
$\Phi_\Psi(Y)
   =
   {\phi\circ\Psi}(A)
   \in W_n^\Gamma(A).$

We next show that $\Phi_{\Psi}$ is $\Gamma$-coordinate affine. Since each $\gamma_j(A)$ lies in $\mathcal{R}_A^\Gamma$ and  $\Psi: \mathcal{R}_{A}^{\Gamma} \rightarrow \mathcal{R}_{B}^{\Gamma}$, the image $\Psi(\gamma_j(A))$  lies in $\mathcal{R}_{B}^{\Gamma} = \operatorname{span}\{I, \gamma_1(B), \dots, \gamma_r(B)\}.$ Therefore, there exist scalars $\lambda_{j0}, \lambda_{j1}, \dots, \lambda_{jr}$ such that
\begin{equation}\label{eq:psi_expansion}
\Psi(\gamma_j(A)) = \lambda_{j0}I + \sum_{l=1}^{r} \lambda_{jl}\gamma_l(B), \quad 1 \le j \le r.
\end{equation}
Applying $\phi$ to both sides of \eqref{eq:psi_expansion} yields
\begin{equation}\label{eq:phi_applied}
(\phi \circ \Psi)(\gamma_j(A)) = \lambda_{j0}I_n + \sum_{l=1}^{r} \lambda_{jl}\gamma_l(Y),
\end{equation}
where we used the fact that $\phi(I) = I_n$ and $\phi(\gamma_l(B)) = \gamma_l(\phi(B)) = \gamma_l(Y).$ 

Because $\Psi$ is admissible, the composition $\phi \circ \Psi$ is a $\Gamma$-ucp map and thus
\begin{equation}\label{eq:commute}
\textstyle (\phi \circ \Psi)(\gamma_j(A)) = \gamma_j\big((\phi \circ \Psi)(A)\big) = \gamma_j(\Phi_{\Psi}(Y)).
\end{equation}
Combining \eqref{eq:phi_applied} and \eqref{eq:commute}, we obtain that for all $1 \le j \le r,$
\begin{equation}\label{eq:c-aff}
\gamma_j(\Phi_{\Psi}(Y)) = \lambda_{j0}I_n + \sum_{l=1}^{r} \lambda_{jl}\gamma_l(Y),
\end{equation}
proving that $\Phi_{\Psi}$ is $\Gamma$-coordinate affine.

Continuity of $\Phi_{\Psi}$ now follows directly from \eqref{eq:c-aff}. Recalling that $\gamma_i(X) = X_i$ for $1 \le i \le \gv,$ we have
\[
\Phi_{\Psi}(Y)_i = \lambda_{i0}I_n + \sum_{l=1}^{r} \lambda_{il}\gamma_l(Y), \quad 1 \le i \le \gv.
\]
Because each coordinate component $\Phi_{\Psi}(Y)_i$ is an affine linear function in the entries of $\Gamma(Y)$, the map $\Phi_{\Psi}$ is continuous on each level.
\end{proof}

\begin{remark}\label{rem:above}\rm

The  $\Gamma$-coordinate affine hypothesis in
Proposition~\ref{prop:gamma-affine-induces-gamma-ucp} is necessary.  Indeed, merely
being a matrix $\Gat$-affine map is too weak. For example, let
$
   \Gamma=(x,x^2)
$
and let
\[
   \bm{K}=\bm{L}=\{X\in\bbS^1 \mid \|X\|\leq 1\}.
\]
The map
\[
   \Phi(X)=X^3
\]
is matrix $\Gamma$-affine, since $\Gamma$-pairs for
$\Gamma=(x,x^2)$ are precisely compressions to reducing subspaces for
$X$.  However, $\Phi$ does not induce a map
\[
   \widehat{\bm{L}}=\spann\{I,Y,Y^2\}
   \to
 \widehat{\bm{K}}=\spann\{I,Y,Y^2\}.
\]
Indeed, such a map would have to send $Y$ to $Y^3$, but $Y^3$ does not
belong to $\operatorname{span}\{I,Y,Y^2\}$.    
Hence preservation of $\Gamma$-pairs alone does not necessarily give a
well-defined dual map on the associated $\Gamma$-operator systems.\hfill\qed
\end{remark}

An exception to the point made in Remark \ref{rem:above} is matrix convexity as matrix affine maps are clearly $\Gamma$-coordinate affine (with $\Gamma(x)=x$).
Similarly, every ucp map is automatically admissible, since ucp maps
compose. These two observations give rise to the following corollary.

\begin{corollary}
\begin{enumerate}[\rm(1)]
    \item Let $\bm{K}$ and $\bm{L}$ be compact matrix convex sets with $0\in K_1,L_1$. Any matrix affine map $\Phi: \bm{K} \rightarrow \bm{L}$ induces a ucp map $\Psi: \widehat{\bm{L}} \rightarrow \widehat{\bm{K}}$ between their corresponding operator systems.
    \item Let $\cR_A$ and $\cR_B$ be operator systems generated by tuples $A$ and $B,$ respectively. Any ucp  map $\Psi: \cR_A \rightarrow \cR_B$ induces a continuous matrix affine map $\Phi: W(B) \rightarrow W(A)$ between the corresponding matrix ranges.
\end{enumerate}
\end{corollary}

 \subsection{Functorial duality}

 The map $F$ sending a compact $\Gamma$-convex set $\bm{K}$ to the $\Gamma$-operator system $\widehat{\bm{K}}$ and a $\Gamma$-coordinate affine map to the corresponding admissible $\Gamma$-ucp map (as described 
in Propositions \ref{prop:morph-sets} and 
\ref{prop:morph-opsys}, respectively) is a {contravariant} functor. Indeed, it is clear that $F$ preserves identities and the next lemma shows that it reverses composition.

\begin{lemma}\label{lem:preceding}
For any pair of $\Gamma$-coordinate affine maps $\Phi_1: \bm{K} \to \bm{L}$ and $\Phi_2: \bm{L} \to \bm{M}$  
\begin{equation}\label{eq:compose}
    F(\Phi_2 \circ \Phi_1) = F(\Phi_1) \circ F(\Phi_2),
\end{equation}
which translates under Proposition \ref{prop:morph-sets} to  $\Psi_{\Phi_2 \circ \Phi_1} = \Psi_{\Phi_1} \circ \Psi_{\Phi_2}.$
\end{lemma}

\begin{proof}
Since $\Psi_{\Phi_2 \circ \Phi_1}$ and $\Psi_{\Phi_1} \circ \Psi_{\Phi_2}$ are  unital linear maps defined on the operator system $\widehat{{\bm{M}}},$ they are uniquely and completely determined by their values on the generators $\gamma_j(\widehat{Y}_{\bm{M}})$ for $1 \le j \le r.$

The left-hand side of \eqref{eq:compose} evaluated on an arbitrary generator yields
\begin{equation}\label{eq:LHScompose}
   \text{LHS} = F(\Phi_2 \circ \Phi_1)(\gamma_j(\widehat{Y}_{\bm{M}})) = \Psi_{\Phi_2 \circ \Phi_1}(\gamma_j(\widehat{Y}_{\bm{M}})) = \gamma_j\bigl((\Phi_2 \circ \Phi_1)(\widehat{Y}_{\bm{K}})\bigr)  =\gamma_j\Big( \bigoplus_{X \in {\bm{K}}} \Phi_2(\Phi_1(X)) \Big),
\end{equation}
since $\widehat{Y}_{\bm{K}} = \bigoplus_{X \in {\bm{K}}} X.$
On the other hand, the right-hand side of \eqref{eq:compose} yields
\begin{equation}\label{eq:RHScompose}
\begin{split}
   \text{RHS} & = \bigl(F(\Phi_1) \circ F(\Phi_2)\bigr)(\gamma_j(\widehat{Y}_{\bm{M}})) = (\Psi_{\Phi_1} \circ \Psi_{\Phi_2})(\gamma_j(\widehat{Y}_{\bm{M}})) = \Psi_{\Phi_1}\bigl( \Psi_{\Phi_2}(\gamma_j(\widehat{Y}_{\bm{M}})) \bigr)\\ & = \Psi_{\Phi_1}\bigl( \gamma_j(\Phi_2(\widehat{Y}_{\bm{L}})) \bigr).
\end{split}
\end{equation}
Because $\Phi_2$ is a $\Gamma$-coordinate affine map, there exist fixed scalar coefficients $\lambda_{j0}, \lambda_{j1}, \dots, \lambda_{jr}$ such that for every $Y \in {\bm{L}},$
\begin{equation}\label{eq:coord-aff}
    \gamma_j(\Phi_2(Y)) = \lambda_{j0}I + \sum_{l=1}^r \lambda_{jl}\gamma_l(Y)
\end{equation}
and since $\widehat{Y}_{\bm{L}} = \bigoplus_{Y \in {\bm{L}}} Y,$ we have
\begin{equation*}
    \gamma_j(\Phi_2(\widehat{Y}_{\bm{L}})) = \lambda_{j0}I + \sum_{l=1}^r \lambda_{jl}\gamma_l(\widehat{Y}_{\bm{L}}).
\end{equation*}
Plugging this into the expression \eqref{eq:RHScompose} for the right-hand side gives
\begin{equation*}
    \text{RHS} = \Psi_{\Phi_1}\Big( \lambda_{j0}I + \sum_{l=1}^r \lambda_{jl}\gamma_l(\widehat{Y}_{\bm{L}}) \Big) = \lambda_{j0}I + \sum_{l=1}^r \lambda_{jl} \Psi_{\Phi_1}(\gamma_l(\widehat{Y}_{\bm{L}})) = \lambda_{j0}I + \sum_{l=1}^r \lambda_{jl} \gamma_l(\Phi_1(\widehat{Y}_{\bm{K}})).
\end{equation*}

Continuing the computation \eqref{eq:LHScompose} of the left-hand side, note that \eqref{eq:coord-aff} implies
\begin{equation*}
    \gamma_j\bigl(\Phi_2(\Phi_1(X))\bigr) = \lambda_{j0}I + \sum_{l=1}^r \lambda_{jl}\gamma_l(\Phi_1(X))
\end{equation*}
since for every $X \in {\bm{K}}$, we have  $\Phi_1(X) \in {\bm{L}}$. Hence,
\begin{align*}
\text{LHS} & =  \gamma_j\Big( \bigoplus_{X \in {\bm{K}}} \Phi_2(\Phi_1(X)) \Big) =   \bigoplus_{X \in {\bm{K}}} \gamma_j\bigl(\Phi_2(\Phi_1(X))\bigr)\\ & =  \lambda_{j0}I + \sum_{l=1}^r \lambda_{jl} \Big( \bigoplus_{X \in {\bm{K}}} \gamma_l(\Phi_1(X)) \Big) = \lambda_{j0}I + \sum_{l=1}^r \lambda_{jl} \gamma_l(\Phi_1(\widehat{Y}_{\bm{K}})).
\end{align*}

Therefore, $\text{LHS} = \text{RHS}$ on all generators $\gamma_j(\widehat{Y}_{\bm{M}})$ and $F$ indeed reverses composition.
\end{proof}

We  summarize the discussion of this section as follows:
 
\begin{theorem}\label{th:gamma-categorical-duality}
Let $\Gat\mathsf{Conv}_\Gamma$ denote the category whose objects are compact
$\Gamma$-convex sets $\bm K\subseteq \mathbb S^g$ satisfying the hypotheses
of Theorem~\ref{th:g-dual}(b), and whose morphisms are 
$\Gamma$-coordinate affine maps.  Let $\mathsf{OS}_\Gamma$
denote the category of the corresponding  $\Gamma$-operator systems,
with admissible $\Gamma$-ucp maps as morphisms.

The assignments
\[
   \bm K\mapsto \widehat{\bm K},
   \qquad
   \Phi\mapsto \Psi_\Phi
\]
and
\[
   \mathcal R_A^\Gamma\mapsto W^\Gamma(A),
   \qquad
   \Psi\mapsto \Phi_\Psi
\]
define contravariant functors
\[
   \Gat\mathsf{Conv}_\Gamma
   \longleftrightarrow
   \mathsf{OS}_\Gamma.
\]
These functors are mutually inverse up to the natural identifications
\[
   W^\Gamma(\widehat{\bm K})=\bm K
\quad
\text{and}
\quad
   \widehat{\,W^\Gamma(A)\,}\cong \mathcal R_A^\Gamma
\]
from Theorem~\ref{th:g-dual}.  Hence
$\Gat\mathsf{Conv}_\Gamma$ and $\mathsf{OS}_\Gamma$ are
contravariantly equivalent categories.

In particular, two such $\Gamma$-convex sets $\bm K$ and $\bm L$ are
isomorphic by a $\Gamma$-coordinate affine bijection if and only if their
dual $\Gamma$-operator systems $\widehat{\bm K}$ and
$\widehat{\bm L}$ are isomorphic by admissible $\Gamma$-ucp maps.
\end{theorem}

\begin{proof}
Proposition~\ref{prop:morph-sets} shows that every $\Gamma$-coordinate affine
map $\Phi:\bm K\to\bm L$ induces an admissible $\Gamma$-ucp map
\[
   \Psi_\Phi:\widehat{\bm L}\to\widehat{\bm K}.
\]
Lemma \ref{lem:preceding} shows that this assignment reverses composition, and it
clearly sends identity maps to identity maps.  Thus
$\bm K\mapsto \widehat{\bm K}$ is a contravariant functor.

Conversely, Proposition~\ref{prop:morph-opsys} shows that every admissible
$\Gamma$-ucp map
\[
   \Psi:\mathcal R_A^\Gamma\to\mathcal R_B^\Gamma
\]
induces a continuous $\Gamma$-coordinate affine map
\[
   \Phi_\Psi:W^\Gamma(B)\to W^\Gamma(A),
\]
again contravariantly.  The object-level identifications are exactly those of
Theorem~\ref{th:g-dual}.  Under these identifications, the formulas defining
$\Psi_\Phi$ and $\Phi_\Psi$ show that applying the two constructions in
succession recovers the original morphism.  Hence the two contravariant functors
are mutually inverse up to natural isomorphism.
\end{proof}

}

  \end{document}